\documentclass[11pt,a4paper]{article}
\usepackage[utf8]{inputenc}
\usepackage{amsmath}
\usepackage{pifont}
\usepackage{bm}
\usepackage[american]{babel}
\usepackage[all]{xy}
\usepackage{enumitem}
\usepackage{anysize}
\usepackage[dvipsnames]{xcolor}
\usepackage{graphicx}
\usepackage{stmaryrd}
\usepackage{amsfonts,amstext,amssymb,amsmath,amsthm,pspicture,bigints}

\usepackage{fullpage}
\usepackage{moreverb}
\usepackage{pifont}
\usepackage{pst-all}
\usepackage[left=2cm,right=2cm,top=2cm,bottom=2cm]{geometry}
\usepackage{esint}
\usepackage{fourier-orns}
\usepackage{mathrsfs}
\usepackage{array}
\newcommand{\T}{\mathbb{T}}
\usepackage{hyperref}
\newcommand{\ii }{{\rm i} }
\usepackage{authblk}
\newcolumntype{C}[1]{>{\centering\arraybackslash}b{#1}}
\newcolumntype{R}[1]{>{\raggedleft\arraybackslash}b{#1}}
\newcolumntype{L}[1]{>{\raggedright\arraybackslash}b{#1}}
\newcolumntype{M}[1]{>{\centering}m{#1}}
\usepackage{bbm}
\numberwithin{equation}{section}
\usepackage{tikz}

\newcommand{\R}{\mathbb{R}}

\numberwithin{equation}{section}
\newtheorem{theorem}{Theorem}[section]
\newtheorem{proposition}[theorem]{Proposition}
\newtheorem{lemma}[theorem]{Lemma}
\newtheorem{remark}[theorem]{Remark}

\newtheorem{coro}[theorem]{Corollary}

\usepackage[textmath,displaymath,floats,graphics,delayed]{preview}
\title{Desingularization of  time-periodic vortex motion\\  in bounded domains via KAM tools}
\author{Zineb Hassainia \qquad Taoufik Hmidi \qquad Emeric Roulley}

\date{}
\begin{document}
	\maketitle
	\begin{abstract}
	We examine the Euler equations within a simply-connected bounded domain. The dynamics of a single point vortex are  governed by a Hamiltonian system, with most of its energy levels corresponding to time-periodic motion. We show that for the single point vortex, under certain non-degeneracy conditions, it is possible to desingularize most of these trajectories into time-periodic concentrated vortex patches. We provide concrete examples of these non-degeneracy conditions, which are satisfied by a broad class of domains, including convex ones. 
	The proof uses Nash-Moser scheme and KAM techniques, in the spirit of the recent work on the leapfrogging motion \cite{HHM23}, combined with complex geometry tools. 
	Additionally, we employ a vortex duplication mechanism to generate synchronized time-periodic motion of multiple vortices. This approach can be, for instance,  applied to desingularize the motion of two symmetric dipoles (with  four vortices) in a disc or a rectangle. To our knowledge, this is the first result showing the existence of non-rigid  time-periodic motion for Euler equations in generic simply-connected bounded domain. This answers an open problem that has been pointed, for example, in \cite{BS18}.
		\end{abstract}

	\tableofcontents

	\section{Introduction}

In this study, we investigate the vortex dynamics governed by the 2D incompressible Euler equations in a  simply-connected domain $\mathbf{D}$. These equations are formulated in the velocity-vorticity form as follows:
\begin{equation}\label{Euler-Eq}
	\partial_{t}\boldsymbol{\omega} + u \cdot \nabla \boldsymbol{\omega} = 0,\qquad u = \nabla^{\perp}\Delta_{\mathbf{D}}^{-1}\omega,\qquad \nabla^{\perp} = (-\partial_{y},\partial_{x}),
\end{equation}
where $\Delta_{\mathbf{D}}$ is the standard  Laplacian supplemented with Dirichlet boundary condition.
This system has been studied for a long time and is still the subject of intensive activities touching multiple facets. Here, we don't in any case  pretend to make a complete review of the main significant results, which is an ambitious program,  but we shall simply focus on two important subjects that fit with the scope of the paper. The first one   is related to the existence and uniqueness of global solutions which is proved at the level of integrable  and bounded vortices for smooth domains by Yudovich in his seminal work \cite{Y63}. However, the problem turns  out to be more tricky for rough domains, for instance when the boundary contains corners or fractal sets. The  existence part for very general domains is done by G\'erard-Varet and Lacave \cite{GVL13}. As to the uniqueness part, only partial results have been performed in the last decade, under some  angle constraints or positive vorticity condition, see for instance  \cite{HZ21, HZ22, LZ19}. We also refer the reader to the recent work of Agrawal and Nahmod \cite{Agrawal} where the uniqueness is obtained under some constraints on  $C^{1,\alpha}$ domains.  Another important topic that will be explored in part in this paper is related to the emergence of  coherent structures.
One particularly significant facet of this field revolves around relative equilibria, which represent equilibrium states in a rigid body frame.
 They stand   as one of the main  central field in  the study of vortex motion, garnering considerable theoretical and experimental examination.
  For a comprehensive exploration of bifurcation techniques used to construct   rigid time-periodic solutions  stemming  from various  stationary radial states, we refer the  reader to an exhaustive list that outlines numerous pertinent questions and findings in this domain \cite{B82,CCG2,CCG4,HHHM16,HFMV,GHM,GHS,gomez2019symmetry,HMW20,HM3,HM2,HMV13}. Additional steady solutions have been discovered in \cite{CJS24,EJ20}.   We observe that all these results are concerned with  rigid motion where the shape of the vorticity is  conserved without any deformation over the time. Very recently, new structures on  quasi-periodic vortex patches for Euler equations have been constructed via KAM methods near ellipses \cite{BHM22}, Rankine vortices \cite{HR22}, and annuli \cite{HHR23}. Similar studies have been conducted over the past few years for more active scalar equations such as  \cite{GSIP23, HHM21, HR21, R22}.  More investigations can be found in \cite{CL23,CF13,EPT22,FMM23,FM24,FM24-1}.
On the other hand, as we shall see later, the point vortex system offers multiple configurations of non-rigid motion where the dynamics can be tracked explicitly as for the leapfrogging motion.  A classic example of this is provided by Love \cite{Love}, who studied two symmetric dipoles. Recently, Hassainia, Hmidi and Masmoudi \cite{HHM23} have proven the desingularization of this system, using KAM tools.

The primary objective of this paper is to construct non-rigid, time-periodic solutions for the Euler equations in simply connected bounded domains. This will be achieved by desingularizing the motion of point vortices using KAM approach. It is important to note that in bounded domains, the point vortex system, where the vorticity is concentrated at a finite number of points $z_1,...,z_N$ (with $N \in \mathbb{N}^*$) with circulations $\Gamma_1, ..,\Gamma_N$,  follows  the ODE  
\begin{equation}\label{pnt vort syst}
	\forall k\in\llbracket1,N\rrbracket,\quad\frac{d\xi_{k}(t)}{dt}=\sum_{\ell=1\atop\ell\neq k}^{N}\frac{\ii\Gamma_{\ell}}{\pi}\partial_{\overline{\xi}}G_{\mathbf{D}}\big(\xi_{\ell}(t),\xi_{k}(t)\big)+\frac{\ii\Gamma_{k}}{2\pi}\partial_{\overline{\xi}}\mathcal{R}_{\mathbf{D}}\big(\xi_{k}(t)\big).
\end{equation}
Here,  $G_{\mathbf{D}}$ and $\mathcal{R}_{\mathbf{D}}$ are the Green and Robin functions of the domain $\mathbf{D}$ whose  definitions are provided in \mbox{Section \ref{sec Green Robin}.}  Notice that in \cite{H58}, Helmholtz  introduced   this asymptotic system in  the flat  case $\mathbf{D} = \mathbb{R}^2$ with $G_{\mathbb{R}^2}(z, w) = \log |z - w|$ and $\mathcal{R}_{\mathbb{R}^2} \equiv 0$, to describe  the behavior of mutual interactions of several concentrated vortices.  Later on, Kirchhoff and Routh  \cite{K76, R81} recast this system in the  Hamiltonian form, with the Hamiltonian
$$H(\xi_1,..,\xi_N) = \sum_{1 \leqslant k < \ell \leqslant N} \tfrac{\Gamma_{k} \Gamma_{\ell}}{2\pi} G_{\mathbf{D}}(\xi_k, \xi_{\ell}) + \sum_{k=1}^{N} \tfrac{\Gamma_{k}^2}{4\pi} \mathcal{R}_{\mathbf{D}}(\xi_k).$$
%
%
%
The point vortex system provides a tractable model describing in suitable regimes the  fluid motion. Despite its simplicity, it captures the main  features of vortex interactions and serves as a foundation for understanding more intricate fluid behaviors. A huge literature has been performed over the past century addressing various aspects of the point vortex system. While it is impossible to cover all the significant contributions in this area, we will concentrate on the  most relevant topics related to this paper. Below, we will highlight some key results obtained for the flat case $\mathbf{D}=\R^2,$  which can be appropriately adapted for more general domains.
 The Cauchy-Lipschitz Theorem ensures that the trajectories of the vortices are well-defined and smooth as long as any two  point vortices remain distant from each other. Collapse in finite time may occur with signed circulations as documented in \cite{A79, A10, G77, GP22, H08, KS18, N75}, although this phenomenon is rare \cite{D22, MP84}. On the other hand, this Hamiltonian system enjoys three important independent conservation laws, which are in involution: total circulation, linear impulse and angular impulse. This implies that the point vortex system is completely integrable for $N\leqslant3$ and starts to exhibit chaotic motion from $N=4,$ for instance \mbox{see \cite{BBPV94}.} Within the complex structure of  point vortex system, certain long-lived or coherent structures emerge. A specific class is given by configurations keeping their geometric form unchanged throughout their evolution. 
 An illustration of this is given by Thomson polygons where the point vortices are located on the vertices of a regular polygon with the same circulation. This configuration rotates uniformly about its center. A general review on this topic can be found in \cite{A03}. One can desingularize such configurations into steady vortex patch motion through various methods such as variational techniques \cite{CLZ21,CQZZ21,CWWZ21,SV10,T85} or gluing method \cite{DDMW20}. The use of the contour dynamics approach, more efficient to track the dynamics, was first developped by Hmidi and Mateu for symmetric vortex pairs \cite{HM17}. The asymmetric case was solved in \cite{HH20} and the global bifurcation has been addressed in \cite{GH23}. For symmetric configurations involving more point vortices, we refer the reader to Garc\'ia's works on K\'arm\'an vortex street \cite{G19} and Thomson polygons \cite{G20}. Finally, the desingularization of  general configurations satisfying a natural non-degeneracy condition has been exposed in \cite{HW22}.\\
 
%

 In this work, we focus on the motion of a single point vortex within a bounded simply-connected domain. In this special case, only the second term in the right-hand side of \eqref{pnt vort syst} persists and the equation takes the   Hamiltonian  form
			\begin{align}\label{Hamilt-0}\dot{\xi}(t)=\nabla^{\perp}H_{\mathbf{D}}\big(\xi(t)\big),\qquad\nabla^{\perp}\triangleq\begin{pmatrix}
				-\partial_{y}\\
				\partial_{x}
			\end{pmatrix},\qquad H_{\mathbf{D}}\triangleq\tfrac{\Gamma}{4\pi}\mathcal{R}_{\mathbf{D}}.
			\end{align}
This 2D system will be explored with more details in  Section \ref{sec pt vortex motion}. In particular, we infer from \eqref{Range} that the set of energy levels is an unbounded connected interval, that is, 
\begin{equation}\label{Range0}
	H_{\mathbf{D}}(\mathbf{D})\triangleq[\lambda_{\textnormal{\tiny{min}}},\infty).
	\end{equation}	
 By virtue of  the Hamiltonian structure, all the particle trajectories are closed curves defined by  the energy levels 
 $$ 
 \mathcal{E}_{\lambda}=\big\{z\in \mathbb{C}, \quad H_{\mathbf{D}}(z)=\lambda\big\},\qquad\lambda\geqslant \lambda_{\textnormal{\tiny{min}}}.
 $$
 The orbit for $\lambda= \lambda_{\textnormal{\tiny{min}}}$ is a single point given by the lowest energy level of the  critical points to the Hamiltonian.
 It is important to mention that the regular values of the Hamiltonian give rise to time-periodic trajectories and, in view of Sard's theorem, it is the generic situation. However, the study of critical points of the Robin function is a complex and delicate subject and the literature contains only a few known results on this topic. For example, when the domain is  convex and bounded, there is only one critical point and the Hamiltonian is strictly convex. This implies that all the orbits are periodic and enclose convex subdomains, see \cite{Cafar,Gust2}. However, except in radial domains, when the trajectory is closed, the particle does not generally exhibit rigid body motion. In fact, it remains far from equilibrium states (critical points) and its trajectory is significantly influenced and deformed by the geometry of the domain. More periodic configurarions with multiple vortices near the boundary were discovered by Bartsch and Sacchet in \cite{BS18}.
 Desingularizing these trajectories using classical solutions to the Euler equations that replicate similar time-periodic dynamics is particularly challenging, especially due to the time-space resonance, which can disrupt the formation of confined structures. Our primary objective is to achieve this construction using KAM theory, which requires that the energy levels belong to a suitable Cantor set, carefully constructed through an extraction procedure. This approach builds upon the recent work on the leapfrogging phenomenon established in \cite{HHM23}.\\
  To state our main result,  we will introduce some necessary objects  and  assumptions. Let $\lambda_{\textnormal{\tiny{min}}}<\lambda_*<\lambda^*$ such that  for any $\lambda\in[\lambda_*,\lambda^*]$ the orbit  $\mathcal{E}_{\lambda}$ is periodic with minimal  period ${\mathtt{T}(\lambda)}$ and parametrized by ${t\in\mathbb{R}\mapsto \xi_\lambda(t)}$. This assumption  holds provided that the interval  $[\lambda_*,\lambda^*]$ does not contain any critical value for the Hamiltonian. 
  According to Sard's theorem, the set of critical values has zero Lebesgue measure. We believe that this set, which is trivially compact, is actually discrete--likely even finite--due to the specific structure of the Hamiltonian. However, we have not yet been able to produce  any  proof of this formal conjecture.
  Remark that the map $\lambda\in[\lambda_*,\lambda^*]\mapsto {\mathtt{T}(\lambda)}$ is real analytic  regardless of   the boundary regularity of  the domain $\mathbf{D}$.   Next, for each $\lambda\in[\lambda_*,\lambda^*],$ we introduce  the $\mathtt{T}(\lambda)-$periodic matrix:
\begin{equation*}
			\mathbb{A}_{\lambda}(t)=\begin{pmatrix}
				\mathtt{u}_{\lambda}(t) & \mathtt{v}_{\lambda}(t)\\
				\overline{\mathtt{v}_{\lambda}(t)} & \overline{\mathtt{u}_{\lambda}(t)}
			\end{pmatrix},\qquad\mathtt{u}_{\lambda}(t)=-\tfrac{\ii }{2}e^{2\mathcal{R}_{\mathbf{D}}(\xi_{\lambda}(t))},\qquad\mathtt{v}_{\lambda}(t)=\tfrac{\ii }{4}\big[\partial_{z}\mathcal{R}_{\mathbf{D}}\big(\xi_{\lambda}(t)\big)\big]^2.
		\end{equation*}
		 We associate to this matrix  its   fundamental matrix $\mathscr{M}_{\lambda}$, which  solves the ODE
		\begin{align*}
		\partial_{t}\mathscr{M}_{\lambda}(t)=\mathbb{A}_{\lambda}(t)\mathscr{M}_{\lambda}(t),\qquad \mathscr{M}_{\lambda}(0)=\textnormal{Id}.
		\end{align*}
	 Then the monodromy matrix is defined  by $\mathscr{M}_{\lambda}\big(\mathtt{T}(\lambda)\big)$ and its spectrum is denoted by $\hbox{sp}\big(\mathscr{M}_\lambda\big(\mathtt{T}(\lambda)\big)\big).$ We emphasize that due to the zero trace of $\mathbb{A}_{\lambda}(t)$, the monodromy matrix belongs to the special linear group $\textnormal{SL}(2,\mathbb{C})$. Our first main result reads as follows.
\begin{theorem}\label{main-theorem}
		Let $\mathbf{D}$ be a simply connected bounded domain and $\lambda_{\textnormal{\tiny{min}}}<\lambda_*<\lambda^*$. Assume that:
		\begin{enumerate}
						\item Non-degeneracy of the period: 
			\begin{equation*}
				\min_{\lambda\in[\lambda_*,\lambda^*]}\left|\mathtt{T}^\prime(\lambda)\right|>0.
			\end{equation*}
			\item Spectral assumption: 
			\begin{equation*}
				\forall\lambda\in[\lambda_*,\lambda^*],\quad {1 \notin \textnormal{sp}\big(\mathscr{M}_\lambda\big(\mathtt{T}(\lambda)\big)\big)}.
			\end{equation*}
		\end{enumerate}
		Then,  $\exists \,\varepsilon_0>0$ such that $\forall \,\varepsilon\in(0,\varepsilon_0),$ there exists a Cantor like set ${\mathscr{C}_{\varepsilon}}\subset[\lambda_*,\lambda^*],$ with
		$$\lim_{\varepsilon\to0}|\mathscr{C}_{\varepsilon}|=\lambda^*-\lambda_*,$$
		and for any ${\lambda\in\mathscr{C}_{\varepsilon}},$ there exists a solution to Euler equation taking the form
			 $$
		\forall t\in\mathbb{R},\quad \omega(t)=\frac{1}{\varepsilon^2}{\bf{1}}_{D_t^\varepsilon}, 	\qquad D_t^\varepsilon={\xi_\lambda(t)}+\varepsilon O_t^\varepsilon, 			 $$
			 with
			 $$
			\forall t\in\mathbb{R},\quad D_{t+{\mathtt{T}(\lambda)}}^\varepsilon=D_t^\varepsilon,\qquad {\xi_\lambda}\big(t+{\mathtt{T}(\lambda)}\big)={\xi_\lambda}(t).
	 $$

	\end{theorem}
	
	\begin{remark}
	\begin{enumerate}
	\item The previous theorem applies to generic simply-connected bounded domains, which are not necessarily a perturbation of explicit domains, such as discs, ellipses or rectangles. 
	\item In contrast to the leapfrogging vortex motion {\rm\cite{HHM23}}, the period $\mathtt{T}(\lambda)$ is not explicit. This explains the need of adding the first assumption in the previous theorem.  We believe that this latter is universally valid, as the energy levels of critical points of the period function are expected to constitute a countable set of isolated points. Indeed, it would be surprising to find domains for which the period remains constant on a non-trivial interval of energy levels. 
	\item As for the spectral assumption, verifying it in its original form is difficult due to the intricate nature of the trajectories of periodic orbits, but we believe it holds for the majority of energy levels. We mention that the spectral assumption corresponds to the classical one appearing in finite dimensional dynamical systems.  
	\end{enumerate}
	\end{remark}
	Next, we will present a practical implication of Theorem \ref{main-theorem}, connecting its assumptions to the geometry of the domain  $\mathbf{D}$ and applying it to relevant standard examples. This leads us to the following corollary.
 \begin{coro}\label{cor-special} Assume that the bounded domain $\mathbf{D}$ is convex.
Then, the conclusion of Theorem $\ref{main-theorem}$  holds true for any $\lambda_{\textnormal{min}}<\lambda_*<\lambda^*$ provided that the conformal mapping  $F:\mathbb{D}\to \mathbf{D}$ with  $F(0)=\xi_0$ satisfies
		\begin{equation*}
				{\left| \tfrac{F^{(3)}(0)}{F^\prime(0)}\right|\not\in\Big\{2\sqrt{1-\tfrac{1}{n^2}},\,\, n\in\mathbb{N}^*\cup\{\infty\}\Big\}}.
			\end{equation*}
			Here, $\mathbb{D}$ stands for  the open unit disc of the complex plane and ${\xi_0}$ is the unique critical point of Robin function.
\end{coro}
\begin{remark}
\begin{enumerate}
\item The assumption in the previous corollary is an open condition. Therefore, for a given domain $\mathbf{D}$ satisfying the assumptions of Corollary \ref{cor-special} one might expect that the same conclusion holds for all convex domains $\mathbf{D}_\varepsilon$ which are $\varepsilon$-perturbation of $\mathbf{D}$. This is a consequence of the continuity of the Riemann mapping with respect to the domain, known as Carath\'eodory's kernel Theorem. 
\item More refined version, where the domain is not necessarily convex but the Robin function has only one critical point, is given by Corollary {\rm \ref{convex-dom}}. 
\end{enumerate}
\end{remark}
Corollary \ref{cor-special} applies to   domains  such as  rectangles and ellipses provided that the aspect ratio avoids a specific discrete set. This issue will be thoroughly examined in Section \ref{Sec-admissible}, as outlined in Propositions \ref{prop ellipse} and \ref{prop rectangle}. Below, we provide a formal statement.
\begin{proposition} The following properties hold true.
	\begin{enumerate}
		\item Let $\mathbf{D}$ be an ellipse with  semi-axes $0<b\leqslant a$. Then there exists a countable set $\mathcal{G}_{E}\subset(1,\infty)$ such that Corollary $\ref{cor-special}$ applies if and only if $\frac{a}{b}\notin\mathcal{G}_{E}.$
		\item Let $\mathbf{D}$ be a rectangle with sides $0<l\leqslant L$. Then there exists a countable set $\mathcal{G}_{R}\subset(0,1)$ such that Corollary $\ref{cor-special}$ applies if and only if $\tfrac{l}{L}\notin\mathcal{G}_{R}.$
	\end{enumerate}
		\end{proposition}
We shall now present another interesting  result related to  the  duplication method and its application in generating   synchronized multiple time periodic vortices. The duplication principle reads as follows. We refer the reader to Section \ref{sec dup meth} for the complete proof.
\begin{proposition}
	Let $\mathbf{D}$ be a simply-connected bounded domain whose boundary contains a non-trivial segment $[z,w]$ with $z\neq w$ such that
	$$\partial\mathbf{D}\cap(z,w)=[z,w].$$
	We denote $\mathtt{S}$ the reflexion through the axis $(z,w).$
	Let $\omega\in L^\infty \big(\mathbb{R};L^\infty_{c}(\mathbf{D})\big)$ be a global weak solution to Euler equations in $\mathbf{D}.$ Then, there exists $\omega^{\#}\in L^\infty \big(\mathbb{R};L^\infty_{c}(\mathbf{D}^{\star})\big)$ a global in time weak solution to Euler equations in $\mathbf{D}^{\star}\triangleq\textnormal{Int}(\overline{\mathbf{D}}\cup\overline{\mathtt{S}(\mathbf{D})})$ taking the form
	\begin{equation*}
		\omega^{\#}(t,x)\triangleq\begin{cases}
			\omega(t,x), &\textnormal{if }x\in\mathbf{D},\\
			-\omega(t,\mathtt{S}x), & \textnormal{if } x\in\mathtt{S}(\mathbf{D}).
		\end{cases}
	\end{equation*}
	In particular, if $\omega$ is a time periodic vortex patch, then $\omega^{\#}$ is a time periodic counter-rotating pair of patches.
\end{proposition}
This construction can be repeated as long as the new domain retains the same properties as the initial one. In this manner, we can generate multiple synchronized vortices, with their dynamics replicating  that of a single vortex, up to appropriate reflections. To illustrate, let's apply the previous statement to construct synchronized non-rigid periodic motion of four vortex patches within the unit disc. In the planar case, the desingularization of the four point vortices has been studied by  Davila, del Pino, Musso and  Parmeshwar in \cite{DDMP23} using gluing techniques. It is worth to point out that in their case the problem is dispersive and one only can describe the asymptotic dynamics. However in our setting the presence of the boundary constraints the trajectories to be closed and therefore obtain a periodic motion. 
In the next proposition we provide our result in this direction. A more detailed discussion on 
$m$-sectors can be found in Section \ref{sec m sectors}.
\begin{proposition}
	Let $\mathbb{D}_{2}$ be the first quadrant of the unit disc, namely
	$$\mathbb{D}_2\triangleq\big\{(x,y)\in\mathbb{D}\quad\textnormal{s.t.}\quad x>0\quad\textnormal{and}\quad y>0\big\}.$$
	Denote by  $\mathtt{S}_{x}$ and $\mathtt{S}_{y}$  the reflexion with respect to the horizontal and vertical axis, respectively. Then, the following properties hold true.
	\begin{enumerate}
		\item The domain $\mathbb{D}_2$ satisfies the assumptions of Corollary $\ref{cor-special}$ with critical point
		$$\xi_0\triangleq\tfrac{1}{\sqrt{2}}\left(4+\sqrt{17}\right)^{-\frac{1}{4}}(1+\ii).$$
		\item Let $\mathcal{E}_{\lambda}\subset\mathbb{D}_2$ be a point vortex periodic orbit that can be desingularized into a periodic vortex patch motion $\mathbf{1}_{D_t}\in L^{\infty}\big(\mathbb{R};L_c^{\infty}(\mathbb{D}_2)\big).$ Then,  the function 
		$$\omega^\star\triangleq\mathbf{1}_{D_{t}}-\mathbf{1}_{\mathtt{S}_xD_{t}}-\mathbf{1}_{\mathtt{S}_yD_t}+\mathbf{1}_{-D_t}\in L^{\infty}\big(\mathbb{R};L_c^{\infty}(\mathbb{D})\big)$$
		is a time periodic  solution to the Euler equations in the unit disc $\mathbb{D}$. 
	\end{enumerate}
\end{proposition}

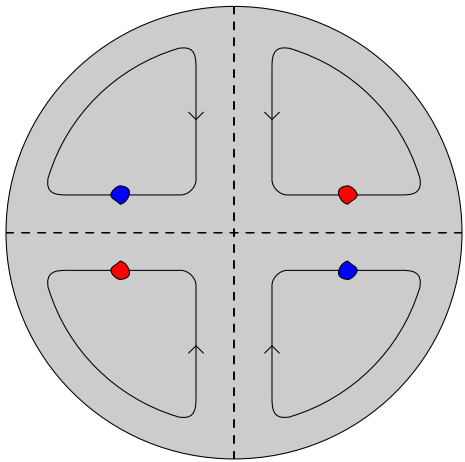
\begin{figure}[!h]
	\begin{center}
		\begin{tikzpicture}[scale=1]
			\filldraw[draw=black,fill=black!20] (0,0) circle (3cm);
			\draw[dashed] (-3,0)--(3,0);
			\draw[dashed] (0,-3)--(0,3);
			\draw (0.7,0.5)--(2.2,0.5);
			\draw (0.5,0.7)--(0.5,2.2);
			\draw (0.4,1.6)--(0.5,1.5);
			\draw (0.6,1.6)--(0.5,1.5);
			\draw (0.7,0.5)..controls+(-0.1,0) and +(0,-0.1)..(0.5,0.7);
			\draw (0.5,2.2)..controls+(0,0.2) and +(-0.2,0.04)..(0.75,2.44);
			\tikzset{
				partial ellipse/.style args={#1:#2:#3}{
					insert path={+ (#1:#3) arc (#1:#2:#3)}
				}
			}
			\draw (0,0) [partial ellipse=17:73:2.55cm and 2.55cm];
			\draw (2.2,0.5)..controls+(0.12,0) and +(0.06,-0.25)..(2.44,0.75);
			\filldraw[draw=black,fill=red] (1.5,0.38) .. controls +(0.05,0.05) and +(0,-0.05) .. (1.62,0.5) .. controls +(0,0.05) and +(0.04,0).. (1.5,0.62).. controls +(-0.04,0) and +(-0.03,0.1) ..(1.38,0.5).. controls +(0.03,-0.1) and +(-0.01,-0.01)..(1.5,0.38);
			%
			%
			\begin{scope}[yscale=-1,xscale=1]
				\draw[dashed] (-3,0)--(3,0);
				\draw[dashed] (0,-3)--(0,3);
				\draw (0.7,0.5)--(2.2,0.5);
				\draw (0.5,0.7)--(0.5,2.2);
				\draw (0.4,1.6)--(0.5,1.5);
				\draw (0.6,1.6)--(0.5,1.5);
				\draw (0.7,0.5)..controls+(-0.1,0) and +(0,-0.1)..(0.5,0.7);
				\draw (0.5,2.2)..controls+(0,0.2) and +(-0.2,0.04)..(0.75,2.44);
				\tikzset{
					partial ellipse/.style args={#1:#2:#3}{
						insert path={+ (#1:#3) arc (#1:#2:#3)}
					}
				}
				\draw (0,0) [partial ellipse=17:73:2.55cm and 2.55cm];
				\draw (2.2,0.5)..controls+(0.12,0) and +(0.06,-0.25)..(2.44,0.75);
				\filldraw[draw=black,fill=blue] (1.5,0.38) .. controls +(0.05,0.05) and +(0,-0.05) .. (1.62,0.5) .. controls +(0,0.05) and +(0.04,0).. (1.5,0.62).. controls +(-0.04,0) and +(-0.03,0.1) ..(1.38,0.5).. controls +(0.03,-0.1) and +(-0.01,-0.01)..(1.5,0.38);
			\end{scope}
			\begin{scope}[yscale=-1,xscale=-1]
				\draw[dashed] (-3,0)--(3,0);
				\draw[dashed] (0,-3)--(0,3);
				\draw (0.7,0.5)--(2.2,0.5);
				\draw (0.5,0.7)--(0.5,2.2);
				\draw (0.4,1.6)--(0.5,1.5);
				\draw (0.6,1.6)--(0.5,1.5);
				\draw (0.7,0.5)..controls+(-0.1,0) and +(0,-0.1)..(0.5,0.7);
				\draw (0.5,2.2)..controls+(0,0.2) and +(-0.2,0.04)..(0.75,2.44);
				\tikzset{
					partial ellipse/.style args={#1:#2:#3}{
						insert path={+ (#1:#3) arc (#1:#2:#3)}
					}
				}
				\draw (0,0) [partial ellipse=17:73:2.55cm and 2.55cm];
				\draw (2.2,0.5)..controls+(0.12,0) and +(0.06,-0.25)..(2.44,0.75);
				\filldraw[draw=black,fill=red] (1.5,0.38) .. controls +(0.05,0.05) and +(0,-0.05) .. (1.62,0.5) .. controls +(0,0.05) and +(0.04,0).. (1.5,0.62).. controls +(-0.04,0) and +(-0.03,0.1) ..(1.38,0.5).. controls +(0.03,-0.1) and +(-0.01,-0.01)..(1.5,0.38);
			\end{scope}
			\begin{scope}[yscale=1,xscale=-1]
				\draw[dashed] (-3,0)--(3,0);
				\draw[dashed] (0,-3)--(0,3);
				\draw (0.7,0.5)--(2.2,0.5);
				\draw (0.5,0.7)--(0.5,2.2);
				\draw (0.4,1.6)--(0.5,1.5);
				\draw (0.6,1.6)--(0.5,1.5);
				\draw (0.7,0.5)..controls+(-0.1,0) and +(0,-0.1)..(0.5,0.7);
				\draw (0.5,2.2)..controls+(0,0.2) and +(-0.2,0.04)..(0.75,2.44);
				\tikzset{
					partial ellipse/.style args={#1:#2:#3}{
						insert path={+ (#1:#3) arc (#1:#2:#3)}
					}
				}
				\draw (0,0) [partial ellipse=17:73:2.55cm and 2.55cm];
				\draw (2.2,0.5)..controls+(0.12,0) and +(0.06,-0.25)..(2.44,0.75);
				\filldraw[draw=black,fill=blue] (1.5,0.38) .. controls +(0.05,0.05) and +(0,-0.05) .. (1.62,0.5) .. controls +(0,0.05) and +(0.04,0).. (1.5,0.62).. controls +(-0.04,0) and +(-0.03,0.1) ..(1.38,0.5).. controls +(0.03,-0.1) and +(-0.01,-0.01)..(1.5,0.38);
			\end{scope}
		\end{tikzpicture}
	\end{center}
	\caption{Choreography of four synchronized vortex patches in the unit disc.}
\end{figure}

Next, we will discuss the key ideas behind the proof of Theorem \ref{main-theorem}. To the best of our knowledge, this result marks the first construction of classical solutions for Euler equations involving non-rigid periodic motion in bounded domains. Our proof is inspired by the approach used in \cite{HHM23}, which combines a desingularization procedure with the Nash-Moser scheme and KAM theory to tackle degenerate quasi-linear transport equations driven by time-space periodic coefficients associated with point vortex motion. This method is remarkably robust, making it applicable to various significant ordered structures in geophysical flows.
As we will see later, our proof not only addresses the challenges of  small divisor problem caused by time-space resonances but also overcomes several other substantial challenges. These include degeneracy in the time direction, the degeneracy of the mode 1 in the leading term, and the invertibility of an operator with variable coefficients. In addition, a supplementary difficulty in our setting arises from the geometry of the fixed domain $\mathbf{D}$, which is not explicit in the general statement. The key idea to overcome this issue is to employ complex analysis using the Riemann conformal mapping, which encodes the geometrical properties of $\mathbf{D}$. This procedure naturally leads to the appearance of interesting geometrical quantities, such as the Schwarzian derivative and the conformal radius.
 \\
Let us now  outline the fundamental  steps  of the proof of Theorem \ref{main-theorem} and discuss  the various  technical  challenges mentioned before.\\

\ding{202} \textit{Contour dynamics equations and linearization.} As we shall see in Section \ref{desing-vortex}, we will look for a solution to  Euler equations  \eqref{Euler-Eq} in the form
\begin{equation*}
		\boldsymbol{\omega}_{\varepsilon}(t)=\tfrac{1}{\varepsilon^2}{\bf{1}}_{{D}_{t}^\varepsilon},
	\end{equation*}
	where $\varepsilon\in(0,1)$ is small enough, the domain   $D_{t}^{\varepsilon}$ is given by  
	\begin{equation*}
		{D}_{t}^\varepsilon \triangleq   \varepsilon  {O}_{t}^\varepsilon+\xi_\lambda(t)
	\end{equation*}
	and ${O}_{t}^\varepsilon$ is a simply connected domain localized around the unit disc. The core of the vortex follows the dynamics of a point vortex, that is,    \begin{equation*}
		\dot{\overline \xi_\lambda}(t)=- \tfrac\ii2\partial_{{z}} \mathcal{R}_{\bf{D}}\big(\xi_\lambda(t)\big).
	\end{equation*} 
	We assume that the orbit $t\mapsto \xi_\lambda(t)$ is $\mathtt{T}(\lambda)-$periodic with $\lambda$ its energy level. The goal is to construct $\mathtt{T}(\lambda)$-periodic solution meaning that  	
	\begin{equation*}
		\forall t\in\R,\quad  {O}_{t+\mathtt{T}(\lambda)}^\varepsilon={O}_{t},
	\end{equation*}
	with  ${O}_{t}^\varepsilon$ being  a simply connected domain localized around the unit disc.   To provide a more precise description, we will parametrize the boundary $ \partial {O}_{t,1}^\varepsilon$ as follows
\begin{align*}
\theta\in\T\mapsto \sqrt{1+2\varepsilon r( \omega(\lambda) t,\theta)}\,e^{\ii\theta},\qquad \omega(\lambda)=\tfrac{2\pi}{\mathtt{T}(\lambda)},
\end{align*}
with $r:(\varphi,\theta)\in\mathbb{T}^2\mapsto r(\varphi,\theta)\in\R$ being a smooth periodic  function. We need to reparamterize the point trajectory as follows
$$
\xi_\lambda(t)=p\big(\omega(\lambda) t\big),\quad \hbox{with}\quad p:\mathbb{T}\mapsto \mathbb{C}.
$$
Then from the contour dynamics equation, see  \eqref{functional scrG}, we find by taking $G\triangleq\tfrac{ \mathbf{G}}{\varepsilon}$
\begin{equation}\label{functional scrG00}
		\begin{aligned}
			{G}(r)(\varphi,\theta)&\triangleq\varepsilon^2\omega(\lambda)\partial_\varphi r(\varphi,\theta)-\tfrac{1}{2}\partial_\theta \textnormal{Re}\Big\{\partial_{z} \mathcal{R}_{\bf{D}}\big(p(\varphi)\big)R(\varphi,\theta)e^{\ii \theta}\Big\}\\ &\quad+\tfrac{1}{\varepsilon}\partial_{\theta}\Big[\Psi_1\big(r,z(\varphi,\theta)\big)+\Psi_2\big(r,z(\varphi,\theta)\big)\Big]=0,
		\end{aligned}
	\end{equation}
where $\Psi_1$ describes the induced effect, whereas $\Psi_2$ describes the boundary interaction.
The  linearization  around a small state $r$ is described in Proposition \ref{proposition linear op exp}. Actually, we get the following asymptotic structure,
\begin{align*}
			d_{r}{G}(r)[h]&=\varepsilon^2\omega(\lambda)\partial_\varphi  h+\partial_{\theta}\big[\mathbf{V}^{\varepsilon}(r)h \big]-\tfrac{1}{2}\mathbf{H}[h]+\varepsilon^2\partial_{\theta}\mathbf{Q}_1[h]+\varepsilon^2\partial_{\theta}\mathcal{R}_1^{\varepsilon}[h]+\varepsilon^3\partial_{\theta}\mathcal{R}_2^{\varepsilon}[h],
		\end{align*}
		where $\mathbf{H}$ denotes the toroidal  Hilbert transform and the function $\mathbf{V}^{\varepsilon}(r)$ decomposes as follows
			\begin{equation*}
				\mathbf{V}^{\varepsilon}(r)\triangleq\tfrac{1}{2}-\tfrac{\varepsilon}{2}r+\varepsilon^2\big(\tfrac{1}{2}\mathtt{g}+V_1^{\varepsilon}(r)\big)+\varepsilon^3V_2^\varepsilon (r),
			\end{equation*}
		with
		\begin{align*}
			\mathtt{g}(\varphi,\theta)&\triangleq\textnormal{Re}\left\lbrace\mathtt{w}_2\big(p(\varphi)\big)e^{2\ii\theta}\right\rbrace,\qquad
			\mathtt{w}_2(p)\triangleq \big(\partial_z\mathcal{R}_{\mathbf{D}}(p)\big)^2+\tfrac13 S(\Phi)(p).
		\end{align*}
In the expression above $S(\Phi)$ is the classical Schwazian derivative, defined in  \eqref{Swart}.
However, the real-valued function $r\mapsto V_1^\varepsilon(r)$ is quadratic while   $r\mapsto V_2^\varepsilon(r)$ is {affine}.
As to the   operator $\mathbf{Q}_1,$ it is a of finite rank localizing in the spatial modes $\pm1$ and takes the integral form
\begin{align*}
	\mathbf{Q}_1[h](\varphi,\theta)&\triangleq\int_{\mathbb{T}}h(\eta)\left[\frac{\cos(\theta-\eta)}{r_{\mathbf{D}}^2\big(p(\varphi)\big)}+\frac16\textnormal{Re}\left\lbrace e^{\ii(\theta+\eta)} S(\Phi)\big(p(\varphi)\big)\right\rbrace\right]d\eta,\qquad r_{\mathbf{D}}\triangleq e^{-\mathcal{R}_{\mathbf{D}}}.
	\end{align*}
The operator $\mathcal{R}_1^\varepsilon(r)$ exhibits a smoothing effect in space and depends quadratically on $r$. On the other hand, the operator $\mathcal{R}_2^\varepsilon(r)$ also smoothes in space, but its dependence on $r$ is affine. The main challenge, in inverting the linearized operator,  involves a small divisor problem caused by the time-space resonance, where the time direction degenerates as $\varepsilon$ approaches zero. To solve this issue, 
 it is necessary to work within Cantor sets on the parameter $\lambda$  that satisfy Diophantine conditions, which also degenerate with rate of order $\varepsilon^{2+\delta}, \delta>0$. This allows to ensure an almost full Lebesgue measure to these sets after the multiple steps applied in the Nash-Moser scheme. Note that the rate of degeneracy in $\varepsilon$ has significant consequences  in constructing a right inverse for the linearized operator, which results in a loss of regularity and leads to a divergent control of order $\varepsilon^{-2-\delta}$.
Nevertheless, when initializing the Nash-Moser scheme with $r=0$, the subsequent iteration corresponds to the term $(d_{r}{G}(0))^{-1}[{G}(0)]$. However, from Lemma \ref{lem G0}, we infer  that ${G}(0)$ is of size $O(\varepsilon)$.
As a result, the estimate of $(d_{r}{G}(0))^{-1}[{G}(0)]$ 
 will exhibit a divergent behavior  as $\varepsilon$ becomes small, making this approach unsuitable for the current scheme. Therefore, before proving the invertibility of the linearized operator and applying the Nash-Moser scheme, we need to construct a better approximate solution for \eqref{functional scrG00}. This construction is provided by Lemma \ref{approxim repsi 0}, which ensures the existence of a function $r_\varepsilon:[\lambda_*,\lambda^*]\times \T^2\to\R$ satisfying
$$\|r_\varepsilon\|_{s}^{\textnormal{Lip}(\gamma)}\lesssim 1\qquad\hbox{and}\qquad\|G(\varepsilon r_\varepsilon)\|_{s}^{\textnormal{Lip}(\gamma)}\lesssim \varepsilon^4.$$
The proof uses the fact that  $G(0)$ is  localized outside the modes $\pm1$, as established in Lemma \ref{lem G0}. Then, we introduce a rescaling of the function $G$ as stated in \eqref{def func F} that takes the form
 \begin{equation}\label{definition functional bfF}
		\mathbf{F}(\rho)\triangleq\tfrac{1}{\varepsilon^{1+\mu}}G(\varepsilon r_\varepsilon+\varepsilon^{1+\mu}\rho),
		\end{equation}
where $\mu \in (0,1)$ is  a free parameter. Note that the parameter $\delta$  will be carefully selected later in the Nash-Moser iteration, in relation to $\mu$ and the measure of the final Cantor set.\\

\ding{203} \textit{Construction of an approximate right inverse.}
In Proposition \ref{Nash-Moser}, we present the formulation of a Nash-Moser iteration tailored to our problem, based on \cite[Prop. 6.1]{HHM23}. The implementation of this iteration requires the construction an approximate right inverse for the linearized operator of the functional $\mathbf{F}$, defined in \eqref{definition functional bfF}, at small state $\rho$. The process follows as the following plan.
First, we expand the linearized operator in terms of $\varepsilon$. Next, we use the KAM method, combined with an appropriate preliminary steps, to conjugate the transport equation into an operator with constant coefficients. This step makes appearing a Cantor-type set.  Then, we reduce a truncated operator at a suitable order in $\varepsilon$, addressing the challenge posed by the degeneracy of the spatial modes $\pm1$, utilizing the monodromy matrix. The inversion on the remaining modes requires a second Cantor-type set. Finally, the approximate right inverse is derived through a perturbative argument.\\

\ding{204} \textit{Nash-Moser scheme and measure of the final Cantor set.}
In the Section \ref{Nash Moser}, we build a non-trivial solution to the equation \eqref{functional scrG00} by means of Nash-Moser iteration scheme. The method is now classical and we borrow the result from \cite[Prop. 6.1]{HHM23}. As stated in Corollary \ref{prop-construction}, the solutions are obtained modulo a suitable choice of the energy level $\lambda$ among a Cantor set. Using the non degeneracy of the point vortex frequency $\lambda\mapsto\omega(\lambda)$ we prove in Lemma \ref{Lem-Cantor-measu} a lower bound for the Lebesgue measure of this final Cantor set.\\

\section{Motion of a single vortex}\label{sec pt vortex motion}
	The point vortex system in bounded domains is well-explored in the literature  and a lot of important facts illustrating  the boundary effects  have been established. We refer for instance to \cite{Gusta1,Newton}. Our primary goal here is to examine certain results related to the motion of a single vortex and analyze its orbit. Specifically, we are interested in the existence of periodic orbits, which are closely linked to the geometry of the domain.

\subsection{Green and Robin functions}\label{sec Green Robin}
In this section, we will gather some fundamental results on Green functions associated with bounded simply connected domains $\mathbf{D}.$ This topic is well-documented in the literature; for instance, see \cite{Evans} for more detailed information. To begin, let's recall that the stream function $\psi$ associated with a given vorticity  $\boldsymbol{\omega}$ is the unique solution to the elliptic problem
\begin{eqnarray*}	           
       \begin{cases}
          	\Delta\psi=\boldsymbol{\omega}, &\text{ in $\mathbf{D}$},\\
        \psi=0,&\text{ for all $z\in\partial \mathbf{D}$}. 
       \end{cases}
\end{eqnarray*}
Then $\psi=\Delta^{-1}\boldsymbol{\omega}$ is linked to   $\boldsymbol{\omega}$ through the following integral formula
	\begin{align}\label{stream-V0}\psi(t,z)=\frac{1}{2\pi}\int_{\mathbf{D}}G_{\mathbf{D}}(z,w)\boldsymbol{\omega}(t,w)dA(w),
	\end{align}
	where $G_{\mathbf{D}}$ denotes the \textit{Green function} in $\mathbf{D}$, which is defined as the unique solution to the elliptic equation
	\begin{eqnarray}   \label{Eulereq}	           
       \begin{cases}
          	\Delta_z G_{\mathbf{D}}(z,w)=2\pi\delta_{w}(z), &\text{ in $\mathbf{D}$},\\
         G_{\mathbf{D}}(z,w)=0,&\text{ for all $z\in\partial \mathbf{D}$}. 
       \end{cases}
\end{eqnarray}
	In addition, the function $G_{\mathbf{D}}$ decomposes as follows,
	\begin{align}\label{GS}
	G_{\mathbf{D}}(z,w)=\log|z-w|+K(z,w),
	\end{align}
	where $K:\mathbf{D}\times\mathbf{D}\rightarrow\mathbb{R}$ is the regular part of the Green function, which is harmonic in each variable $z$ and $w.$ In the particular case of the unit disc $\mathbb{D}$, one gets
	 	$$\qquad  G_{\mathbb{D}}(z,w)=\log\left|\frac{z-w}{1-z\overline{w}}\right|,\qquad\textnormal{i.e.}\qquad K(z,w)=-\log|1-z\overline{w}|.$$
	It is a classical fact \cite[Chap. 6]{A66}, that the Green function is linked to the conformal mappings $\Phi:\mathbf{D}\to\mathbb{D}$ as follows
	\begin{align}\label{Green-conf}
	G_{\mathbf{D}}(z,w)= G_{\mathbb{D}}\big(\Phi(z),\Phi(w)\big).
	\end{align}
	 Therefore, we easily deduce from \eqref{GS}
	 \begin{align}\label{Rob-double}
	\forall z\neq w\in\mathbf{D},\quad K(z,w)= \log\left|\frac{\frac{\Phi(z)-\Phi(w)}{z-w}}{1-\Phi(z)\overline{\Phi(w)}}\right|.
	\end{align}
	The \textit{Robin function} is defined by
	 \begin{align}\label{Robin}
	 	\forall z\in\mathbf{D},\quad\mathcal{R}_{\mathbf{D}}(z)\triangleq K(z,z)=\lim_{w\to z}K(z,w).
	 \end{align}
 Thus we get from \eqref{Rob-double} that the Robin function can be explicitly linked to the conformal mapping as follows,
	\begin{align}\label{Robin1}\nonumber 
	\forall z\in\mathbf{D},\quad\mathcal{R}_{\mathbf{D}}(z)&=\log\left(\frac{|\Phi^\prime(z)|}{1-|\Phi(z)|^2}\right)\\
	&=-\log\big(r_{\mathbf{D}}(z)\big),
	\end{align}
	where $r_{\mathbf{D}}$ is the conformal radius at a point $ z\in \mathbf{D}$  is defined by
	\begin{align}\label{Conf-R}
	r_{\mathbf{D}}(z)\triangleq\frac{1-|\Phi(z)|^2}{|\Phi^\prime(z)|}\cdot
	\end{align}
	In particular, $\mathcal{R}_{\mathbf{D}}$ is real analytic on $\mathbf{D}$ and, according to \cite{Bandle,Fied,Gusta1}, one may deduce from a direct computation that $\mathcal{R}_{\mathbf{D}}$ satisfies the Liouville  equation
	\begin{align}\label{Conf-R1}
	\Delta 	\mathcal{R}_{\mathbf{D}}=4 e^{2	\mathcal{R}_{\mathbf{D}}}\quad\hbox{in }\mathbf{D}.
	\end{align}
	Notice that Robin function satisfies the global bounds, see for instance \cite{Gust2}
	\begin{equation}\label{boundary behaviour Robin}
		\forall z\in \mathbf{D},\quad-\log\big(4\delta(z)\big)\leqslant \mathcal{R}_{\mathbf{D}}(z)\leqslant -\log\delta(z),
	\end{equation}
	with $\delta$ the distance to the boundary $\partial \mathbf{D}$ defined by
	$$
	\delta(z)\triangleq\inf_{w\in \partial \mathbf{D}}|z-w|.
	$$
	We point out that when  the domain $\mathbf{D}$ is  bounded and convex, the Robin function is strictly convex and admits only one non-degenerate critical point, see \cite{Cafar,Gust2}. Later, we will need the following identity which follows from a direct differentiation of    \eqref{Robin1}
	\begin{align}\label{Robin1N}
	\partial_z\mathcal{R}_{\mathbf{D}}(z)
	&=\frac12\frac{\Phi^{\prime\prime}(z)}{\Phi^\prime(z)}+\frac{\Phi^\prime(z)\overline{\Phi(z)}}{1-|\Phi(z)|^2}\cdot
		\end{align}

	
%
			\subsection{Hamiltonian structure and periodic orbits}
	In this section we continue to  assume that  the domain $\mathbf{D}$ is a simply connected bounded domain. The dynamics of a single vortex  $\omega=\Gamma \delta_{\xi(t)}$ inside the domain $\mathbf{D}$ is governed  by the Robin function, introduced in \eqref{Robin}, through the Hamiltonian complex equation, see for instance \cite{Gust2, Newton},
	\begin{equation}\label{Vort-leap}
		\dot{\xi}(t)=\ii \tfrac{\Gamma}{2\pi}\partial_{\overline{z}} \mathcal{R}_{\mathbf{D}}\big(\xi(t)\big).
			\end{equation}
			Notice that this equation can be recast in terms of a 2d real    Hamiltonian  system,
			\begin{align}\label{Hamilt}\dot{\xi}(t)=\nabla^{\perp}H_{\mathbf{D}}\big(\xi(t)\big),\qquad\nabla^{\perp}\triangleq\begin{pmatrix}
				-\partial_{y}\\
				\partial_{x}
			\end{pmatrix},\qquad H_{\mathbf{D}}\triangleq\tfrac{\Gamma}{4\pi}\mathcal{R}_{\mathbf{D}}.
			\end{align}
	Therefore the trajectories $t\mapsto \xi(t)$ are globally well-defined and lie  in the level sets of the Hamiltonian $H_{\mathbf{D}}.$ 
	However, exploring  critical points, such as their number and structure, together with the  shape of level sets within general domains remains a captivating pursuit, with only a sparse collection of results currently available. Notably, as indicated by \eqref{boundary behaviour Robin}, the growth of the Hamiltonian near the boundary implies the  existence of at least one critical point. In convex domains, it is shown in  \cite{Cafar,Gust2} that this point is unique and the Hamiltonian is strictly convex. As a by-product, the phase space is foliated by periodic orbits surrounding  convex subdomains.
	The uniqueness of critical points has been extended to more general domains. Actually, if we denote  the inverse conformal mapping $F\triangleq\Phi^{-1}:\mathbb{D}\rightarrow\mathbf{D},$ then we infer from  \eqref{Robin1N} that the critical points $F(\xi)$ to Robin function are given by the  $\xi\in\mathbb{D}$ which are the  roots of Grakhov's equation \cite{Grakov}
	 \begin{equation}\label{Grakhov}
	 \frac{F^{\prime\prime}(\xi)}{F^\prime(\xi)}=\frac{2\overline\xi}{1-|\xi|^2}\cdot
	 \end{equation}
	According to  \cite{CO94}, the uniqueness is established  under the Nehari univalence  criterion
	\begin{equation}\label{cond F crit pnt}
		\forall z\in\mathbb{D},\quad|S(F)(z)|\leqslant\tfrac{2}{(1-|z|^2)^2},
	\end{equation}
	where  $S$ is the Schwarzian derivative defined by
	\begin{align}\label{Swart}
	S(F)\triangleq\left(\tfrac{F''}{F'}\right)'-\tfrac{1}{2}\left(\tfrac{F''}{F'}\right)^2=\tfrac{F^{(3)}}{F'}-\tfrac{3}{2}\left(\tfrac{F''}{F'}\right)^2.
	\end{align}
We emphasize  that the  condition \eqref{cond F crit pnt} is satisfied by convex domains, see \cite{CDO11,N76}. More results in the same spirit   leading to the uniqueness of critical points  can be found in \cite{K21,KK18}. 
In general domains, a little is known on the critical points. For instance, by  invoking Sard theorem we can deduce weak results such as  the set of critical values 
	\begin{equation}\label{Singular}
	\mathcal{C}\triangleq\big\{ H_{\bf{D}}(z)\quad \hbox{s.t.}\quad \nabla H_{\bf{D}}(z)=0\big\}
	\end{equation}
	is compact with zero Lebesgue measure. Furthermore, the range of the Hamiltonian $H_{\bf{D}}$ is a connected set of $\mathbb{R}$ and takes the form
	\begin{equation}\label{Range}
	H_{\mathbf{D}}(\mathbf{D})\triangleq[\lambda_{\textnormal{\tiny{min}}},\infty).
	\end{equation}	
	Remark that $\lambda_{\textnormal{\tiny{min}}}\in \mathcal{C}$ and the set $H_{\mathbf{D}}(\mathbf{D})\setminus \mathcal{C}$ is open. For each regular energy value $\lambda \in H_{\mathbf{D}}(\mathbf{D})\setminus\mathcal{C}$ all the connected components of $H_{\mathbf{D}}^{-1}(\{\lambda\})$ are periodic orbits diffeomorphic to a circle, see for instance  \cite[Prop. 2.1]{LP04}. We  denote by $\mathtt{T}(\lambda)$ the  minimal period of each periodic  orbit $\mathcal{E}_\lambda$. This period can be recovered from area $A(\lambda)$ of the domain enclosed by $\mathcal{E}_\lambda$ as follows
$$\mathtt{T}(\lambda)=\frac{dA(\lambda)}{d\lambda}\cdot
$$
In the following proposition we shall collect some basic classical results on planar Hamiltonian dynamical system. They follow from the fact that the Hamiltonian $\mathcal{R}_{\bf{D}}$ is real analytic inside the domain $\mathbf{D}.$
	\begin{proposition}\label{prop-analy}
	Let  
$[\lambda_*,\lambda^*]\subset H_{\mathbf{D}}(\mathbf{D})\setminus\mathcal{C},$ and consider  a continuous family of periodic orbits $(\mathcal{E}_\lambda)_{\lambda\in[\lambda_*,\lambda^*]}$. Then the following results hold true.
\begin{enumerate}
\item The period map $\lambda\in [\lambda_*,\lambda^*]\mapsto \mathtt{T}(\lambda)$ is analytic.
\item Each periodic orbit $\mathcal{E}_\lambda$ admits an analytic parametrization $s\in\mathbb{T}\mapsto p_{\lambda}(s)$.
\item The map $\lambda\in [\lambda_*,\lambda^*]\mapsto p_\lambda$ is analytic.
\end{enumerate} 
	\end{proposition}
	
		\subsection{Period asymptotics}
	The purpose of this section is to analyze the asymptotic behaviors of the period of point vortex orbits in two extreme regimes: near an elliptic critical point and close to the domain boundary. Near an elliptic critical point, the vortex exhibits slow motion, whereas near the boundary, it rotates rapidly.
	\subsubsection{Near an elliptic critical point}
	Let us consider $\xi_0\in\mathbf{D}$ an elliptic critical point of the Robin function, such as the  point associated with the global minimum. Then 
	\begin{equation}\label{crit pnt}
		\partial_{z}\mathcal{R}_{\mathbf{D}}(\xi_0)=0.
	\end{equation}
	We can always choose (and this is done in a unique way) the conformal mapping so that
	$$\Phi(\xi_0)=0\qquad\textnormal{and}\qquad\Phi'(\xi_0)>0.$$
	According to \eqref{Robin1N}, the condition \eqref{crit pnt} is equivalent to 
	$$\Phi''(\xi_0)=0.$$
	Recall from \eqref{Vort-leap} that the point vortex motion  writes 
	\begin{equation}\label{dynamique pnt vx}
		\dot{\xi}=\tfrac{\ii\Gamma}{2\pi}\partial_{\overline z}\mathcal{R}_{\mathbf{D}}(\xi).
	\end{equation}
	By means of Morse's Lemma, locally the trajectory of the point vortex will be close to an ellipse. Therefore, one can choose a parametrization of the orbit using polar formulation as follows
	\begin{equation}\label{splitting xi}
		\xi(t)=\xi_0+\mathtt{R}\big(\lambda,\Theta(t)\big)e^{\ii\Theta(t)},
	\end{equation}
	where the energy level $\lambda\in H_{\mathbf{D}}(\mathbf{D})$ is taken close to  $\lambda_{\star}\triangleq H_{\mathbf{D}}(\xi_0).$
	In particular
	$$\mathtt{R}(\lambda_{\star},\Theta)=0.$$
	Inserting the ansatz \eqref{splitting xi} into \eqref{dynamique pnt vx} gives
	$$\dot{\Theta}=\frac{\Gamma}{2\pi\mathtt{R}(\lambda,\Theta)}\textnormal{Re}\big\{(\partial_{z}\mathcal{R}_{\mathbf{D}})\big(\xi_0+\mathtt{R}(\lambda,\Theta)e^{\ii\Theta}\big)e^{\ii\Theta}\big\}.$$
	It follows that
	$$\mathtt{T}(\lambda)=\frac{2\pi}{\Gamma}\int_0^{2\pi}\frac{\mathtt{R}(\lambda,\Theta)}{\textnormal{Re}\big\{(\partial_{{\xi}}\mathcal{R}_{\mathbf{D}})\big(\xi_0+\mathtt{R}(\lambda,\Theta)e^{\ii\Theta}\big)e^{\ii\Theta}\big\}}d\Theta.$$
	Performing a Taylor expansion, we find
	\begin{align*}\partial_{z}\mathcal{R}_{\mathbf{D}}\big(\xi_0+\mathtt{R}\big(\lambda,\Theta(t)\big)e^{\ii\Theta(t)}\big)=&\mathtt{R}(\lambda,\Theta)e^{\ii\Theta}\tfrac{\Phi^{(3)}(\xi_0)}{2\Phi'(\xi_0)}+\mathtt{R}(\lambda,\Theta)e^{-\ii\Theta}\big(\Phi'(\xi_0)\big)^2+O\big(\mathtt{R}^2(\lambda,\Theta)\big).
	\end{align*}
	Therefore, we deduce that
	\begin{align*}
		\mathtt{T}(\lambda_{\star})&\triangleq \lim_{\lambda\to \lambda_\star}\mathtt{T}(\lambda)=\frac{2\pi}{\Gamma}\int_{0}^{2\pi}\frac{d\Theta}{\mathtt{a}+|\mathtt{b}|\cos(2\Theta+\arg(\mathtt{b}))}\\
		&=\frac{4\pi}{\Gamma}\int_{0}^{\pi}\frac{d\Theta}{\mathtt{a}+|\mathtt{b}|\cos(\Theta)},
	\end{align*}
	where
	\begin{equation}\label{def a b}
\mathtt{a}\triangleq \big(\Phi'(\xi_0)\big)^2\qquad\textnormal{and}\qquad\mathtt{b}\triangleq\frac{\Phi^{(3)}(\xi_0)}{2\Phi'(\xi_0)}\cdot
	\end{equation}
	Finally, using \cite[p. 402]{GR15} or a direct computation based on the residue Theorem, we find
	\begin{equation}\label{T0 crit}
\mathtt{T}(\lambda_{\star})=\frac{4\pi^2}{\Gamma\sqrt{\mathtt{a}^2-|\mathtt{b}|^2}}\cdot
\end{equation}
Notice that this result was established in a different way in \cite{Gusta1}.

	\begin{remark}
	\begin{enumerate}
	\item In \eqref{T0 crit} the condition $\mathtt{a}>|\mathtt{b}|$ is required. This corresponds to the ellipticity of the critical point $\xi_0$ and it is necessary for the dynamics to be confined locally near $\xi_0$, as explained in~{\rm \cite{DI21}}.  
	
	\item	A more refined analysis based on the expansion of the orbit leads to the first order expansion
		\begin{align*}
			\mathtt{T}(\lambda)&=\frac{4\pi^2}{\Gamma\sqrt{\mathtt{a}^2-|\mathtt{b}|^2}}+\frac{16\pi^2 (\lambda-\lambda_{\star})}{\Gamma(\sqrt{\mathtt{a}^2-|\mathtt{b}|^2})^7}\Big[ \tfrac23\mathtt{a}^3|\mathtt{c}|^2+\mathtt{a}|\mathtt{b}|^2|\mathtt{c}|^2 +\tfrac53 \mathtt{a}  \textnormal{Re}\{\mathtt{c}^2\overline{\mathtt{b}}^2\}-\tfrac32(  \mathtt{a}-  |\mathtt{b}|^2)\textnormal{Re}\{\mathtt{d}\overline{\mathtt{b}}^2\}
			\\ &\qquad\qquad\qquad\qquad\qquad\qquad\qquad-\tfrac32\mathtt{a}^4|\mathtt{b}|^2-\mathtt{a}^6+\tfrac52\mathtt{a}^2|\mathtt{b}|^4\Big]+O\big((\lambda-\lambda_{\star})^{\frac32}\big),
		\end{align*}
		where
		$$\mathtt{c}\triangleq\frac{\Phi^{(4)}(\xi_0)}{4\Phi'(\xi_0)},\qquad\mathtt{d}\triangleq\frac{\Phi^{(5)}(\xi_0)}{12\Phi'(\xi_0)}-\left(\frac{\Phi^{(3)}(\xi_0)}{2\Phi'(\xi_0)}\right)^2\cdot
		$$
		\end{enumerate}
	\end{remark}
	
	\subsubsection{Near the boundary}
	In what follows,  we intend to study the asymptotic of $\mathtt{T}(\lambda)$ as $\lambda$ approaches infinity.  We will begin with a weaker result, which we believe remains valid even at a significant distance from the boundary.
	\begin{lemma}\label{lemm-non-degen}
		Let $\mathbf{D}$ be a simply-connected bounded domain. Then, all the orbits near the boundary are closed and  the period function  cannot be constant near the boundary $\partial\mathbf{D}.$ 
	\end{lemma}
	
	\begin{proof}
	First, close to the boundary all the orbits are periodic since there is no  critical points for Robin function near the boundary. Assume by contraction that the period function  is constant on some interval $[\overline\lambda,\infty)$. Then from  $\mathtt{T}=A'$, we deduce that the area is an affine function in this range of  $\lambda$, namely there exists $(\alpha,\beta)\in\mathbb{R}^2$ such that
		$$\forall\lambda\in[\overline\lambda,\infty)\quad A(\lambda)=\alpha\lambda+\beta.$$
		Since $A(\lambda)$ is bounded by the area of $\mathbf{D}$, then taking the limit $\lambda\to\infty$ forces $\alpha=0.$ Hence, the area is constant in $\lambda.$ Now,  if $\overline\lambda\leqslant\lambda_1<\lambda_2,$ then the domain enclosed by the periodic orbit $\mathcal{E}_{\lambda_1}$ is strictly embedded in the domain enclosed by the periodic orbit $\mathcal{E}_{\lambda_2},$ implying in turn $A(\lambda_1)<A(\lambda_2).$ This gives the contradiction.
	\end{proof}
We emphasize that the previous argument applies to  any simply connected bounded domain without any regularity assumption on the boundary. When the boundary is supposed to be  slightly smooth, specifically more regular than $C^1$, then we obtain a more precise  estimate of the period. This result, was obtained for instance in  \cite{Gusta1}.
	
	\begin{proposition}\label{prop asymp bnd}
	Let  $\mathbf{D}$ be  a simply-connected bounded domain with boundary of regularity $C^{2}.$ Then, near the boundary, we have the following asymptotic of the period
	\begin{equation}\label{asymp bnd}
		\mathtt{T}(\lambda)\underset{\lambda\to\infty}{=}\frac{2\pi L_{\partial\mathbf{D}}}{\Gamma}e^{-\frac{4\pi\lambda}{\Gamma}}+O(e^{-\frac{8\pi\lambda}{\Gamma}}),
	\end{equation}
	where $L_{\partial\mathbf{D}}$ denotes the length of the boundary $\partial\mathbf{D}.$
	\end{proposition}
\begin{proof}
According to \cite{P92}, the conformal mapping $\mathbf{F}:\mathbb{D}\to\mathbf{D}$ admits a smooth extension to the boundary, still denoted $\mathbf{F}.$ In addition,
$$\mathbf{F}(\mathbb{T})=\partial\mathbf{D}.$$
This allows us to consider the following parametrization of the boundary of $\partial\mathbf{D}$
$$\forall\theta\in\mathbb{T},\quad X(\theta)\triangleq\mathbf{F}(e^{\ii\theta}).$$
Referring to \cite[Cor. 14]{FG97}, the Robin function admits the following behavior close to a point $b\in\partial\mathbf{D}$
$$\mathcal{R}_{\mathbf{D}}\big(b-s\nu(b)\big)\underset{s\to0}{=}-\log(2s)+O(s^2),
$$
where $\nu(b)$ denotes the unit outward normal vector to the boundary $\partial\mathbf{D}$ at the point $b$. For a periodic orbit corresponding to the energy level $\lambda$, we have in view of \eqref{Hamilt}
$$s=s(\lambda,b)\underset{\lambda\to\infty}{=}\tfrac{1}{2}e^{-\frac{4\pi\lambda}{\Gamma}}+O(e^{-\frac{8\pi\lambda}{\Gamma}}),$$
uniformly in $b\in \partial\mathbf{D}$. We obtain a parametrization of the periodic orbit by setting
\begin{equation}\label{param bnd nrj}
\gamma_{\lambda}(\theta)\triangleq X(\theta)-s\big(\lambda,X(\theta)\big)\frac{\ii X'(\theta)}{|X'(\theta)|}=\mathbf{F}(e^{\ii\theta})+s\big(\lambda,X(\theta)\big)\frac{\mathbf{F}'(e^{\ii\theta})e^{\ii\theta}}{|\mathbf{F}'(e^{\ii\theta})|}\cdot
\end{equation}
The Gauss-Green formula gives the area of the enclosed domain 
\begin{equation}\label{area gm}
A(\lambda)=-\frac{1}{2}\int_{0}^{2\pi}\textnormal{Im}\left\{\overline{\gamma_{\lambda}(\theta)}\partial_{\theta}\gamma_{\lambda}(\theta)\right\}d\theta.
\end{equation}
Inserting \eqref{param bnd nrj} into \eqref{area gm}, we obtain the following asymptotic
$$A(\lambda)\underset{s\to0}{=}A_{\mathbf{D}}-\tfrac{1}{4}e^{-\frac{4\pi\lambda}{\Gamma}}\textnormal{Im}\left\{\int_0^{2\pi}\overline{F(e^{\ii\theta})}\partial_{\theta}\left(\frac{\mathbf{F}'(e^{\ii\theta})e^{\ii\theta}}{|\mathbf{F}'(e^{\ii\theta})|}\right)+\frac{\overline{\mathbf{F}'(e^{\ii\theta})}e^{-\ii\theta}}{|\mathbf{F}'(e^{\ii\theta})|}\ii e^{\ii\theta}\mathbf{F}'(e^{\ii\theta})d\theta\right\}+O(e^{-\frac{8\pi\lambda}{\Gamma}}),$$
where $A_{\mathbf{D}}$ is the area of the domain $\mathbf{D}.$ Integrating by parts yields
$$A(\lambda)\underset{s\to0}{=}A_{\mathbf{D}}-\tfrac{1}{2}e^{-\frac{4\pi\lambda}{\Gamma}}\int_0^{2\pi}|\mathbf{F}'(e^{\ii\theta})|d\theta+O(e^{-\frac{8\pi\lambda}{\Gamma}}).
$$
On the other hand, as
$$L_{\partial\mathbf{D}}\triangleq\int_0^{2\pi}|\mathbf{F}'(e^{\ii\theta})|d\theta$$
is the length of the boundary $\partial\mathbf{D},$ then we deduce that
$$A(\lambda)\underset{\lambda\to\infty}{=}A_{\mathbf{D}}-\tfrac12{L_{\partial\mathbf{D}}}e^{-\frac{4\pi\lambda}{\Gamma}}+O(e^{-\frac{8\pi\lambda}{\Gamma}}).$$
From this, we derive the  following asymptotic of the period,
$$\mathtt{T}(\lambda)\underset{\lambda\to\infty}{=}\frac{2\pi L_{\partial\mathbf{D}}}{\Gamma}e^{-\frac{4\pi\lambda}{\Gamma}}+O(e^{-\frac{8\pi\lambda}{\Gamma}}).$$
This concludes the proof of Proposition \ref{prop asymp bnd}.
\end{proof}
\begin{remark}
	The asymptotic \eqref{asymp bnd} only involves the length $L_{\partial\mathbf{D}}.$ This suggests that (at least at first order) this asymptotic remains valid for less regular domains, namely only rectifiable.
\end{remark}

\section{Main results}
In this section, we expanded upon the main results partially discussed in the introduction, providing  a broader generalization covering more than convex domains. We shall first see how to desingularize a single vortex using the contour dynamics. Afterwards, we will state our main results on the existence of time periodic solutions near the single vortex. 
	\subsection{Desingularization of a single vortex}\label{desing-vortex}
	In what follows, we shall consider the motion of a  single vortex inside a simply connected bounded domain at a non degenerate   energy level $\lambda\in H_{\mathbf{D}}(\mathbf{D})\setminus\mathcal{C}$ of the Hamiltonian $H_{\mathbf{D}}$ as described in Section \ref{sec pt vortex motion}. Without loss of generality, we can make the normalization 
	\begin{align}\label{choice circ}
	\Gamma=\pi
	\end{align} and we assume that the orbit is closed with a minimal \mbox{period  $\mathtt{T}(\lambda)$.} We denote by   $t\mapsto\xi(t)$ the vortex orbit which satisfies according to \eqref{Vort-leap} the Hamiltonian equation
	\begin{equation}\label{Ham conj}
		\dot{\overline \xi}(t)=- \tfrac\ii2\partial_{{z}} \mathcal{R}_{\bf{D}}\big(\xi(t)\big).
	\end{equation} The goal is to construct $\mathtt{T}(\lambda)$-periodic solution in  the vortex patch form whose core follows asymptotically   the point vortex motion. Given $\varepsilon\in(0,1)$, we will look for solutions to \eqref{Euler-Eq} in the form
	\begin{equation}\label{omegateps}
		\boldsymbol{\omega}_{\varepsilon}(t)=\tfrac{1}{\varepsilon^2}{\bf{1}}_{{D}_{t}^\varepsilon},
	\end{equation}
	where $D_{t}^{\varepsilon}$ is given by  
	\begin{equation*}
		{D}_{t}^\varepsilon \triangleq   \varepsilon  {O}_{t}^\varepsilon+\xi(t)
	\end{equation*}
	and ${O}_{t}^\varepsilon$ is a simply connected domain localized around the unit disc. Due to the mass conservation, one has
	$$|D_{t}^{\varepsilon}|=|D_{0}^{\varepsilon}|.$$
	We will normalize the area of the initial patch  to $\pi$ so that
	$$\lim_{\varepsilon\to0}\boldsymbol{\omega}_{\varepsilon}(t)=\pi\delta_{\xi(t)}\quad\textnormal{in}\quad \mathcal{D}'(\mathbf{D}),$$
	which is coherent with the choice of the point vortex circulation \eqref{choice circ}.
	In this particular case, and in view of \eqref{stream-V0} and \eqref{GS}, the stream function 
	takes the form
	\begin{equation}\label{def stream}
		\begin{split}
			\psi(t,z)
			=\frac{1}{2\pi\varepsilon^2}\int_{D_{t}^{\varepsilon}}\log(|z-\zeta|)dA(\zeta)+\frac{1}{2\pi\varepsilon^2}\int_{D_{t}^{\varepsilon}}K(z,\zeta)dA(\zeta).
		\end{split}	
	\end{equation}
	Let $ \gamma(t,\cdot): \mathbb{T}  \mapsto  \partial {O}_{t}^\varepsilon$ 
	be any smooth  parametrization of the boundary $\partial{O}_{t}^\varepsilon$, that gives in turn a parametrization  of the boundary $\partial {D}_{t}^\varepsilon$  through
	\begin{equation}\label{param Dk} 
		w(t,\cdot)\triangleq\varepsilon  \gamma(t,\cdot)+\xi(t).
	\end{equation}	
	According to \cite[p. 174]{HMV13},  the vortex patch equation writes 	
	\begin{equation}\label{vortex patches equation}
		\partial_{t}w(t,\theta)\cdot \mathbf{n}\big(t,w(t,\theta)\big)+\partial_{\theta}\big[\psi\big(t,w(t,\theta)\big)\big]=0,
	\end{equation}	
	where  $\mathbf{n}(t,\cdot)$ refers to an outward normal vector to the boundary $\partial D_{t}^{\varepsilon}$. Note that, from \eqref{def stream}, we readily  get
	\begin{align*}
		\partial_{\theta}\big[\psi\big(t,w(t,\theta)\big)\big]
		&= \frac{1}{2\pi \varepsilon^2}\partial_{\theta}\bigg[\int_{D_{t}^{\varepsilon}}\log(|w(t,\theta)-\zeta|)dA(\zeta)+\int_{D_{t}^{\varepsilon}}K(w(t,\theta),\zeta)dA(\zeta)\bigg].\nonumber
	\end{align*}
	Therefore,  using  \eqref{param Dk} with  a suitable change of variable we conclude that	
\begin{align}\label{partial psi}
		\nonumber \partial_{\theta}\big[\psi\big(t,w(t,\theta)\big)\big]&= \frac{1}{2\pi }\partial_{\theta}\int_{O_{t}^{\varepsilon}}\log(|\gamma(t,\theta)-  \zeta |)dA(\zeta)\\
		&\quad+\frac{1}{2\pi }\partial_{\theta}\int_{O_{t}^{\varepsilon}}K\big(\varepsilon  \gamma(t,\theta)+\xi(t),\varepsilon \zeta+\xi(t)\big)dA(\zeta).
\end{align}
Now, identifying $\mathbb{C}$ with $\mathbb{R}^{2}$ 
	and  making the choice  $\mathbf{n}\big(t,w(t,\theta)\big)=-\ii\partial_{\theta}w(t,\theta)$ we get
	\begin{align}\label{ident1}
		\nonumber\partial_{t}w(t,\theta)\cdot \mathbf{n}(t,w(t,\theta))&=\textnormal{Im}\Big\{\partial_{t}\overline{w(t,\theta)}\partial_{\theta}w(t,\theta)\Big\}\\ &=\varepsilon^2 \textnormal{Im}\Big\{ \partial_{t}\overline{\gamma(t,\theta)}\partial_{\theta}  \gamma(t,\theta)\Big\}+\varepsilon \textnormal{Im}\Big\{\partial_{t}\overline{\xi(t)}\partial_{\theta}  \gamma(t,\theta)\Big\}.
	\end{align}
Thus, combining \eqref{vortex patches equation}, \eqref{partial psi} and 	\eqref{ident1} we obtain the equation
\begin{equation}\label{vortex patch equation new}
\begin{aligned}
	&\varepsilon^2 \textnormal{Im}\Big\{ \partial_{t}\overline{\gamma(t,\theta)}\partial_{\theta}  \gamma(t,\theta)\Big\}+\varepsilon \textnormal{Im}\Big\{\partial_{t}\overline{\xi(t)}\partial_{\theta}  \gamma(t,\theta)\Big\}\\ &+\frac{1}{2\pi }\partial_{\theta}\int_{O_{t}^{\varepsilon}}\log(|\gamma(t,\theta)-\zeta |)dA(\zeta) +\frac{1}{2\pi }\partial_{\theta}\int_{O_{t}^{\varepsilon}}K\big(\varepsilon  \gamma(t,\theta)+\xi(t),\varepsilon \zeta+\xi(t)\big)dA(\zeta)=0.
			\end{aligned}
\end{equation}
Consequently, using \eqref{Ham conj}, the relation \eqref{vortex patch equation new} writes
	\begin{align}\label{new vortex patches system}
		\nonumber	\varepsilon^2 \textnormal{Im}\Big\{ \partial_{t}\overline{\gamma(t,\theta)}\partial_{\theta}  \gamma(t,\theta)\Big\}&-\tfrac{\varepsilon}{2} \textnormal{Re}\Big\{\partial_{{z}} \mathcal{R}_{\bf{D}}\big(\xi(t)\big)\partial_{\theta}  \gamma(t,\theta)\Big\} +\frac{1}{2\pi }\partial_{\theta}\int_{O_{t}^{\varepsilon}}\log(|\gamma(t,\theta)-\zeta |)dA(\zeta)\\
			&\quad +\frac{1}{2\pi }\partial_{\theta}\int_{O_{t}^{\varepsilon}}K\big(\varepsilon  \gamma(t,\theta)+\xi(t),\varepsilon \zeta+\xi(t)\big)dA(\zeta)=0.
		\end{align}
Recall that we are looking for  $\mathtt{T}(\lambda)$-periodic solution whose frequency is  
	$$\omega(\lambda)\triangleq\frac{2\pi}{\mathtt{T}(\lambda)}\cdot$$
	Therefore, we will work with the structure
	\begin{align}\label{lien-12}
		\xi(t)=p\big(\omega(\lambda) t\big), 
		\qquad \gamma(t,\theta)=z\big(\omega(\lambda) t,\theta\big),
	\end{align}
	where $p:\mathbb{T}\to\mathbb{C}$ is $2\pi$-periodic function and the boundary is parametrized as
	\begin{equation}\label{def zk}
		{z}(\varphi):\begin{array}[t]{rcl}
			\mathbb{T} & \mapsto & \partial {O}_{\varphi}^\varepsilon\\
			\theta & \mapsto & z (\varphi,\theta)
		\end{array}\quad \textnormal{with}\qquad\begin{array}[t]{rl}
			z (\varphi,\theta)&\triangleq R(\varphi,\theta)e^{\ii\theta},\\
			R(\varphi,\theta)&\triangleq\big(1+2\varepsilon r( \varphi,\theta)\big)^{\frac{1}{2}}.
		\end{array}
	\end{equation}  
	Thus, 
	 using \eqref{new vortex patches system} and polar coordinates,   the curve  $z$ satisfies the equation 
	\begin{equation*}
		\begin{aligned}
			&\varepsilon^2 \omega(\lambda)\textnormal{Im}\Big\{\partial_{\varphi}\overline{z(\varphi,\theta)}\partial_{\theta}  z(\varphi,\theta)\Big\}-\tfrac{\varepsilon}{2}\partial_\theta \textnormal{Re}\Big\{\partial_{{\xi}} \mathcal{R}_{\bf{D}}\big(p(\varphi)\big)R(\varphi,\theta)e^{\ii \theta}\Big\}\\ &
			\quad+\partial_{\theta}\Big[\Psi_1\big(r,z(\varphi,\theta)\big)+\Psi_2\big(r,z(\varphi,\theta)\big)\Big]=0,
		\end{aligned}
	\end{equation*}
	where
	\begin{equation}\label{def-Psi2}
		\begin{aligned}
			\Psi_1(r, z)&\triangleq\int_{\mathbb{T}}\int_{0}^{ R(t,\eta)}\log\big(\big| z- l e^{\ii\eta}\big|\big)l d l d\eta,\\  \Psi_2(r, z)&\triangleq\int_{\mathbb{T}}\int_{0}^{ R(\varphi,\eta)}K\big(p(\varphi)+\varepsilon z,p(\varphi)+\varepsilon  l e^{\ii\eta}\big)l d l d\eta.
		\end{aligned}
	\end{equation}
	Here, we are using the notation
	$$
	\int_{\mathbb{T}}..=\frac{1}{2\pi} \int_0^{2\pi}..
	$$
	Straightforward computations, using \eqref{def zk}, lead to 
	\begin{equation*}
		\begin{aligned}
			&\textnormal{Im}\Big\{ \partial_{\varphi}\overline{z(\varphi,\theta)}\partial_{\theta}  z(\varphi,\theta)\Big\} =\varepsilon\partial_\varphi r(\varphi ,\theta).
		\end{aligned}
	\end{equation*}
	It follows that  the radial deformation $r$, defined through \eqref{def zk}, satisfies the following nonlinear transport equation,
	\begin{equation}\label{functional scrG}
		\begin{aligned}
			\mathbf{G}(r)(\varphi,\theta)&\triangleq\varepsilon^3\omega(\lambda)\partial_\varphi r(\varphi,\theta)-\tfrac{\varepsilon}{2}\partial_\theta \textnormal{Re}\Big\{\partial_{z} \mathcal{R}_{\bf{D}}\big(p(\varphi)\big)R(\varphi,\theta)e^{\ii \theta}\Big\}\\ &\quad+\partial_{\theta}\Big[\Psi_1\big(r,z(\varphi,\theta)\big)+\Psi_2\big(r,z(\varphi,\theta)\big)\Big]=0.
		\end{aligned}
	\end{equation}
This is the final form of the contour dynamics equation   to which we have to show the existence of  solutions for small values of  $\varepsilon$.

\subsection{General statement}
In this paragraph, we intend to give our main result from which we can deduce several interesting consequences. For this we aim, we need to introduce some basic objects stemming from the periodic single vortex described by the periodic parametrization $\varphi\in\mathbb{T}\mapsto p(\varphi)$ as in \eqref{lien-12}. Notice that this orbit is associated with  the level energy $\lambda$ and it admits a minimal period $\mathtt{T}(\lambda).$ We define the following matrix with $2\pi-$ periodic coefficients,
\begin{equation}\label{Matrix A0}
			\mathbb{A}_{\lambda}(\varphi)=\begin{pmatrix}
				\mathtt{u}_{\lambda}(\varphi) & \mathtt{v}_{\lambda}(\varphi)\\
				\overline{\mathtt{v}_{\lambda}(\varphi)} & \overline{\mathtt{u}_{\lambda}(\varphi)}
			\end{pmatrix},\qquad\mathtt{u}_{\lambda}(\varphi)=-\frac{\ii \mathtt{T}(\lambda)}{4\pi}r_{\mathbf{D}}^{-2}\big(p_{\lambda}(\varphi)\big),\qquad\mathtt{v}_{\lambda}(\varphi)=\frac{\ii \mathtt{T}(\lambda)}{8\pi}\big(\partial_{z}\mathcal{R}_{\mathbf{D}}\big(p_{\lambda}(\varphi)\big)\big)^2,
		\end{equation}
		where $\mathcal{R}_{\mathbf{D}}$ denotes  the Robin function stated in \eqref{Robin} and $r_{\mathbf{D}}$ is the conformal radius defined in \eqref{Conf-R}. To this matrix $\mathbb{A}_{\lambda}$, we associate the fundamental matrix $\mathscr{M}_{\lambda}$ which satisfies the linear ODE
		\begin{align}\label{Monod-p1}
		\partial_{\varphi}\mathscr{M}_{\lambda}(\varphi)=\mathbb{A}_{\lambda}(\varphi)\mathscr{M}_{\lambda}(\varphi),\qquad \mathscr{M}_{\lambda}(0)=\textnormal{Id}.
		\end{align}
		Since $\textnormal{Tr}(\mathbb{A}_{\lambda}(\varphi))=0$, then from Abel's Theorem $\textnormal{det}\mathscr{M}_{\lambda}(\varphi)=1.$ Hence, one can easily show the equivalence relation on  the monodromy matrix $\mathscr{M}_{\lambda}(2\pi)$ 
		$$
		\textnormal{Tr}\big(\mathscr{M}_\lambda(2\pi)\big)\neq 2\qquad\Longleftrightarrow \qquad 1 \notin \hbox{sp}(\mathscr{M}_\lambda(2\pi)),
		$$
		where  $\hbox{sp}(\mathscr{M}_\lambda(2\pi))$ denotes the spectrum of the matrix $\mathscr{M}_\lambda(2\pi)$.
Now, we are in a position to state our main result. 
	\begin{theorem}\label{main thm}
		Let $\mathbf{D}$ be a simply-connected bounded domain in $\mathbb{R}^2.$   Let $\lambda_*<\lambda^*$ such that $[\lambda_*,\lambda^*]\subset H_{\mathbf{D}}(\mathbf{D})\setminus \mathcal{C},$ with  the definition \eqref{Singular}. Assume the following conditions. 
		\begin{enumerate}
			\item Non-degeneracy of the period: 
			\begin{equation}\label{non degeneracy period thm}
				\min_{\lambda\in[\lambda_*,\lambda^*]}\left|T^\prime(\lambda)\right|>0.
			\end{equation}
			\item Spectral assumption on the monodromy matrix: 
			\begin{equation}\label{non degeneracy monodromy thm}
				\forall\lambda\in[\lambda_*,\lambda^*],\quad1 \notin \textnormal{sp}\big(\mathscr{M}_\lambda(2\pi)\big).
			\end{equation}
		\end{enumerate}
		Then, there exists $\varepsilon_0>0$ such that for any $\varepsilon\in(0,\varepsilon_0),$ there exists a Cantor set $\mathscr{C}_{\varepsilon}\subset[\lambda_*,\lambda^*],$ with
		$$\lim_{\varepsilon\to0}|\mathscr{C}_{\varepsilon}|=\lambda^*-\lambda_*,$$
		such that for any $\lambda\in\mathscr{C}_{\varepsilon},$ we can construct a time-periodic solution  to \eqref{Euler-Eq} in the form
		$$\boldsymbol{\omega}(t)=\frac{1}{\varepsilon^2}\mathbf{1}_{D_{t}^{\varepsilon}},\qquad D_{t}^{\varepsilon}=p_{\lambda}\left(\tfrac{2\pi}{\mathtt{T}(\lambda)}t\right)+\varepsilon O_{t}^{\varepsilon},$$
		with
		$$O_{t}^{\varepsilon}=\left\lbrace\ell e^{\ii\theta},\quad\theta\in[0,2\pi],\quad0\leqslant\ell<\sqrt{1+\varepsilon r\left(\tfrac{2\pi}{\mathtt{T}(\lambda)}t,\theta\right)}\right\rbrace,$$
		where $r\in H^s(\mathbb{T}^2)$ for $ s$ large enough and  admits the asymptotics 
		$$r(\varphi,\theta)=-\varepsilon\textnormal{Re}\left\lbrace\left[\Big(\partial_z\mathcal{R}_{\mathbf{D}}\big(p_{\lambda}(\varphi)\big)\Big)^2+\tfrac13 S(\Phi)\big(p_{\lambda}(\varphi)\big)\right]e^{2\ii\theta}\right\rbrace+O(\varepsilon^{1+\mu}),$$
		for some $\mu\in(0,1)$ and $S(\Phi)$ stands for the Schwarzian derivative of the conformal mapping $\Phi:\mathbf{D}\to\mathbb{D}$ defined through \eqref{Swart}.
	\end{theorem}
	\begin{remark}
	We believe that the non-degeneracy of the period in Theorem \ref{main thm} holds true for most energy levels. It would be surprising if the period were constant over a non-trivial open interval. Nevertheless, partial results have been obtained  in Lemma $\ref{lemm-non-degen}$ and Remark $\ref{T0 crit}$  that support this fact.\\
As to the monodromy condition stated in  theorem $\ref{main thm},$ it  sounds to be generic and  in the same time very challenging   to check for a given geometry. In  Section $\ref{Sec-admissible},$ we will provide 
 sufficient conditions to ensure the validity of this condition
 \end{remark}
 \begin{coro}\label{convex-dom}
	Theorem $\ref{main thm}$ holds true under the following assumptions
	\begin{enumerate}
	\item The  set defined by \eqref{Singular} contains only one critical point, that is, $\mathcal{C}=\{\xi_0\}$.
	\item  The conformal mapping  $F:\mathbb{D}\to \mathbf{D}$ with  $F(0)=\xi_0$ satisfies
		\begin{equation}\label{cond1 thm crit}
			{\left| \tfrac{F^{(3)}(0)}{F^\prime(0)}\right|\not\in\Big\{2\sqrt{1-\tfrac{1}{n^2}},\,\, n\in\mathbb{N}^*\cup\{\infty\}\Big\}}.
			\end{equation}
	\end{enumerate}
	\end{coro}
	Notice that this corollary generalizes Corollary \ref{cor-special}, seen in the introduction, as any bounded convex domain has only one critical point, as discussed in \cite{Cafar,Gust2}. As we will see in Section \ref{Sec-admissible}, assumption 2 in the preceding corollary holds for almost all rectangles and ellipses. Now, we intend to prove Corollary~ \ref{convex-dom}.
		
	\begin{proof} As the domain $\mathbf{D}$ is assumed to be a simply-connected bounded domain  with only one critical point. Then, the phase portrait of  the single vortex motion is foliated by periodic orbits. More precisely, for any level energy $\lambda>\lambda_{\textnormal{\tiny{min}}}$, see \eqref{Range}, the orbit $H_{\mathbf{D}}^{-1}(\{\lambda\})$ is a compact  trajectory. Applying Proposition \ref{prop-analy}, we infer that the map $\lambda\in (\lambda_{\textnormal{\tiny{min}}},\infty)\mapsto \mathtt{T}(\lambda)$ is analytic. On the other hand, Lemma \ref{lemm-non-degen} asserts that the period map is not constant near the boundary. Therefore the set of zeroes of the map $\lambda\in (\lambda_{\textnormal{\tiny{min}}},\infty)\mapsto \mathtt{T}^\prime(\lambda)$ is a discrete set denoted \mbox{by $\mathcal{Z}$.} Hence,  for any compact interval $[\lambda_*,\lambda^*]\subset (\lambda_{\textnormal{\tiny{min}}},\infty)\backslash \mathcal{Z}$ we have
	$$
				\min_{\lambda\in[\lambda_*,\lambda^*]}\left|\mathtt{T}^\prime(\lambda)\right|>0.
			$$
			which guarantees the validity of the assumption  \eqref{non degeneracy period thm}. Next, let us move to the second assumption \eqref{non degeneracy monodromy thm}. Combining \eqref{Matrix A0}, \eqref{Monod-p1} with Proposition \ref{prop-analy}, we infer that the mapping 
			\begin{align}\label{psi-inf}
			\psi: \lambda\in (\lambda_{\textnormal{\tiny{min}}},\infty)\mapsto \textnormal{Tr}\big(S_\lambda(2\pi)\big)- 2
			\end{align}
			is real analytic. Now, we shall show that this map is not identically zero  under the assumption 2. of \mbox{Corollary \ref{convex-dom}.}
			According to \eqref{Matrix A0}, \eqref{def a b}, \eqref{Conf-R} and \eqref{T0 crit}, we have at the critical point ($\lambda=\lambda_{\textnormal{\tiny{min}}}$),
	$$\mathtt{u}_{\lambda_{\textnormal{\tiny{min}}}}(\varphi)=-\frac{\ii}{4\pi}\frac{\mathtt{T}(\lambda_{\textnormal{\tiny{min}}})}{r_{\mathbf{D}}^{2}(\xi_0)}=-\ii\frac{\mathtt{a}}{\sqrt{\mathtt{a}^2-|\mathtt{b}|^2}}\qquad\textnormal{and}\qquad\mathtt{v}_{\lambda_{\textnormal{\tiny{min}}}}(\varphi)=\frac{\ii}{2\sqrt{\mathtt{a}^2-|\mathtt{b}|^2}}\Big(\partial_{z}\mathcal{R}_{\mathbf{D}}(\xi_0)\Big)^2=0,
	$$
	with
	\begin{align*}
		\mathtt{a}=\frac{1}{|F'(0)|^2}\qquad\textnormal{and}\qquad\mathtt{b}=-\frac{F^{(3)}(0)}{2[F'(0)]^3}\cdot
	\end{align*}
This implies that
	$$\mathtt{u}_{\lambda_{\textnormal{\tiny{min}}}}(\varphi)=-\frac{2\ii}{\sqrt{4-|S(F)(0)|^2}},\qquad S(F)(0)=\tfrac{F^{(3)}(0)}{F'(0)}
	$$
	and 
	$$\mathbb{A}_{\lambda_{\textnormal{\tiny{min}}}}=\tfrac{2\ii}{\sqrt{4-|S(F)(0)|^2}}\begin{pmatrix}
		-1 & 0\\
		0 & 1
	\end{pmatrix},  $$
	which is a constant matrix. From the equation \eqref{Monod-p1}, we obtain 
	$$\partial_\varphi \mathscr{M}_{\lambda_{\textnormal{\tiny{min}}}}(\varphi)=\mathbb{A}_{\lambda_{\textnormal{\tiny{min}}}} \mathscr{M}_{\lambda_{\textnormal{\tiny{min}}}}(\varphi),\qquad \mathscr{M}_{\lambda_{\textnormal{\tiny{min}}}}(0)=\textnormal{Id}.
	$$
	Consequently,
	\begin{align*}
		\forall \varphi\in\mathbb{R}, \quad \mathscr{M}_{\lambda_{\textnormal{\tiny{min}}}}(\varphi)=e^{\varphi \mathbb{A}_{\lambda_{\textnormal{\tiny{min}}}}}=\begin{pmatrix}
			e^{-\frac{2\ii\varphi}{\sqrt{4-|S(F)(0)|^2}}} & 0\\
			0 & e^{\frac{2\ii\varphi}{\sqrt{4-|S(F)(0)|^2}}}
		\end{pmatrix}.
	\end{align*}
	 Then taking the trace, we infer
	\begin{align*}
		\textnormal{Tr}\big(\mathscr{M}_{\lambda_{\textnormal{\tiny{min}}}}(2\pi)\big)=2\cos\Big(\tfrac{4\pi}{\sqrt{4-|S(F)(0)|^2}}\Big).
	\end{align*}
	Thus,
	\begin{align*}
		\textnormal{Tr}\big(\mathscr{M}_{\lambda_{\textnormal{\tiny{min}}}}(2\pi)\big)=2\qquad\Longleftrightarrow\qquad|S(F)(0)|\in\left\lbrace2\sqrt{1-\tfrac{1}{n^2}},\,\, n\in\mathbb{N}^*\cup\{\infty\}\right\rbrace.
	\end{align*}
	Therefore, in view of \eqref{psi-inf}, we get 
	\begin{align}\label{cond-secon}
	\psi(\lambda_{\textnormal{\tiny{min}}})\neq 0 \qquad\Longleftrightarrow\qquad|S(F)(0)|\notin\left\lbrace2\sqrt{1-\tfrac{1}{n^2}},\,\, n\in\mathbb{N}^*\cup\{\infty\}\right\rbrace.
	\end{align}
	Hence, the condition of the r.h.s. of \eqref{cond-secon} together with the analyticity of $\psi$ ensures that the set of zeroes of $\psi$ denoted \mbox{by $\mathcal{Z}^\prime$} is discrete.  Hence by imposing $[\lambda_*,\lambda^*]\subset (\lambda_{\textnormal{\tiny{min}}},\infty)\backslash( \mathcal{Z}\cup \mathcal{Z}^\prime)$ the conditions \eqref{non degeneracy period thm} and \eqref{non degeneracy monodromy thm} are satisfied and Theorem \ref{main thm} can be applied.
	This ends the proof of Corollary \ref{convex-dom}.
	\end{proof}
	\section{Unit disc and rigidity}
In this section, we will focus on the specific case of the unit disc  $\mathbf{D}=\mathbb{D}$. As we shall see below, the conditions  \eqref{non degeneracy period thm} and \eqref{non degeneracy monodromy thm} are satisfied ensuring the applicability of  Theorem \ref{main thm}. Furthermore, by leveraging the radial symmetry of the disc, we show  that it is possible to desingularize all the admissible energy levels without involving Cantor sets, and achieve this through rigid patch motion.
 \subsection{Validity of Theorem \ref{main thm}.}
 We will study the point vortex motion in the unit disc and show that  the particle follows a circular orbit. Afterwards, we will investigate the assumptions of  Theorem \ref{main thm}. First, 
recall that the conformal mapping here is just the identity. In particular,
\begin{equation}\label{conform disc}
	\Phi_{\mathbb{D}}(z)=z,\qquad\Phi_{\mathbb{D}}'(z)=1,\qquad\Phi_{\mathbb{D}}''(z)=0.
\end{equation}
In addition, from  the discussion of Sections \ref{sec Green Robin} and \ref{sec pt vortex motion} we know that the Hamiltonian $H_{\mathbb{D}}$ expresses as
$$
H_{\mathbb{D}}(z)=\tfrac{1}{4}\mathcal{R}_{\mathbb{D}}(z),\qquad\mathcal{R}_{\mathbb{D}}(z)=-\log(1-|z|^2).
$$
Therefore, from this structure, we observe that this Hamiltonian admits a unique critical point located at the origin and its range $H_{\mathbb{D}}(\mathbb{D})=[0,\infty).$
Let us consider a point vortex orbit $\pi\delta_{\xi(t)}$   on an energy level $\lambda>0$, parametrized as follows
$$\xi(t)=q(t)e^{\ii\Theta(t)}.$$
In view of \eqref{Vort-leap}, the point vortex dynamics is given by
$$\tfrac{\ii}{2}\,\overline{\partial_{z}\mathcal{R}_{\mathbb{D}}\big(\xi(t)\big)}=\dot{\xi}(t)=\big(\dot{q}(t)+\ii\dot{\Theta}(t)q(t)\big)e^{\ii\Theta(t)}.$$
According to \eqref{conform disc} and \eqref{Robin1N}, we have
$$\partial_{z}\mathcal{R}_{\mathbb{D}}\big(\xi(t)\big)=\frac{\overline{\xi(t)}}{1-|\xi(t)|^2}=\frac{q(t)e^{-\ii\Theta(t)}}{1-q^2(t)}\cdot$$
Putting together the previous two identities and identifying real/imaginary parts yield
$$\dot{q}(t)=0\qquad\textnormal{and}\qquad\dot{\Theta}(t)=\frac{1}{2(1-q^2(t))}\cdot$$
Integrating, we infer
$$q\equiv q_0(\lambda)\triangleq \sqrt{1-e^{-4\lambda}}\qquad\textnormal{and}\qquad\Theta(t)=\omega(\lambda)t,\qquad\omega(\lambda)\triangleq\frac{1}{2\big(1-q_0^2(\lambda)\big)}=\frac{e^{4\lambda}}{2}\cdot$$
Hence, the orbit is a circle  of radius $q_0(\lambda)$ and with minimal  period
$$\mathtt{T}(\lambda)=\frac{2\pi}{\omega(\lambda)}=4\pi e^{-4\lambda}.$$
Notice that this formulae  is compatible  with the asymptotic \eqref{asymp bnd}, given that  $L_{\partial\mathbb{D}}=2\pi$ and $\Gamma=\pi.$ Clearly  the period function is strictly monotone ensuring the condition \eqref{non degeneracy period thm} is met. It remains to check the assumption \eqref{non degeneracy monodromy thm}. Before doing so, we should  emphasize that Corollary \ref{convex-dom}  cannot be applied here because the second assumption is not satisfied. \\
From \eqref{Matrix A0} and \eqref{Monod-p1}, one can check by uniqueness of the Cauchy problem that the fundamental matrix $\mathscr{M}_\lambda$ enjoys the same structure as the generator $\mathbb{A}_{\lambda},$ that is, 
\begin{equation}\label{resolvante shape2}
	\mathscr{M}_\lambda(\varphi)=\begin{pmatrix}
		a_{\lambda}(\varphi) & b_{\lambda}(\varphi)\\
		\overline{b_{\lambda}(\varphi)} & \overline{a_{\lambda}(\varphi)}
	\end{pmatrix},\qquad a_{\lambda}(0)=1,\qquad b_{\lambda}(0)=0.
\end{equation}
By straightforward computation we get 
	\begin{align*}
		0=\partial_{\varphi}\mathscr{M}_\lambda-\mathbb{A}_{\lambda}\mathscr{M}_\lambda=\begin{pmatrix}
			a_{\lambda}^\prime-\mathtt{u}_{\lambda} a_{\lambda}-\mathtt{v}_{\lambda}\overline{b_{\lambda}} & b_{\lambda}^\prime-\mathtt{u}_{\lambda} b_{\lambda}-\mathtt{v}_{\lambda}\overline{a_{\lambda}}\\
			\overline{b_{\lambda}^\prime-\mathtt{u}_{\lambda} b_{\lambda}-\mathtt{v}_{\lambda}\overline{a_{\lambda}}} & \overline{a_{\lambda}^\prime-\mathtt{u}_{\lambda} a_{\lambda}-\mathtt{v}_{\lambda}\overline{b_{\lambda}}}
		\end{pmatrix},
	\end{align*}
	which is equivalent to solve the following system
	$$\begin{cases}
		a^\prime-\mathtt{u}_{\lambda} a_{\lambda}-\mathtt{v}_{\lambda}\overline{b_{\lambda}}=0,\\
		b_{\lambda}^\prime-\mathtt{u}_{\lambda} b_{\lambda}-\mathtt{v}_{\lambda}\overline{a_{\lambda}}=0.
	\end{cases}$$
	We set
	\begin{equation}\label{def U V}
		U(\varphi)\triangleq a_{\lambda}(\varphi)W_{\mathtt{u}_{\lambda}}^{-}(\varphi),\qquad V(\varphi)\triangleq b_{\lambda}(\varphi)W_{\mathtt{u}_{\lambda}}^{-}(\varphi),\qquad W^\pm_{\mathtt{u}_{\lambda}}(\varphi)\triangleq e^{\pm\int_{0}^{\varphi}\mathtt{u}_{\lambda}(\tau)d\tau}.
	\end{equation}
	From the initial dat in  \eqref{resolvante shape2} we deduce that
	\begin{equation}\label{U0 and V0}
		U(0)=1,\qquad V(0)=0.
	\end{equation}
	From straightforward computations we infer  
	\begin{equation*}
		U'=\mathtt{v}_{\lambda}\overline{V}W_{-2\ii\textnormal{Im}(\mathtt{u}_{\lambda})}^{-1},\qquad 
		V'=\mathtt{v}_{\lambda}\overline{U}W_{-2\ii\textnormal{Im}(\mathtt{u}_{\lambda})}^{-1},
	\end{equation*}
	with 
	\begin{equation*}
		U'(0)=0,\qquad V'(0)=\frac{\ii}{4\omega(\lambda)}\Big(\partial_{z}\mathcal{R}_{\mathbf{D}}\big(p(0)\big)\Big)^2.
	\end{equation*}
	Therefore, we deduce that both  $U$ and $V$ solve the following homogeneous linear scalar differential equation of second order
	\begin{equation}\label{EDO2}
		y^{\prime\prime}+\left(2\ii\textnormal{Im}(\mathtt{u}_{\lambda})-\tfrac{\mathtt{v}_{\lambda}^\prime}{\mathtt{v}_{\lambda}}\right)y^\prime-|\mathtt{v}_{\lambda}|^2y=0.
	\end{equation}
	Solving this  equation allows to recover the complex coefficients of  the fundamental matrix. In the particular case of the unit disc, the computations can be made explicitly, since the entries of the matrix $\mathbb{A}_{\lambda}(\varphi)$ in \eqref{Mat A} are
$$\mathtt{u}_{\lambda}=-\frac{\ii}{1-q_0^2(\lambda)}=-\ii e^{4\lambda},\qquad\mathtt{v}_{\lambda}(\varphi)=\ii\zeta(\lambda) e^{-2\ii\varphi},$$
with
$$\zeta(\lambda)\triangleq\frac{q_0^2(\lambda)}{2(1-q_0^2(\lambda))}=\frac{e^{4\lambda}-1}{2}\in(0,\infty).$$
Inserting this into \eqref{EDO2}, we find that $U$ solves the  second order EDO with constant coefficients,
$$y^{\prime\prime}-4\ii\zeta(\lambda) y^\prime-\zeta^2(\lambda)y=0.$$
Solving explicitly, we infer 
$$U(\varphi)=\tfrac{\sqrt{3}-2}{2\sqrt{3}} e^{\varphi\mu_+(\lambda)}+\tfrac{\sqrt{3}+2}{2\sqrt{3}} e^{\varphi\mu_-(\lambda)},\qquad \mu_{\pm}(\lambda)\triangleq\ii\zeta(\lambda)(2\pm\sqrt{3}).$$
Coming back to \eqref{def U V}, we find after simplifications
$$\textnormal{Re}\big(a_{\lambda}(2\pi)\big)=\cos\big(2\pi\zeta(\lambda)\sqrt{3}\big).$$
Finally,
\begin{equation}\label{non degeneracy monodromy unit disc}
	\textnormal{Tr}\big(\mathscr{M}_\lambda(2\pi)\big)=2\textnormal{Re}\big(a_{\lambda}(2\pi)\big)\neq2\qquad\Longleftrightarrow\qquad\lambda\notin\left\{\tfrac{1}{4}\log\left(1+\tfrac{2k}{\sqrt{3}}\right),\, k\in\mathbb{N}^*\right\}.
\end{equation}
Thus, under the condition
$$
[\lambda_*,\lambda^*]\subset (0,\infty)\backslash \left\{\tfrac{1}{4}\log\left(1+\tfrac{2k}{\sqrt{3}}\right),\, k\in\mathbb{N}^*\right\},
$$
Theorem   \ref{main thm} applies. However, as we shall see below, using the radial symmetry of of the unit disc we may derive  a better result. In fact,  we can remove this restriction on the energy levels and  desingularize all the positive energy levels with rigid rotating patches.
 \subsection{Rigid rotation}
 Now that we have described the periodic orbits of point vortices within the circular domain, we turn our attention to their desingularization into uniformly rotating patch motion. To achieve this, we insert the ansatz
\begin{equation}
	\xi(t)=q e^{\ii\Omega t}\quad\textnormal{and}\quad \gamma(t,\theta)=e^{\ii \Omega t} \gamma_0(\theta)
\end{equation}
into the equation \eqref{vortex patch equation new}, leading to 
\begin{equation}\label{vortex patch equation new rigid}
\begin{aligned}
	&-\varepsilon^2\tfrac{\Omega}{2}\partial_\theta\big( |\gamma_0(\theta)|^2\big)-\varepsilon q \Omega \partial_{\theta}\textnormal{Re}\big\{    \gamma_0(\theta)\big\}+\frac{1}{2\pi }\partial_{\theta}\int_{O_{0}^{\varepsilon}}\log(|\gamma_0(\theta)-\zeta |)dA(\zeta) \\ &-\frac{1}{2\pi }\partial_{\theta}\int_{O_{0}^{\varepsilon}}\log\big(\big|1-\big(\varepsilon  \overline{\gamma_0(\theta)}+q\big)\big(\varepsilon \zeta+q\big)\big)\big|dA(\zeta)=0.
			\end{aligned}
\end{equation}

\begin{figure}[!h]
	\begin{center}
		\begin{tikzpicture}[scale=0.6]
			\filldraw[draw=black,fill=black!20] (0,0) circle (4);
			\draw[draw=black] (0,0) circle (2);
			\draw[->=latex] (0,0)--(2,0);
			\node at (1,-0.3) {$q_0$};
			\filldraw[draw=black,fill=red,rotate=120] (2,0) ellipse (0.5cm and 0.3cm);
			\node at (-1,1.732) {$\bullet$};
			\draw[dashed] (0,0)--(-1,1.732);
			\node at (3,0) {$\mathbb{D}$};
			\draw[->=latex] (1,0) arc (0:120:1);
			\node at (0.9,1.1) {$\Omega t$};
		\end{tikzpicture}
	\end{center}
	\caption{Rigid rotation in the unit disc.}
\end{figure}
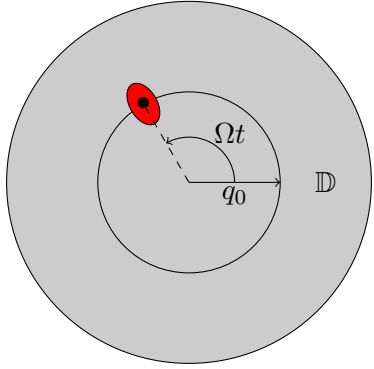

\noindent We shall parametrize the boundary of the domain $O_{0}^{\varepsilon}$ as follows,
$$
\gamma_0(\theta)=R(\theta)e^{\ii\theta}, \qquad R(\theta)\triangleq\sqrt{1+2\varepsilon q  r(\theta)}.
$$
Hence, the contour equation takes the form 
\begin{equation}\label{def calG}
\begin{aligned}
	\mathcal{G}(\varepsilon,q,\Omega,r)&\triangleq \varepsilon^2{\Omega}\partial_\theta r(\theta)+ \Omega\partial_{\theta}\big[R(\theta)\cos(\theta)\big]-\frac{1}{2\pi \varepsilon q}\partial_{\theta}\int_{O_{0}^{\varepsilon}}\log(|R(\theta)e^{\ii\theta}-\zeta |)dA(\zeta)\\
	&\quad+\frac{1}{2\pi \varepsilon q}\partial_{\theta}\int_{O_{0}^{\varepsilon}}\log\big(\big|1-\big(\varepsilon  R(\theta)e^{-\ii\theta}+q\big)\big(\varepsilon \zeta +q\big)\big|\big) dA(\zeta)=0.
			\end{aligned}
\end{equation}
The goal is to establish solutions to this nonlinear equation. To achieve this, we plan to apply the implicit function theorem within appropriate function spaces. Specifically, we will work with H\"older spaces, given  \mbox{for $s\in(0,1)$} by
$$\mathbb{Y}_{\textnormal{\tiny{odd}}}^{s}\triangleq\left\lbrace r:\mathbb{T}\to\mathbb{R}\quad\textnormal{s.t.}\quad\forall \theta\in\mathbb{T},\,\,r(\theta)=\sum_{n=1}^{\infty}\alpha_{n}\sin(n\theta),\quad\alpha_{n}\in\mathbb{R},\quad\|r\|_{C^{s}(\mathbb{T})}<\infty\right\rbrace$$
equipped with the norm
$$\|r\|_{C^{s}(\mathbb{T})}\triangleq\|r\|_{L^{\infty}(\mathbb{T})}+\sup_{(\theta_1,\theta_2)\in\mathbb{T}^2\atop\theta_1\neq\theta_2}\frac{|r(\theta_1)-r(\theta_2)|}{|\theta_1-\theta_2|^{s}}$$
and
$$\mathbb{X}_{\textnormal{\tiny{even}}}^{1,s}\triangleq\left\lbrace r:\mathbb{T}\to\mathbb{R}\quad\textnormal{s.t.}\quad\forall \theta\in\mathbb{T},\,\,r(\theta)=\sum_{n=1}^{\infty}\alpha_{n}\cos(n\theta),\quad\alpha_{n}\in\mathbb{R},\quad\|r\|_{C^{1,s}(\mathbb{T})}<\infty\right\rbrace,$$
endowed with the norm
$$\|r\|_{C^{1,s}(\mathbb{T})}\triangleq\|r\|_{L^{\infty}(\mathbb{T})}+\|r'\|_{C^{s}(\mathbb{T})}.$$
For any $a>0$, we define  the open ball
$$B_{\mathbb{X}_{\textnormal{\tiny{even}}}^{1,s}}(a)\triangleq\left\lbrace r\in \mathbb{X}_{\textnormal{\tiny{even}}}^{1,s}\quad\textnormal{s.t.}\quad\|r\|_{C^{1,s}(\mathbb{T})}<a\right\rbrace.$$
Now, we state the main result of this section.
\begin{theorem}\label{thm rot disc}
	Let $s\in(0,1)$. For any $\delta\in(0,1)$, there exist positive numbers $\varepsilon_0,a_0>0$, such that 
	\begin{enumerate}[label=(\roman*)]
		\item The functional $\mathcal{G}:(-\varepsilon_0,\varepsilon_0)\times(-1+\delta,1-\delta)\times\mathbb{R}\times B_{\mathbb{X}_{\textnormal{\tiny{even}}}^{1,s}}(a_0)\to\mathbb{Y}_{\textnormal{\tiny{odd}}}^{s}$ is well-defined and of class $C^1$.
		\item We have the equivalence
		$$\forall q\in(-1+\delta,1-\delta),\quad \mathcal{G}(0,q,\Omega,0)=0\qquad\Leftrightarrow\qquad\Omega=\Omega_0\triangleq\frac{1}{2(1-q^2)}\cdot$$
		\item The linear operator $d_{(\Omega,r)}\mathcal{G}(0,q,\Omega_0,0):\mathbb{R}\times\mathbb{X}_{\textnormal{\tiny{even}}}^{1,s}\to\mathbb{Y}_{\textnormal{\tiny{odd}}}^{s}$ is an isomorphism.
		\item There exist $C^1-$functions
		$$\Omega:(-\varepsilon_0,\varepsilon_0)\times(-1+\delta,1-\delta)\to\mathbb{R}\quad\textnormal{and}\quad r:(-\varepsilon_0,\varepsilon_0)\times(-1+\delta,1-\delta)\to B_{\mathbb{X}_{\textnormal{\tiny{even}}}^{1,s}}(a_0)$$
		satisfying
		$$\Omega(0,q)=\Omega_0\qquad\textnormal{and}\qquad r(0,q)=0,$$
		such that 
		$$\forall(\varepsilon,q)\in(-\varepsilon_0,\varepsilon_0)\times(-1+\delta,1-\delta),\quad\mathcal{G}\big(\varepsilon,q,\Omega(\varepsilon,q),r(\varepsilon,q)\big)=0.$$
	\end{enumerate}
\end{theorem}
Before giving the proof, we shall make some comments.
\begin{remark} 
	\begin{enumerate}
		\item Following exactly the same lines as \cite[Sec. 5.4]{HMW20}, we actually can show that the boundary of the uniformly rotating vortex patches inside the unit disc is analytic.
		\item Observe that 
		\begin{equation}\label{OMG00}
			\Omega(0,0)=\tfrac{1}{2}\cdot
		\end{equation}
		Now, according to \cite{HHHM16}, in the case $\overline{q}_0=0,$ the uniformly rotating solutions with amplitude $1$ bifurcate from the disc of radius $\varepsilon\in(0,1)$ at the values	
		$$\Omega_n=\frac{n-1+\varepsilon^{2n}}{2n}\cdot$$
		Thus, for a disc with amplitude $\tfrac{1}{\varepsilon^2}$ the spectrum is given by
		$$\Omega_n=\frac{n-1+\varepsilon^{2n}}{2n\varepsilon^2},$$
		which converges when $\varepsilon\to0$ only for $n=1$ (corresponding to $1$-fold solutions) and the limiting value is $\tfrac{1}{2},$ which is consistant with \eqref{OMG00} and Theorem \ref{thm rot disc}.
	\end{enumerate}
\end{remark}
\begin{proof}
$(i)$ We denote
\begin{align*}
	\mathcal{I}_1(\varepsilon,q,r)(\theta)&\triangleq \frac{1}{2\pi\varepsilon q}\partial_{\theta}\int_{O_{0}^{\varepsilon}}\log(|R(\theta)e^{\ii\theta}-\zeta|)dA(\zeta),\\
	\mathcal{I}_2(\varepsilon,q,r)(\theta)&\triangleq \frac{1}{2\pi \varepsilon q}\partial_{\theta}\int_{O_{0}^{\varepsilon}}\log\big(\big|1-\big(\varepsilon  R(\theta)e^{-\ii\theta}+q\big)\big(\varepsilon \zeta +q\big)\big|\big)dA(\zeta).
\end{align*}
According to \eqref{def calG}, it suffices to study the well-posedness and regularity of the terms $\mathcal{I}_1$ and $\mathcal{I}_2,$ the other terms being obvious. In view of \cite[Lemma 2.1]{HR22} we may write
\begin{equation}\label{integ slf ind}
\mathcal{I}_1(\varepsilon,q,r)(\theta)=\frac{1}{\varepsilon q}\int_{\mathbb{T}}\log(|R(\theta)e^{\ii\theta}-R(\eta)e^{\ii\eta} |)\partial^2_{\theta\eta}\big[R(\theta)R(\eta)\sin(\theta-\eta)\big]d\eta
\end{equation}
and 
\begin{equation}\label{integ slf ind2}
\begin{aligned}
	\mathcal{I}_2(\varepsilon,q,r)(\theta)&=-\frac{1}{2}\partial_{\theta}r(\theta)\int_{\mathbb{T}}\tfrac{R^2(\eta)}{R^2(\theta)}d\eta\\
	&\quad -\frac{1}{2\varepsilon q}\int_{\mathbb{T}}\log\big(\big|1-\big(\varepsilon  R(\theta)e^{-\ii\theta}+q\big)\big(\varepsilon R(\eta)e^{\ii\eta} +q\big)\big|\big) \partial^2_{\theta\eta}\Big[\tfrac{R(\eta)}{R(\theta)}\sin(\theta-\eta)\Big]d\eta.
			\end{aligned}
\end{equation}
First, we mention that the regularity with respect to the variable $r$ and the parity property of $\mathcal{I}_1$ have already been studied in previous works, see \cite[Sec. 3.3]{HHRZ24}. As for $\mathcal{I}_2$, since $\overline{q}_0<1,$ then taking $\varepsilon_0$ and $a_0$ small enough, the integrand inside \eqref{integ slf ind2} is not singular. This gives the desired regularity in $r$ with respect to the desired function spaces (loss of only one derivative). For the parity property, it is obtained by immediate changes of variables. Thus, it remains to proves that the quantities \eqref{integ slf ind}-\eqref{integ slf ind2} are actually not singular in $\varepsilon q.$ From the identities
\begin{align*}
\log(|R(\theta)e^{\ii\theta}- R(\eta)e^{\ii\eta} |)&=\log\big(\big|e^{\ii\theta}- e^{\ii\eta}\big|\big)+\log\big(\big|1+\tfrac{(R(\theta)-1)e^{\ii\theta}-(R(\eta)-1)e^{\ii\eta}}{e^{\ii\theta}-e^{\ii\eta}} \big|\big),\\
\forall |z|<1,\quad \log|1+z|&=\sum_{k=1}^\infty\frac{(-1)^{k+1}}{k}\textnormal{Re}\big\{z^k\big\},
\end{align*}
we obtain the following decomposition
\begin{equation}\label{I1}
	\mathcal{I}_1=\mathcal{I}_{1,1}+\mathcal{I}_{1,2}+\mathcal{I}_{1,3},
\end{equation}
with
\begin{equation}\label{I1-123}
\begin{aligned}
	\mathcal{I}_{1,1}(\varepsilon,q)(\theta)&\triangleq\frac{1}{\varepsilon q}\int_{\mathbb{T}}\log(|e^{\ii\theta}-e^{\ii\eta} |)\sin(\theta-\eta)d\eta,\\
	\mathcal{I}_{1,2}(\varepsilon,q,r)(\theta)&\triangleq\frac{1}{\varepsilon q}\int_{\mathbb{T}}\log(|e^{\ii\theta}-e^{\ii\eta} |)\partial^2_{\theta\eta}\big[\big(R(\theta)R(\eta)-1\big)\sin(\theta-\eta)\big]d\eta,\\
	\mathcal{I}_{1,3}(\varepsilon,q,r)(\theta)&\triangleq\frac{1}{\varepsilon q}\sum_{k=1}^\infty\frac{(-1)^{k+1}}{k}\int_{\mathbb{T}}\textnormal{Re}\Big\{\Big(\tfrac{(R(\theta)-1)e^{\ii\theta}-(R(\eta)-1)e^{\ii\eta}}{e^{\ii\theta}-e^{\ii\eta}}\Big)^k \Big\}\partial^2_{\theta\eta}\big[R(\theta)R(\eta)\sin(\theta-\eta)\big]d\eta.
\end{aligned}
\end{equation}
First observe that by parity argument,
\begin{equation}\label{I11}
\mathcal{I}_{1,1}(\varepsilon,q,r)=0.	
\end{equation}
Moreover, expanding the second term we get
\begin{equation*}
\begin{aligned}
	\mathcal{I}_{1,2}(\varepsilon,q,r)(\theta)&=\varepsilon q \int_{\mathbb{T}}\log(|e^{\ii\theta}-e^{\ii\eta} |)\frac{\partial_\theta r(\theta)\partial_\eta r(\eta)}{R(\theta)R(\eta)}\sin(\theta-\eta)d\eta\\
	&\quad+\int_{\mathbb{T}}\log(|e^{\ii\theta}-e^{\ii\eta} |)\left(\frac{\partial_\theta r(\theta)}{R(\theta)}R(\eta)-\frac{\partial_\eta r(\eta)}{R(\eta)}R(\theta)\right)\cos(\theta-\eta)d\eta\\
	&\quad+\frac{1}{\varepsilon q}\int_{\mathbb{T}}\log\big(\big|e^{\ii\theta}-e^{\ii\eta} |)\big(R(\theta)R(\eta)-1\big)\sin(\theta-\eta)d\eta.
\end{aligned}
\end{equation*}
Remark that one may write
$$R(\theta)R(\eta)-1 =\big(R(\theta)-1\big)\big(R(\eta)-1\big)+\big(R(\theta)-1\big)+\big(R(\eta)-1\big)$$
and, by Taylor formula,
\begin{equation}\label{Taylor R-1}
	R(\theta)-1=\varepsilon q\int_0^1\frac{(\partial_\theta r)(s\theta)}{R(s\theta)}ds.
\end{equation}
Therefore, we get
\begin{equation}\label{I12}
	\begin{aligned}
		\mathcal{I}_{1,2}(\varepsilon,q,r)(\theta)&=\varepsilon q \int_{\mathbb{T}}\log(|e^{\ii\theta}-e^{\ii\eta} |)\frac{\partial_\theta r(\theta)\partial_\eta r(\eta)}{R(\theta)R(\eta)}\sin(\theta-\eta)d\eta\\
		&\quad+\int_{\mathbb{T}}\log(|e^{\ii\theta}-e^{\ii\eta} |)\left(\frac{\partial_\theta r(\theta)}{R(\theta)}R(\eta)-\frac{\partial_\eta r(\eta)}{R(\eta)}R(\theta)\right)\cos(\theta-\eta)d\eta\\
		&\quad+\varepsilon q\int_{\mathbb{T}}\int_{0}^{1}\int_{0}^{1}\log\big(\big|e^{\ii\theta}-e^{\ii\eta} |)\frac{(\partial_\theta r)(s_1\theta)}{R(s_1\theta)}\frac{(\partial_\eta r)(s_2\eta)}{R(s_2\eta)}\sin(\theta-\eta)ds_1ds_2d\eta\\
		&\quad+\int_{\mathbb{T}}\int_{0}^{1}\log\big(\big|e^{\ii\theta}-e^{\ii\eta} |)\frac{(\partial_\theta r)(s_1\theta)}{R(s_1\theta)}\sin(\theta-\eta)ds_1d\eta\\
		&\quad+\int_{\mathbb{T}}\int_{0}^{1}\log\big(\big|e^{\ii\theta}-e^{\ii\eta} |)\frac{(\partial_\eta r)(s_2\eta)}{R(s_2\eta)}\sin(\theta-\eta)ds_2d\eta.
	\end{aligned}
\end{equation}
This implies that $\mathcal{I}_{1,2}$ is not singular in $\varepsilon q$ and actually smooth with respect to $\varepsilon$ and $q$. In addition, one obtains from \eqref{I12} that
\begin{equation}\label{I12 equal 0}
	\mathcal{I}_{1,2}(0,q,0)=0.
\end{equation} 
As for the last term in \eqref{I1} we use the identity
\begin{align*}
\textnormal{Re}\Big\{\Big(\tfrac{(R(\theta)-1)e^{\ii\theta}-(R(\eta)-1)e^{\ii\eta}}{e^{\ii\theta}-e^{\ii\eta}}\Big)^k \Big\}&=\textnormal{Re}\Big\{\Big(R(\theta)-1+\tfrac{R(\theta)-R(\eta)}{e^{\ii\theta}-e^{\ii\eta}}e^{\ii\eta}\Big)^k \Big\}\\
&=\textnormal{Re}\Big\{\Big(R(\theta)-1+\tfrac{R(\theta)-R(\eta)}{2\ii\sin(\frac{\theta-\eta}{2})}e^{\ii\frac{\eta-\theta}{2}}\Big)^k \Big\}.
\end{align*}
Using one more time Taylor formula, we can write
\begin{equation}\label{Taylor I31}
	\begin{aligned}
		R(\theta)-R(\eta)&=\varepsilon q\big(r(\theta)-r(\eta)\big)\mathtt{I}(\varepsilon,q,r,\theta,\eta),\\
		\mathtt{I}(\varepsilon,q,r,\theta,\eta)&\triangleq\int_{0}^{1}\frac{1}{\sqrt{1+2\varepsilon q\left(r(\eta)+s(r(\theta)-r(\eta))\right)}}ds.
	\end{aligned} 
\end{equation}
Combining \eqref{Taylor R-1}, \eqref{Taylor I31} and \eqref{I1-123}, we infer
\begin{equation*}
	\begin{aligned}
	&\mathcal{I}_{1,3}(\varepsilon,q,r)(\theta)\\
	&=\sum_{k=1}^{\infty}\tfrac{(-1)^{k+1}}{k}(\varepsilon q)^{k-1}\int_{\mathbb{T}}\textnormal{Re}\left\lbrace\left(\int_{0}^{1}\tfrac{(\partial_\theta r)(s\theta)}{R(s\theta)}ds+\tfrac{r(\theta)-r(\eta)}{2\ii\sin\left(\frac{\theta-\eta}{2}\right)}\mathtt{I}(\varepsilon, q,r,\theta,\eta)\right)^k\right\rbrace\partial_{\theta\eta}^{2}\big[R(\theta)R(\eta)\sin(\theta-\eta)\big]d\eta.	
	\end{aligned}
\end{equation*}
Therefore, $\mathcal{I}_{1,3}$ is not singular in $\varepsilon q$ and actually smooth with respect to $\varepsilon$ and $q.$ In addition,
\begin{equation}\label{I13 equal 0}
	\mathcal{I}_{1,3}(0,q,0)=0.
\end{equation}
The previous calculations show that $\mathcal{I}_1$ is not singular in $\varepsilon q$ and can actually been prolongated in $\varepsilon$ on an interval of the form $(-\varepsilon_0,\varepsilon_0)$ and in $q$ in an interval of the form $(\overline{q}_0-a_0,\overline{q}_0+a_0).$ Moreover, gathering \eqref{I1}, \eqref{I11}, \eqref{I12 equal 0} and \eqref{I13 equal 0} gives
\begin{equation}\label{I1 equal 0}
	\mathcal{I}_1(0,q,0)=0.
\end{equation}
For $\mathcal{I}_2(\varepsilon,q,r)$ we write
$$\mathcal{I}_2=\mathcal{I}_{2,1}+\mathcal{I}_{2,2}+\mathcal{I}_{2,3},$$
with
\begin{equation*}
\begin{aligned}
	\mathcal{I}_{2,1}(r)(\theta) &\triangleq-\frac{1}{2}\frac{\partial_{\theta}r(\theta)}{R^2(\theta)},
	\\
	\mathcal{I}_{2,2}(\varepsilon,q,r)(\theta)&\triangleq\frac{1}{\varepsilon q}\sum_{k=1}^\infty\frac{1}{k}\int_{\mathbb{T}}\textnormal{Re}\Big\{\Big(\big(\varepsilon  R(\theta)e^{-\ii\theta}+q\big)\big(\varepsilon R(\eta)e^{\ii\eta} +q\big)\Big)^k\Big\}\sin(\theta-\eta)d\eta,
	\\
	\mathcal{I}_{2,3}(\varepsilon,q,r)(\theta)&\triangleq\frac{1}{\varepsilon q}\sum_{k=1}^\infty\frac{1}{k}\int_{\mathbb{T}}\textnormal{Re}\Big\{\Big(\big(\varepsilon  R(\theta)e^{-\ii\theta}+q\big)\big(\varepsilon R(\eta)e^{\ii\eta} +q\big)\Big)^k\Big\} \partial^2_{\theta\eta}\Big[\Big(\tfrac{R(\eta)}{R(\theta)}-1\Big)\sin(\theta-\eta)\Big]d\eta.
\end{aligned}
\end{equation*}
First observe that 
\begin{equation}\label{I21 equal 0}
	\mathcal{I}_{2,1}(0)=0.
\end{equation}
Notice that for any $k\in\mathbb{N}^*,$
\begin{equation*}
\begin{aligned}
		\mathcal{J}_k(\varepsilon,q,r)(\theta)&\triangleq\frac{1}{\varepsilon q}\int_{\mathbb{T}}\textnormal{Re}\Big\{\Big(\big(\varepsilon  R(\theta)e^{-\ii\theta}+q\big)\big(\varepsilon R(\eta)e^{\ii\eta} +q\big)\Big)^k\Big\} \sin(\theta-\eta)d\eta\\
	&=\frac{1}{\varepsilon q}\sum_{n=0}^k\sum_{m=0}^k\binom{k}{n}\binom{k}{m}q^{2k-m-n} \varepsilon^{n+m} R^n(\theta)\textnormal{Re}\Big\{   e^{-\ii n\theta} \int_{\mathbb{T}}R^m(\eta) e^{\ii m\eta}  \sin(\theta-\eta)d\eta\Big\}\\
	&=\frac{1}{\varepsilon q}\sum_{n=0}^k\sum_{m=0}^k\binom{k}{n}\binom{k}{m} q^{2k-m-n} \varepsilon^{n+m} R^n(\theta)\textnormal{Re}\Big\{   e^{-\ii n\theta} \int_{\mathbb{T}} e^{\ii m\eta}  \sin(\theta-\eta)d\eta \Big\}\\
	&\quad+\frac{1}{\varepsilon q}\sum_{n=0}^k\sum_{m=0}^k\binom{k}{n}\binom{k}{m} q^{2k-m-n} \varepsilon^{n+m} R^n(\theta)\textnormal{Re}\Big\{   e^{-\ii n\theta} \int_{\mathbb{T}}(R^m(\eta)-1) e^{\ii m\eta}  \sin(\theta-\eta)d\eta\Big\}.
	\end{aligned}
\end{equation*}
Also, straightforward computations give
$$\textnormal{Re}\Big\{e^{-\ii n\theta} \int_{\mathbb{T}} e^{\ii m\eta}  \sin(\theta-\eta)d\eta \Big\}=\begin{cases}
	0, & \textnormal{if}\quad m\neq 1,\\
	-\tfrac12\sin\big((n-1)\theta\big), &\textnormal{if}\quad m=1.
\end{cases}$$
Therefore,
\begin{equation*}
\begin{aligned}
		\mathcal{J}_{k}(\varepsilon,q,r)(\theta)&=\frac{k}{2}q^{2(k-1)}\sin(\theta)+\mathcal{J}_{k,1}(\varepsilon,q,r)(\theta),\\
		\mathcal{J}_{k,1}(\varepsilon,q,r)(\theta)&\triangleq-\frac{k}{2 }\sum_{n=2}^k\binom{k}{n}q^{2(k-1)-n} \varepsilon^{n} R^n(\theta)\sin\big((n-1)\theta\big)\\
	&\quad+\frac{1}{\varepsilon q}\sum_{n=0}^k\sum_{m=0}^k\begin{pmatrix}
		k\\ n
	\end{pmatrix}\begin{pmatrix}
		k\\ m
	\end{pmatrix}q^{2k-m-n} \varepsilon^{n+m} R^n(\theta)\textnormal{Re}\Big\{   e^{-\ii n\theta} \int_{\mathbb{T}}(R^m(\eta)-1) e^{\ii m\eta}  \sin(\theta-\eta)d\eta \Big\}.
			\end{aligned}
\end{equation*}
Observe that
$$R^{m}(\eta)-1=\big(R(\eta)-1\big)\sum_{\ell=0}^{m-1}R^{m-1-\ell}(\eta).$$
Hence, making appeal to \eqref{Taylor R-1}, we infer
\begin{equation*}
	\begin{aligned}
		&\mathcal{J}_{k,1}(\varepsilon,q,r)(\theta)= -\frac{k}{2 }\sum_{n=2}^k\binom{k}{n}q^{2(k-1)-n} \varepsilon^{n} R^n(\theta)\sin\big((n-1)\theta\big)\\
		&\quad+\sum_{n=0}^k\sum_{m=0}^k\sum_{\ell=0}^{m-1}\begin{pmatrix}
			k\\ n
		\end{pmatrix}\begin{pmatrix}
			k\\ m
		\end{pmatrix}q^{2k-m-n} \varepsilon^{n+m} R^n(\theta)\textnormal{Re}\Big\{ e^{-\ii n\theta} \int_{\mathbb{T}}\int_0^1\tfrac{(\partial_{\eta}r)(s\eta)}{R(s\eta)}dsR^{m-1-\ell}(\eta)e^{\ii m\eta}\sin(\theta-\eta)d\eta \Big\}.
	\end{aligned}
\end{equation*}
We have removed the singularity and it is immediate that
\begin{equation}\label{Jk1 equal 0}
	\mathcal{J}_{k,1}(0,q,0)=0.
\end{equation}
Thus, we have the following decomposition
\begin{align*}
	\mathcal{I}_{2,2}(\varepsilon,q,r)(\theta)&=\tfrac{1}{2}\sum_{k=1}^{\infty}q^{2(k-1)}\sin(\theta)+\sum_{k=1}^{\infty}\tfrac{1}{k}\mathcal{J}_{k,1}(\varepsilon,q,r)(\theta)\\
	&\triangleq\frac{1}{2(1-q^2)}\sin(\theta)+\widetilde{\mathcal{I}}_{2,2}(\varepsilon,q,r)(\theta).
\end{align*}
From what preceeds, $\mathcal{I}_{2,2}$ and $\widetilde{\mathcal{I}}_{2,2}$ are smooth in $\varepsilon$ and $q$. Moreover, \eqref{Jk1 equal 0} implies
\begin{equation}\label{I22tilde equal 0}
	\widetilde{\mathcal{I}}_{2,2}(0,q,0)=0.
\end{equation}
Besides, another use of \eqref{Taylor R-1} yields
\begin{align*}
	\mathcal{I}_{2,3}(\varepsilon,q,r)(\theta)&=\frac{1}{\varepsilon q}\sum_{k=1}^\infty\frac{1}{k}\int_{\mathbb{T}}\textnormal{Re}\Big\{\Big(\big(\varepsilon  R(\theta)e^{-\ii\theta}+q\big)\big(\varepsilon R(\eta)e^{\ii\eta} +q\big)\Big)^k\Big\} \partial^2_{\theta\eta}\Big[\Big(\tfrac{R(\eta)}{R(\theta)}-1\Big)\sin(\theta-\eta)\Big]d\eta\\
	&=-\sum_{k=1}^\infty\frac{1}{k}\int_{\mathbb{T}}\textnormal{Re}\Big\{\Big(\big(\varepsilon  R(\theta)e^{-\ii\theta}+q\big)\big(\varepsilon R(\eta)e^{\ii\eta} +q\big)\Big)^k\Big\} \partial_{\theta}\Big[\int_{0}^{1}\tfrac{(\partial_{\eta}r)(s\eta)}{R(\theta)R(s\eta)}ds\cos(\theta-\eta)\Big]d\eta\\
	&\quad+\sum_{k=1}^\infty\frac{1}{k}\int_{\mathbb{T}}\textnormal{Re}\Big\{\Big(\big(\varepsilon  R(\theta)e^{-\ii\theta}+q\big)\big(\varepsilon R(\eta)e^{\ii\eta} +q\big)\Big)^k\Big\} \partial_{\theta}\Big[\tfrac{(\partial_{\eta}r)(\eta)}{R(\theta)R(\eta)}\sin(\theta-\eta)\Big]d\eta.
\end{align*}
One readily sees that $\mathcal{I}_{2,3}$ is smooth and that
\begin{equation}\label{I23 equal 0}
	\mathcal{I}_{2,3}(0,q,0)=0.
\end{equation}
Combining the foregoing calculations, we find that
\begin{equation}\label{vortex patch equation new rigid2}
	\mathcal{G}(\varepsilon,q,\Omega,r)=\varepsilon^2\Omega\partial_{\theta}r(\theta)+\Omega\partial_{\theta}\big[R(\theta)\cos(\theta)\big]+\frac{1}{2(1-q^2)}\sin(\theta)+\mathcal{I}_{3}(\varepsilon,q,r)(\theta),
\end{equation}
where
\begin{align*}
	\mathcal{I}_{3}&\triangleq\mathcal{I}_1+\mathcal{I}_{2,1}+\widetilde{\mathcal{I}}_{2,2}+\mathcal{I}_{2,3}
\end{align*}
is a smooth function. Finally, putting together \eqref{I1 equal 0}, \eqref{I21 equal 0}, \eqref{I22tilde equal 0} and \eqref{I23 equal 0}, we get
\begin{equation}\label{I3 equal 0}
	\mathcal{I}_{3}(0,q,0)=0.
\end{equation}
\noindent $(ii)$ From \eqref{vortex patch equation new rigid2} and \eqref{I3 equal 0}, we find
\begin{equation}\label{vortex patch equation new rigid2 r=0}
\mathcal{G}(0,q,\Omega,0)=-\Big[\Omega-\tfrac{1}{2(1-q^2)}\Big]\sin(\theta).
\end{equation}
This gives the desired equivalence.\\

\noindent $(iii)$ The linearized operator of $\mathcal{G}$ with respect to $(\Omega,r)$ at $(\varepsilon, \Omega,r)=(0, \Omega_0,0)$ is given by
$$d_{(\Omega,r)}\mathcal{G}(0,q,\Omega_0,0)[(\widehat\Omega,h)]=-\widehat\Omega\sin(\theta)-\tfrac12\big[\partial_\theta-\mathbf{H} \big]h(\theta).$$
Given $g\in\mathbb{Y}_{\textnormal{\tiny{odd}}}^{s}$ in the form 
$$g(\theta)=\sum_{n=1}^{\infty}g_n\sin(n\theta),\quad g_n\in \mathbb{R}.$$
Then, we choose
$$\widehat{\Omega}=-g_1\qquad\textnormal{and}\qquad\forall n\geqslant2,\quad h_n=-\frac{g_n}{nB_n},\qquad B_n\triangleq\frac{n-1}{2n}$$
so that setting
$$h(\theta)=\sum_{n=2}^{\infty}h_n\cos(n\theta),$$
we find
$$d_{(\Omega,r)}\mathcal{G}(0,q,\Omega_0,0)[(\widehat\Omega,h)]=g.$$
Moreover, since for any $n\geqslant2,$ we have $B_n\in(\tfrac{1}{4},\tfrac{1}{2}),$ we immediately get from Cauchy-Schwarz and Bessel inequalities
$$\|h\|_{L^{\infty}(\mathbb{T})}\lesssim \sum_{n=2}^{\infty}\frac{|g_n|}{n}\lesssim\|g\|_{L^{2}(\mathbb{T})}\lesssim\|g\|_{C^{s}(\mathbb{T})}.$$
We can also write
$$h'(\theta)=\sum_{n=2}^{\infty}\frac{g_n}{B_n}\sin(n\theta)=(\widetilde{g}\ast g)(\theta)+2\big[g(\theta)-g_1\sin(\theta)\big],$$
where
$$\widetilde{g}(\theta)\triangleq-\sum_{n=2}^{\infty}\frac{2}{n-1}\cos(n\theta),\qquad(\widetilde{g}\ast g)(\theta)\triangleq2\int_{\mathbb{T}}\widetilde{g}(\theta-\eta)g(\eta)d\eta.$$
Notice that $\widetilde{g}\in L^{2}(\mathbb{T})\subset L^{1}(\mathbb{T}).$ Thus, using that $L^{1}(\mathbb{T})\ast C^{s}(\mathbb{T})\to C^{s}(\mathbb{T})$, we obtain $\widetilde{g}\ast g\in C^{s}(\mathbb{T}).$ Hence $h'\in C^{s}(\mathbb{T})$. This proves that $h\in\mathbb{X}_{\textnormal{\tiny{even}}}^{1,s}$ and therefore that $d_{(\Omega,r)}\mathcal{G}(0,q,\Omega_0,0):\mathbb{X}_{\textnormal{\tiny{even}}}^{1,s}\to\mathbb{Y}_{\textnormal{\tiny{odd}}}^{s}$ is an isomorphism.\\

\noindent $(iv)$ It follows from the previous points by applying the Implicit Function Theorem.\\
The proof of Theorem \ref{thm rot disc} is now complete.
\end{proof}

		\section{Applications of Corollary \ref{convex-dom}}\label{Sec-admissible}
		In this section, we focus on the application of Corollary \ref{convex-dom} to specific domains. Since the domains we will consider are convex, the first assumption of this corollary is automatically satisfied. Therefore, we only need to verify the second assumption in Corollary \ref{convex-dom} for the admissibility of these domains. This condition requires knowledge of the critical point and the conformal mapping.
We have two strategies to check this condition. The first strategy involves providing the conformal mapping that satisfies the required constraint, from which we can derive the geometry. The second strategy involves defining the geometry first and then verifying the constraint. This second approach that we will develop here, which  is particularly challenging due to the complex structure of the conformal mapping, even for simple geometries.
We will focus on ellipses, rectangles, and more generally, polygonal domains. All these examples are convex bounded domains making  Corollary \ref{convex-dom}   applicable whenever its second assumption holds true, that is,
\begin{equation}\label{cond-nondeg}
			{\left| \tfrac{F^{(3)}(0)}{F^\prime(0)}\right|\not\in\Big\{2\sqrt{1-\tfrac{1}{n^2}},\,\, n\in\mathbb{N}^*\Big\}}.
			\end{equation}
			with $F:\mathbb{D}\to \mathbf{D}$ the conformal mapping such that $F(0)=\xi_0$,  is the critical point of the Robin function. In terms of the conformal mapping $\Phi=F^{-1}: \mathbf{D}\to \mathbb{D}$, this condition is equivalent to
			\begin{equation}\label{cond-nondeg1}
			{\left| \tfrac{\Phi^{(3)}(0)}{(\Phi^\prime(0))^3}\right|\not\in\Big\{2\sqrt{1-\tfrac{1}{n^2}},\,\, n\in\mathbb{N}^*\Big\}}.
			\end{equation}
To check one of these conditions, we need to get access to the conformal mapping structures.

		\subsection{Elliptic domains}
		We want  to check the validity of the  Corollary \ref{convex-dom} with  one of the most simplest domain shapes, namely ellipses. As we have mentioned before, since ellipses are convex, the first assumption of this corollary is automatically satisfied. Thus it remains to check  the second assumption reformulated in \eqref{cond-nondeg}.   			
		Consider  $\mathbf{D}=\mathbf{E}_{a}$ the domain delimited by  the renormalized ellipse with semi-axis $a>1$ and $b=1,$
		\begin{align*}
			\mathbf{E}_{a}\triangleq\Big\{(x,y)\in\mathbb{R}^2\quad\textnormal{s.t.}\quad\tfrac{x^2}{a^2}+y^2<1\Big\}.
		\end{align*}
		Using the symmetry of this ellipse, we can show that the critical point of Robin function will be given by $\xi_0=0.$ On the other hand, it is well known \cite{KS08} (see also \cite[p. 296]{N52}, that the conformal mapping $\Phi:\mathbf{E}_{a}\to\mathbb{D}$ is given by
		\begin{equation}\label{conformal map ellipse}
			\Phi(z)=\sqrt{k}\,\mathtt{sn}\left(\frac{2K(k)}{\pi}\sin^{-1}\left(\frac{z}{\sqrt{a^2-1}}\right);k\right),
		\end{equation}
		where $K:[0,1)\to[\tfrac{\pi}{2},\infty)$ is the complete Legendre elliptic integral of first kind defined by
		\begin{equation}\label{complete elliptic def}
			K(x)\triangleq\int_0^{\frac\pi2}\tfrac{d\varphi}{\sqrt{1-x^2\sin^2(\varphi)}}.
		\end{equation}
		The application $\mathtt{sn}$ is the Jacobi elliptic sinus amplitudinus function and $k$ is defined through
		\begin{equation}\label{Gk}
			G(k)=\frac{2}{\pi}\sinh^{-1}\left(\frac{2a}{a^2-1}\right),
		\end{equation}
		with $G:(0,1]\to\mathbb{R}$ the strictly decreasing function given by
		\begin{equation}\label{def modulus G}
			G(x)\triangleq\frac{K\big(\sqrt{1-x^2}\big)}{K(x)}\cdot
		\end{equation}
		The previous relation gives a bijective correspondance between $k\in(0,1)$ and $a>1$, with
		\begin{align*}
			a\to1\,\,\leftrightarrow\,\, k\to0\qquad\textnormal{and}\qquad a\to\infty\,\,\leftrightarrow\,\, k\to1.
		\end{align*}
		Our result reads as follows.
		\begin{proposition}\label{prop ellipse}
			Let $a>1$ and $k=k(a)\in(0,1)$ defined through \eqref{Gk}. Introduce the function
			\begin{align*}
				g(k)\triangleq\left(1-\frac{1}{4k^2}\left(\big(1+k^2\big)-\left(\frac{\pi}{2K(k)}\right)^2\right)^2\right)^{-\frac{1}{2}}.
			\end{align*}
			The function $g:[0,1)\to\mathbb{R}_+$ is continuous, strictly increasing and satisfies
			\begin{align*}
				g(0)=1\qquad\textnormal{and}\qquad\lim_{k\to1}g(k)=\infty.
			\end{align*}
			For any $n\in\mathbb{N}^*,$ there exists a unique $k_n\in[0,1)$ (and therefore a unique $a_n>1$) such that
			\begin{align*}
				g(k_n)=n.
			\end{align*}
			Finally, for any $a\in(1,\infty)\setminus\{a_n,\,n\in\mathbb{N}\}$, the non-degeneracy condition \eqref{cond-nondeg1} holds.		\end{proposition}
		
		\begin{figure}[!h]
			\centering
			\includegraphics[width=7.5cm]{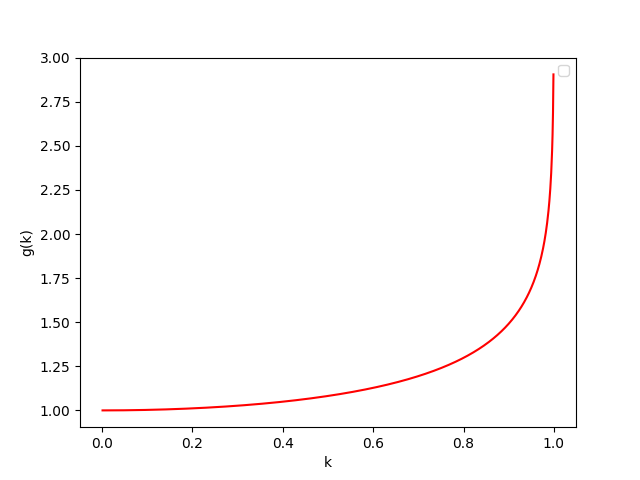}
			\caption{Graph of the function $g$.}
		\end{figure}
		
		\begin{proof}
			First, we recall  the following power series expansions, see for instance  \cite[p. 81]{AS64},
			\begin{align*}
				\sin^{-1}(z)=\sum_{n=0}^{\infty}\frac{(2k-1)!!}{(2k)!!}\frac{z^{2k+1}}{2k+1}=z+\tfrac{1}{6}z^3+\sum_{n=2}^{\infty}\frac{(2k-1)!!}{(2k)!!}\frac{z^{2k+1}}{2k+1}
			\end{align*}
			and, see \cite{S76},
			\begin{align*}
				\mathtt{sn}(z,k)=z-\tfrac{1}{6}\big(1+k^2\big)z^3+\sum_{n=2}^{\infty}b_n(k)z^{2n+1},\qquad b_n(k)\in\mathbb{R}.
			\end{align*}
			Inserting this into \eqref{conformal map ellipse}, we find
			\begin{align*}
				\Phi(z)=a_1(k)z+\sum_{n=3}^{\infty}a_n(k)z^n,
			\end{align*}
			and  in particular, we have 
			\begin{align*}
				a_1(k)=\frac{2K(k)\sqrt{k}}{\pi\sqrt{a^2-1}},\qquad a_3(k)=\frac{K(k)\sqrt{k}}{3\pi(a^2-1)^{\frac{3}{2}}}\left[1-\frac{4K^2(k)}{\pi^2}\big(1+k^2\big)\right].
			\end{align*}
			Recall that the mapping  $x\mapsto K(x)$ is continuous increasing on $[0,1)$ with
			\begin{equation}\label{limits K}
				K(0)=\tfrac{\pi}{2},\qquad\lim_{x\to1}K(x)=\infty.
			\end{equation}
			In particular,
			\begin{align*}
				K(k)>\tfrac{\pi}{2}\cdot
			\end{align*}
			Hence,
			\begin{align*}
				\frac{|a_3(k)|}{a_1^3(k)}&=\frac{\pi^2}{24kK^2(k)}\left(\frac{4K^2(k)}{\pi^2}\big(1+k^2\big)-1\right)\\
				&=\frac{1}{6k}\left(\big(1+k^2\big)-\left(\frac{\pi}{2K(k)}\right)^2\right).
			\end{align*}
			Denote by $M$ the arithmetic-geometric mean function defined on $\{(\alpha,\beta)\in(0,\infty)^2\}$ by
			\begin{align*}
				M(\alpha,\beta)\triangleq\lim_{n\to\infty}\alpha_n=\lim_{n\to\infty}\beta_n,
			\end{align*}
			where the sequences $(\alpha_n)_{n\in\mathbb{N}}$ and $(\beta_n)_{n\in\mathbb{N}}$ are given by
			\begin{align*}
				(\alpha_0,\beta_0)=(\alpha,\beta)\qquad\textnormal{and}\qquad\forall n\in\mathbb{N},\quad(\alpha_{n+1},\beta_{n+1})=\left(\tfrac{\alpha_n+\beta_n}{2},\sqrt{\alpha_n\beta_n}\right).
			\end{align*}
			According to \cite[p. 66-67]{MM99}, we have
			\begin{equation}\label{majoration Kk}
				K(k)=\frac{\pi}{2M(1,\sqrt{1-k^2})}\cdot
			\end{equation}
			 Now, we turn to condition \eqref{cond-nondeg1} which is equivalent to 
			\begin{equation}\label{n gk}
				\begin{aligned}
					&\frac{|a_3(k)|}{a_1^3(k)}\neq\frac{1}{3}\sqrt{1-\tfrac{1}{n^2}}
					\Longleftrightarrow&\quad n\neq\left(1-\frac{1}{4k^2}\left(\big(1+k^2\big)-\left(\frac{\pi}{2K(k)}\right)^2\right)^2\right)^{-\frac{1}{2}}\triangleq g(k).
				\end{aligned}
			\end{equation}
			The function $K$ admits the following power series expansion, see \cite[p. 591]{AS64}
			\begin{align*}
				K(k)=\frac{\pi}{2}\sum_{n=0}^{\infty}\left(\frac{(2n)!}{2^{2n}(n!)^2}\right)^2k^{2n}.
			\end{align*}
			Consequently, the map  $g:[0,1)\to\mathbb{R}_+$ is continuous and satisfies, in view of \eqref{limits K},
			\begin{equation}\label{limits g}
				g(0)=1\qquad\textnormal{and}\qquad\lim_{k\to1}g(k)=\infty.
			\end{equation}
			Now, let us turn to the monotonicity of $g$. The derivative of $g$ is
			\begin{align*}
				g'(k)&=\frac{g^3(k)}{2}\left(\big(1+k^2\big)-\left(\frac{\pi}{2K(k)}\right)^2\right)\left[\frac{\left(k+\frac{\pi^2K'(k)}{4K^3(k)}\right)}{k^2}-\frac{\big(1+k^2\big)-\left(\frac{\pi}{2K(k)}\right)^2}{2k^3}\right]\\
				&=\frac{g^3(k)}{4k^3}\left(\big(1+k^2\big)-\left(\frac{\pi}{2K(k)}\right)^2\right)\left[k^2-1+\left(\frac{\pi}{2K(k)}\right)^2\left(1+\frac{2kK'(k)}{K(k)}\right)\right].
			\end{align*}
			From  \eqref{majoration Kk}, we deduce that
			\begin{align*}
				K(k)<\frac{\pi}{2\sqrt{1-k^2}},
			\end{align*}
			which implies
			\begin{equation*}
				(1-k^2)\left(\frac{2K(k)}{\pi}\right)^2<1.
			\end{equation*}
			Since $K$ is positive and strictly increasing, we have proved that
			\begin{equation}\label{Maj0}
				(1-k^2)\left(\frac{2K(k)}{\pi}\right)^2<1+\frac{2kK'(k)}{K(k)}\cdot
			\end{equation}
			Finally, \eqref{Maj0} implies $g'(k)>0$, and therefore $g$ is strictly increasing on $[0,1)$. Combined with \eqref{limits g}, we obtain that $g:[0,1)\to[1,\infty)$ is a bijection. Hence, from  \eqref{n gk} we get  the proof of Proposition \ref{prop ellipse}.
		\end{proof}
		
		%
		
		\subsection{Polygonal domains}
		The main goal in this section is to explore some polygonal shapes that satisfy the assumption \eqref{cond-nondeg}. This will be implemented through the Schwarz-Christoffel mapping. The Schwarz-Christoffel transformation is a powerful tool in complex analysis that describes the conformal mapping of unit disc  onto the interior of a polygon. This mapping provides a useful framework for examining the geometric properties and behaviors of various polygonal domains, particularly in the context of the non-degeneracy condition \eqref{cond-nondeg}.
We begin this section with the following remark showing that for highly symmetric polygons, the condition \eqref{cond-nondeg} fails.  This observation underscores the necessity of exploring less symmetric, more irregular polygonal shapes to satisfy the given assumption.

		\begin{remark}
			Let $m\geqslant 3$ be an integer and consider a regular polygon $\mathbf{P}_m$ with $m$ sides, centered at $0$ and with length side
			\begin{align*}
				\frac{2^{1-\frac{4}{m}}\Gamma^2\left(\frac{1}{2}-\frac{1}{m}\right)}{m\Gamma\left(1-\frac{2}{m}\right)}\cdot
			\end{align*}
			It is well-known, see for instance $\cite[p. 196]{N52},$ that the associated conformal mapping $F:\mathbb{D}\to\mathbf{P}_m$ is given by
			\begin{equation}\label{SC reg poly}
				F(z)=\int_{0}^{z}\frac{d\xi}{(1-\xi^m)^{\frac{2}{m}}}\cdot
			\end{equation}
			From this we deduce that
			\begin{align*}
				F(0)=F''(0)=F^{(3)}(0)=0,\qquad F'(0)=1.
			\end{align*}
			Consequently, $0$ is the unique (since $\mathbf{P}_m$ is convex) critical point of the Robin function $\mathcal{R}_{\mathbf{P}_m}$ and
			\begin{align*}
				S(F)(0)=0.
			\end{align*}
			Thus, the condition   \eqref{cond-nondeg} fails.
		\end{remark}
		The previous remark teaches us that we need to look for less symmetric configurations. This is what we will explore in the next subsections. Before entering into details, let us discuss the general theory of conformal mapping for polygonal domains discovered by Schwarz and Christoffel. We refer the reader to \cite[p. 189]{N52} for the theory. The Schwarz-Christoffel mapping provides a way to map the unit disk conformally onto the interior of the polygon $\mathbf{P}$. This transformation is essential for understanding the geometric properties of the polygonal domain.
Consider $\mathbf{P}$ a polygon with $m\geqslant 3$ vertices labelled $z_1,$ $z_2,$ ..., $z_m.$ We denote $\pi\alpha_1,$ $\pi\alpha_2,$ ..., $\pi\alpha_m$ as its interior angles, see Figure \ref{Figure polygon}. The associated exterior angles $\pi\mu_1,$ $\pi\mu_2$, ..., $\pi\mu_{m}$ are defined by
		\begin{align*}
			\forall k\in\llbracket 1,m\rrbracket,\quad\mu_{k}\triangleq1-\alpha_{k}\in(-1,1)
		\end{align*}
		and  they satisfy the identity
		\begin{equation}\label{sum ext angles tri}
			\sum_{k=1}^{m}\mu_k=2.
		\end{equation}
		Then, there exist $\alpha,\beta\in\mathbb{C},$ determining respectively the size and the position of the polygon $\mathbf{P}$, and distinct angles $\theta_1$, $\theta_2$, ... , $\theta_{m}\in[0,2\pi)$ such that the following application maps conformally the unit disc $\mathbb{D}$ onto the interior of the polygon $\mathbf{P}$
		\begin{equation}\label{Swart-Christ}
			F(z)=\alpha\bigintsss_{0}^{z}\prod_{k=1}^{m}\big(\xi-e^{\ii\theta_k}\big)^{-\mu_k}d\xi+\beta,
		\end{equation}
		with the additional property
		\begin{align*}
			\forall k\in\llbracket 1,m\rrbracket,\quad F\big(e^{\ii\theta_k}\big)=z_k.
		\end{align*}
		
		\begin{figure}[!h]
			\begin{center}
				\begin{tikzpicture}[scale=0.85]
					\filldraw[draw=black,fill=black!20] (-4,0) circle (2);
					\node at (-3,1.7) {$\bullet$};
					\node at (-2.6,2) {$e^{\ii\theta_1}$};
					\node at (-5,1.7) {$\bullet$};
					\node at (-5,2) {$e^{\ii\theta_2}$};
					\node at (-5.93,0.5) {$\bullet$};
					\node at (-6.3,0.7) {$e^{\ii\theta_3}$};
					\node at (-5.93,-0.5) {$\bullet$};
					\node at (-6.3,-0.7) {$e^{\ii\theta_4}$};
					\node at (-4.1,-2) {$\bullet$};
					\node at (-4.1,-2.3) {$e^{\ii\theta_5}$};
					\node at (-4,0) {$\mathbb{D}$};
					\draw[>=latex,->](-0.5,0)--(0.5,0);
					\node at (0,0.3) {$F$};
					\filldraw[draw=black,fill=black!20] (2,-1) -- (7,-2) -- ++(94:4) -- ++(210:2.5) -- ++(120:2)-- cycle;
					\node at (7.2,-2.2) {$z_1$};
					\node at (6.9,2.2) {$z_2$};
					\node at (4.65,1.1) {$z_3$};
					\node at (3.5,2.7) {$z_4$};
					\node at (1.8,-1.2) {$z_5$};
					\node at (6.6,-1.7) {$\pi\alpha_1$};
					\node at (6.37,1.4) {$\pi\alpha_2$};
					\node at (4.6,0.5) {$\pi\alpha_3$};
					\node at (3.63,1.7) {$\pi\alpha_4$};
					\node at (2.55,-0.8) {$\pi\alpha_5$};
					\draw (6.95,-1.3) arc (94:140:1.4);
					\draw (3,-1.2) arc (0:54:1.4);
					\draw (4.35,1.1) arc (120:400:0.5);
					\draw (5.8,1.45) arc (200:298:0.7);
					\draw (3.16,1.6) arc (240:300:0.9);
					\node at (5,-0.5) {$\mathbf{P}$};
					\node at (2,-1) {$\bullet$};
					\node at (7,-2) {$\bullet$};
					\node at (6.73,2) {$\bullet$};
					\node at (4.57,0.74) {$\bullet$};
					\node at (3.56,2.45) {$\bullet$};
				\end{tikzpicture}
			\end{center}
			\caption{Schwarz Christoffel conformal mapping for polygonal domains.}\label{Figure polygon}
		\end{figure}
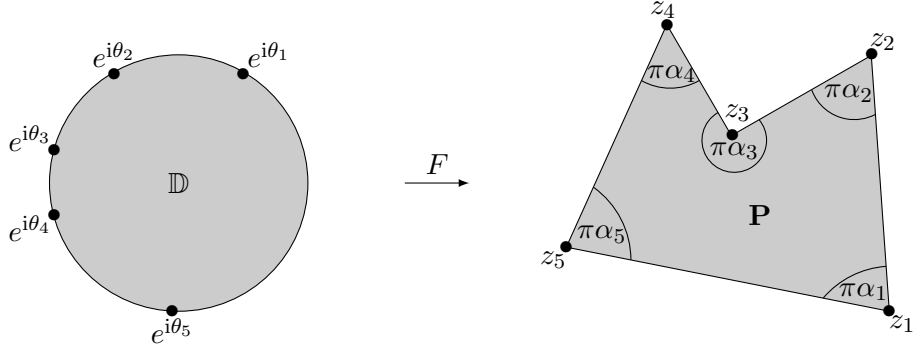
		
		\noindent Observe that the map in \eqref{SC reg poly} is nothing but  the Schwarz-Christoffel conformal mapping with 
		\begin{align*}
			\alpha=1,\qquad \beta=0\qquad\textnormal{and}\qquad\forall k\in\llbracket1,m\rrbracket,\quad\theta_k=\tfrac{2k\pi}{m}\cdot
		\end{align*}
		It is readily seen from the expression of $F$ that
		\begin{equation}\label{trii}
			\frac{F''(z)}{F'(z)}=-\sum_{k=1}^{m}\frac{\mu_k}{z-e^{\ii\theta_k}}\cdot
		\end{equation}
		Differentiating \eqref{trii}, we get
		\begin{align*}
			\frac{F^{(3)}(z)}{F'(z)}-\left(\frac{F''(z)}{F'(z)}\right)^2=\sum_{k=1}^{m}\frac{\mu_k}{(z-e^{\ii\theta_k})^2}\cdot
		\end{align*}
		Therefore, by vitue of \eqref{Swart}, the Schwarzian derivative of $F$ is given by 
		\begin{equation}\label{Swart poly}
			S(F)(z)=\sum_{k=1}^{m}\frac{\mu_k}{(z-e^{\ii\theta_k})^2}-\frac{1}{2}\left(\sum_{k=1}^{m}\frac{\mu_k}{z-e^{\ii\theta_k}}\right)^2.
		\end{equation}
		Notice that  the convexity of the polygon $\mathbf{P}$ is equivalent to require
		$$\forall k\in\llbracket1,m\rrbracket,\quad\mu_{k}>0.$$
			\subsubsection{Symmetric convex polygons}
	Here we will explore the non-degeneracy condition \eqref{cond-nondeg} for symmetric polygons and see a concrete  application for rectangles.  We consider   a polygon $\mathbf{P}$ with $2m$ vertices ($m\geqslant2$) denoted $(z_k)_{1\leqslant k\leqslant2m}$ subject to the following configuration,
	\begin{enumerate}
		\item (Symmetry with respect to $0$) The vertices satisfy the following properties
		\begin{align}\label{symm-pol1}
			\forall k\in\llbracket 1,m\rrbracket,\quad\,z_{m+k}=-z_k,\qquad \textnormal{Im}(z_k)\geqslant0.
		\end{align}
		\item (Convexity) The exterior angles satisfy
		\begin{equation}\label{muk pos}
			\forall k\in\llbracket1,2m\rrbracket,\quad\mu_k>0.
		\end{equation}
	\end{enumerate}
	
	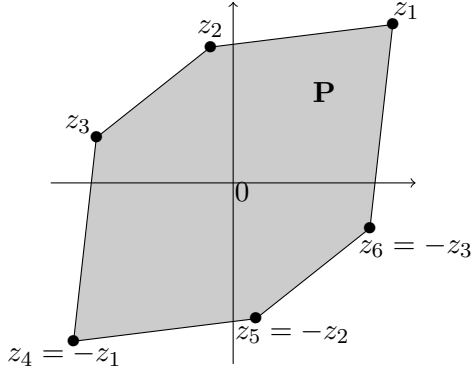
\begin{figure}[!h]
		\begin{center}
			\begin{tikzpicture}[scale=0.6]
				\draw[fill=black!20] (-3.5,-3.5)--(0.5,-3)--(3,-1)--(3.5,3.5)--(-0.5,3)--(-3,1)--cycle; 
				\draw[->](-4,0)--(4,0);
				\draw[->](0,-4)--(0,4);
				\node at (-3.5,-3.5) {$\bullet$};
				\node at (0.5,-3) {$\bullet$};
				\node at (3,-1) {$\bullet$};
				\node at (3.5,3.5) {$\bullet$};
				\node at (-0.5,3) {$\bullet$};
				\node at (-3,1) {$\bullet$};
				\node at (3.8,3.8) {$z_1$};
				\node at (-3.7,-3.8) {$z_4=-z_1$};
				\node at (-0.5,3.4) {$z_2$};
				\node at (1.3,-3.3) {$z_5=-z_2$};
				\node at (-3.4,1.3) {$z_3$};
				\node at (4,-1.4) {$z_6=-z_3$};
				\node at (2,2) {$\mathbf{P}$};
				\node at (0.2,-0.2) {$0$};
			\end{tikzpicture}
			\caption{Symmetric convex polygons.}
		\end{center}
	\end{figure}
	\noindent The symmetry property \eqref{symm-pol1} imposes that
	\begin{align}\label{symm-pol2}
		\forall k\in\llbracket1,m\rrbracket,\quad\mu_{m+k}=\mu_k.
	\end{align}
	Putting together \eqref{symm-pol2} and \eqref{sum ext angles tri} (with $m$ replaced by $2m$), we deduce that
	\begin{align}\label{conf-2}
		\sum_{k=1}^{m}\mu_k=1.
	\end{align} 
	Recall from \eqref{Swart-Christ} that the Schwarz-Christoffel mapping satisfying $F(0)=0$ and $F'(0)>0$  takes the form
	\begin{equation*}
		F(z)=\alpha\bigintsss_{0}^{z}\prod_{k=1}^{2m}\big(\xi-e^{\ii\theta_k}\big)^{-\mu_k}d\xi,\qquad\alpha=\widetilde{\alpha}\prod_{k=0}^{2m-1}\left(-e^{\ii\theta_k}\right)^{\mu_k},\qquad\widetilde{\alpha}>0.
	\end{equation*}
	Due to the symmetry, we can impose to the angles $\theta_k$ the constraints
	\begin{align}\label{symm-pol3}
		0\leqslant\theta_1<\theta_2<...<\theta_m<\pi\qquad\textnormal{and}\qquad\forall k\in\llbracket1,m\rrbracket,\quad\theta_{k+m}=\theta_k+\pi
	\end{align}
	leading, with \eqref{symm-pol1}, to
	\begin{equation}\label{Swart sym conv poly1}
		F(z)=\widetilde{\alpha}\bigintsss_{0}^{z}\prod_{k=1}^{m}\big(\xi^2-e^{2\ii\theta_k}\big)^{-\mu_k}d\xi.
	\end{equation}
	Next, we will show the following result.
	\begin{proposition}\label{prop sym poly}
		Consider $\mathbf{P}$ a symmetric convex polygonal domain associated with the Schwarz-Christoffel conformal mapping \eqref{Swart sym conv poly1}. Assume that
		\begin{equation}\label{hyp sym poly}
			\sum_{k=1}^{m}{\mu_k}\cos(2\theta_k)\notin\left\{\sqrt{1-\tfrac{1}{n^2}},\, n\in\mathbb{N}^*\right\}.
		\end{equation}
		Then, the origin $0$ is the unique critical point of the Robin function $\mathcal{R}_{\mathbf{P}}$ and the condition \eqref{cond-nondeg} is satisfied.
	\end{proposition}
	\begin{proof} 
		Recall that by construction $F(0)=0$ and $F'(0)>0.$ Moreover, a direct differentiation in \eqref{Swart sym conv poly1} gives
		\begin{align}\label{Fs sur Fp poly}
			F^{\prime\prime}(0)=0.
		\end{align}
		It follows that $z=0$ is a solution to  Grakhov's equation \eqref{Grakhov}. Therefore $z=0$ is a critical point to the Robin function associated with the domain $\mathbf{P}.$ As the domain $\mathbf{P}$ is convex, then this critical point is unique. Now, using \eqref{Swart poly}, \eqref{trii}, \eqref{Fs sur Fp poly}, \eqref{symm-pol2} and \eqref{symm-pol3} yield
		\begin{align}\label{symm-poly4}
			\nonumber S(F)(0)=\frac{F^{(3)}(0)}{F^\prime(0)}&=\sum_{k=1}^{2m}{\mu_k}{e^{-2\ii\theta_k}}\\
			\nonumber &=\sum_{k=1}^{m}{\mu_k}{e^{-2\ii\theta_k}}+\sum_{k=1}^{m}{\mu_{m+k}}{e^{-2\ii\theta_{m+k}}}\\
			&=2\sum_{k=1}^{m}{\mu_k}\cos(2\theta_k).
		\end{align}
Thus the condition \eqref{cond-nondeg} is equivalent to  \eqref{hyp sym poly}. This concludes the proof of Proposition \ref{prop sym poly}.
	\end{proof}
	\subsubsection{Rectangles}
	As an application of the previous proposition, we intend to study rectangular domains for which the condition \eqref{cond-nondeg} is explicitly described. In this context, the non-degeneracy condition \eqref{cond-nondeg} will be related to the aspect ratio which should avoid a discrete set of values. Our main result reads as follows. 
	\begin{proposition}\label{prop rectangle}
		Let $\mathbf{D}$ be a rectangle with sides $0<l<L$. Then the condition \eqref{cond-nondeg} is satisfied if and only if
		\begin{equation}\label{def Gn}
			\tfrac{l}{L}\notin\big\{G_n,\, n\in\mathbb{N}^\star\big\},\qquad G_n\triangleq G\left(\sqrt{\tfrac{1+\sqrt{1-\frac{1}{n^2}}}{2}}\right),
		\end{equation}
		where $G$ has been introduced in \eqref{def modulus G}.
	\end{proposition}
	
	\begin{figure}[!h]
		\centering
		\includegraphics[width=7.5cm]{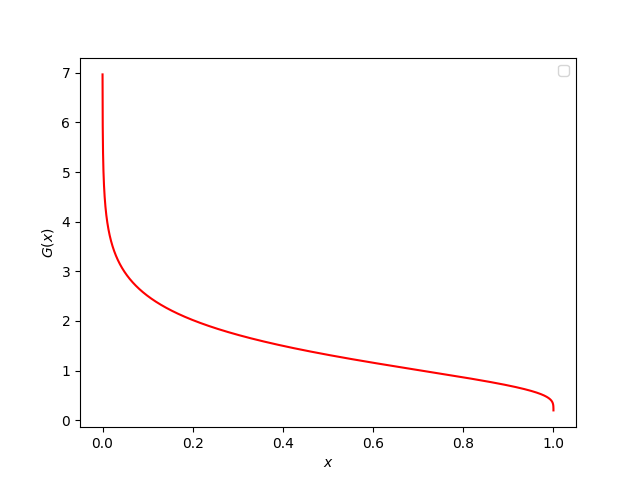}
		\includegraphics[width=7.5cm]{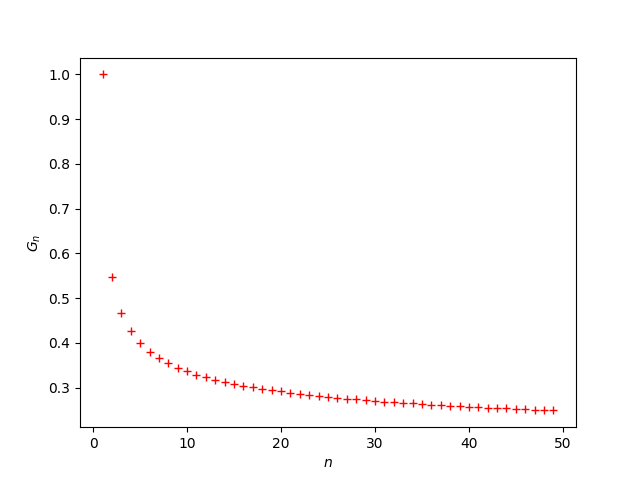}
		\caption{Graph of the function $G$ and the sequence $(G_n)_{n\in\mathbb{N}^*}$.}
	\end{figure}
	
	\begin{proof}
		Without loss of generality, we can  shall chose the rectangle $\mathbf{D}$ to be  symmetric with respect to the real and imaginary axis,
		\begin{equation}\label{thetak rectangle}
			\theta_1\in(0,\tfrac\pi4],\qquad \theta_2=\pi-\theta_1,\qquad \theta_3=\theta_1+\pi,\qquad \theta_4=\theta_2+\pi.
		\end{equation}
		Moreover, 
		\begin{equation}\label{muk rectangle}
			\mu_1=\mu_2=\frac12,
		\end{equation}
		implying in turn
		\begin{align*}
			\sum_{k=1}^{2}{\mu_k}\cos(2\theta_k)&=\cos(2\theta_1).
		\end{align*}
		Therefore, the condition \eqref{hyp sym poly} becomes
		\begin{align*}
			\cos(2\theta_1)\notin\left\{\sqrt{1-\tfrac{1}{n^2}},\, n\in\mathbb{N}^*\right\},
		\end{align*}
		which is equivalent to 
		\begin{align}\label{cond-2}
			\cos(\theta_1)\notin\left\{\sqrt{\tfrac{1+\sqrt{1-\tfrac{1}{n^2}}}{2}},\, n\in\mathbb{N}^*\right\}.
		\end{align}
		According to \eqref{Swart sym conv poly1}, \eqref{thetak rectangle} and \eqref{muk rectangle}, we have 
		\begin{align*}
			F(z)&=\widetilde{\alpha}\int_0^z\frac{d\xi}{\sqrt{(\xi^2-e^{2\ii\theta_1})(\xi^2-e^{-2\ii\theta_1})}},\qquad\widetilde{\alpha}>0.
		\end{align*}
		The length of the side joining $z_1$ and $z_2$ is given by
		\begin{align}
			L=\int_{\theta_1}^{\pi-\theta_1}|F^\prime(e^{\ii t})|dt.
		\end{align}
		From straightfoward computations we infer
		\begin{align*}
			|F^\prime(e^{\ii t})|&=\frac{\widetilde{\alpha}}{\sqrt{|e^{2\ii t}-e^{2\ii\theta_1}||e^{2\ii t}-e^{-2\ii\theta_1}|}}\\
			&=\frac{\widetilde{\alpha}}{2\sqrt{|\sin(t-\theta_1)||\sin(t+\theta_1)|}}\\
			&=\frac{\widetilde{\alpha}}{\sqrt{2}\sqrt{|\cos(2\theta_1)-\cos(2t)|}}\cdot
		\end{align*}
		Therefore, by change of variables, we get with $\beta=\pi-2\theta_1$,
		\begin{align*}
			\frac{L}{\widetilde{\alpha}}
			&=\frac{1}{\sqrt{2}}\int_{\theta_1}^{\pi-\theta_1}\frac{dt}{\sqrt{|\cos(2\theta_1)-\cos(2t)|}}\\
			&=\frac{1}{2\sqrt{2}}\int_{2\theta_1}^{\pi}\frac{dt}{\sqrt{|\cos(2\theta_1)-\cos(t)|}}+\frac{1}{2\sqrt{2}}\int_{\pi}^{2\pi-2\theta_1}\frac{dt}{\sqrt{|\cos(2\theta_1)-\cos(t)|}}\\
			&=\int_{0}^{\beta}\frac{dt}{\sqrt{2\cos(t)-2\cos(\beta)}}\\
			&=\frac{\pi}{2}{P_{-\frac12}}\big(\cos(\beta)\big)=\frac{\pi}{2}{P_{-\frac12}}\big(-\cos(2\theta_1)\big),
		\end{align*}
		where $P_\nu$ is the Legendre functions of degree $\nu$ related to the hypergeometric function $\mathtt{F}$ through
		\begin{align*}
			P_\nu(z)=\mathtt{F}\left(-\nu,\nu+1;1;\tfrac{1-z}{2}\right).
		\end{align*}
		We refer the reader to \cite[p. 337 and 561]{AS64} for the justification of the previous identities involving special functions. Hence
		\begin{align*}
			\frac{L}{\widetilde{\alpha}}&=\frac{\pi}{2}\mathtt{F}\left(\tfrac12,\tfrac12;1;\tfrac{1+\cos(2\theta_1)}{2}\right)\\
			&=\frac{\pi}{2}\mathtt{F}\left(\tfrac12,\tfrac12;1;\cos^2(\theta_1)\right)\\
			&=K\big(\cos(\theta_1)\big),
		\end{align*}
		with $K$ the complete elliptic function defined in \eqref{complete elliptic def}. The last equality comes from \cite[p. 591]{AS64}. The length $l$ of the side $[z_1,z_4]$ satisfies
		\begin{align*}
			\frac{l}{\widetilde{\alpha}}
			&=\frac{1}{\sqrt{2}}\int_{-\theta_1}^{\theta_1}\frac{dt}{\sqrt{|\cos(2\theta_1)-\cos(2t)|}}\\
			&=\int_{0}^{2\theta_1}\frac{dt}{\sqrt{2\cos(t)-2\cos(2\theta_1)}}\\
			&=\frac{\pi}{2}{P_{-\frac12}}\big(\cos(\theta_1)\big).
		\end{align*}
		Hence
		\begin{align*}
			\frac{l}{\widetilde{\alpha}}&=\frac{\pi}{2}\mathtt{F}\left(\tfrac12,\tfrac12;1;\tfrac{1-\cos(2\theta_1)}{2}\right)\\
			&=\frac{\pi}{2}\mathtt{F}\left(\tfrac12,\tfrac12;1;\sin^2(\theta_1)\right)\\
			&=K\big(\sin(\theta_1)\big).
		\end{align*}
		Therefore for $ \theta_1\in(0,\tfrac\pi4]$
		\begin{align}
			\frac{l}{L}&=\frac{K\big(\sin(\theta_1)\big)}{K\big(\cos(\theta_1)\big)}=G\big(\cos(\theta_1)\big),
		\end{align}
		where $G$ has been introduced in \eqref{def modulus G}. We can easily check that $G$ is strictly decreasing. It follows that \eqref{cond-2} is satisfied if and only if
		\begin{align*}
			\tfrac{l}{L}\notin\big\{G_n, n\in\mathbb{N}^*\big\},\qquad G_n\triangleq G\left(\sqrt{\tfrac{1+\sqrt{1-\frac{1}{n^2}}}{2}}\right).
		\end{align*}
		This concludes the proof of Proposition \ref{prop rectangle}.
	\end{proof}

	\section{Vortex duplication}\label{Sec-duplic}
	The aim of this section is to explore the duplication method achievable through a standard reflection scheme. This mechanism will be employed to construct multi-vortex configurations from a single point vortex, ultimately leading to vortex synchronization.
To illustrate this, we will discuss two cases. The first case involves a rectangular domain divided into several identical rectangular cells, each filled with a time-periodic vortex patch and  exhibiting alternating circulations. The second case involves a disc divided into several identical sectors.
%
 	\subsection{Duplication method}\label{sec dup meth}
	Let $\mathbf{D}$ be a bounded simply connected domain such that its boundary $\partial\mathbf{D}$ contains a non-trivial segment $[z,w]$ with $z\neq w$ and
	\begin{align}\label{sym-taou}
		\partial\mathbf{D}\cap(z,w)=[z,w],
	\end{align}
	where $(z,w)$ denotes the infinite line joining the points $z$ and $w.$ Now, denote by  $\mathtt{S}$ the reflection mapping  through the line $(z,w)$, namely
	\begin{align*}
		\forall \, x\in \mathbb{C},\quad \mathtt{S}(x)=2\textnormal{Re}\big\{x(\overline{z-w})\big\}\frac{z-w}{|z-w|}-x.
	\end{align*}
	Define
	\begin{align*}
		\mathbf{D}^\prime\triangleq\mathtt{S}(\mathbf{D})=\big\{\mathtt{S}(x),\,x\in\mathbf{D}\big\}.
	\end{align*}
	Notice that $\mathbf{D}\cap \mathbf{D}^\prime=[z,w].$ We set
	\begin{align*}
		\mathbf{D}^\star\triangleq\textnormal{Int}\big(\overline{\mathbf{D}}\cup\overline{\mathbf{D}^\prime}\big).
	\end{align*}
	Let  $\omega_0\in L^\infty_{c}(\mathbf{D})$ and $\omega_0^\star\in L^\infty_{c}(\mathbf{D}^\star)$  and consider  Euler equations,  
	\begin{equation}\label{Euler-Eq0}
		\begin{cases}
			\partial_{t}{\omega}+u\cdot\nabla\omega=0,\quad x\in\mathbf{D},\\
			u=\nabla^{\perp}(\Delta_{\mathbf{D}})^{-1}{\omega},\\
			\omega(0,x)=\omega_0
		\end{cases}
	\end{equation}
	and
	\begin{equation}\label{Euler-Eq1}
		\begin{cases}
			\partial_{t}{\omega}^\star+u^\star\cdot\nabla\omega^\star=0,\quad x\in\mathbf{D}^\star,\\
			u^\star=\nabla^{\perp}(\Delta_{\mathbf{D}^\star})^{-1}{\omega}^\star,\\
			\omega^\star(0,x)=\omega_0^\star.
		\end{cases}
	\end{equation}
	For a function $f:\mathbf{D}\to\mathbb{R}$ compactly supported in $\mathbf{D}$, we define its reflection $f^\#:\mathbf{D}^\star\to\mathbb{R}$ as
	\begin{equation}\label{Euler-Eq3}
		f^\#(x)\triangleq\begin{cases}
			f(x), & \textnormal{if }x\in\mathbf{D},
			\\
			-f(\mathtt{S}x), & \textnormal{if } x\in \mathbf{D}^\prime.
		\end{cases}
	\end{equation}
	We intend to prove the following result.
	\begin{proposition}\label{reflect}
		Let $\mathbf{D}$ be a  simply connected bounded domain satisfying \eqref{sym-taou}. Given  $\omega_0\in L^\infty_c(\mathbf{D})$  such that the system \eqref{Euler-Eq0} admits a global weak solution $\omega\in L^\infty \big(\mathbb{R};L^\infty_{c}(\mathbf{D})\big)$.   Let  $\omega_0^\star=\omega_0^\#$, then 
		the system \eqref{Euler-Eq1} admits  a weak global  solution $\omega^\star\in L^\infty \big(\mathbb{R};L^\infty_{c}(\mathbf{D}^\star)\big)$,  such that, 
		\begin{equation*}
			\forall t\in\mathbb{R},\quad \omega^\star(t)=\omega(t)^\#.
		\end{equation*}
		In particular, if $\omega$ is time periodic then $ \omega^\star$ is a time periodic counter-rotating pairs.
	\end{proposition}
	\begin{remark}
	The existence of at least one  solution to $\eqref{Euler-Eq0}$  satisfying  $\omega\in L^\infty \big(\mathbb{R};L^\infty_{c}(\mathbf{D})\big)$,  particularly with   compact support for each time, can be guaranteed    with a piecewise smooth boundary $\partial \mathbf{D}$ that has a   finite number of corners with angles  smaller than $\pi$, see \cite{LZ19}. Uniqueness can be achieved in slightly smoother domains, as detailed in $\cite{Agrawal,GVL13,HZ21,HZ22,LZ19}.$ It is important to note that  in  Theorem $\ref{main thm}$ and Corollary $\ref{convex-dom},$ we do not impose any regularity conditions on the boundary. In fact, the time-periodic solutions we construct are inherently compactly supported in space due to the construction scheme. 	\end{remark}
	\begin{proof}
Let $\omega$ be  a compactly supported solution of \eqref{Euler-Eq0} whose stream function is $\psi.$ Let us show that  the function $\omega^\#$ is a solution to \eqref{Euler-Eq1}. Define
\begin{align}\label{psi23}
\nonumber \forall x\in \mathbf{D}^\star,\quad  \psi^\#(t,x)&\triangleq\int_{ \mathbf{D}^\star}G_{ \mathbf{D}^\star}(x,y)\omega^\#(t,y)dy\\
&=\int_{ \mathbf{D}}\big[G_{ \mathbf{D}^\star}(x,y)- G_{ \mathbf{D}^\star}(x,\mathtt{S}y) \big]\omega(t,y)dy.
\end{align} 
Assume for a while that
\begin{align}\label{Green-func1}
\forall x, y\in \mathbf{D},\quad G_{ \mathbf{D}^\star}(x,y)- G_{ \mathbf{D}^\star}(x,\mathtt{S}y)=G_{ \mathbf{D}}(x,y)
\end{align}
and
\begin{align}\label{Green-func2}
\forall x\in \mathbf{D}^\prime,\quad\forall y\in \mathbf{D},\quad G_{ \mathbf{D}^\star}(x,y)- G_{ \mathbf{D}^\star}(x,\mathtt{S}y)=-G_{ \mathbf{D}}(\mathtt{S}x,y).
\end{align}	
Then, we deduce from \eqref{psi23} that 
\begin{align*}
\nonumber \forall x\in \mathbf{D},\quad  \psi^\#(t,x)&=\int_{ \mathbf{D}}G_{ \mathbf{D}}(x,y)\omega(t,y)dy\\
&=\psi(t,x)
\end{align*} 
and
\begin{align*}
\nonumber \forall x\in \mathbf{D}^\prime,\quad  \psi^\#(t,x)&=-\int_{ \mathbf{D}}G_{ \mathbf{D}}(\mathtt{S} x,y)\omega(t,y)dy\\
&=-\psi(t,\mathtt{S}x).
\end{align*} 
Therefore, we get
\begin{align*}
			\nabla^\perp \psi^\#(t,x)\cdot \nabla \omega^\#(t,x)=\begin{cases}
				(\nabla^\perp \psi\cdot \nabla \omega)(t,x), \quad x\in \mathbf{D},\\
				-(\nabla^\perp \psi\cdot \nabla \omega)(t,\mathtt{S}x), \quad x\in\mathbf{D}^\prime.
			\end{cases}
		\end{align*}
Consequently, we find
$$
\partial_{t}{\omega}^\#+\nabla^\perp \psi^\#\cdot\nabla\omega^\#=0,\quad \hbox{in}\quad \mathbf{D}^\star
$$
implying that ${\omega}^\#$ is a solution to \eqref{Euler-Eq1}. This gives the desired result. It remains to check the identities \eqref{Green-func1} and \eqref{Green-func2}. Before checking these identities, we will first show 
\begin{align}\label{sym-green}
			\forall\, x,y\in\mathbf{D}^\star,\quad G_{\mathbf{D}^\star}(x,\mathtt{S}y)=G_{\mathbf{D}^\star}(\mathtt{S}x,y).
		\end{align}
		Its proof follows from the fact that
		\begin{align*}
			\Delta_x [G_{\mathbf{D}^\star}(\mathtt{S}x,\mathtt{S}y)]&=[\Delta_x G_{\mathbf{D}^\star}](\mathtt{S}x,\mathtt{S}y)\\
			&=\delta_{\mathtt{S} x}(\mathtt{S}y)\\
			&=\delta_x(y).
		\end{align*}
		In addition, as $\mathtt{S}\mathbf{D}^\star=\mathbf{D}^\star$ and for $x\in\partial\mathbf{D}^\star$ we have $\mathtt{S}x\in\partial\mathbf{D}^\star$, then 
		\begin{align*}
			\forall y\in\mathbf{D},\quad \forall\,x\in\partial\mathbf{D},\quad G_{\mathbf{D}^\star}(\mathtt{S}x,\mathtt{S}y)=0.
		\end{align*}
		By uniqueness of the Green function, we find that
		\begin{align*}
			\forall\, x,y\in\mathbf{D}^\star,\quad G_{\mathbf{D}^\star}( \mathtt{S}x,\mathtt{S}y)=G_{\mathbf{D}^\star}(x,y),
		\end{align*}
		which implies \eqref{sym-green} in view of the identity $\mathtt{S}^2=\hbox{Id}.$
	Let's start with checking the identity \eqref{Green-func1}. For this aim, we  introduce the auxiliary  function
\begin{align}\label{green-lo}
			\forall\, x,y\in\mathbf{D},\quad G_{1}(x,y)\triangleq G_{\mathbf{D}^\star}(x,y)-G_{\mathbf{D}^\star}(x,\mathtt{S}y).
		\end{align}
First, for $x,y\in \mathbf{D}$, as  $\mathtt{S}y\notin {\mathbf{D}}$, we have 
		\begin{align*}
			\Delta_x[G_{1}(x,y)]&=\Delta_x G_{\mathbf{D}^\star}(x,y)-\Delta_x  G_{\mathbf{D}^\star}(x,\mathtt{S}y)\\
			&=\delta_x(y)-\delta_{x}(\mathtt{S}y)\\
			&=\delta_x(y).
		\end{align*}
One may easily verify that 
		\begin{align*}
			\partial\mathbf{D}=A\cup[z,w],
		\end{align*}
		where $A\subset \partial\mathbf{D}^\star$. As $\mathtt{S}(\mathbf{D}^\star)=\mathbf{D}^\star$, then  for $y\in\mathbf{D}$, we obtain from \eqref{green-lo} and \eqref{Eulereq}
		\begin{align*}
			\forall\, x\in A, \quad G_{1}(x,y)=0-0=0.
		\end{align*}
		However, for $x\in [z,w]$, we have $\mathtt{S}x=x$ and thus we find once again  from \eqref{green-lo} and \eqref{sym-green}
		\begin{align*}
			\forall\, y\in\mathbf{D},\quad G_{1}(x,y)&=G_{\mathbf{D}^\star}(x,y)-G_{\mathbf{D}^\star}(\mathtt{S}x,y)\\
			&=0.
		\end{align*}
		Therefore $G_1$ satisfies the same elliptic equation as $G_{\mathbf{D}}$. By uniqueness, we find
		\begin{align*}
			\forall x,y\in\mathbf{D},\quad G_1(x,y)=G_{\mathbf{D}}(x,y).
		\end{align*}
		Together with \eqref{sym-green}, it yields the identity \eqref{Green-func1}. It remains to prove \eqref{Green-func2}. We write by virtue of \eqref{sym-green}
		\begin{align}\label{Green-func02}
\nonumber\forall x\in \mathbf{D}^\prime,\quad\forall y\in \mathbf{D},\quad G_{ \mathbf{D}^\star}(x,y)- G_{ \mathbf{D}^\star}(x,\mathtt{S}y)&= G_{ \mathbf{D}^\star}(\mathtt{S}^2x,y)- G_{ \mathbf{D}^\star}(\mathtt{S}^2x,\mathtt{S}y)\\
&=G_{ \mathbf{D}^\star}(\mathtt{S}x,\mathtt{S}y)- G_{ \mathbf{D}^\star}(\mathtt{S}x,y).
\end{align}
On the other hand, as $\mathtt{S} x\in \mathbf{D},$ then applying \eqref{Green-func1}
\begin{align}\label{Green-func3}
 G_{ \mathbf{D}^\star}(\mathtt{S} x,y)- G_{ \mathbf{D}^\star}(\mathtt{S} x,\mathtt{S}y)=G_{ \mathbf{D}}(\mathtt{S} x,y).
\end{align}	
Putting together \eqref{Green-func3} with \eqref{Green-func02} yields
\begin{align*}
\forall x\in \mathbf{D}^\prime,\quad\forall y\in \mathbf{D},\quad G_{ \mathbf{D}^\star}(x,y)- G_{ \mathbf{D}^\star}(x,\mathtt{S}y)&= -G_{ \mathbf{D}}(\mathtt{S} x,y).
\end{align*}
		This achieves the proof of \eqref{Green-func2}.
		
	\end{proof}
	Combining  Theorem \ref{main thm} together with Proposition \ref{reflect} allows to generate synchronized  pairs of counter-rotating time periodic patches. More precisely, we obtain the following result.
	\begin{coro}
		Let  $\mathbf{D}$ be a simply-connected bounded domain such that Theorem \ref{main thm} holds true. Assume the existence of a non-trivial segment $[z,w]$ such that 
		\begin{align*}
			\partial\mathbf{D}\cap(z,w)=[z,w].
		\end{align*}
		Then,  counter-rotating time periodic vortex patches exist as solutions to Euler equations \eqref{Euler-Eq1} in the domain $\mathbf{D}^\star,$ generated in the spirit of Proposition \ref{reflect} by duplication.
	\end{coro}
	In the remainder of this section, we shall explore how to iterate the duplication process  when the boundary of the initial domain $\mathbf{D}$ contains more than one segment. 
	Let $\mathbf{D}$ be a domain such that its boundary $\partial\mathbf{D}$ contains two  non-trivial segments $[z_1,w_1]$ and  $[z_2,w_2]$ and
	\begin{align*}
		\forall j\in\{1,2\},\quad\partial\mathbf{D}\cap(z_j,w_j)=[z_j,w_j].
	\end{align*}
	Notice that the line $(z_1,w_1)$ will necessary intersect (at infinity when they are parallel) the line $(z_2,w_2)$ at a point $I$ outside the segment $[z_1,w_1].$ We denote by $\theta$ the geometric angle $\widehat{z_1 I z_2}$, which belongs to $[0,\pi).$ We denote by $\mathtt{S}_j$ the reflection with respect to the axis $(z_j,w_j)$. Denote 
	\begin{align*}
		\mathbf{D}_{\star,1}\triangleq\textnormal{Int}\big(\overline{\mathbf{D}}\cup\overline{\mathtt{S}_1\mathbf{D}}\big).
	\end{align*}
	Then, for $\theta\in[0,\tfrac\pi2)$ we still get that
	\begin{align*}
		\forall j\in\{1,2\},\quad\partial\mathbf{D}_{\star,1}\cap(z_2,w_2)=[z_2,w_2].
	\end{align*}
	Therefore, we can generate a new duplication by introducing
	\begin{align*}
		\mathbf{D}_{\star,2}\triangleq\textnormal{Int}\big(\overline{\mathbf{D}_{\star,1}}\cup\overline{\mathtt{S}_2\mathbf{D}_{\star,1}}\big).
	\end{align*}
	The new domain is  simply connected. Notice that we have excluded $\theta=\frac\pi2$ because in that case the resulting domain would not be simply connected and would contain a hole. However, we can reach $\theta=\tfrac{\pi}{2}$ if we allow  $w_1=w_2$ which permits the duplication of  rectangles. 

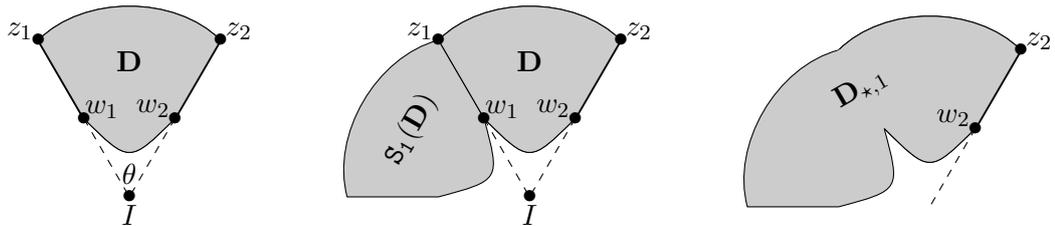
\begin{figure}[!h]
	\begin{center}
		\hspace{1.7cm}\begin{tikzpicture}[scale=1.2]
			\draw[line width=1pt] (-0.5,0.866)--(-1,1.732);
			\draw[line width=1pt] (0.5,0.866)--(1,1.732);
			\draw[dashed] (-0.5,0.866)--(0,0);
			\draw[dashed] (0.5,0.866)--(0,0);
			\node at (0,0) {$\bullet$};
			\node at (0,-0.2) {$I$};
			\draw (0.5,0.866) ..controls +(-0.5,-0.5) and +(0.5,-0.5).. (-0.5,0.866);
			\draw (1,1.732) ..controls +(-0.5,0.5) and +(0.5,0.5).. (-1,1.732);
			\draw [fill=black!20]
			(0.5,0.866) .. controls +(-0.5,-0.5) and +(0.5,-0.5) .. (-0.5,0.866) -- (-1,1.732)
			.. controls +(0.5,0.5) and +(-0.5,0.5) .. (1,1.732) -- (0.5,0.866);
			\node at (-0.5,0.866) {$\bullet$};
			\node at (0.5,0.866) {$\bullet$};
			\node at (-1,1.732) {$\bullet$};
			\node at (1,1.732) {$\bullet$};
			\node at (0,1.5) {$\mathbf{D}$};
			\node at (-0.3,0.966) {$w_1$};
			\node at (0.26,0.98) {$w_2$};
			\node at (-1.2,1.832) {$z_1$};
			\node at (1.2,1.832) {$z_2$};
			\node at (0,0.25) {$\theta$};
		\end{tikzpicture}
		\qquad
		\begin{tikzpicture}[scale=1.2]
			\draw[line width=1pt] (-0.5,0.866)--(-1,1.732);
			\draw[line width=1pt] (0.5,0.866)--(1,1.732);
			\draw[dashed] (-0.5,0.866)--(0,0);
			\draw[dashed] (0.5,0.866)--(0,0);
			\node at (0,0) {$\bullet$};
			\node at (0,-0.2) {$I$};
			\draw (0.5,0.866) ..controls +(-0.5,-0.5) and +(0.5,-0.5).. (-0.5,0.866);
			\draw (1,1.732) ..controls +(-0.5,0.5) and +(0.5,0.5).. (-1,1.732);
			\draw [fill=black!20]
			(0.5,0.866) .. controls +(-0.5,-0.5) and +(0.5,-0.5) .. (-0.5,0.866) -- (-1,1.732)
			.. controls +(0.5,0.5) and +(-0.5,0.5) .. (1,1.732) -- (0.5,0.866);
			\node at (0,1.5) {$\mathbf{D}$};
			\draw [rotate=60,fill=black!20]
			(0.5,0.866) .. controls +(-0.5,-0.5) and +(0.5,-0.5) .. (-0.5,0.866) -- (-1,1.732)
			.. controls +(0.5,0.5) and +(-0.5,0.5) .. (1,1.732) -- (0.5,0.866);
			\node at (-0.5,0.866) {$\bullet$};
			\node at (0.5,0.866) {$\bullet$};
			\node at (-1,1.732) {$\bullet$};
			\node at (1,1.732) {$\bullet$};
			\node[rotate=60] at (-1.3,0.75) {$\mathtt{S_1}(\mathbf{D})$};
			\node at (-0.3,0.966) {$w_1$};
			\node at (0.26,0.98) {$w_2$};
			\node at (-1.2,1.832) {$z_1$};
			\node at (1.2,1.832) {$z_2$};
		\end{tikzpicture}
		\qquad
		\begin{tikzpicture}[scale=1.2]
			\draw[line width=1pt] (0.5,0.866)--(1,1.732);
			\draw[dashed] (0.5,0.866)--(0,0);
			\draw (0.5,0.866) ..controls +(-0.5,-0.5) and +(0.5,-0.5).. (-0.5,0.866);
			\draw (1,1.732) ..controls +(-0.5,0.5) and +(0.5,0.5).. (-1,1.732);
			\draw [fill=black!20]
			(0.5,0.866) .. controls +(-0.5,-0.5) and +(0.5,-0.5) .. (-0.5,0.866) -- (-1,1.732)
			.. controls +(0.5,0.5) and +(-0.5,0.5) .. (1,1.732) -- (0.5,0.866);
			\draw [rotate=60,fill=black!20]
			(0.5,0.866) .. controls +(-0.5,-0.5) and +(0.5,-0.5) .. (-0.5,0.866) -- (-1,1.732)
			.. controls +(0.5,0.5) and +(-0.5,0.5) .. (1,1.732) -- (0.5,0.866);
			\node at (0.5,0.866) {$\bullet$};
			\node at (1,1.732) {$\bullet$};
			\node at (0.26,0.98) {$w_2$};
			\node at (1.2,1.832) {$z_2$};
			\draw[color=black!20,line width=1pt] (-0.5,0.866)--(-1,1.732);
			\node[rotate=30] at (-0.75,1.3) {$\mathbf{D}_{\star,1}$};
			\node at (0,-0.2) {$ $};
		\end{tikzpicture}
	\end{center}
	\caption{Duplication iteration : 1st step ($\theta=\tfrac{\pi}{3}$)}
\end{figure}

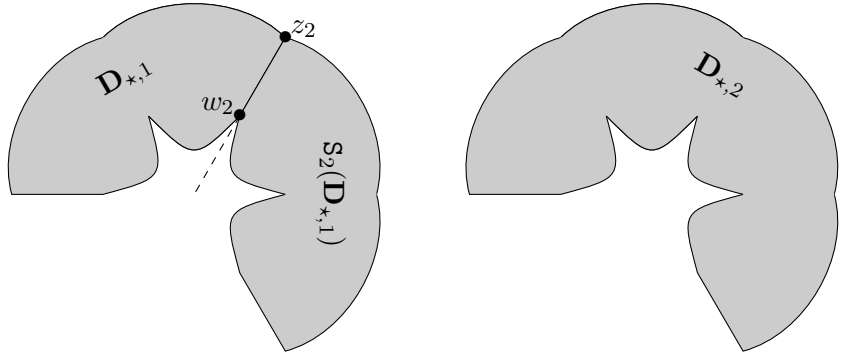
\begin{figure}
	\begin{center}
		\begin{tikzpicture}[scale=1.2]
			\draw[line width=1pt] (0.5,0.866)--(1,1.732);
			\draw[dashed] (0.5,0.866)--(0,0);
			\draw (0.5,0.866) ..controls +(-0.5,-0.5) and +(0.5,-0.5).. (-0.5,0.866);
			\draw (1,1.732) ..controls +(-0.5,0.5) and +(0.5,0.5).. (-1,1.732);
			\draw [fill=black!20]
			(0.5,0.866) .. controls +(-0.5,-0.5) and +(0.5,-0.5) .. (-0.5,0.866) -- (-1,1.732)
			.. controls +(0.5,0.5) and +(-0.5,0.5) .. (1,1.732) -- (0.5,0.866);
			\draw [rotate=-60,fill=black!20]
			(0.5,0.866) .. controls +(-0.5,-0.5) and +(0.5,-0.5) .. (-0.5,0.866) -- (-1,1.732)
			.. controls +(0.5,0.5) and +(-0.5,0.5) .. (1,1.732) -- (0.5,0.866);
			\draw [rotate=-120,fill=black!20]
			(0.5,0.866) .. controls +(-0.5,-0.5) and +(0.5,-0.5) .. (-0.5,0.866) -- (-1,1.732)
			.. controls +(0.5,0.5) and +(-0.5,0.5) .. (1,1.732) -- (0.5,0.866);
			\draw [rotate=60,fill=black!20]
			(0.5,0.866) .. controls +(-0.5,-0.5) and +(0.5,-0.5) .. (-0.5,0.866) -- (-1,1.732)
			.. controls +(0.5,0.5) and +(-0.5,0.5) .. (1,1.732) -- (0.5,0.866);
			\node at (0.5,0.866) {$\bullet$};
			\node at (1,1.732) {$\bullet$};
			\node at (0.26,0.98) {$w_2$};
			\node at (1.2,1.832) {$z_2$};
			\draw[color=black!20,line width=1pt] (-0.5,0.866)--(-1,1.732);
			\draw[color=black!20,line width=1pt] (1,0)--(2,0);
			\node[rotate=30] at (-0.75,1.3) {$\mathbf{D}_{\star,1}$};
			\node[rotate=-90] at (1.5,0) {$\mathtt{S}_{2}(\mathbf{D}_{\star,1})$};
		\end{tikzpicture}\qquad
		\begin{tikzpicture}[scale=1.2]
			\draw (0.5,0.866) ..controls +(-0.5,-0.5) and +(0.5,-0.5).. (-0.5,0.866);
			\draw (1,1.732) ..controls +(-0.5,0.5) and +(0.5,0.5).. (-1,1.732);
			\draw [fill=black!20]
			(0.5,0.866) .. controls +(-0.5,-0.5) and +(0.5,-0.5) .. (-0.5,0.866) -- (-1,1.732)
			.. controls +(0.5,0.5) and +(-0.5,0.5) .. (1,1.732) -- (0.5,0.866);
			\draw [rotate=-60,fill=black!20]
			(0.5,0.866) .. controls +(-0.5,-0.5) and +(0.5,-0.5) .. (-0.5,0.866) -- (-1,1.732)
			.. controls +(0.5,0.5) and +(-0.5,0.5) .. (1,1.732) -- (0.5,0.866);
			\draw [rotate=-120,fill=black!20]
			(0.5,0.866) .. controls +(-0.5,-0.5) and +(0.5,-0.5) .. (-0.5,0.866) -- (-1,1.732)
			.. controls +(0.5,0.5) and +(-0.5,0.5) .. (1,1.732) -- (0.5,0.866);
			\draw [rotate=60,fill=black!20]
			(0.5,0.866) .. controls +(-0.5,-0.5) and +(0.5,-0.5) .. (-0.5,0.866) -- (-1,1.732)
			.. controls +(0.5,0.5) and +(-0.5,0.5) .. (1,1.732) -- (0.5,0.866);
			\draw[color=black!20,line width=1pt] (-0.5,0.866)--(-1,1.732);
			\draw[color=black!20,line width=1pt] (0.5,0.866)--(1,1.732);
			\draw[color=black!20,line width=1pt] (1,0)--(2,0);
			\node[rotate=-30] at (0.75,1.3) {$\mathbf{D}_{\star,2}$};
		\end{tikzpicture}
	\end{center}
	\caption{Duplication iteration : 2nd step}
\end{figure}

	\newpage
	\subsection{Application to vortex synchronization}
	Our aim here is to apply the duplication method explained in the Section \ref{Sec-duplic} to study the desingularization of a highly symmetric system of $2^m$ point vortices. Due to the symmetry, the problem reduces to the evolution a single point vortex within a primitive fundamental  cell, in the application it will be a rectangle or an  $m$-sector.
	\subsubsection{Time periodic vortex choreography within rectangular  cells}
	Let us consider a rectangular fundamental domain $\mathbf{R}$ with length $L$ and width $l.$ Assume that
	$$\tfrac{l}{L}\not\in\{G_n,\,n\in\mathbb{N}^*\},$$
	where the sequence $(G_n)_{n\in\mathbb{N}^*}$ has been introduced in \eqref{def Gn}. Therefore, Proposition \ref{prop rectangle} applies and by virtue of Corollary \ref{convex-dom} {\it most} of the point vortex periodic orbits can be desingularized into  periodic vortex patch solutions. Given  such dynamics, one can apply and iterate  the duplication method  to construct a simply connected domain formed by $2^m$ cells (copies of $\mathbf{R}$) within  which time  periodic multi-vortex patch motion occurs. Therefore one obtains  time periodic symmetric choreography of $2^m$ patches as illustrated in Figure \ref{rectangle choreography}.
	
	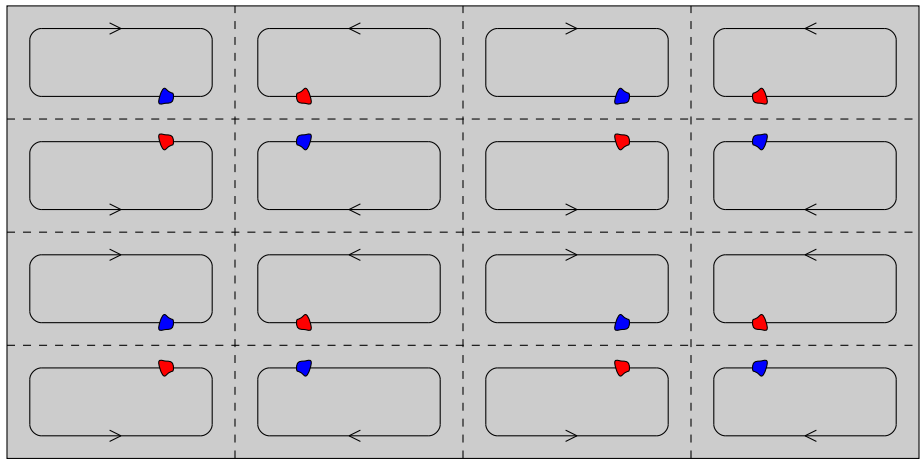
\begin{figure}[!h]
		\begin{center}
			\begin{tikzpicture}[scale=1.5]
				\draw [fill=black!20]
				(-4,-2)--(-4,2)--(4,2)--(4,-2)--cycle;
				\draw [dashed] (-2,-2)--(-2,2);
				\draw [dashed] (0,-2)--(0,2);
				\draw [dashed] (2,-2)--(2,2);
				\draw [dashed] (-4,-1)--(4,-1);
				\draw [dashed] (-4,0)--(4,0);
				\draw [dashed] (-4,1)--(4,1);
				\draw (0.3,0.2)--(1.7,0.2);
				\draw (1.8,0.3)--(1.8,0.7);
				\draw (1.7,0.8)--(0.3,0.8);
				\draw (0.2,0.7)--(0.2,0.3);
				\draw (0.9,0.25)--(1,0.2);
				\draw (0.9,0.15)--(1,0.2);
				\draw (1.7,0.2) .. controls +(0.05,0) and +(0,-0.05) .. (1.8,0.3);
				\draw (1.8,0.7) .. controls +(0,0.05) and +(0.05,0) .. (1.7,0.8);
				\draw (0.3,0.8) .. controls +(-0.05,0) and +(0,0.05) .. (0.2,0.7);
				\draw (0.2,0.3) .. controls +(0,-0.05) and +(-0.05,0) .. (0.3,0.2);
				\filldraw[draw=black,fill=red] (1.4,0.74) .. controls +(0.05,0.05) and +(0,-0.05) .. (1.46,0.8) .. controls +(0,0.05) and +(0.04,0).. (1.4,0.86).. controls +(-0.04,0) and +(-0.03,0.1) ..(1.34,0.8).. controls +(0.03,-0.1) and +(-0.01,-0.01)..(1.4,0.74);
				%
				%
				\draw[shift={(0,-2)}] (0.3,0.2)--(1.7,0.2);
				\draw[shift={(0,-2)}] (1.8,0.3)--(1.8,0.7);
				\draw[shift={(0,-2)}] (1.7,0.8)--(0.3,0.8);
				\draw[shift={(0,-2)}] (0.2,0.7)--(0.2,0.3);
				\draw[shift={(0,-2)}] (0.9,0.25)--(1,0.2);
				\draw[shift={(0,-2)}] (0.9,0.15)--(1,0.2);
				\draw[shift={(0,-2)}] (1.7,0.2) .. controls +(0.05,0) and +(0,-0.05) .. (1.8,0.3);
				\draw[shift={(0,-2)}] (1.8,0.7) .. controls +(0,0.05) and +(0.05,0) .. (1.7,0.8);
				\draw[shift={(0,-2)}] (0.3,0.8) .. controls +(-0.05,0) and +(0,0.05) .. (0.2,0.7);
				\draw[shift={(0,-2)}] (0.2,0.3) .. controls +(0,-0.05) and +(-0.05,0) .. (0.3,0.2);
				\filldraw[shift={(0,-2)},draw=black,fill=red] (1.4,0.74) .. controls +(0.05,0.05) and +(0,-0.05) .. (1.46,0.8) .. controls +(0,0.05) and +(0.04,0).. (1.4,0.86).. controls +(-0.04,0) and +(-0.03,0.1) ..(1.34,0.8).. controls +(0.03,-0.1) and +(-0.01,-0.01)..(1.4,0.74);
				%
				%
				\draw[shift={(-4,0)}] (0.3,0.2)--(1.7,0.2);
				\draw[shift={(-4,0)}] (1.8,0.3)--(1.8,0.7);
				\draw[shift={(-4,0)}] (1.7,0.8)--(0.3,0.8);
				\draw[shift={(-4,0)}] (0.2,0.7)--(0.2,0.3);
				\draw[shift={(-4,0)}] (0.9,0.25)--(1,0.2);
				\draw[shift={(-4,0)}] (0.9,0.15)--(1,0.2);
				\draw[shift={(-4,0)}] (1.7,0.2) .. controls +(0.05,0) and +(0,-0.05) .. (1.8,0.3);
				\draw[shift={(-4,0)}] (1.8,0.7) .. controls +(0,0.05) and +(0.05,0) .. (1.7,0.8);
				\draw[shift={(-4,0)}] (0.3,0.8) .. controls +(-0.05,0) and +(0,0.05) .. (0.2,0.7);
				\draw[shift={(-4,0)}] (0.2,0.3) .. controls +(0,-0.05) and +(-0.05,0) .. (0.3,0.2);
				\filldraw[shift={(-4,0)},draw=black,fill=red] (1.4,0.74) .. controls +(0.05,0.05) and +(0,-0.05) .. (1.46,0.8) .. controls +(0,0.05) and +(0.04,0).. (1.4,0.86).. controls +(-0.04,0) and +(-0.03,0.1) ..(1.34,0.8).. controls +(0.03,-0.1) and +(-0.01,-0.01)..(1.4,0.74);
				%
				%
				\draw[shift={(-4,-2)}] (0.3,0.2)--(1.7,0.2);
				\draw[shift={(-4,-2)}] (1.8,0.3)--(1.8,0.7);
				\draw[shift={(-4,-2)}] (1.7,0.8)--(0.3,0.8);
				\draw[shift={(-4,-2)}] (0.2,0.7)--(0.2,0.3);
				\draw[shift={(-4,-2)}] (0.9,0.25)--(1,0.2);
				\draw[shift={(-4,-2)}] (0.9,0.15)--(1,0.2);
				\draw[shift={(-4,-2)}] (1.7,0.2) .. controls +(0.05,0) and +(0,-0.05) .. (1.8,0.3);
				\draw[shift={(-4,-2)}] (1.8,0.7) .. controls +(0,0.05) and +(0.05,0) .. (1.7,0.8);
				\draw[shift={(-4,-2)}] (0.3,0.8) .. controls +(-0.05,0) and +(0,0.05) .. (0.2,0.7);
				\draw[shift={(-4,-2)}] (0.2,0.3) .. controls +(0,-0.05) and +(-0.05,0) .. (0.3,0.2);
				\filldraw[shift={(-4,-2)},draw=black,fill=red] (1.4,0.74) .. controls +(0.05,0.05) and +(0,-0.05) .. (1.46,0.8) .. controls +(0,0.05) and +(0.04,0).. (1.4,0.86).. controls +(-0.04,0) and +(-0.03,0.1) ..(1.34,0.8).. controls +(0.03,-0.1) and +(-0.01,-0.01)..(1.4,0.74);
				%
				%
				\begin{scope}[yscale=1,xscale=-1]
					\draw (0.3,0.2)--(1.7,0.2);
					\draw (1.8,0.3)--(1.8,0.7);
					\draw (1.7,0.8)--(0.3,0.8);
					\draw (0.2,0.7)--(0.2,0.3);
					\draw (0.9,0.25)--(1,0.2);
					\draw (0.9,0.15)--(1,0.2);
					\draw (1.7,0.2) .. controls +(0.05,0) and +(0,-0.05) .. (1.8,0.3);
					\draw (1.8,0.7) .. controls +(0,0.05) and +(0.05,0) .. (1.7,0.8);
					\draw (0.3,0.8) .. controls +(-0.05,0) and +(0,0.05) .. (0.2,0.7);
					\draw (0.2,0.3) .. controls +(0,-0.05) and +(-0.05,0) .. (0.3,0.2);
					\filldraw[draw=black, fill=blue] (1.4,0.74) .. controls +(0.05,0.05) and +(0,-0.05) .. (1.46,0.8) .. controls +(0,0.05) and +(0.04,0).. (1.4,0.86).. controls +(-0.04,0) and +(-0.03,0.1) ..(1.34,0.8).. controls +(0.03,-0.1) and +(-0.01,-0.01)..(1.4,0.74);
				\end{scope}
				\begin{scope}[shift={(4,0)},yscale=1,xscale=-1]
					\draw (0.3,0.2)--(1.7,0.2);
					\draw (1.8,0.3)--(1.8,0.7);
					\draw (1.7,0.8)--(0.3,0.8);
					\draw (0.2,0.7)--(0.2,0.3);
					\draw (0.9,0.25)--(1,0.2);
					\draw (0.9,0.15)--(1,0.2);
					\draw (1.7,0.2) .. controls +(0.05,0) and +(0,-0.05) .. (1.8,0.3);
					\draw (1.8,0.7) .. controls +(0,0.05) and +(0.05,0) .. (1.7,0.8);
					\draw (0.3,0.8) .. controls +(-0.05,0) and +(0,0.05) .. (0.2,0.7);
					\draw (0.2,0.3) .. controls +(0,-0.05) and +(-0.05,0) .. (0.3,0.2);
					\filldraw[draw=black,fill=blue] (1.4,0.74) .. controls +(0.05,0.05) and +(0,-0.05) .. (1.46,0.8) .. controls +(0,0.05) and +(0.04,0).. (1.4,0.86).. controls +(-0.04,0) and +(-0.03,0.1) ..(1.34,0.8).. controls +(0.03,-0.1) and +(-0.01,-0.01)..(1.4,0.74);
				\end{scope}
				\begin{scope}[shift={(0,-2)},yscale=1,xscale=-1]
					\draw (0.3,0.2)--(1.7,0.2);
					\draw (1.8,0.3)--(1.8,0.7);
					\draw (1.7,0.8)--(0.3,0.8);
					\draw (0.2,0.7)--(0.2,0.3);
					\draw (0.9,0.25)--(1,0.2);
					\draw (0.9,0.15)--(1,0.2);
					\draw (1.7,0.2) .. controls +(0.05,0) and +(0,-0.05) .. (1.8,0.3);
					\draw (1.8,0.7) .. controls +(0,0.05) and +(0.05,0) .. (1.7,0.8);
					\draw (0.3,0.8) .. controls +(-0.05,0) and +(0,0.05) .. (0.2,0.7);
					\draw (0.2,0.3) .. controls +(0,-0.05) and +(-0.05,0) .. (0.3,0.2);
					\filldraw[draw=black,fill=blue] (1.4,0.74) .. controls +(0.05,0.05) and +(0,-0.05) .. (1.46,0.8) .. controls +(0,0.05) and +(0.04,0).. (1.4,0.86).. controls +(-0.04,0) and +(-0.03,0.1) ..(1.34,0.8).. controls +(0.03,-0.1) and +(-0.01,-0.01)..(1.4,0.74);
				\end{scope}
				\begin{scope}[shift={(4,-2)},yscale=1,xscale=-1]
					\draw (0.3,0.2)--(1.7,0.2);
					\draw (1.8,0.3)--(1.8,0.7);
					\draw (1.7,0.8)--(0.3,0.8);
					\draw (0.2,0.7)--(0.2,0.3);
					\draw (0.9,0.25)--(1,0.2);
					\draw (0.9,0.15)--(1,0.2);
					\draw (1.7,0.2) .. controls +(0.05,0) and +(0,-0.05) .. (1.8,0.3);
					\draw (1.8,0.7) .. controls +(0,0.05) and +(0.05,0) .. (1.7,0.8);
					\draw (0.3,0.8) .. controls +(-0.05,0) and +(0,0.05) .. (0.2,0.7);
					\draw (0.2,0.3) .. controls +(0,-0.05) and +(-0.05,0) .. (0.3,0.2);
					\filldraw[draw=black,fill=blue] (1.4,0.74) .. controls +(0.05,0.05) and +(0,-0.05) .. (1.46,0.8) .. controls +(0,0.05) and +(0.04,0).. (1.4,0.86).. controls +(-0.04,0) and +(-0.03,0.1) ..(1.34,0.8).. controls +(0.03,-0.1) and +(-0.01,-0.01)..(1.4,0.74);
				\end{scope}
				\begin{scope}[shift={(-4,0)},yscale=-1,xscale=1]
					\draw (0.3,0.2)--(1.7,0.2);
					\draw (1.8,0.3)--(1.8,0.7);
					\draw (1.7,0.8)--(0.3,0.8);
					\draw (0.2,0.7)--(0.2,0.3);
					\draw (0.9,0.25)--(1,0.2);
					\draw (0.9,0.15)--(1,0.2);
					\draw (1.7,0.2) .. controls +(0.05,0) and +(0,-0.05) .. (1.8,0.3);
					\draw (1.8,0.7) .. controls +(0,0.05) and +(0.05,0) .. (1.7,0.8);
					\draw (0.3,0.8) .. controls +(-0.05,0) and +(0,0.05) .. (0.2,0.7);
					\draw (0.2,0.3) .. controls +(0,-0.05) and +(-0.05,0) .. (0.3,0.2);
					\filldraw[draw=black,fill=blue] (1.4,0.74) .. controls +(0.05,0.05) and +(0,-0.05) .. (1.46,0.8) .. controls +(0,0.05) and +(0.04,0).. (1.4,0.86).. controls +(-0.04,0) and +(-0.03,0.1) ..(1.34,0.8).. controls +(0.03,-0.1) and +(-0.01,-0.01)..(1.4,0.74);
				\end{scope}
				\begin{scope}[shift={(0,2)},yscale=-1,xscale=1]
					\draw (0.3,0.2)--(1.7,0.2);
					\draw (1.8,0.3)--(1.8,0.7);
					\draw (1.7,0.8)--(0.3,0.8);
					\draw (0.2,0.7)--(0.2,0.3);
					\draw (0.9,0.25)--(1,0.2);
					\draw (0.9,0.15)--(1,0.2);
					\draw (1.7,0.2) .. controls +(0.05,0) and +(0,-0.05) .. (1.8,0.3);
					\draw (1.8,0.7) .. controls +(0,0.05) and +(0.05,0) .. (1.7,0.8);
					\draw (0.3,0.8) .. controls +(-0.05,0) and +(0,0.05) .. (0.2,0.7);
					\draw (0.2,0.3) .. controls +(0,-0.05) and +(-0.05,0) .. (0.3,0.2);
					\filldraw[draw=black,fill=blue] (1.4,0.74) .. controls +(0.05,0.05) and +(0,-0.05) .. (1.46,0.8) .. controls +(0,0.05) and +(0.04,0).. (1.4,0.86).. controls +(-0.04,0) and +(-0.03,0.1) ..(1.34,0.8).. controls +(0.03,-0.1) and +(-0.01,-0.01)..(1.4,0.74);
				\end{scope}
				\begin{scope}[shift={(-4,2)},yscale=-1,xscale=1]
					\draw (0.3,0.2)--(1.7,0.2);
					\draw (1.8,0.3)--(1.8,0.7);
					\draw (1.7,0.8)--(0.3,0.8);
					\draw (0.2,0.7)--(0.2,0.3);
					\draw (0.9,0.25)--(1,0.2);
					\draw (0.9,0.15)--(1,0.2);
					\draw (1.7,0.2) .. controls +(0.05,0) and +(0,-0.05) .. (1.8,0.3);
					\draw (1.8,0.7) .. controls +(0,0.05) and +(0.05,0) .. (1.7,0.8);
					\draw (0.3,0.8) .. controls +(-0.05,0) and +(0,0.05) .. (0.2,0.7);
					\draw (0.2,0.3) .. controls +(0,-0.05) and +(-0.05,0) .. (0.3,0.2);
					\filldraw[draw=black,fill=blue] (1.4,0.74) .. controls +(0.05,0.05) and +(0,-0.05) .. (1.46,0.8) .. controls +(0,0.05) and +(0.04,0).. (1.4,0.86).. controls +(-0.04,0) and +(-0.03,0.1) ..(1.34,0.8).. controls +(0.03,-0.1) and +(-0.01,-0.01)..(1.4,0.74);
				\end{scope}
				\begin{scope}[yscale=-1,xscale=1]
					\draw (0.3,0.2)--(1.7,0.2);
					\draw (1.8,0.3)--(1.8,0.7);
					\draw (1.7,0.8)--(0.3,0.8);
					\draw (0.2,0.7)--(0.2,0.3);
					\draw (0.9,0.25)--(1,0.2);
					\draw (0.9,0.15)--(1,0.2);
					\draw (1.7,0.2) .. controls +(0.05,0) and +(0,-0.05) .. (1.8,0.3);
					\draw (1.8,0.7) .. controls +(0,0.05) and +(0.05,0) .. (1.7,0.8);
					\draw (0.3,0.8) .. controls +(-0.05,0) and +(0,0.05) .. (0.2,0.7);
					\draw (0.2,0.3) .. controls +(0,-0.05) and +(-0.05,0) .. (0.3,0.2);
					\filldraw[draw=black,fill=blue] (1.4,0.74) .. controls +(0.05,0.05) and +(0,-0.05) .. (1.46,0.8) .. controls +(0,0.05) and +(0.04,0).. (1.4,0.86).. controls +(-0.04,0) and +(-0.03,0.1) ..(1.34,0.8).. controls +(0.03,-0.1) and +(-0.01,-0.01)..(1.4,0.74);
				\end{scope}
				\begin{scope}[shift={(0,2)},yscale=-1,xscale=-1]
					\draw (0.3,0.2)--(1.7,0.2);
					\draw (1.8,0.3)--(1.8,0.7);
					\draw (1.7,0.8)--(0.3,0.8);
					\draw (0.2,0.7)--(0.2,0.3);
					\draw (0.9,0.25)--(1,0.2);
					\draw (0.9,0.15)--(1,0.2);
					\draw (1.7,0.2) .. controls +(0.05,0) and +(0,-0.05) .. (1.8,0.3);
					\draw (1.8,0.7) .. controls +(0,0.05) and +(0.05,0) .. (1.7,0.8);
					\draw (0.3,0.8) .. controls +(-0.05,0) and +(0,0.05) .. (0.2,0.7);
					\draw (0.2,0.3) .. controls +(0,-0.05) and +(-0.05,0) .. (0.3,0.2);
					\filldraw[draw=black,fill=red] (1.4,0.74) .. controls +(0.05,0.05) and +(0,-0.05) .. (1.46,0.8) .. controls +(0,0.05) and +(0.04,0).. (1.4,0.86).. controls +(-0.04,0) and +(-0.03,0.1) ..(1.34,0.8).. controls +(0.03,-0.1) and +(-0.01,-0.01)..(1.4,0.74);
				\end{scope}
				\begin{scope}[shift={(4,2)},yscale=-1,xscale=-1]
					\draw (0.3,0.2)--(1.7,0.2);
					\draw (1.8,0.3)--(1.8,0.7);
					\draw (1.7,0.8)--(0.3,0.8);
					\draw (0.2,0.7)--(0.2,0.3);
					\draw (0.9,0.25)--(1,0.2);
					\draw (0.9,0.15)--(1,0.2);
					\draw (1.7,0.2) .. controls +(0.05,0) and +(0,-0.05) .. (1.8,0.3);
					\draw (1.8,0.7) .. controls +(0,0.05) and +(0.05,0) .. (1.7,0.8);
					\draw (0.3,0.8) .. controls +(-0.05,0) and +(0,0.05) .. (0.2,0.7);
					\draw (0.2,0.3) .. controls +(0,-0.05) and +(-0.05,0) .. (0.3,0.2);
					\filldraw[draw=black,fill=red] (1.4,0.74) .. controls +(0.05,0.05) and +(0,-0.05) .. (1.46,0.8) .. controls +(0,0.05) and +(0.04,0).. (1.4,0.86).. controls +(-0.04,0) and +(-0.03,0.1) ..(1.34,0.8).. controls +(0.03,-0.1) and +(-0.01,-0.01)..(1.4,0.74);
				\end{scope}
				\begin{scope}[shift={(4,0)},yscale=-1,xscale=-1]
					\draw (0.3,0.2)--(1.7,0.2);
					\draw (1.8,0.3)--(1.8,0.7);
					\draw (1.7,0.8)--(0.3,0.8);
					\draw (0.2,0.7)--(0.2,0.3);
					\draw (0.9,0.25)--(1,0.2);
					\draw (0.9,0.15)--(1,0.2);
					\draw (1.7,0.2) .. controls +(0.05,0) and +(0,-0.05) .. (1.8,0.3);
					\draw (1.8,0.7) .. controls +(0,0.05) and +(0.05,0) .. (1.7,0.8);
					\draw (0.3,0.8) .. controls +(-0.05,0) and +(0,0.05) .. (0.2,0.7);
					\draw (0.2,0.3) .. controls +(0,-0.05) and +(-0.05,0) .. (0.3,0.2);
					\filldraw[draw=black,fill=red] (1.4,0.74) .. controls +(0.05,0.05) and +(0,-0.05) .. (1.46,0.8) .. controls +(0,0.05) and +(0.04,0).. (1.4,0.86).. controls +(-0.04,0) and +(-0.03,0.1) ..(1.34,0.8).. controls +(0.03,-0.1) and +(-0.01,-0.01)..(1.4,0.74);
				\end{scope}
				\begin{scope}[yscale=-1,xscale=-1]
					\draw (0.3,0.2)--(1.7,0.2);
					\draw (1.8,0.3)--(1.8,0.7);
					\draw (1.7,0.8)--(0.3,0.8);
					\draw (0.2,0.7)--(0.2,0.3);
					\draw (0.9,0.25)--(1,0.2);
					\draw (0.9,0.15)--(1,0.2);
					\draw (1.7,0.2) .. controls +(0.05,0) and +(0,-0.05) .. (1.8,0.3);
					\draw (1.8,0.7) .. controls +(0,0.05) and +(0.05,0) .. (1.7,0.8);
					\draw (0.3,0.8) .. controls +(-0.05,0) and +(0,0.05) .. (0.2,0.7);
					\draw (0.2,0.3) .. controls +(0,-0.05) and +(-0.05,0) .. (0.3,0.2);
					\filldraw[draw=black,fill=red] (1.4,0.74) .. controls +(0.05,0.05) and +(0,-0.05) .. (1.46,0.8) .. controls +(0,0.05) and +(0.04,0).. (1.4,0.86).. controls +(-0.04,0) and +(-0.03,0.1) ..(1.34,0.8).. controls +(0.03,-0.1) and +(-0.01,-0.01)..(1.4,0.74);
				\end{scope}
			\end{tikzpicture}
		\end{center}
		\caption{Rectangle cells choreography}\label{rectangle choreography}
	\end{figure}
	
	\subsubsection{Duplication of $m$-sectors: application to the unit disc}\label{sec m sectors}
	Here, we aim to investigate time-periodic multi-vortex motion in sectors, similar to our previous discussion on rectangles.
	Let  $\mathbf{D}$ be a simply connected bounded  domain  centered in $0$ such that there exists  $m\in\mathbb{N}^*$ with 
	\begin{equation}\label{def bfD mfold}
		\mathbf{D}=\textnormal{Int}\left(\bigcup_{k=0}^{2m-1}e^{\frac{\ii k\pi}{m}}\overline{\mathbf{D}_m}\right),
	\end{equation}
	where $\mathbf{D}_m$ is the fundamental cell given by the $m$-sector
	\begin{equation}\label{def m-sectorD}
		\mathbf{D}_m\triangleq\big\{z\in\mathbf{D}\quad\textnormal{s.t.}\quad0<\arg(z)<\tfrac{\pi}{m}\big\}.
	\end{equation}
	As examples, we can take discs, ellipses with $m=2$ or  regular polygons  of $2^m$ sides.
	
		
Using the image method, we can establish a connection between the Green functions of the cell  $\mathbf{D}_m$ and the full domain  $\mathbf{D}.$
The result is as follows, with a proof analogous to the initial part of the Lemma. \ref{lemma-crit}.
	\begin{proposition}
		The Green function of the domain $\mathbf{D}_m$ defined through \eqref{def m-sectorD} takes the form 
		\begin{align*}
			G_{\mathbf{D}_m}(z,w)=\sum_{k=0}^{m-1}\Big[G_{\mathbf{D}}\big(z,\omega_m^kw\big)-G_{\mathbf{D}}\big(z,\omega_m^k\overline{w}\big)\Big], \quad \omega_{m}\triangleq e^{\frac{2\ii\pi}{m}}.
		\end{align*}
	\end{proposition}
	Now, we shall focus on  the particular case of the unit disc $\mathbb{D}.$
	Define the sets
	\begin{align}\label{dom+}\mathbb{D}_+\triangleq\big\{z\in\mathbb{D}\quad\textnormal{s.t.}\quad\textnormal{Im}(z)>0\big\},\qquad  \mathbb{H}_+\triangleq\big\{z\in\mathbb{C}\quad\textnormal{s.t.}\quad\textnormal{Im}(z)>0\big\}
	\end{align}
	and for $m\in\mathbb{N}^*$ we define the unit $m$-sector
	\begin{equation}\label{def sector Dm}
		\mathbb{D}_m\triangleq\big\{z\in\mathbb{D}\quad\textnormal{s.t.}\quad 0<\arg(z)<\tfrac{\pi}{m}\big\}.
	\end{equation}
	 For the specific case of the disc, the computations are explicit.
	\begin{lemma}\label{lemma-crit}
		The Green function of the unit $m$-sector $\mathbb{D}_m$ defined in \eqref{def sector Dm} takes the form 
		$$G_{\mathbb{D}_m}(z,w)=\sum_{k=0}^{m-1}\Big[G_{\mathbb{D}}\big(z,\omega_m^kw\big)-G_{\mathbb{D}}\big(z,\omega_m^k\overline{w}\big)\Big].$$
		Its associated Robin function writes
		$$\mathcal{R}_{\mathbb{D}_m}(z)=-\log\left({1-|z|^2}\right)-\log\left|\tfrac{z-\overline{z}}{1-z^2}\right|+\sum_{k=1}^{m-1}\Big[G_{\mathbb{D}}\big(z,\omega_m^kz\big)-G_{\mathbb{D}}\big(z,\omega_m^k\overline{z}\big)\Big]$$
		and admits a unique critical point $\xi_m\in\mathbb{D}_m$ given by
		\begin{equation}\label{def crit xim}
			\xi_m=t_me^{\frac{\ii\pi}{2m}},\qquad t_m\triangleq\left(2m+\sqrt{4m^2+1}\right)^{-\frac{1}{2m}}.
		\end{equation}
	\end{lemma}
	
\begin{figure}[!h]
	\begin{center}
		\begin{tikzpicture}[scale=2]
			\filldraw[draw=black,fill=black!20] (1,0) arc(0:180:1) -- (0,0) --cycle;
			\node at (0,0.486) {$\bullet$};
			\node at (0,0.7) {$\xi_{1}$};
		\end{tikzpicture}
		\qquad
		\begin{tikzpicture}[scale=2]
			\filldraw[draw=black,fill=black!20] (1,0) arc(0:90:1) -- (0,0) --cycle;
			\node at (0.419,0.419) {$\bullet$};
			\node at (0.55,0.55) {$\xi_{2}$};
		\end{tikzpicture}
		\qquad
		\begin{tikzpicture}[scale=2]
			\filldraw[draw=black,fill=black!20] (1,0) arc(0:45:1) -- (0,0) --cycle;
			\node at (0.653,0.270) {$\bullet$};
			\node at (0.7,0.42) {$\xi_{4}$};
		\end{tikzpicture}
	\end{center}
	\caption{Position of the critical point $\xi_{m}$ in the circular sector $D_m$ for $m\in\{1,2,4\}.$}
\end{figure}
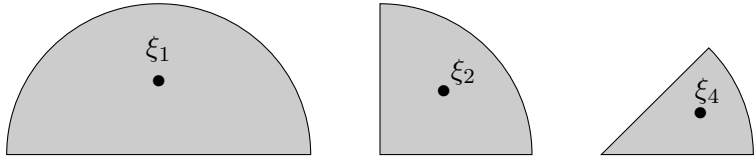

	\begin{proof}
		By the image method \cite{K53}, finding the Green function in $\mathbb{D}_{m}$ amounts to solve in $\mathbb{D}$ the equation
		\begin{eqnarray}   \label{Eulereqm}	        
			\begin{cases}
				\Delta_z G_{\mathbb{D}_m}(z,w)=\displaystyle2\pi\sum_{k=0}^{m-1}\delta_{\omega_m^kw}(z)-2\pi\sum_{k=0}^{m-1}\delta_{\omega_m^k\overline{w}}(z), &\text{ in $\mathbb{D}$}, \\
				G_{\mathbb{D}_m}(z,w)=0,&\text{ for all $z\in\partial \mathbb{D}$}. 
			\end{cases}
		\end{eqnarray}
		By a uniqueness argument, we conclude that
		$$G_{\mathbb{D}_m}(z,w)=\sum_{k=0}^{m-1}\Big[G_{\mathbb{D}}\big(z,\omega_m^kw\big)-G_{\mathbb{D}}\big(z,\omega_m^k\overline{w}\big)\Big].$$
		By isolating the term $k=0$, we can write
		\begin{align*}
			G_{\mathbb{D}_m}(z,w)=\log\left|\tfrac{z-w}{1-z\overline{w}}\right|-\log\left|\tfrac{z-\overline{w}}{1-zw}\right|+\sum_{k=1}^{m-1}\Big[G_{\mathbb{D}}\big(z,\omega_m^kw\big)-G_{\mathbb{D}}\big(z,\omega_m^k\overline{w}\big)\Big].
		\end{align*}
		Thus Robin Function takes the form
		\begin{align*}
			\mathcal{R}_{\mathbb{D}_m}(z)=-\log\left({1-|z|^2}\right)-\log\left|\tfrac{z-\overline{z}}{1-z^2}\right|+\sum_{k=1}^{m-1}\Big[G_{\mathbb{D}}\big(z,\omega_m^kz\big)-G_{\mathbb{D}}\big(z,\omega_m^k\overline{z}\big)\Big].
		\end{align*}
		Straightforward computations yield for any $k\in\llbracket1,m-1\rrbracket,$
		\begin{align*}
			\partial_z\Big(G_{\mathbb{D}}\big(z,\omega_m^kz\big)\Big)=\frac{1}{2z}+\overline{z}\,\hbox{Re}\left(\tfrac{\omega_m^k}{1-\omega_m^k |z|^2}\right)
		\end{align*}
		and for any $k\in\llbracket0,m-1\rrbracket,$
		\begin{align*}
			\partial_z\Big(G_{\mathbb{D}}\big(z,\omega_m^k\overline{z}\big)\Big)=\tfrac{1}{z-\omega_m^k\overline{z}}+\tfrac{\omega_m^{-k}{z}}{1-\omega_m^{-k}z^2}\cdot
		\end{align*}
		Hence
		\begin{align*}
			\partial_z\mathcal{R}_{\mathbb{D}_m}(z)=\frac{m-1}{2z}+\sum_{k=0}^{m-1}\overline{z}\,\hbox{Re}\left(\tfrac{\omega_m^k}{1-\omega_m^k|z|^2}\right)+\tfrac{1}{\omega_m^k\overline{z}-z}+\tfrac{\omega_m^{-k}{z}}{\omega_m^{-k}z^2-1}\cdot
		\end{align*}
		The critical point $\xi_m\in\mathbb{D}_m$ is unique since $\mathbb{D}_m$ is convex and it solves the equation
		\begin{equation}\label{eq crit xim}
			\frac{m-1}{2\xi_m}+\sum_{k=0}^{m-1}\overline{\xi_m}\,\hbox{Re}\left(\tfrac{\omega_m^k}{1-\omega_m^k|\xi_m|^2}\right)+\tfrac{1}{\omega_m^k\overline{\xi_m}-\xi_m}+\tfrac{\omega_m^{-k}{\xi_m}}{\omega_m^{-k}\xi_m^2-1}=0.
		\end{equation}
		By symmetry, we expect 
		\begin{equation}\label{def xim}
			\xi_m=t_m\,\omega_{4m},\qquad t_m\in(0,1).
		\end{equation}
		Notice that
		\begin{equation}\label{relations omgm}
			\omega_{4m}=e^{\frac{\ii\pi}{2m}},\qquad\omega_{4m}^{4}=\omega_m,\qquad\omega_{4m}^{4m}=1.
		\end{equation}
		Inserting \eqref{def xim} into \eqref{eq crit xim} and using \eqref{relations omgm} yields
		\begin{align*}
			\frac{m-1}{2t_m}+\sum_{k=0}^{m-1} t_m\,\hbox{Re}\left(\tfrac{\omega_{4m}^{4k}}{1-\omega_{4m}^{4k} t_m^2} \right)+\tfrac{1}{t_m(\omega_{4m}^{4k-2}-1)}+\tfrac{\omega_{4m}^{2-4k}{t_m}}{\omega_{4m}^{2-4k}t_m^2-1}=0. \end{align*}
		In addition, the third identity in \eqref{relations omgm} implies
		\begin{align*}
			\sum_{k=0}^{m-1} \hbox{Re}\left(\tfrac{\omega_{4m}^{4k}}{1-\omega_{4m}^{4k}t_m^2}\right)&=\tfrac12\sum_{k=0}^{m-1} \tfrac{\omega_{4m}^{4k}}{1-\omega_{4m}^{4k}t_m^2} +\tfrac12\sum_{k=0}^{m-1} \tfrac{\omega_{4m}^{-4k}}{1-\omega_{4m}^{-4k}t_m^2}\\
			&=\sum_{k=0}^{m-1} \tfrac{\omega_{4m}^{4k}}{1-\omega_{4m}^{4k} t_m^2}
		\end{align*}
		and
		\begin{align*}
			\sum_{k=0}^{m-1} \tfrac{\omega_{4m}^{2-4k}{t}}{\omega_{4m}^{2-4k}t_m^2-1}=\sum_{k=0}^{m-1}\tfrac{\omega_{4m}^{2+4k}t_m}{\omega_{4m}^{2+4k}t_m^2-1}\cdot
		\end{align*}
		Therefore,
		\begin{align*}
			\left(\frac{m-1}{2}+\sum_{k=0}^{m-1}\tfrac{1}{\omega_{4m}^{4k-2}-1}\right)+\sum_{k=0}^{m-1} \tfrac{\omega_{4m}^{4k}t_m^2}{1-\omega_{4m}^{4k}t_m^2}+\tfrac{\omega_{4m}^{2+4k}t_m^2}{\omega_{4m}^{2+4k}t_m^2-1}=0,
		\end{align*}
		which is equivalent to
		\begin{align*}
			\left(\frac{m-1}{2}+\sum_{k=0}^{m-1}\tfrac{1}{\omega_{4m}^{4k-2}-1}\right)+\sum_{k=0}^{m-1}\tfrac{1}{1-\omega_{4m}^{4k}t_m^2}+\tfrac{1}{\omega_{4m}^{4k+2}t_m^2-1}=0.
		\end{align*}
		Writing
		\begin{align*}
			\tfrac{1}{1-\omega_{4m}^{4k} t_m^2}&=\sum_{j=0}^\infty \omega_{4m}^{4kj} t_m^{2j}\\
			&=\sum_{\ell=0}^\infty\sum_{j=\ell m}^{(\ell+1) m-1}\omega_{4m}^{4kj} t_m^{2j}\\
			&=\sum_{\ell=0}^\infty t_m^{2\ell m}\sum_{j=0}^{m-1} \omega_{4m}^{4kj} t_m^{2j}\\
			&=\frac{P_k(t_m)}{1-t_m^{2m}},
		\end{align*}
		with
		\begin{align*}
			P_k(x)\triangleq\sum_{j=0}^{ m-1} \omega_{4m}^{4kj}x^{2j}.
		\end{align*}
		Then
		\begin{align*}
			\tfrac{1}{\omega_{4m}^{4k+2}t_m^2-1}=\tfrac{P_k(\omega_{4m} t_m)}{\omega_{4m}^{2m}t_m^{2m}-1}=-\tfrac{P_k(\omega_{4m} t_m)}{t_m^{2m}+1}\cdot
		\end{align*}
		In addition, notice that from \eqref{relations omgm}, we have
		\begin{align*}
			P_k\left(\omega_{4m}^{-1} \right)&=\sum_{j=0}^{m-1}\left(\omega_{4m}^{4k-2}\right)^j\\
			&=\frac{1-\omega_{4m}^{(4k-2)m}}{1-\omega_{4m}^{4k-2}}\\
			&=\frac{2}{1-\omega_{4m}^{4k-2}}\cdot
		\end{align*}
		Thus
		\begin{align*}
			\left(\tfrac{m-1}{2}-\tfrac12\sum_{k=0}^{m-1}P_k\left(\omega_{4m}^{-1} \right)\right)-\sum_{k=0}^{m-1} \tfrac{P_k(t_m)}{t^{2m}-1}+\tfrac{P_k(\omega_{4m} t_m)}{t^{2m}+1}=0. \end{align*}
		By exchanging finite sums, we can check that for any $x\in\mathbb{C}$,
		\begin{align*}
			\sum_{k=0}^{m-1}{P_k(x)}&=\sum_{k=0}^{m-1}\sum_{j=0}^{m-1}\omega_{4m}^{4kj}x^{2j}\\
			&=m+\sum_{j=1}^{m-1}x^{2j}\sum_{k=0}^{m-1}\omega_{4m}^{4kj}\\
			&=m,
		\end{align*}
		since from \eqref{relations omgm} we have for any $j\in\llbracket 1,m-1\rrbracket,$
		$$\sum_{k=0}^{m-1}\omega_{4m}^{4kj}=\frac{1-\omega_{4m}^{4jm}}{1-\omega_{4m}^{4j}}=0.$$
		It follows that
		\begin{align*}
			\tfrac12+\tfrac{m}{t_m^{2m}-1}+\tfrac{m}{t_m^{2m}+1}=0.
		\end{align*}
		Thus, $t_m$ is solution to the equation
		\begin{align*}
			t_m^{4m}+4mt_m^{2m}-1=0.
		\end{align*}
		Hence, under the constraint $t_m\in(0,1)$, we find
		$$t_m=\left(2m+\sqrt{4m^2+1}\right)^{-\frac{1}{2m}}.$$
		This achieves the proof of Lemma \ref{lemma-crit}.
	\end{proof}
	Now, we prove the following result.
	\begin{proposition}\label{disc-chor}
		Let $m\in\mathbb{N}^*$ and consider $\Phi_m:\mathbb{D}_m\to\mathbb{D}$ the unique Riemann mapping such that
		$$\Phi_m(\xi_m)=0\qquad\Phi_m^\prime(\xi_m)>0,$$
		where $\xi_m\in\mathbb{D}_m$ is the unique critical point of the Robin function $\mathcal{R}_{\mathbb{D}_{m}}$ given in \eqref{def crit xim}. Then, the  \mbox{condition \eqref{cond1 thm crit}} is satisfied, with  $F=\Phi^{-1}_m$,  and Corollary $\ref{convex-dom}$ holds true. As a consequence, for any $w\in \mathbb{D}_{m}$,
		the initial configuration
		$$
		\omega_0=\displaystyle\pi\sum_{k=0}^{m-1}\delta_{w\omega_m^k}(z)-\pi\sum_{k=0}^{m-1}\delta_{\overline{w}\omega_m^k}(z)
		$$
		generates time-periodic solution to the point vortex system, which  can be desingularized with time periodic vortex patches.
	\end{proposition}
	\begin{proof}
		It is known that the following mappings are conformal
		$$\phi_+:\begin{array}[t]{rcl}
			\mathbb{D}_+ & \to & \mathbb{H}_+,\vspace{0.05cm}\\
			z & \mapsto & -\tfrac12(z+\tfrac{1}{z})
		\end{array}\qquad \phi_0:\begin{array}[t]{rcl}
			\mathbb{H}_+ & \to & \mathbb{D}\\
			z & \mapsto & \tfrac{z-\ii}{z+\ii},
		\end{array}$$
		where we refer to \eqref{dom+} for the definition of $\mathbb{D}_+ $ and $\mathbb{H}_+.$
		Therefore $\phi\triangleq \phi_0\circ\phi_+: \mathbb{D}_+\to \mathbb{D}$ is conformal. One can check that
		$$\phi(z)=\tfrac{z^2+2i z+1}{z^2-2i z+1}\cdot$$
		On the other hand, the mapping $z\in \mathbb{D}_m\mapsto z^m \in \mathbb{D}_+$ is conformal. Therefore the mapping 
		$$\phi_m:\begin{array}[t]{rcl}
			\mathbb{D}_m & \to & \mathbb{D}\\
			z & \mapsto & \phi(z^m)=\tfrac{z^{2m}+2i z^m+1}{z^{2m}-2i z^m+1}
		\end{array}$$
		is conformal. We want to construct the unique conformal mapping $\Phi_m:\mathbb{D}_m\mapsto\mathbb{D}$ such that
		$$\Phi_m(\xi_m)=0,\qquad\Phi_m'(\xi_m)>0,$$
		where $\xi_m$ is the critical point defined in \eqref{def crit xim}. Now, set
		$$a_m\triangleq\phi_m(\xi_m)=\tfrac{t_m^{2m}+2 t_m^m-1}{t_m^{2m}-2 t_m^m-1}\in(-1,1)$$
		and denote by $T_m:\mathbb{D}\to \mathbb{D}$ the conformal map (Moebius transform)
		$$T_m(z)=e^{i\theta_m}\frac{z-a_m}{1-a_m\, z}$$
		for some $\theta_m\in\mathbb{R}$ that will be fixed later.
		Then set
		$$\Phi_m\triangleq T_m\circ\phi_m.$$
		From direct computations we show that
		$$\Phi_m(\xi_m)=0$$
		and, from \eqref{def crit xim},
		$$\phi_m^\prime(\xi_m)=\alpha_m e^{i(\pi-\frac{\pi}{2m})},\qquad\alpha_m\triangleq\frac{4mt_m^{m-1}(1+t_{m}^{2m})}{(1+2t_m^m-t_m^{2m})^2}>0.$$
		Therefore
		$$\Phi'_m(\xi_m)=\frac{\alpha_m}{1-a_m^2}e^{i(\theta+\pi-\frac{\pi}{2m})}.$$
		By taking $\theta_m\triangleq\frac{\pi}{2m}-\pi$, we infer that
		$$\Phi'_m(\xi_m)=\frac{\alpha_m}{1-a_m^2}>0.$$
		By Fa\`{a} di Bruno formula, we have that
		$$\Phi_m'=\phi_m'\cdot T_m'\circ\phi_m,\qquad\Phi_m^{(3)}=(\phi_m')^3\cdot T_m^{(3)}\circ\phi_m+3\phi_m'\phi_m''\cdot T_m''\circ\phi_m+\phi_m^{(3)}\cdot T'_m\circ\phi_m.$$
		Helped by Mathematica, we find that
		$$\left| \tfrac{F^{(3)}(0)}{F^\prime(0)}\right|^2=\left|\tfrac{\Phi_m^{(3)}(\xi_m)}{\big(\Phi_m'(\xi_m)\big)^3}\right|^2=\mathtt{A}_m\sqrt{4m^2+1}+\mathtt{B}_m,$$
		where
		$$\mathtt{A}_m\triangleq\frac{-8}{m^8}\big(2m^6+9m^4+6m^2+1\big)\in\mathbb{Q}^*,\qquad\mathtt{B}_m\triangleq\frac{4}{m^8}\big(m^8+24m^6+38m^4+8m^2+2\big)\in\mathbb{Q}.$$
		Notice that for any integer $m\geqslant1$,
		$$\sqrt{4m^2+1}\not\in\mathbb{Q}.$$
		Hence,
		$$\forall n\in\mathbb{N}^*,\quad \mathtt{A}_m\sqrt{4m^2+1}+\mathtt{B}_m\neq4\left(1-\tfrac{1}{n^2}\right).$$
		This proves the condition \eqref{cond1 thm crit}, which concludes the proof of Proposition \ref{disc-chor}.
	\end{proof}

	\section{Proof of Theorem \ref{main thm}}
The goal of this section is to prove Theorem \ref{main thm} by employing  a modified Nash-Moser scheme tin conjunction  with KAM tools.  Specifically, we will utilize the framework developed in \cite{HHM23} to address the leapfrogging phenomenon in the context of Euler equations involving two symmetric pairs of vortices. Our approach begins with an introduction to the functional setting, followed by a detailed presentation of the asymptotic structure of the linearized operator. Subsequently, we will construct a suitable approximate solution to the nonlinear equation \eqref{functional scrG}. Next, we rescale the functional $\mathbf{G}$, defined via \eqref{functional scrG}, and transform the new linear operator into a Fourier multiplier, with a regularizing remainder. This transformation enables us to identify an approximate right inverse of the linearized operator, which is a critical component of the Nash-Moser scheme.
%
\subsection{Functional setting}
We denote by $\mathbb{T}^2$ the two-dimensional torus and any  $h:\mathbb{T}^{2}\to \mathbb{C}$ in the class $L^2(\mathbb{T}^2,\mathbb{C})$ admits the Fourier expansion,
\begin{align*}
	h=\sum_{(l,j)\in\mathbb{Z}^{2 }}h_{l,j}\,\mathbf{e}_{l,j},\qquad \mathbf{e}_{l,j}(\varphi,\theta)\triangleq e^{\ii(l\varphi+j\theta)},\qquad h_{l,j}\triangleq\big\langle h,\mathbf{e}_{l,j}\big\rangle_{L^{2}(\mathbb{T}^{2},\mathbb{C})}.
\end{align*}
The Hilbert  $L^2(\mathbb{T}^2,\mathbb{C})$ is equipped with the scalar product 
\begin{align*}
	\big\langle h_1,h_2\big\rangle_{L^2(\mathbb{T}^2,\mathbb{C})}\triangleq\int_{\mathbb{T}^2}h_1(\varphi,\theta)\overline{h_2(\varphi,\theta)}d\varphi d\theta=\sum_{(l,j)\in\mathbb{Z}^2}h_{1,l,j}\overline{h_{2,l,j}}.
\end{align*}
The following notation is of constant use in this work
\begin{align}\label{conv-norm}
	\int_{\mathbb{T}}f(x)dx\triangleq\frac{1}{2\pi}\int_{0}^{2\pi}f(x)dx.
\end{align}
The Sobolev spaces of regularity $s\in\mathbb{R}$ is given by
\begin{align*}
	H^{s}(\mathbb{T}^2,\mathbb{C})\triangleq\Big\lbrace h\in L^{2}(\mathbb{T}^2,\mathbb{C})\quad\textnormal{s.t.}\quad\|h\|_{H^s}<\infty\Big\rbrace,
\end{align*}
where
\begin{align*}
	\|h\|_{H^s}^2\triangleq\sum_{(l,j)\in\mathbb{Z}^2}\langle l,j\rangle^{2s}|h_{l,j}|^2,\qquad\langle l,j\rangle^2\triangleq 1+|l|^2+|j|^2.
\end{align*}
The subspace with zero space average functions is denoted
\begin{align*}
	H^{s}_0(\mathbb{T}^{2},\mathbb{C})\triangleq\left\{ h\in H^{s}(\mathbb{T}^{2},\mathbb{C}) \textnormal\quad\hbox{s.t.}\quad\forall \varphi\in\mathbb{T},\,\int_{\mathbb{T}}h(\varphi,\theta) d\theta=0\right\}.
\end{align*}
Given a nonempty subset  $\mathscr{O}$  of $\mathbb{R}$ and $\gamma\in(0,1]$,  we define the Banach spaces
\begin{align*}
	\textnormal{Lip}_{\gamma}(\mathscr{O},H^{s})&\triangleq\Big\lbrace h:\mathscr{O}\rightarrow H^{s}(\mathbb{T}^2,\mathbb{C})\quad\textnormal{s.t.}\quad\|h\|_{s}^{{\textnormal{Lip}(\gamma)}}<\infty\Big\rbrace,\\
	\textnormal{Lip}_{\gamma}(\mathscr{O},\mathbb{C})&\triangleq\Big\lbrace h:\mathscr{O}\rightarrow\mathbb{C}\quad\textnormal{s.t.}\quad\|h\|^{{\textnormal{Lip}(\gamma)}}<\infty\Big\rbrace,
\end{align*}
with
\begin{align*}
	\|h\|_{s}^{\textnormal{Lip}(\gamma)}&\triangleq\sup_{\lambda\in{\mathscr{O}}}\|h(\lambda,\cdot)\|_{H^{s}}+\gamma\sup_{(\lambda_1,\lambda_2)\in\mathscr{O}^2\atop\lambda_{1}\neq\lambda_{2}}\frac{\|h(\lambda_1,\cdot)-h(\lambda_2,\cdot)\|_{H^{s-1}}}{|\lambda_1-\lambda_2|},
\\
	\|h\|^{{\textnormal{Lip}(\gamma)}}&\triangleq\sup_{\lambda\in{\mathscr{O}}}|h(\lambda)|+\gamma\sup_{(\lambda_1,\lambda_2)\in\mathscr{O}^2\atop\lambda_{1}\neq\lambda_{2}}\frac{|h(\lambda_1)-h(\lambda_2)|}{|\lambda_1-\lambda_2|}\cdot
\end{align*}
We emphasize that in Section \ref{Nash Moser} related to Nash-Moser scheme we find it convenient to use the notation
\begin{align}\label{Norm-not}
	\|h\|_{s,\mathscr{O}}^{\textnormal{Lip}(\gamma)}=\|h\|_{s}^{\textnormal{Lip}(\gamma)}.
\end{align}
 In the sequel, we consider the list of numbers with the constraints
\begin{equation}\label{cond1}
	\gamma\in(0,1],\qquad\tau>1,\qquad  S\geqslant s\geqslant s_0>3.
\end{equation}
Along this section, we will extensively use the following classical result related to the product law over weighted Sobolev spaces. Given  $(\gamma,s_{0},s)$ satisfying \eqref{cond1}.
Let $h_{1},h_{2}\in \textnormal{Lip}_{\gamma}(\mathscr{O},H^{s}).$ Then $h_{1}h_{2}\in \textnormal{Lip}_{\gamma}(\mathscr{O},H^{s})$ and 
		\begin{equation}\label{ prod law}
			\| h_{1}h_{2}\|_{s}^{\textnormal{Lip}(\gamma)}\lesssim\| h_{1}\|_{s_{0}}^{\textnormal{Lip}(\gamma)}\| h_{2}\|_{s}^{\textnormal{Lip}(\gamma)}+\| h_{1}\|_{s}^{\textnormal{Lip}(\gamma)}\| h_{2}\|_{s_{0}}^{\textnormal{Lip}(\gamma)}.
		\end{equation}
We  need some anisotropic spaces to describe  the spatial smoothing effects of some non-local operators. The anisotropic Sobolev spaces of regularity $s_1,s_2\in\mathbb{R}$ is given by
\begin{align*}
	H^{s_1,s_2}(\mathbb{T}^2,\mathbb{C})\triangleq\Big\lbrace h\in L^{2}(\mathbb{T}^2,\mathbb{C})\quad\textnormal{s.t.}\quad\|h\|_{H^{s_1,s_2}}<\infty\Big\rbrace,
\end{align*}
where
\begin{align*}
	\|h\|_{H^{s_1,s_2}}^2\triangleq\sum_{(l,j)\in\mathbb{Z}^2}\langle l,j\rangle^{2s_1}\langle j\rangle^{2s_2}|h_{l,j}|^2.
\end{align*}
The corresponding Lipschitz norm for external parameters
\begin{align*}
	\|h\|_{s_1,s_2}^{\textnormal{Lip}(\gamma)}\triangleq\sup_{\lambda\in{\mathscr{O}}}\|h(\gamma,\cdot)\|_{H^{s_1,s_2}}+\gamma\sup_{(\lambda_1,\lambda_2)\in\mathscr{O}^2\atop\lambda_{1}\neq\lambda_{2}}\frac{\|h(\lambda_1,\cdot)-h(\lambda_2,\cdot)\|_{H^{s_1-1,s_2}}}{|\lambda_1-\lambda_2|}\cdot
\end{align*}
To conclude this section, we will briefly recall the Hilbert transform on $\mathbb{T}$ . 
Let $h:\mathbb{T}\to\mathbb{R}$ be a continuous function with a zero average, we define its Hilbert transform by 
\begin{align}\label{def Hilbert trans}
	\mathbf{H}h(\theta)\triangleq\int_{\mathbb{T}}h(\eta)\cot\left(\tfrac{\eta-\theta}{2}\right)d\eta=-\partial_{\theta}\int_{\mathbb{T}}h(\eta)\log\left(\sin^2\left(\tfrac{\theta-\eta}{2}\right)\right)d\eta.
\end{align}
The integrals are understood in the principal value sense. It is a classical that $\mathbf{H}$ is a Fourier multiplier with  
$$\forall j\in\mathbb{Z}^*,\quad \mathbf{H} \mathbf{e}_j(\theta)=\ii\,\hbox{sign}(j)  \mathbf{e}_j(\theta).$$

\subsection{Asymptotic structure of the linearized operator}
Our aim here is to analyze the asymptotic behavior in $\varepsilon$ of the differential $d_r\mathbf{G}(r)$ at a small state $r$ for the functional $\mathbf{G}$ introduced in \eqref{functional scrG}. The main result is stated  in Proposition \ref{proposition linear op exp} below. As a preliminary step and  in view of \eqref{functional scrG}-\eqref{def-Psi2}, we need to  perform the $\varepsilon$-expansion, up to order 2, of the kernel
$$K(p+\varepsilon z,p+\varepsilon\xi)$$
using its representation via the conformal mapping \eqref{Rob-double}. 
\begin{lemma}\label{lemma expand K}
	Let   $V$ be an open set such that $\overline{V}\subset \mathbf{D}$. Denote also $\mathbb{D}_2$ the disc of radius $2$ centered in the origin. There exist $\varepsilon_0>0$ and $K_3\in C^\infty_b(V\times\mathbb{D}_2^2)$ such that for any $\varepsilon\in(0,\varepsilon_0)$ and $(p,z,\xi)\in V\times\mathbb{D}_2^2$, we have
	\begin{equation}\label{expansion K}
		K(p+\varepsilon z,p+\varepsilon\xi)=\mathcal{R}_{\mathbf{D}}(p)+\varepsilon K_1(p,z,\xi)+\varepsilon^2K_2(p,z,\xi)+\varepsilon^3 K_3(p,z,\xi),
	\end{equation}
	with
	\begin{align*}
		K_1(p,z,\xi)&\triangleq\textnormal{Re}\Big\{\partial_z\mathcal{R}_{\mathbf{D}}(p)\, (z+\xi)\Big\},\\
		K_2(p,z,\xi)&\triangleq\textnormal{Re}\Big\{\Big(\tfrac{1}{6}S(\Phi)(p)+\tfrac{1}{2}\big(\partial_z\mathcal{R}_{\mathbf{D}}(p)\big)^2\Big)(z^2+\xi^2)+\tfrac{1}{6}S(\Phi)(p)z\xi+\tfrac{1}{r_{\mathbf{D}}^2(p)}z\overline\xi\Big\},
	\end{align*}
	where we have used the notations \eqref{Conf-R}, \eqref{Robin1N} and \eqref{Swart}.
\end{lemma}
\begin{proof}
	Throughout the proof, we shall make use of the following formulae
	\begin{eqnarray}
		 \partial_{\varepsilon}\big(\log|1+g(\varepsilon)|\big)&=&\textnormal{Re}\left\lbrace\tfrac{\partial_{\varepsilon}g(\varepsilon)}{1+g(\varepsilon)}\right\rbrace,\label{depsi1 log}\\
		\partial_{\varepsilon}^2\big(\log|1+g(\varepsilon)|\big)&=&\textnormal{Re}\left\lbrace\tfrac{(1+g(\varepsilon))\partial_{\varepsilon}^2g(\varepsilon)-(\partial_{\varepsilon}g(\varepsilon))^2}{(1+g(\varepsilon))^2}\right\rbrace.\label{depsi2 log}
	\end{eqnarray}
	 In general, we can write
	\begin{align*}
	\Phi(z)=bz+f(z),\qquad b>0,\qquad f\textnormal{ analytic in } \mathbf{D}.
	\end{align*}
	As $\Phi$ is univalent, then we infer that for any compact subset $\mathbf{K}\subset\mathbf{D},$
	\begin{equation*}
		\max_{\zeta\in\mathbf{K}}|f'(\zeta)|<b.
	\end{equation*}
	Therefore, applying  \eqref{Rob-double} we deduce that
	\begin{align}\label{Green K}
		\nonumber K(z,w)&=-\log|1-\Phi(z)\overline{\Phi(w)}|+\log(b)+\log\left|1+\tfrac{1}{b}\tfrac{f(z)-f(w)}{z-w}\right|\\
		&\triangleq -\log|1+g_1(z,w)|+\log(b)+\log|1+g_2(z,w)|.
	\end{align}
	It follows that
	\begin{equation}\label{Green K2}
		K(p+\varepsilon z,p+\varepsilon\xi)=-\log|1+g_1(p+\varepsilon z,p+\varepsilon\xi)|+\log(b)+\log|1+g_2(p+\varepsilon z,p+\varepsilon\xi)|.
	\end{equation}
		Define
	\begin{align*}
	g_1(\varepsilon)&\triangleq g_1(p+\varepsilon z,p+\varepsilon\xi)\\
	&=-\Phi\big(p+\varepsilon z\big)\overline{\Phi\big(p+\varepsilon\xi\big)}.
	\end{align*}
	Then straightforward calculations based on Taylor expansion of the holomorphic function $\Phi$ yield for small $\varepsilon$
	\begin{align*}
		g_1(\varepsilon)&{=}-|\Phi(p)|^2-\varepsilon\big(\Phi(p)\overline{\Phi'(p)}\,\overline{\xi}+\overline{\Phi(p)}\Phi'(p)z\big)\\
		&\quad-\tfrac{\varepsilon^2}{2}\Big(\overline{\Phi(p)}\Phi''(p)z^{2}+\Phi(p)\overline{\Phi''(p)}\,\overline{\xi}^{2}+2|\Phi'(p)|^2z\overline{\xi}\Big)+O(\varepsilon^3).
	\end{align*}	
This implies
\begin{align*}
	g_1(0)&=-|\Phi(p)|^2,\\
	\partial_{\varepsilon}g_1(0)&=-\big(\Phi(p)\overline{\Phi'(p)}\overline{\xi}+\overline{\Phi(p)}\Phi'(p)z\big),\\
	\partial_{\varepsilon}^2g_1(0)&=-\Big(\overline{\Phi(p)}\Phi''(p)z^{2}+\Phi(p)\overline{\Phi''(p)}\overline{\xi}^{2}+2|\Phi'(p)|^2z\overline{\xi}\Big).
\end{align*}
	Combining these identities with \eqref{depsi1 log} and  \eqref{depsi2 log} allow to get for small $\varepsilon$
	\begin{equation}\label{log g1}
		\begin{aligned}
			\log|1+g_1(\varepsilon)|&{=}\log\big(1-|\Phi(p)|^2\big)-\varepsilon\textnormal{Re}\left\lbrace\tfrac{\Phi(p)\overline{\Phi'(p)}\overline{\xi}+\overline{\Phi(p)}\Phi'(p)z}{1-|\Phi(p)|^2}\right\rbrace\\
			&\quad-\tfrac{\varepsilon^2}{2}\textnormal{Re}\left\lbrace\tfrac{\overline{\Phi(p)}\Phi''(p)z^2+\Phi(p)\overline{\Phi''(p)}\overline{\xi}^{2}+2|\Phi'(p)|^2z\overline{\xi}}{1-|\Phi(p)|^2}\right\rbrace\\
			&\quad-\tfrac{\varepsilon^2}{2}\textnormal{Re}\left\lbrace\left(\tfrac{\Phi(p)\overline{\Phi'(p)}\overline{\xi}+\overline{\Phi(p)}\Phi'(p)z}{1-|\Phi(p)|^2}\right)^2\right\rbrace+O(\varepsilon^3).
		\end{aligned}
	\end{equation}
	Notice that $O(\varepsilon^3)$ is a smooth function
	On the other hand, one writes
	\begin{align*}
	g_2(\varepsilon)&\triangleq \tfrac{1}{b}\tfrac{f(p+\varepsilon z)-f(p+\varepsilon\xi)}{\varepsilon(z-\xi)}\\
	&=\tfrac{1}{b}\int_{0}^{1}f'\big(p+\varepsilon\xi+s\varepsilon(z-\xi)\big)ds.
	\end{align*}
	Therefore, for any $n\in\mathbb{N},$
	$$\partial_{\varepsilon}^ng_2(\varepsilon)=\frac{1}{b}\int_{0}^{1}f^{(n+1)}\big(p+\varepsilon\xi+s\varepsilon(z-\xi)\big)\big(\xi+s(z-\xi)\big)^{n}ds,$$
	implying  in turn that
	\begin{align}\label{despin g20}
		\partial_{\varepsilon}^ng_2(0)&=\tfrac{f^{(n+1)}(p)}{b(n+1)}\tfrac{z^{n+1}-\xi^{n+1}}{z-\xi}=\tfrac{f^{(n+1)}(p)}{b(n+1)}\sum_{k=0}^{n}z^{k}\xi^{n-k}.
	\end{align}
	Putting together \eqref{depsi1 log}, \eqref{depsi2 log} and \eqref{despin g20}, we deduce for small $\varepsilon$
	\begin{align*}
		\log|1+g_{2}(\varepsilon)|&{=}\log\left|1+\tfrac{f'(p)}{b}\right|+\varepsilon\textnormal{Re}\left\lbrace\tfrac{f''(p)(z+\xi)}{2(b+f'(p))}\right\rbrace\\
		&\quad+\tfrac{\varepsilon^2}{6}\textnormal{Re}\left\lbrace\tfrac{f^{(3)}(p)}{b+f'(p)}\big(z^{2}+\xi^{2}+z\xi\big)\right\rbrace\\
		&\quad-\tfrac{\varepsilon^2}{8}\textnormal{Re}\left\lbrace\left(\tfrac{f''(p)(z+\xi)}{b+f'(p)}\right)^2\right\rbrace+O(\varepsilon^3).
	\end{align*}
Since 
	$\Phi'(z)=b+f'(z),$ $\Phi''(z)=f''(z)$ and $\Phi^{(3)}(z)=f^{(3)}(z)$,
	then
	we end up with
	\begin{equation}\label{log g2}
		\begin{aligned}
			\log|1+g_{2}(\varepsilon)|&=-\log(b)+\log\left|\Phi'(p)\right|+\varepsilon\textnormal{Re}\left\lbrace\tfrac{\Phi''(p)(z+\xi)}{2\Phi'(p)}\right\rbrace\\
			&\quad+\tfrac{\varepsilon^2}{6}\textnormal{Re}\left\lbrace\tfrac{\Phi^{(3)}(p)}{\Phi'(p)}\big(z^{2}+\xi^{2}+z\xi\big)\right\rbrace\\
			&\quad-\tfrac{\varepsilon^2}{8}\textnormal{Re}\left\lbrace\left(\tfrac{\Phi''(p)(z+\xi)}{\Phi'(p)}\right)^2\right\rbrace+O(\varepsilon^3).
		\end{aligned}
	\end{equation}
	We observe that $O(\varepsilon^3)=\varepsilon^3K_3(p,z,\xi)$ with $K_3$ being   a smooth function in $(p,z,\xi)\in V\times\mathbb{D}_2^2$. Inserting \eqref{log g1} and \eqref{log g2} into \eqref{Green K2} and using \eqref{Conf-R}, \eqref{Robin1N}  and \eqref{Swart} give the desired result.
\end{proof}
The next target is to provide an  asymptotic expansion in $\varepsilon$ of the linearized operator $\mathbf{G}$, defined in \eqref{functional scrG},  at  a small state $r$. In what follows, and throughout the remainder of this paper, we denote
$$
\mathscr{O}\triangleq [\lambda_*,\lambda^*],
$$
where $ \lambda_*$ and $\lambda^*$ are as defined in Theorem \ref{main thm}.
	\begin{proposition}\label{proposition linear op exp}
	There exists $\varepsilon_0\in(0, 1)$ such that if $r$ is smooth with zero average in space and
	$$
	\varepsilon\leqslant\varepsilon_0\quad\textnormal{and}\quad\|r\|_{s_0+2}^{\textnormal{Lip}(\gamma)}\leqslant 1,
	$$ 
		then, the linearized operator of the map  $\mathbf{G}$,  at a small state $r$ in the direction $h\in L^2_0(\mathbb{T}^2;\mathbb{R})$ is given by
		\begin{align*}
			d_{r}\mathbf{G}(r)[h]&=\varepsilon^3\omega(\lambda)\partial_\varphi  h+\varepsilon\partial_{\theta}\big[\mathbf{V}^{\varepsilon}(r)h \big]-\tfrac{\varepsilon}{2}\mathbf{H}[h]+\varepsilon^3\partial_{\theta}\mathbf{Q}_1[h]+\varepsilon^3\partial_{\theta}\mathcal{R}_1^{\varepsilon}[h]+\varepsilon^4\partial_{\theta}\mathcal{R}_2^{\varepsilon}[h],
		\end{align*}	
		with the following properties.
		\begin{enumerate}
			\item The function $\mathbf{V}^{\varepsilon}(r)$ decomposes as follows
			\begin{equation*}
				\mathbf{V}^{\varepsilon}(r)\triangleq\tfrac{1}{2}-\tfrac{\varepsilon}{2}r+\varepsilon^2\big(\tfrac{1}{2}\mathtt{g}+V_1^{\varepsilon}(r)\big)+\varepsilon^3V_2^\varepsilon (r),
			\end{equation*}
		with
		\begin{align}\label{ttw2}
			\mathtt{g}(\varphi,\theta)&\triangleq\textnormal{Re}\left\lbrace\mathtt{w}_2\big(p(\varphi)\big)e^{2\ii\theta}\right\rbrace,\qquad
			\mathtt{w}_2(p)\triangleq \big(\partial_z\mathcal{R}_{\mathbf{D}}(p)\big)^2+\tfrac13 S(\Phi)(p)
		\end{align}
	and $V_1^{\varepsilon}(r),$ $V_2^{\varepsilon}(r)$ satisfy the estimates
	\begin{align*}
	\|V_1^{\varepsilon}(r)\|_{s}^{\textnormal{Lip}(\gamma)}&\lesssim\|r\|_{s+1}^{\textnormal{Lip}(\gamma)}\|r\|_{s_0+1}^{\textnormal{Lip}(\gamma)},\\		\|V_2^{\varepsilon}(r)\|_{s}^{\textnormal{Lip}(\gamma)}&\lesssim1+\|r\|_{s}^{\textnormal{Lip}(\gamma)},\\
		\nonumber \|\Delta_{12}V_1^{\varepsilon}(r)\|_{s}^{\textnormal{Lip}(\gamma)}&\lesssim\|\Delta_{12}r\|_{s+1}^{\textnormal{Lip}(\gamma)}+\|\Delta_{12}r\|_{s_0+1}^{\textnormal{Lip}(\gamma)}\max_{k\in\{1,2\}}\|r_k^{\varepsilon}(r)\|_{s+1}^{\textnormal{Lip}(\gamma)},\\
		\|\Delta_{12}V_2^{\varepsilon}(r)\|_{s}^{\textnormal{Lip}(\gamma)}&\lesssim\|\Delta_{12}r\|_{s}^{\textnormal{Lip}(\gamma)}+\|\Delta_{12}r\|_{s_0}^{\textnormal{Lip}(\gamma)}\max_{k\in\{1,2\}}\|r_k^{\varepsilon}(r)\|_{s}^{\textnormal{Lip}(\gamma)}.
	\end{align*}
\item The operator $\mathbf{H}$ is the Hilbert transform on the torus and the other non-local terms are integral operators in the form
\begin{align}
	\mathbf{Q}_1[h](\varphi,\theta)&\triangleq\int_{\mathbb{T}}h(\eta)\mathtt{Q}\big(p(\varphi),\theta,\eta\big)d\eta, \label{def bfQ1}\\
	\mathcal{R}_j^{\varepsilon}[h](\varphi,\theta)&\triangleq\int_{\mathbb{T}}h(\eta)\mathtt{K}_j^{\varepsilon}(r)(\varphi,\theta,\eta)d\eta,\quad j\in\{1,2\},\label{scrRk}
\end{align}
with, see \eqref{Conf-R} and \eqref{Swart} for the definitions of the functions below, 
\begin{align}
	\mathtt{Q}(p,\theta,\eta)&\triangleq\frac{\cos(\theta-\eta)}{r_{\mathbf{D}}^2(p)}\cos(\theta-\eta)+\frac16\textnormal{Re}\left\lbrace e^{\ii(\theta+\eta)} S(\Phi)(p)\right\rbrace\label{ttQ}
\end{align}
and $\mathtt{K}_1^{\varepsilon}(r)$, $\mathtt{K}_2^{\varepsilon}(r)$ satisfy the following estimates
\begin{align}
	\|\mathtt{K}_1^{\varepsilon}(r)\|_{s}^{\textnormal{Lip}(\gamma)}&\lesssim\|r\|_{s+1}^{\textnormal{Lip}(\gamma)}\|r\|_{s_0+1}^{\textnormal{Lip}(\gamma)},\label{e-scrK1}\\
	\|\mathtt{K}_2^{\varepsilon}(r)\|_{s}^{\textnormal{Lip}(\gamma)}&\lesssim1+\|r\|_{s}^{\textnormal{Lip}(\gamma)}.\label{e-scrK2}
\end{align}
		\end{enumerate}
		
	\end{proposition}
	\begin{proof}
	Throughout the proof, we may disregard the dependence on the variable $\varphi$, specifying it only when it is relevant.
		Differentiating \eqref{functional scrG} with respect to $r$ in the direction $h$ gives
		\begin{equation}\label{linearized st0}
			\begin{aligned}
				d_r\mathbf{G}(r)[h](\theta)&=\varepsilon^3\omega(\lambda)\,\partial_\varphi  h(\theta)-\tfrac{\varepsilon^2}{2}\partial_\theta\textnormal{Re}\Big\{\partial_{{z}} \mathcal{R}_{\bf{D}}(p)e^{\ii \theta}\Big\}\tfrac{h(\theta)}{R(\theta)}
				\\ &\quad+\partial_\theta\bigg[d_{r}\Big(\Psi_1\big(r,z(\theta)\big)\Big)[h](\theta)+d_{r}\Big(\Psi_2\big(r,z(\theta)\big)\Big)[h](\theta)\bigg].
			\end{aligned}
		\end{equation} 
	The linearized of $\Psi_1$ has been computed in \cite[Prop. 3.1]{HHM23} and is given by
	\begin{align*}
		\partial_\theta d_{r}\Big(\Psi_1\big(r,z(\theta)\big)\Big)[h](\theta)&=\varepsilon\partial_{\theta}\Big[\big(\tfrac{1}{2}-\tfrac{\varepsilon}{2}r(\theta)+\varepsilon^2 V_1^{\varepsilon}(r)(\theta)\big)h(\theta)\Big]-\tfrac{\varepsilon}{2}\mathbf{H}[h](\theta)+\varepsilon^3\partial_{\theta}\mathcal{R}_1^{\varepsilon}[h](\theta),
	\end{align*}
where $V_1^{\varepsilon}$ satisfies the desired estimates and $\mathcal{R}_1^{\varepsilon}$ is  an integral operator with the description \eqref{scrRk} and \eqref{e-scrK1}.
		Linearizing  \eqref{def-Psi2}, one readily finds
	\begin{align}\label{NonL-L}
		\nonumber d_{r}\Big(\Psi_{2}\big(r,z(\theta)\big)\Big)[h](\theta)&=\varepsilon\int_{\mathbb{T}}h(\eta)K\big(\varepsilon z(\theta)+p,\varepsilon z(\eta)+p\big)d\eta\\
	\nonumber	&\quad+\int_{\mathbb{T}}\int_{0}^{R(\eta)}d_{r}\Big(K\big(\varepsilon z(\theta)+p,\varepsilon le^{\ii\eta}+p\big)\Big)[h](\theta)ldld\eta\\
		&\triangleq J_1[h](\theta)+J_2[h](\theta).
	\end{align}
We shall start with expanding the nonlocal term $J_1[h]$ using the structure of the kernel \eqref{expansion K}, 
\begin{align*}
		\partial_\theta J_1[h](\theta)&=\varepsilon\partial_\theta\sum_{j=1}^3\varepsilon^j\int_{\mathbb{T}}h(\eta)K_j\big(p,z(\theta), z(\eta)\big)d\eta.	\end{align*}
		We have used that the test function $h$ is of zero spatial average.
From $K_1$, detailed in Lemma \ref{lemma expand K},  we obtain
\begin{align*}
		\partial_\theta \int_{\mathbb{T}}h(\eta)K_1\big(p,z(\theta), z(\eta)\big)d\eta&= \partial_\theta \textnormal{Re} \left\lbrace\partial_{z}\mathcal{R}_{\mathbf{D}}(p)\int_{\mathbb{T}}h(\eta)\big(z(\theta)+z(\eta)\big)d\eta
		\right\rbrace.
		\end{align*}
		Using one time more that  $h$ has a zero spatial average, we deduce that
		\begin{align*}
			\partial_\theta \int_{\mathbb{T}}h(\eta)K_1\big(p,z(\theta), z(\eta)\big)d\eta=0.
		\end{align*}
	Regarding the second term, in view of Lemma \ref{lemma expand K} and using once again that $h$ has a zero average, we obtain
		\begin{align*}
			\partial_\theta \int_{\mathbb{T}}h(\eta)K_2\big(p,z(\theta), z(\eta)\big)d\eta&= \partial_\theta \textnormal{Re} \left\lbrace\tfrac{1}{6}S(\Phi)(p)z(\theta)\int_{\mathbb{T}}h(\eta)z(\eta)d\eta
			\right\rbrace\\
			&\quad+\frac{1}{r_{\mathbf{D}}^{2}(p)}\partial_\theta \textnormal{Re} \left\lbrace z(\theta)\int_{\mathbb{T}}h(\eta)\overline{z(\eta)}d\eta\right\rbrace.
		\end{align*}
	Applying the parametrization \eqref{def zk} we infer	
		\begin{align*}
		\partial_\theta \int_{\mathbb{T}}h(\eta)K_2\big(p,z(\theta), z(\eta)\big)d\eta&= \partial_\theta \textnormal{Re} \left\lbrace\tfrac{1}{6}S(\Phi)(p)\int_{\mathbb{T}}h(\eta)e^{\ii(\theta+\eta)}d\eta
		\right\rbrace\\
		&\quad+\frac{1}{r_{\mathbf{D}}^{2}(p)}\partial_\theta \textnormal{Re} \left\lbrace\int_{\mathbb{T}}h(\eta)e^{\ii(\theta-\eta)}d\eta
		\right\rbrace+\varepsilon \partial_\theta \mathcal{R}_{1,2}^{\varepsilon}[h],
		\end{align*}	
	with $\mathcal{R}_{1,2}^{\varepsilon}$ a kernel operator, whose kernel satisfies the estimate \eqref{e-scrK2}. Gathering the preceding identities gives	
\begin{align*}
	\partial_\theta J_1[h]=\varepsilon^3\partial_\theta \mathbf{Q}_1[h]+\varepsilon^4\partial_{\theta}\mathcal{R}_2^{\varepsilon}[h],
\end{align*}
where $\mathbf{Q}_1$ is a nonlocal operator localizing on the modes $\pm1$ and defined by
\begin{equation}
	\mathbf{Q}_1[h](\theta)\triangleq\int_{\mathbb{T}}h(\eta)\mathtt{Q}(p,\theta,\eta)d\eta,
\end{equation}
where the kernel takes the form, 
\begin{align*}
	\mathtt{Q}(p,\theta,\eta)&\triangleq\frac{\cos(\theta-\eta)}{r_{\mathbf{D}}^2(p)}+\frac16\textnormal{Re}\left\lbrace e^{\ii(\theta+\eta)} S(\Phi)(p)\right\rbrace.
\end{align*}
The remainder operator $\mathcal{R}_2^{\varepsilon}$ has the form
\begin{align*}
	\mathcal{R}_2^{\varepsilon}[h]\triangleq\mathcal{R}_{1,2}^{\varepsilon}[h]+\int_{\mathbb{T}}h(\eta)K_3\big(p,z(\theta), z(\eta)\big)d\eta.
\end{align*}
From Lemma \ref{lemma expand K}, we infer that $K_3\in C^\infty_b(V\times\mathbb{D}_2^2)$. Therefore $\mathcal{R}_2^{\varepsilon}$ is an integral operator whose kernel obeys to the estimate of  \eqref{e-scrK2}.\\
Next, let us move to the estimate of the local term in \eqref{NonL-L}. Performing  Taylor formula with the integral contribution $\int_0^R..$ allows to get
	\begin{equation}\label{expansion Psi2}
		\begin{aligned}
			J_2[h](\theta)&=\int_{\mathbb{T}}\int_{0}^{1}d_{r}\Big(K\big(\varepsilon z(\theta)+p,\varepsilon  l e^{\ii\eta}+p\big)\Big)[h](\theta)l d l d\eta\\
			&\quad+\varepsilon\int_{\mathbb{T}}\int_0^1\tfrac{r(\eta)}{{\sqrt{1+2\varepsilon \tau r(\eta)}}}d_{r}\Big(K\big(\tau\varepsilon z(\theta)+p,\tau\varepsilon e^{\ii\eta}+p\big)\Big)[h](\theta)d\eta d\tau,
		\end{aligned}
	\end{equation}
where $\mathbf{L}_{\varepsilon}[r]$ is a non-local term with smooth kernel (depending on $K$) at least quadratic in $r.$
By using the decomposition \eqref{expansion K}, we obtain
\begin{align}\label{diff-chatl}
	\nonumber &d_{r}\Big(K\big(\varepsilon z(\theta)+p,\varepsilon  l e^{\ii\eta}+p\big)\Big)[h](\theta)=\varepsilon^2\tfrac{h(\theta)}{R(\theta)}\textnormal{Re}\left\lbrace\partial_{z}\mathcal{R}_{\mathbf{D}}(p)e^{\ii\theta}\right\rbrace\\
	\nonumber&\quad+\varepsilon^3h(\theta)\textnormal{Re}\left\lbrace\mathtt{w}_2(p)e^{2\ii\theta}\right\rbrace+\varepsilon^3\tfrac{h(\theta)l}{6R(\theta)}\textnormal{Re}\left\lbrace S(\Phi)(p)e^{\ii(\theta+\eta)}\right\rbrace+\varepsilon^{3}\tfrac{h(\theta)l}{r_{\mathbf{D}}^2(p)R(\theta)}\textnormal{Re}\left\lbrace e^{\ii(\theta-\eta)}\right\rbrace\\
	&\quad+\varepsilon^4\tfrac{h(\theta)}{R(\theta)}\textnormal{Re}\Big\{\partial_z K_3(p,z(\theta),l e^{\ii \eta})\, e^{\ii \theta}\Big\},
\end{align}
with $\mathtt{w}_2$ as in \eqref{ttw2}. Hence, the zero space average condition for $r$ implies
\begin{align*}
	\varepsilon\int_{\mathbb{T}}\int_0^1\tfrac{r(\eta)}{{\sqrt{1+2\varepsilon \tau r(\eta)}}}d_{r}\Big(K\big(\tau\varepsilon z(\theta)+p,\tau\varepsilon e^{\ii\eta}+p\big)\Big)[h](\theta)d\eta d\tau=\varepsilon^4V_{2,1}(r) h(\theta)
\end{align*}
and we can check in a straightforward manner that $V_{2,1}^\varepsilon$ satisfies the estimates of  $V_2^\varepsilon$ in Proposition~\ref{proposition linear op exp}. Coming back to \eqref{expansion Psi2} and using \eqref{diff-chatl} we find
\begin{align*}
	&\int_{\mathbb{T}}\int_{0}^{1}d_{r}\Big(K\big(\varepsilon z(\theta)+p(\varphi),\varepsilon  l e^{\ii\eta}+p(\varphi)\big)\Big)[h](\theta)l d l d\eta\\
	&{=}\tfrac{\varepsilon^2}{2}\tfrac{h(\theta)}{R(\theta)}\textnormal{Re}\left\lbrace\partial_z\mathcal{R}_{\mathbf{D}}(p) e^{\ii\theta}\right\rbrace+\tfrac{\varepsilon^3}{2}\textnormal{Re}\left\lbrace\mathtt{w}_2\big(p(\varphi)\big)e^{2\ii\theta}\right\rbrace h(\theta)+\varepsilon^4 V_{2,2}(r) h(\theta).
\end{align*}
		Notice that $V_{2,2}(r)$ satisfies the same estimates of $V_2^\varepsilon$ in Proposition \ref{proposition linear op exp}.
Putting the preceding identities in \eqref{linearized st0} yields to the asymptotic expansion of the linearized operator as described in Proposition~\ref{proposition linear op exp}.
\end{proof}

\subsection{Construction of an approximate solution}\label{sec-approx-sol}
In the following lemma, we provide the leading term in the asymptotic expansion of $\mathbf{G}(0)$ as described by \eqref{functional scrG}. This result will be pertinent later when developing an appropriate scheme to obtain a precise approximation, as discussed in Lemma \ref{approxim repsi}.
	\begin{lemma}\label{lem G0}
	There exists $\varepsilon_0\in(0, 1)$ such that for all $\varepsilon\in(0,\varepsilon_0)$, 
	one has
	\begin{align*}
		\forall (\varphi,\theta)\in\mathbb{T}^2,\quad \mathbf{G}(0)(\varphi,\theta)&=-\tfrac{\varepsilon^2}{2}\textnormal{Im}\left\lbrace\mathtt{w}_2\big(p(\varphi)\big)e^{2\ii\theta}\right\rbrace+O(\varepsilon^3),
	\end{align*}
	where $\mathtt{w}_2$ is defined in \eqref{ttw2}.
	In addition, the function $\mathbf{G}(0)$ does not contain the spatial modes $\pm1$, that is,
	$$
	\int_{\mathbb{T}}\mathbf{G}(0)(\varphi,\theta)e^{\pm \ii \theta}d\theta=0.
	$$
\end{lemma}
\begin{proof} To simplify the notation, we shall remove the dependence in $\varphi$. Substituting $r=0$ into  \eqref{functional scrG} gives
	\begin{align*}
		\mathbf{G}(0)(\theta)&=\partial_{\theta}\Big[-\tfrac{\varepsilon}{2} \textnormal{Re}\Big\{\partial_{{z}} \mathcal{R}_{\bf{D}}(p)e^{\ii \theta}\Big\}+\Psi_1\big(0,e^{\ii\theta}\big)+\Psi_2\big(0,e^{\ii\theta}\big)\Big].
	\end{align*}
	In view of \eqref{def-Psi2}, we have
	\begin{equation*}
		\partial_\theta\big[\Psi_1\big(0,  e^{\ii\theta}\big)\big]=\partial_\theta\bigg[\int_{\mathbb{T}}\int_{0}^{1}\log\big(\big| 1 - l e^{\ii(\eta-\theta)}\big|\big)l d l d\eta\bigg]=0.
	\end{equation*}
	As for the $\Psi_2$ term, one gets by virtue of \eqref{def-Psi2}
	\begin{equation*}
		\Psi_2\big(0,e^{\ii\theta}\big)=\int_{\mathbb{T}}\int_{0}^{1}K\big(p+\varepsilon e^{\ii\theta},p+\varepsilon le^{\ii\eta}\big)ldld\eta.
	\end{equation*}
	We apply Lemma \ref{lemma expand K} with $z=e^{\ii\theta}$ and $\xi= le^{\ii\eta}$ leading in view of  \eqref{ttw2} to
	\begin{align*}
		\partial_{\theta}\Big(\Psi_{2}\big(0,e^{\ii\theta}\big)\Big)&{=}\tfrac{\varepsilon}{2}\partial_\theta\textnormal{Re}\left\lbrace \partial_{{z}} \mathcal{R}_{\bf{D}}(p) e^{\ii\theta}\right\rbrace+\tfrac{\varepsilon^2}{2}\textnormal{Re}\left\lbrace\ii e^{2\ii\theta}\mathtt{w}_2(p)\right\rbrace+O(\varepsilon^3)\\
		&{=}\tfrac{\varepsilon}{2}\partial_\theta\textnormal{Re}\left\lbrace \partial_{{z}} \mathcal{R}_{\bf{D}}(p) e^{\ii\theta}\right\rbrace-\tfrac{\varepsilon^2}{2}\textnormal{Im}\left\lbrace e^{2\ii\theta}\mathtt{w}_2(p)\right\rbrace+O(\varepsilon^3).
	\end{align*}
 Combining the foregoing identities  gives the suitable  expansion for $\mathbf{G}(0)$. The final goal   is to prove the absence of modes $\pm1$ in the whole quantity $\mathbf{G}(0).$ This pertains to show that the Fourier expansion of $\Psi_2\big(0,e^{\ii\theta}\big)-\tfrac{\varepsilon}{2}\textnormal{Re}\left\lbrace \partial_{{z}} \mathcal{R}_{\bf{D}}(p) e^{\ii\theta}\right\rbrace$ does not contain the modes $\pm1.$ 
		Applying \eqref{Green K} we infer 
		\begin{align*}
			K(p+\varepsilon z,p+\varepsilon\xi)=\textnormal{Re}\left\lbrace H(p+\varepsilon z,{p+\varepsilon\xi}) \right\rbrace,
		\end{align*}
		where
		\begin{align*}
			H(z_1,z_2)\triangleq\log(b)+\textnormal{Log}\left(1+\frac{1}{b}\frac{f(z_1)-f(z_2)}{z_1-z_2}\right)-\textnormal{Log}\big(1-\Phi(z_2)\overline{\Phi(z_1)}\big)
		\end{align*}
and  $\textnormal{Log}$ is the principal value of the complex logarithm function.
		Applying this with $z=e^{\ii\theta}$ and $\xi=le^{\ii\eta}$, we obtain
		\begin{align*}
			\Psi_2\big(0,e^{\ii\theta}\big)=\textnormal{Re}\left\lbrace\int_{\mathbb{T}}\int_{0}^{1}H\big(p+\varepsilon e^{\ii\theta},p+\varepsilon le^{\ii\eta}\big)ldld\eta\right\rbrace.
		\end{align*}
		Noting that the function $\xi\in \mathbb{D}_2\mapsto H(p+\varepsilon e^{\ii\theta},p+\varepsilon \xi)$ is holomorphic  for small $\varepsilon$ allows to get the power series expansion
		\begin{align*}
			H\big(p+\varepsilon e^{\ii\theta},p+\varepsilon \xi\big)= \sum_{n=0}^{\infty}c_n(\varepsilon,p,\theta)\xi^n.
		\end{align*}
		Therefore
		\begin{align*}\Psi_2\big(0,e^{\ii\theta}\big)&=\textnormal{Re}\left\lbrace\int_{\mathbb{T}}\int_{0}^{1} \sum_{n=0}^{\infty}c_n(\varepsilon,p,\theta)e^{\ii n\eta}l^{n+1}dld\eta\right\rbrace\\
		&=\textnormal{Re}\left\lbrace\tfrac{1}{2}c_0(\varepsilon,p,\theta)\right\rbrace.
		\end{align*}
		It follows that
		\begin{align*}\Psi_2\big(0,e^{\ii\theta}\big)		&=\textnormal{Re}\left\lbrace\tfrac{1}{2} H\big(p+\varepsilon e^{\ii\theta},p\big)\right\rbrace\\
		&=\textnormal{Re}\left\lbrace\tfrac{1}{2} H_1\big(p+\varepsilon e^{\ii\theta},p\big)\right\rbrace,
		\end{align*}
with
\begin{align*}
			H_1(z_1,z_2)\triangleq\log(b)+\textnormal{Log}\left(1+\frac{1}{b}\frac{f(z_1)-f(z_2)}{z_1-z_2}\right)-\textnormal{Log}\big(1-\Phi(z_1)\overline{\Phi(z_2)}\big).
		\end{align*}
		As the function $\xi\in \mathbb{D}_2\mapsto H_1(p+\varepsilon \xi,p)$ is holomorphic  for small $\varepsilon,$ then
		\begin{align*}
			H_1(p+\varepsilon e^{\ii\theta},p)= \sum_{n=0}^{\infty}\tfrac{\partial_{z_1}^nH_1(p,p)}{n!}\varepsilon^ne^{\ii n\theta},
		\end{align*}
		which implies that
		\begin{align*}\Psi_2\big(0,e^{\ii\theta}\big)		&=\tfrac{1}{2} H_1(p,p)+\tfrac\varepsilon2\textnormal{Re}\left\lbrace{\partial_{z_1}H_1(p,p)}e^{\ii\theta}\right\rbrace+\tfrac{1}{2}\textnormal{Re}\sum_{n=2}^{\infty}\tfrac{\partial_{z_1}^nH_1(p,p)}{n!}\varepsilon^ne^{\ii n\theta}.
		\end{align*}
		One can check that
		\begin{align*}
			H_1(p,p)=\mathcal{R}_{\mathbf{D}}(p)\quad\hbox{and}\quad \partial_{z_1}H_1(p,p)=\partial_z\mathcal{R}_{\mathbf{D}}(p).
		\end{align*}
		Consequently,
			\begin{align*}\Psi_2\big(0,e^{\ii\theta}\big)-\tfrac{\varepsilon}{2}\textnormal{Re}\left\lbrace \partial_{{z}} \mathcal{R}_{\bf{D}}(p) e^{\ii\theta}\right\rbrace
			=\tfrac{1}{2}\mathcal{R}_{\mathbf{D}}(p)+\tfrac{1}{2}\textnormal{Re}\sum_{n\geqslant 2}^{\infty}\tfrac{\partial_{z_1}^nH_1(p,p)}{n!}\varepsilon^ne^{\ii n\theta}.
				\end{align*}
				This shows that the Fourier expansion of $\Psi_2\big(0,e^{\ii\theta}\big)-\tfrac{\varepsilon}{2}\textnormal{Re}\left\lbrace \partial_{{z}} \mathcal{R}_{\bf{D}}(p) e^{\ii\theta}\right\rbrace
			$  does not contain the modes $\pm1.$ This ends the proof of the lemma.
			
\end{proof}
We have already seen in  Lemma \ref{lem G0} that the function $\mathbf{G}(0)$ does not contain the modes $\pm1.$ Furthermore,   Proposition \ref{proposition linear op exp} ensures that $\mathtt{g}$ is localized on the mode $2$. Leveraging these insights and adopting a similar approach as developed in \cite[Lemma 3.3, Proposition 3.4]{HHM23} we derive a suitable approximate solution $r_\varepsilon$. Upon rescaling the functional, the behavior of the linearized operator is quite similar to  the original one. Our result reads as follows.
	\begin{lemma}\label{approxim repsi 0}
		There exists $r_{\varepsilon}\in C^{\infty}(\mathbb{T}^2)$, satisfying  
		\begin{align*}
			\|r_{\varepsilon}+\mathtt{g}\|_{s}^{\textnormal{Lip}(\gamma)}\lesssim\varepsilon,
		\end{align*}
		such that the nonlinear functional 
		\begin{equation}\label{def func F}
		\mathbf{F}(\rho)\triangleq \tfrac{1}{\varepsilon^{2+\mu}} \mathbf{G}( \varepsilon r_\varepsilon+\varepsilon^{1+\mu}\rho),\qquad \mu\in(0,1),
		\end{equation}
		satisfies the following estimates
		$$
		\|\mathbf{F}(0)\|_{s}^{\textnormal{Lip}(\gamma)}\lesssim\varepsilon^{3-\mu}.
		$$
\end{lemma}
	\begin{proof}
		The starting point is to write a Taylor expansion of the functional $\mathbf{G}$ around $0$,
		$$
		\mathbf{G}(r)=\mathbf{G}(0)+\partial_{r}\mathbf{G}(0)[r]+\tfrac{1}{2}\partial_{r}^2\mathbf{G}(0)[r,r]+\tfrac{1}{2}\int_{0}^{1}(1-\tau)^2\partial_{r}^3\mathbf{G}(\tau r)[r,r,r]d\tau.$$
		Then, we look for an approximate solution in the form $r=\varepsilon r_0+\varepsilon^2 r_1.$ Hence,
		\begin{equation}\label{taylor g0}
		\begin{aligned}
			&\mathbf{G}(\varepsilon r_0+\varepsilon^2r_1)=\mathbf{G}(0)+\varepsilon\partial_{r}\mathbf{G}(0)[r_0]+\tfrac{\varepsilon^2}{2}\partial_{r}^2\mathbf{G}(0)[r_0,r_0]+\varepsilon^2\partial_{r}\mathbf{G}(0)[r_1]\\
			&\quad+\varepsilon^3\partial_{r}^2\mathbf{G}(0)[r_0,r_1]+\tfrac{\varepsilon^4}{2}\partial_{r}^2\mathbf{G}(0)[r_1,r_1]+\tfrac{1}{2}\int_{0}^{1}(1-\tau)^2\partial_{r}^{3}\mathbf{G}(\tau r)[r,r,r]d\tau.
		\end{aligned}	
		\end{equation}
				In view of Proposition \ref{proposition linear op exp}  one has
		\begin{equation}\label{taylor r0}
			\begin{aligned}
			&\mathbf{G}(0)+\varepsilon\partial_{r}\mathbf{G}(0)[r_0]+\tfrac{\varepsilon^2}{2}\partial_{r}^2\mathbf{G}(0)[r_0,r_0]=\mathbf{G}(0)+\varepsilon^4\omega(\lambda) \partial_\varphi  r_0\\ & +\varepsilon^2\partial_{\theta}\Big[\big(\tfrac{1}{2}+\tfrac{\varepsilon^2}{2}\mathtt{g}+\varepsilon^3{V}_2^{\varepsilon}(0)\big)r_0 \Big]-\tfrac{\varepsilon^2}{2}\mathbf{H}[h]+\varepsilon^4\partial_{\theta}\mathbf{Q}_1[r_0]-\tfrac{\varepsilon^4}{4}\partial_\theta\big(r_0^2\big)+\varepsilon^5\partial_{\theta}\mathcal{E}^{\varepsilon}(r_0),
		\end{aligned}
	\end{equation}
				with
		$$
		\|\partial_{\theta}\mathcal{E}^{\varepsilon}(r_0)\|_s^{\textnormal{Lip}(\gamma)}\lesssim \|r_0\|_{s+2}^{\textnormal{Lip}(\gamma)}.		$$
		At the main order we shall solve the equation
		$$\tfrac{\varepsilon^2}{2}\big[\partial_{\theta}-\mathbf{H}\big]r_0+\mathbf{G}(0)=0.$$
		According to Lemma \ref{lem G0}, the source term is of size $O(\varepsilon^2)$ and does not contain the modes $0$ and $\pm1$. Therefore, we can solve the previous equation
		\begin{align}\label{moul1}
			\nonumber r_0(\varphi,\theta)&=-2\varepsilon^{-2}\big[\partial_{\theta}-\mathbf{H}\big]^{-1}\mathbf{G}(0)\\ 
			\nonumber &={-}\textnormal{Re}\left\lbrace\mathtt{w}_2\big(p(\varphi)\big)e^{2\ii\theta}\right\rbrace +\varepsilon \underline{r}_0(\varphi,\theta)\\
			&={-}\mathtt{g}(\varphi,\theta)+\varepsilon \underline{r}_0(\varphi,\theta),
		\end{align}
		where $\mathtt{w}_2$ and $\mathtt{g}$ are defined in \eqref{ttw2} and  $\underline{r}_0\in C^{\infty}(\mathbb{T}^2)$. Moreover, $\underline{r}_0$ does not contain the modes $0$ and $\pm1$ which implies according to  \eqref{def bfQ1} and \eqref{ttQ} that
		$$
		\partial_{\theta}\mathbf{Q}_1[r_0]=0.
		$$
Gathering the identities above, we conclude that
\begin{equation}\label{taylor r01}
			\mathbf{G}(0)+\varepsilon\partial_{r}\mathbf{G}(0)[r_0]+\tfrac{\varepsilon^2}{2}\partial_{r}^2\mathbf{G}(0)[r_0,r_0]=\varepsilon^4\mathbf{B}_1+\varepsilon^5\mathbf{B}_2,
		\end{equation}
		where $\mathbf{B}_1$ is given by
		$$
		\mathbf{B}_1\triangleq {-}\omega(\lambda) \partial_\varphi  \mathtt{g}-\tfrac{3}{4}\partial_\theta\big(\mathtt{g}^2\big)
		$$
		and  $\mathbf{B}_2$ satisfies
		$$
		\|\mathbf{B}_2\|_s^{\textnormal{Lip}(\gamma)}\lesssim 1.
		$$
Thus, $\mathbf{B}_1$ contains only even modes and the equation
		 \begin{equation}\label{eq r1}
		 \tfrac{1}{2}\big[\partial_{\theta}-\mathbf{H}\big]r_1+\varepsilon \mathbf{B}_1=0
		 \end{equation}
admits a solution that satisfies the estimates
$$
		\|r_1\|_s^{\textnormal{Lip}(\gamma)}\lesssim \varepsilon.
		$$
Since $r_1$ does not contain the modes $\pm1$ then 
$$
		\partial_{\theta}\mathbf{Q}_1[r_1]=0.
		$$
		Therefore,  by \eqref{moul1}
		\begin{align}\label{moul2}
		\nonumber r_\varepsilon&\triangleq r_0+\varepsilon r_1\\
		&=-\mathtt{g}+\varepsilon \big(\underline{r}_0+r_1\big)
		\end{align}	
		and inserting \eqref{taylor r0}, \eqref{eq r1} into \eqref{taylor g0} we obtain
		$$
		\|\mathbf{G}(\varepsilon r_\varepsilon)\|_s^{\textnormal{Lip}(\gamma)}\lesssim \varepsilon^5.
		$$
	This concludes the proof of Lemma \ref{approxim repsi 0}.	
	\end{proof}
	The next objective is to describe the asymptotic structure of the linearized operator associated with the rescaled function $F$ introduced in \eqref{def func F}.
	\begin{lemma}\label{approxim repsi}
	There exists $\varepsilon_0\in(0,1)$ such that for all $\varepsilon\in(0,\varepsilon_0)$ and  any smooth function $\rho$   with  zero space average  and satisfying
		$$\|\rho\|_{s_0+2}^{\textnormal{Lip}(\gamma)}\leqslant 1,$$ 
the following holds true.
		  The linearized operator $\mathcal{L}_0\triangleq d_\rho \mathbf{F}(\rho)$, see \eqref{def func F},  has the form 
		\begin{equation}\label{def L0}
		\begin{aligned}
\mathcal{L}_0 h&=\varepsilon^2\omega(\lambda)\partial_\varphi  h +\partial_{\theta}\big[\mathbf{V}_1^\varepsilon(\rho)h\big]-\tfrac{1}{2}\mathbf{H}[h]+{\varepsilon^2}\partial_\theta\mathbf{Q}_1[h]+\varepsilon^3 \partial_\theta \mathcal{R}_0^\varepsilon(\rho)[h],
\end{aligned}
\end{equation}
where the function $\mathbf{V}_1^\varepsilon(\rho)$ is given by
\begin{equation}\label{def scr V1}
\begin{aligned}
\mathbf{V}_1^\varepsilon(\rho)&\triangleq \tfrac{1}{2}+{\varepsilon^2}\mathtt{g}- \tfrac{\varepsilon^{2+\mu}}{2} \rho+\varepsilon^{3}{V}^\varepsilon(\rho).
\end{aligned}
\end{equation}
The function $\mathtt{g}$ is defined in Proposition $\ref{proposition linear op exp}$  and ${V}^\varepsilon(\rho)$ enjoys  the estimates: for any $s\geqslant s_0$
\begin{equation}\label{estim V}
\begin{aligned}
\|{V}^\varepsilon(\rho)\|_{s}^{\textnormal{Lip}(\gamma)}&\lesssim 1+ \varepsilon^\mu\|\rho\|_{s+1}^{\textnormal{Lip}(\gamma)},
\\				
\|\Delta_{12}{V}^\varepsilon(\rho)\|_{s}^{\textnormal{Lip}(\gamma)}&\lesssim \varepsilon^\mu \|\Delta_{12}\rho\|_{s+1}^{\textnormal{Lip}(\gamma)}+ \varepsilon^\mu\|\Delta_{12}\rho\|_{s_0+1}^{\textnormal{Lip}(\gamma)}\max_{\ell\in\{1,2\}}\|\rho_{\ell}\|_{s+1}^{\textnormal{Lip}(\gamma)}.
\end{aligned}
\end{equation}
The operator $\mathbf{Q}_1$ is given in Proposition $\ref{proposition linear op exp}$ 
the operator $\mathcal{R}_0^\varepsilon(\rho)$ is an integral operator of the form 
\begin{align*}
\mathcal{R}_0^\varepsilon(\rho)[h](\varphi,\theta)	&\triangleq \int_{\mathbb{T}} h (\varphi,\eta)\mathcal{K}_{0}^\varepsilon(\rho)(\varphi,\theta,\eta)	d\eta,
\end{align*}
where the kernel $\mathcal{K}_0^\varepsilon(\rho)$ satisfies: for any $s\geqslant s_0$
\begin{align}\label{est K0 scrR}
\|\mathcal{K}_0^\varepsilon(\rho)\|_{s}^{\textnormal{Lip}(\gamma)}&\lesssim 1+  \varepsilon^\mu\|\rho\|_{s+1}^{\textnormal{Lip}(\gamma)}.
\end{align}
 	Moreover,  for any $s\geqslant s_0$, one has
 \begin{equation}\label{est F}
\begin{aligned}
		\big\|d_{\rho}\mathbf{F}(\rho)[h]\big\|_{s}^{\textnormal{Lip}(\gamma)}&\lesssim \|h\|_{s+2}^{\textnormal{Lip}(\gamma)}+\|\rho\|_{s+2}^{\textnormal{Lip}(\gamma)}\|h\|_{s_0+2}^{\textnormal{Lip}(\gamma)},\\
		 \big\|d_{\rho}^{2}\mathbf{F}(\rho)[h,h]\big\|_{s}^{\textnormal{Lip}(\gamma)}&\lesssim \varepsilon^2 \|h\|_{s+2}^{\textnormal{Lip}(\gamma)}\|h\|_{s_0+2}^{\textnormal{Lip}(\gamma)}+\varepsilon^2\|\rho\|_{s+2}^{\textnormal{Lip}(\gamma)}\big(\|h\|_{s_0+2}^{\textnormal{Lip}(\gamma)}\big)^2.
	\end{aligned}		
\end{equation}
\end{lemma}
\begin{proof}
Linearizing the functional $\mathbf{F}$ defined by \eqref{def func F}  an using Proposition \ref{proposition linear op exp} we obtain
\begin{align*}
\partial_\rho \mathbf{F}(\rho)[h]&=\tfrac{1}{\varepsilon} (\partial_r\mathbf{G})(\varepsilon r_\varepsilon+\varepsilon^{1+\mu}\rho)[h]\\ \nonumber&=\varepsilon^2\omega_0\partial_\varphi  h +\partial_{\theta}\Big[\Big(\tfrac12- \tfrac{\varepsilon^2}{2} r_\varepsilon- \tfrac{\varepsilon^{2+\mu}}{2} \rho(\varphi,\theta){+}\tfrac{\varepsilon^2}{2}\mathtt{g}\\ 
\nonumber&\quad+\varepsilon^2{V}_1^\varepsilon(\varepsilon r_\varepsilon+\varepsilon^{1+\mu}\rho)+\varepsilon^3{V}_2^\varepsilon(\varepsilon r_\varepsilon+\varepsilon^{1+\mu}\rho)\Big) h \Big]-\tfrac{1}{2}\mathbf{H}[h]\nonumber\\ 
\nonumber&\quad+\varepsilon^2\mathbf{Q}_1[h]+\varepsilon^2 \partial_\theta\mathcal{R}_1^\varepsilon(\varepsilon r_\varepsilon+\varepsilon^{1+\mu}\rho)[h]+\varepsilon^3 \partial_\theta\mathcal{R}_2^\varepsilon(\varepsilon r_\varepsilon+\varepsilon^{1+\mu}\rho)[h].
\end{align*}
By \eqref{moul2} we infer
\begin{equation}\label{est r2}
{-} r_\varepsilon(\varphi,\theta)+\mathtt{g}(\varphi,\theta)\triangleq 2\mathtt{g}(\varphi,\theta)+2\varepsilon r_2(\varphi,\theta), \qquad \textnormal{with}\qquad \|r_2\|_{s}^{\textnormal{Lip}(\gamma)}\lesssim 1. \end{equation}
Setting 
\begin{align*}
{V}^\varepsilon(\rho)&\triangleq  r_2+\tfrac1\varepsilon {V}_1^\varepsilon(\varepsilon r_\varepsilon+\varepsilon^{1+\mu}\rho)+{V}_2^\varepsilon(\varepsilon r_\varepsilon+\varepsilon^{1+\mu}\rho),\\
\mathcal{R}_0^\varepsilon(\rho)[h]&\triangleq \tfrac1\varepsilon\mathcal{R}_1^\varepsilon(\varepsilon r_\varepsilon+\varepsilon^{1+\mu}\rho)[h]+\mathcal{R}_2^\varepsilon(\varepsilon r_\varepsilon+\varepsilon^{1+\mu}\rho)[h],
\end{align*}
gives \eqref{def L0} and \eqref{def scr V1}. The estimates \eqref{estim V}, \eqref{est K0 scrR}, \eqref{est F}  follow easily from Proposition \ref{proposition linear op exp}. 
 This ends the proof of Lemma~\ref{approxim repsi}.
\end{proof}

\subsection{Straightening the transport part}\label{sec reduction transp}
The purpose of this section is to conjugate the linearized operator $\mathcal{L}_0$ in \eqref{def L0} through appropriate symplectic changes of variables so that in the new coordinates, the transport part has  constant coefficients.  Following the approach outlined in \cite[Sec. 4]{HHM23}, we begin by performing   a suitable  change of variable acting only on space to transform the vector field in the advection part into a constant one, up to an error term of small size \mbox{in $\varepsilon$.} Notably, the time degeneracy is advantageous in this step. With this new structure, one may apply the standard KAM reduction method in the spirit of \cite{BM21,HHM21}. 
\begin{proposition}\label{prop CVAR1}
Let $(\gamma,s_0,S)$ as in \eqref{cond1}. There exists ${\varepsilon}_0>0$ 
such that if 
		\begin{equation}\label{smallness CVAR1}
			\varepsilon\leqslant\varepsilon_0,\qquad\|\rho\|_{s_0+2}^{\textnormal{Lip}(\gamma)}\leqslant 1,
		\end{equation}
 there exist $\beta_1$ taking the form
\begin{align*}
	\beta_{1}(\lambda,\varphi,\theta)=\varepsilon b_{1}(\lambda,\varphi)+\varepsilon^{2}b_2(\lambda,\varphi,\theta)+\varepsilon^{2+\mu}\mathfrak{r}(\lambda,\varphi,\theta),
\end{align*}
with 
\begin{equation}\label{def b2}
	b_{2}(\lambda,\varphi,\theta)\triangleq\textnormal{Im}\Big\{\mathtt{w}_2\big(p_{\lambda}(\varphi)\big)e^{\ii2\theta}\Big\},
\end{equation}
where $\mathtt{w}_2(p)$ is defined in \eqref{ttw2} and an invertible symplectic change of variables
\begin{equation}\label{CVAR1}
\mathscr{B}_1 h(\lambda,\varphi,\theta)\triangleq\big(1+\partial_{\theta}\beta_1(\lambda,\varphi,\theta)\big)h\big(\lambda,\varphi,\theta+\beta_1(\lambda,\varphi,\theta)\big)
\end{equation}
such that
\begin{align}\label{conjugaison scrB1}
\mathscr{B}_1^{-1}\Big(\varepsilon^2\omega(\lambda)\partial_\varphi+\partial_\theta\big[\mathbf{V}_1^\varepsilon(\rho)\cdot\big]\Big)\mathscr{B}_1=\varepsilon^2\omega(\lambda)\partial_\varphi+\partial_{\theta}\big[\big(\mathtt{c}_0+\varepsilon^{6}{\mathbf{V}}_2^\varepsilon({\rho})\big)\cdot\big].
\end{align}
Furthermore, we have
\begin{enumerate}
\item The space-time independent function $\lambda\mapsto\mathtt{c}_0(\lambda)$ writes
\begin{equation}\label{ttc0}
\mathtt{c}_0\triangleq\tfrac{1}{2}+\varepsilon^{3}\mathtt{c}_1,\qquad\|\mathtt{c}_1\|^{\textnormal{Lip}(\gamma)}\lesssim 1.\end{equation}
\item The functions $\beta_1$, $\mathfrak{r}$, $b_1$ and  $\mathbf{V}_2^\varepsilon({\rho})$ satisfy the following estimates for all $s\in[s_{0},S],$ 
\begin{equation}\label{est beta1 b1 frakr bfV2}
\begin{aligned}
\|\beta_1\|_{s}^{\textnormal{Lip}(\gamma)} &\lesssim \varepsilon\left(1+\|\rho\|_{s+{2}}^{\textnormal{Lip}(\gamma)}\right),\\
\|b_1\|_{s}^{\textnormal{Lip}(\gamma)}+\|\mathfrak{r}\|_{s}^{\textnormal{Lip}(\gamma)} &\lesssim 1+\|\rho\|_{s+{2}}^{\textnormal{Lip}(\gamma)},\\
\|\mathbf{V}_2^\varepsilon({\rho})\|_{s}^{\textnormal{Lip}(\gamma)}&\lesssim 1+\|\rho\|_{s+{3}}^{\textnormal{Lip}(\gamma)}.
\end{aligned}
\end{equation}
\item The changes of variables $\mathscr{B}_1^{\pm 1}$  satisfy the following estimates for all $s\in[s_{0},S],$
\begin{equation}\label{est scrB1}
\|\mathscr{B}_1^{\pm1}h\|_{s}^{\textnormal{Lip}(\gamma)}\lesssim\|h\|_{s}^{\textnormal{Lip}(\gamma)}+\varepsilon\|\rho\|_{s+3}^{\textnormal{Lip}(\gamma)}\|h\|_{s_{0}}^{\textnormal{Lip}(\gamma)}.
\end{equation} 
	\item Given two states $\rho_{1}$ and $\rho_{2}$ with the smallness property \eqref{smallness CVAR1}, then 
		\begin{align}
			\|\Delta_{12}\mathtt{c}_0\|^{\textnormal{Lip}(\gamma)}&\lesssim\varepsilon^{3}\|\Delta_{12}\rho\|_{s_0+2}^{\textnormal{Lip}(\gamma)},\label{diff ttc0}\\ 
			\|\Delta_{12}\mathbf{V}_2^{\varepsilon}\|_{s}^{\textnormal{Lip}(\gamma)}&\lesssim\|\Delta_{12}\rho\|_{s+3}^{\textnormal{Lip}(\gamma)}+\|\Delta_{12}\rho\|_{s_0+3}^{\textnormal{Lip}(\gamma)}\max_{k\in\{1,2\}}\|\rho_k\|_{s+3}^{\textnormal{Lip}(\gamma)}.\label{diff bfV2}
		\end{align}
\end{enumerate}

\end{proposition}
\begin{proof} 
Let us consider a symplectic diffeomorphism of the torus in the form
\begin{equation}\label{change of variables}
\mathscr{B}_1h(\varphi,\theta)\triangleq\big(1+\partial_{\theta}\beta_1(\varphi,\theta)\big)\mathrm{B}_1h(\varphi,\theta),\qquad\mathrm{B}_1h(\varphi,\theta)\triangleq h\big(\varphi,\theta+\beta_{1}(\varphi,\theta)\big),
\end{equation}
with $\beta_1$ to be chosen later in order to reduce the size of the coefficients. Due to the symplectic nature of $\mathscr{B}_1,$ we have
\begin{align*}
	\langle h\rangle_{\theta}=0\quad\Rightarrow\quad\langle\mathscr{B}_1h\rangle_{\theta}=0.
\end{align*}
Applying Lemma \ref{lem tranfo CVAR}-2, we obtain
\begin{align*}
	\mathscr{B}_1^{-1}\Big(\varepsilon^2\omega(\lambda)\partial_\varphi+\partial_\theta\big[\mathbf{V}_{1}^\varepsilon(\rho)\cdot\big]\Big)\mathscr{B}_1=\varepsilon^2\omega(\lambda)\partial_\varphi+\partial_y\Big[\mathrm{B}_1^{-1}\big(\widehat{\mathbf{V}}_1^\varepsilon(\rho)\big)\cdot\Big],
\end{align*}
with, using \eqref{def scr V1}, 
\begin{align*}
\widehat{\mathbf{V}}_1^\varepsilon(\rho)&\triangleq\varepsilon^2\omega(\lambda)\partial_{\varphi}\beta_1+\mathbf{V}_1^\varepsilon(\rho)\big(1+\partial_\theta\beta_1\big)\\
&=\varepsilon^2\omega(\lambda)\partial_\varphi\beta_1+\Big(\tfrac{1}{2}+\varepsilon^{2}\mathtt{g}-\tfrac{\varepsilon^{2+\mu}}{2}\rho+\varepsilon^{3}{V}^\varepsilon(\rho)\Big)\big(1+\partial_\theta\beta_1\big).
\end{align*}
Select $\beta_1$ in the following form
\begin{equation}\label{def bfrak}
 \beta_1(\lambda,\varphi,\theta)=\beta_{1,1}(\lambda,\varphi)+\varepsilon^{2}\beta_{1,2}(\lambda,\varphi,\theta)+\varepsilon^{4}\beta_{1,3}(\lambda,\varphi,\theta),
\end{equation}
such that  for any $n\in\{1,2,3\},$ the function $\beta_{1,n}$ satisfies the following constraints 
\begin{eqnarray}  \label{constrain bk2}
     \begin{cases}
	 \tfrac{1}{2}\,\partial_{\theta}\beta_{1,2}&=\mathtt{F}_1-\langle \mathtt{F}_1\rangle_{\theta}, \qquad
		\mathtt{F}_1\triangleq-\mathtt{g}+\tfrac{1}{2}\varepsilon^\mu\rho-\varepsilon V^\varepsilon(\rho), \\
 \tfrac{1}{2}\,\partial_{\theta}\beta_{1,3}&=\mathtt{F}_2-\langle \mathtt{F}_2\rangle_{\theta},\qquad
		\mathtt{F}_2\triangleq-\omega(\lambda)\partial_\varphi \beta_{1,2}-\big( \mathtt{g}-\tfrac{1}{2}\varepsilon^{\mu}\rho+\varepsilon V^\varepsilon(\rho)\big)\partial_{\theta}\beta_{1,2}, 
      \end{cases} 
\end{eqnarray}	
and
\begin{equation}  \label{constrain bk1}
		\omega(\lambda)\partial_{\varphi}\beta_{1,1}=-\langle \mathtt{F}_1\rangle_{\theta}-\varepsilon^2 \langle\mathtt{F}_2\rangle_{\theta}+\langle \mathtt{F}_1\rangle_{\theta,\varphi}+\varepsilon^2 \langle\mathtt{F}_2\rangle_{\theta,\varphi}.
\end{equation}
With these choices, one can check that
\begin{align*}
	\widehat{\mathbf{V}}_1^\varepsilon(\rho)=\tfrac{1}{2}-\varepsilon^2\langle \mathtt{F}_1\rangle_{\theta,\varphi}-\varepsilon^4\langle \mathtt{F}_2\rangle_{\theta,\varphi}+\varepsilon^6\widetilde{\mathbf{V}}_1^\varepsilon(\rho),
\end{align*}
with
\begin{align}\label{def bfV1t}
\widetilde{\mathbf{V}}_1^\varepsilon(\rho)\triangleq\omega(\lambda)\partial_\varphi \beta_{1,3}+\big(\mathtt{g}-\tfrac{1}{2}\varepsilon^\mu\rho+\varepsilon V^\varepsilon(\rho)\big)\partial_\theta \beta_{1,3}.
\end{align}
Using in particular \eqref{ttw2}, we have
\begin{equation}\label{avg rho g}
	\langle\rho\rangle_{\theta}=0,\qquad\langle\mathtt{g}\rangle_{\theta}=0.
\end{equation}
Therefore, 
\begin{equation}\label{avg F1}
	\langle\mathtt{F}_1\rangle_{\theta}=-\varepsilon\langle V^{\varepsilon}(\rho)\rangle_{\theta}
\end{equation}
and
\begin{align*}
	\widehat{\mathbf{V}}_1^\varepsilon(\rho)=\tfrac{1}{2}+\varepsilon^{3} \langle {V}^\varepsilon(\rho)\rangle_{\theta,\varphi}-\varepsilon^4\langle F_2\rangle_{\theta,\varphi} +\varepsilon^6 \widetilde{\mathbf{V}}_1^\varepsilon(\rho).
\end{align*}
By defining
\begin{equation}\label{def bfV2}
\mathbf{V}_2^\varepsilon({\rho})\triangleq \mathrm{B}_1^{-1}\widetilde{\mathbf{V}}_1^\varepsilon(\rho),\qquad\mathtt{c}_0\triangleq \tfrac{1}{2}+\varepsilon^{3}\mathtt{c}_1,\qquad\mathtt{c}_1\triangleq\big(\langle {V}^\varepsilon(\rho)\rangle_{\theta,\varphi}-\varepsilon\langle F_2\rangle_{\theta,\varphi}\big),
\end{equation}
we end up with the identity \eqref{conjugaison scrB1}.
The first equation in \eqref{constrain bk2} writes
\begin{align*}
	\tfrac{1}{2}\,\partial_{\theta}\beta_{1,2}(\varphi,\theta)=-\mathtt{g}(\varphi,\theta)+\tfrac{1}{2}\varepsilon^\mu\rho-\varepsilon V^\varepsilon(\rho)+\varepsilon\langle  V^\varepsilon(\rho)\rangle_{\theta}.
\end{align*}
The right hand-side has zero space average according to \eqref{avg rho g}. Consequently, we can solve this equation via Fourier expansion to get from \eqref{ttw2} that
\begin{equation}\label{b1-expand}
\beta_{1,2}=b_{2}+\varepsilon^{\mu}\mathfrak{r}_{2},
\end{equation}
with $b_2$ as in \eqref{def b2} and $\mathfrak{r}_2$ satisfying, in view of \eqref{estim V}, the following esitmate
\begin{align*}
\|\mathfrak{r}_{2}\|_{s}^{\textnormal{Lip}(\gamma)}\lesssim1+\varepsilon^\mu\|\rho\|_{s+1}^{\textnormal{Lip}(\gamma)}.
\end{align*}
As consequence
\begin{align}\label{est b12}
\|\beta_{1,2}\|_{s}^{\textnormal{Lip}(\gamma)}\lesssim1+\|\rho\|_{{s+1}}^{\textnormal{Lip}(\gamma)}.
\end{align}
One can proceed in a similar way with the second equation in \eqref{constrain bk2} and get
\begin{align}\label{est b13}
\|\beta_{1,3}\|_{s}^{\textnormal{Lip}(\gamma)}\lesssim 1+\|\rho\|_{{s+2}}^{\textnormal{Lip}(\gamma)}.
\end{align}
Now, from \eqref{avg F1}, the equation \eqref{constrain bk1} becomes
\begin{align*}
	\omega(\lambda)\partial_\varphi \beta_{1,1}=\varepsilon \big[\langle{V}^\varepsilon(\rho)\rangle_{\theta}-\langle  {V}^\varepsilon(\rho)\rangle_{\theta,\varphi}+\varepsilon \langle\mathtt{F}_2\rangle_{\theta,\varphi}-\varepsilon \langle\mathtt{F}_2\rangle_{\theta}\big].
\end{align*}
The $\varphi$-average in the right hand-side vanishes so we can invert and find from \eqref{estim V},
\begin{equation}\label{est b11}
\beta_{1,1}=\varepsilon b_{1},\qquad\|b_{1}\|_{s}^{\textnormal{Lip}(\gamma)}\lesssim1+\varepsilon^\mu\|\rho\|_{s}^{\textnormal{Lip}(\gamma)}.
\end{equation}
Putting together the identities \eqref{def bfrak},\eqref{b1-expand} and \eqref{est b11} leads to
\begin{align*}
	\beta_1(\varphi,\theta)= \varepsilon b_1(\varphi)+\varepsilon^2b_{2}(\varphi,\theta)+\varepsilon^{2+\mu}\mathfrak{r}(\varphi,\theta),\qquad\textnormal{with}\qquad\mathfrak{r}(\varphi,\theta)\triangleq\mathfrak{r}_{2}(\varphi,\theta)+\varepsilon^{2-\mu}\beta_{1,3}(\varphi,\theta).
\end{align*}
Moreover, from the last identity and the foregoing estimates, we obtain 
\begin{align*}
	\|\beta_1\|_{s}^{\textnormal{Lip}(\gamma)}+\|\mathfrak{r}\|_{s}^{\textnormal{Lip}(\gamma)}\lesssim\varepsilon\big(1+\|\rho\|_{{s+2}}^{\textnormal{Lip}(\gamma)}\big).
\end{align*}
This estimate together with Lemma \ref{lemma estimates CVAR} concludes \eqref{est scrB1}. 
Now, from \eqref{constrain bk2}, \eqref{def bfV2}, \eqref{est b12} and \eqref{estim V}, we infer
\begin{align*}
	\|\mathtt{c}_1\|^{\textnormal{Lip}(\gamma)}\lesssim 1.
\end{align*}
In addition, by \eqref{def bfV1t}, \eqref{estim V}, \eqref{est b13}, and the product law  \eqref{ prod law}, we find
\begin{equation}\label{e-tildebfV1}
  \|\widetilde{\mathbf{V}}_1^\varepsilon(\rho)\|_{s}^{\textnormal{Lip}(\gamma)}\lesssim1+\|\rho\|_{s+{3}}^{\textnormal{Lip}(\gamma)}.
\end{equation}
Applying \eqref{tame comp} together with \eqref{e-tildebfV1} and \eqref{smallness CVAR1}, we get  
\begin{align*}
\|\mathbf{V}_2^\varepsilon(\rho)\|_{s}^{\textnormal{Lip}(\gamma)}&=\|\mathrm{B}_1^{-1}\widetilde{\mathbf{V}}_1^\varepsilon(\rho) \|_{s}^{\textnormal{Lip}(\gamma)}\\
& \leqslant\|\widetilde{\mathbf{V}}_1^\varepsilon(\rho) \|_{s}^{\textnormal{Lip}(\gamma)}\big(1+C\|\beta_1\|_{s_0}^{\textnormal{Lip}(\gamma)}\big)+C\|\beta_1\|_{s}^{\textnormal{Lip}(\gamma)}\|\widetilde{\mathbf{V}}_1^\varepsilon(\rho)\|_{s_0}^{\textnormal{Lip}(\gamma)}\Big)\nonumber\\
&\lesssim 1+\|\rho\|_{s+{3}}^{\textnormal{Lip}(\gamma)}.
\end{align*}
Similarly, from Lemma \ref{lemma estimates CVAR}, the second estimate in \eqref{estim V}, \eqref{smallness CVAR1} we conclude \eqref{diff ttc0}. Finally, by \eqref{ttc0} and \eqref{estim V}, we find \eqref{diff bfV2}.
This ends the proof of Proposition \ref{prop CVAR1}.
\end{proof}
Now, we shall state the final reduction step based on KAM tools. The proof is quite similar to that of {\cite[Prop. 4.2]{HHM23}} and we recommend the reader this references for more  details. To state the result, we need to use the following sequence $(N_n)_{n\in\mathbb{N}\cup\{-1\}}$ given by
	\begin{equation}\label{suite Nn}
		N_{-1}\triangleq1,\qquad\forall n\in\mathbb{N},\quad N_{n}\triangleq N_{0}^{\left(\frac{3}{2}\right)^{n}},
	\end{equation}
for some $N_0\geqslant2.$
\begin{proposition}\label{prop CVAR2}
	
Let $(\gamma,s_0,S)$ as in \eqref{cond1} and 
\begin{equation}\label{choix mu2}
\mu_{2}\geqslant 4\tau+3.
\end{equation}
There exists ${\varepsilon}_0>0$ such that if    
\begin{align}\label{smallness CVAR2}
\gamma\varepsilon^{-2}\leqslant1,\qquad N_{0}^{\mu_{2}}\varepsilon^{6}\gamma^{-1}\leqslant{\varepsilon}_0,\qquad\|\rho\|_{\frac{3}{2}\mu_{2}+2s_{0}+2\tau+4}^{\textnormal{Lip}(\gamma)}\leqslant1,
\end{align}
then, we can find $\mathtt{c}\triangleq \mathtt{c}(\lambda,\rho)\in \textnormal{Lip}_\gamma(\mathscr{O},\mathbb{R})$ and $\beta_2\in\textnormal{Lip}_\gamma(\mathscr{O},H^{S})$
with  the following properties:
\begin{enumerate}
\item The  function $\lambda\in \mathscr{O}\mapsto\mathtt{c}(\lambda,\rho)$  satisfies the following estimate,
\begin{equation*}
\| \mathtt{c}-\mathtt{c}_0 \|^{\textnormal{Lip}(\gamma)}\lesssim {\varepsilon}^{{6}},
\end{equation*}
where $\mathtt{c}_0$ is defined in \eqref{ttc0}.
\item Consider the transformation $\mathscr{B}_2$ given by
\begin{equation}\label{CVAR2}
	\mathscr{B}_2 h(\lambda,\varphi,\theta)\triangleq\big(1+\partial_{\theta}\beta_2(\lambda,\varphi,\theta)\big)h\big(\lambda,\varphi,\theta+\beta_2(\lambda,\varphi,\theta)\big).
\end{equation}
It is invertible with inverse taking the form
\begin{align*}
	\mathscr{B}_2^{-1} h(\lambda,\varphi,\theta)\triangleq\big(1+\partial_{\theta}\widehat{\beta}_2(\lambda,\varphi,\theta)\big)h\big(\lambda,\varphi,\theta+\widehat{\beta}_2(\lambda,\varphi,\theta)\big).
\end{align*}
Moreover, for any $s\in[s_{0},S],$ 
\begin{equation*}
\|\mathscr{B}_2^{\pm 1}h\|_{s}^{\textnormal{Lip}(\gamma)}\lesssim\|h\|_{s}^{\textnormal{Lip}(\gamma)}+{\varepsilon^{{6}}\gamma^{-1}}\|\rho\|_{s+2\tau+{5}}^{\textnormal{Lip}(\gamma)}\|h\|_{s_{0}}^{\textnormal{Lip}(\gamma)}
\end{equation*}
and 
\begin{equation}\label{estim beta2}
\|\widehat{\beta}_2\|_{s}^{\textnormal{Lip}(\gamma)}\lesssim\|\beta_2\|_{s}^{\textnormal{Lip}(\gamma)}\lesssim\varepsilon^{{6}}\gamma^{-1}\left(1+\|\rho\|_{s+2\tau+{4}}^{\textnormal{Lip}(\gamma)}\right).
\end{equation}
\item For any $n\in\mathbb{N}$, if $\lambda$ belongs to the Cantor set
\begin{equation}\label{Cantor set0}
{\mathscr{O}_{n}^{1}(\rho)}\triangleq\bigcap_{(l,j)\in\mathbb{Z}^{2}\atop 1\leqslant |j|\leqslant N_{n}}\Big\lbrace\lambda\in \mathscr{O}\quad\textnormal{s.t.}\quad\big|\varepsilon^2\omega(\lambda)l+j\mathtt{c}(\lambda,\rho)\big|\geqslant{\gamma}{|j|^{-\tau}}\Big\rbrace,
\end{equation}
then the following identity holds
\begin{align}\label{conjugaison scrB2}
\mathscr{B}_2^{-1}\Big(&\varepsilon^2\omega(\lambda)\partial_\varphi+\partial_{\theta}\big[\big(\mathtt{c}_0+\varepsilon^{6}\mathbf{V}_2^\varepsilon(\rho)\big)\cdot\big]\Big)\mathscr{B}_2=\varepsilon^2\omega(\lambda)\partial_\varphi+\mathtt{c}(\lambda,\rho)\partial_{\theta}+\mathtt{E}_{n},
\end{align}
with $\mathtt{E}_{n}$ a linear operator satisfying
\begin{equation}\label{estimate En0}
\|\mathtt{E}_{n}h\|_{s_0}^{\textnormal{Lip}(\gamma)}\lesssim \varepsilon^{{6}} N_{0}^{\mu_{2}}N_{n+1}^{-\mu_{2}}\|h\|_{s_{0}+2}^{\textnormal{Lip}(\gamma)}.
\end{equation}
	\item Given two small states $\rho_{1}$ and $\rho_{2}$ both satisfying the smallness properties \eqref{smallness CVAR2}, we have 
			\begin{align}\label{diff ttc}
				\|\Delta_{12}\mathtt{c}\|^{\textnormal{Lip}(\gamma)}&\lesssim\varepsilon^{3}\| \Delta_{12}\rho\|_{2{s}_{0}+2\tau+3}^{\textnormal{Lip}(\gamma)}.
			\end{align}
\end{enumerate}
\end{proposition}

Now we can state the main result of this subsection  which addresses the conjugation of the linearized \mbox{operator $\mathcal{L}_0$,} as described by \eqref{def func F}, using the changes of coordinates discussed before.  
\begin{proposition}\label{prop CVAR} 
Given the conditions \eqref{cond1}, \eqref{choix mu2} and \eqref{smallness CVAR2}.  Let $\mathscr{B}_1,\mathscr{B}_2$ as in \eqref{CVAR1} and \eqref{CVAR2}. Then, the following properties hold true.
\begin{enumerate}
\item The operator 
\begin{equation*}
\mathscr{B}\triangleq \mathscr{B}_1\mathscr{B}_2
\end{equation*}
writes 
\begin{equation}\label{CVAR}
\mathscr{B}h(\lambda,\varphi,\theta)=\big(1+\partial_{\theta}\beta(\lambda,\varphi,\theta)\big)  h\big(\lambda,\varphi,\theta+\beta(\lambda,\varphi,\theta)\big),
\end{equation}
with 
\begin{align*}
	\beta(\lambda,\varphi,\theta)\triangleq\beta_1(\lambda,\varphi,\theta)+\beta_2\big(\lambda,\varphi,\theta+\beta_1(\lambda,\varphi,\theta)\big)
\end{align*}
satisfying the estimate
\begin{equation*}
\|\beta\|_{s}^{\textnormal{Lip}(\gamma)}\lesssim  
\varepsilon\big(1+\| \rho\|_{s+2\tau+{4}}^{\textnormal{Lip}(\gamma)}\big).
\end{equation*}
Moreover, $\mathscr{B}$ is invertible and satisfies the estimate, for any $s\in[s_0,S]$
\begin{align*}
\|\mathscr{B}^{\pm 1} h\|_{s}^{\textnormal{Lip}(\gamma)}& \lesssim   \|h\|_{s}^{\textnormal{Lip}(\gamma)}+ \| \rho\|_{s+2\tau+{5}}^{\textnormal{Lip}(\gamma)}\|h\|_{s_0}^{\textnormal{Lip}(\gamma)}.
\end{align*}
\item For any $\lambda$ in the Cantor set ${\mathscr{O}_{n}^{1}(\rho)}$, defined in \eqref{Cantor set0},
one has 
\begin{align*}
\mathcal{L}_1\triangleq\mathscr{B}^{-1}\mathcal{L}_0\mathscr{B}&= \varepsilon^2\omega(\lambda)\partial_\varphi +\mathtt{c}(\lambda,\rho)\partial_{\theta}-\tfrac{1}{2}\mathbf{H}+\varepsilon^2\partial_\theta\mathfrak{Q}_1+\partial_{\theta}\mathfrak{R}_{1}^{\varepsilon}+\mathtt{E}_{n},
\end{align*}
where the linear operator $\mathtt{E}_{n}$ is the same as in Proposition $\ref{prop CVAR2}$ and 
\begin{align*}
\mathfrak{Q}_1[h](\varphi,\theta)&\triangleq  \int_{\mathbb{T}} h( \varphi, \eta)\mathtt{Q}_1\big(p(\varphi),\theta,\eta\big)d\eta,\\
\mathtt{Q}_1(p,\theta,\eta)&\triangleq\tfrac{\cos(\theta-\eta)}{r_{\mathbf{D}}^2(p)}-\tfrac{1}{2}\textnormal{Re}\left\lbrace\big(\partial_{z}\mathcal{R}_{\mathbf{D}}(p)\big)^2e^{\ii(\theta+\eta)}\right\rbrace.
\end{align*}
In addition, the operator $\partial_{\theta}\mathfrak{R}_{1}^{\varepsilon}$ satisfies  the estimates: for any $s\in[s_0,S],\,N\geqslant0$, 
\begin{equation*}
\begin{aligned}
\|\partial_{\theta}\mathfrak{R}_{1}^{\varepsilon} h\|_{s,N}^{\textnormal{Lip}(\gamma)}&\lesssim \big(\varepsilon^{2+{\mu}} +\varepsilon^{6} \gamma^{-1}\big) \Big(\|h\|_{{s}}^{\textnormal{Lip}(\gamma)}\big(1+ \|\rho\|_{s_0+2\tau+{5}+N}^{\textnormal{Lip}(\gamma)}\big) + \|h\|_{{s_0}}^{\textnormal{Lip}(\gamma)}\|\rho\|_{s+2\tau+{5}+N}^{\textnormal{Lip}(\gamma)} \Big).
\end{aligned}
\end{equation*}
\end{enumerate}

\end{proposition}

\begin{proof}

${\bf 1}.$ The structure \eqref{CVAR} follows from Lemma \ref{lem tranfo CVAR}-1 together with \eqref{CVAR1} and \eqref{CVAR2}. Then, making use of  Lemma \ref{lemma estimates CVAR}-$1$, \eqref{est beta1 b1 frakr bfV2} and  \eqref{estim beta2}, we get the desired estimates.\\
${\bf 2}.$ According to \eqref{def L0}, \eqref{conjugaison scrB1} and \eqref{conjugaison scrB2}, we get that for any $\lambda\in\mathscr{O}_{n}^{1}(\rho)$ 
\begin{align}\label{conjug A}
\mathscr{B}^{-1}\mathcal{L}_0\mathscr{B}&=\varepsilon^2\omega(\lambda)\partial_{\varphi}+\mathtt{c}(\lambda,\rho)\partial_{\theta}+\mathtt{E}_{n}-\tfrac{1}{2}\mathscr{B}^{-1}\mathbf{H} \mathscr{B}+\varepsilon^2 \mathscr{B}^{-1}\partial_\theta\mathbf{Q}_1  \mathscr{B}+
 \varepsilon^3 \mathscr{B}^{-1} \partial_\theta\mathcal{R}^\varepsilon_0(\rho) \mathscr{B}.
\end{align}
In the sequel, for convenience, we will omit the dependence on $\lambda$ since it does not play any role here. 
Denoting
\begin{align*}
	\mathbf{H}_0\triangleq\mathscr{B}_1^{-1}\mathbf{H} \mathscr{B}_1-\mathbf{H},
\end{align*} 
we infer
\begin{align}\label{commutant H}
\nonumber\mathscr{B}^{-1}\mathbf{H} \mathscr{B}-\mathbf{H}&=\mathscr{B}_2^{-1}\mathscr{B}_1^{-1}\mathbf{H} \mathscr{B}_1\mathscr{B}_2-\mathbf{H}
\\
\nonumber&=\mathscr{B}_2^{-1}\big[\mathscr{B}_1^{-1}\mathbf{H} \mathscr{B}_1-\mathbf{H}\big]\mathscr{B}_2+\mathscr{B}_2^{-1}\mathbf{H}\mathscr{B}_2-\mathbf{H}
\\
&=\mathbf{H}_0+(\mathscr{B}_2^{-1}\mathbf{H}_0\mathscr{B}_2-\mathbf{H}_0)+\mathscr{B}_2^{-1}\mathbf{H}\mathscr{B}_2-\mathbf{H}.
\end{align}
Note that the inverse diffeomorphism $\mathscr{B}_1^{-1}$ admits the form
\begin{equation*}
 \mathscr{B}_1^{-1}h(\varphi,y)=\big(1+\partial_y\widehat{\beta_1}(\varphi,y)\big)h\big(\varphi,y+\widehat{\beta_1}(\varphi,y)\big), 
\end{equation*}
where, by means of  
\cite[Lem. A.6]{HHM23} and Proposition \ref{prop CVAR1}
\begin{align}\label{expand beta1h}
\widehat{\beta_1}(\varphi,\theta)=-\varepsilon\,b_1(\varphi)-\varepsilon^2b_{2}\big(\varphi,\theta\big)+\varepsilon^{2+\mu}\widehat{\mathfrak{r}}\big(\varphi,\theta\big)\qquad\textnormal{with}\qquad  \|\widehat{\mathfrak{r}}\|_{s}^{\textnormal{Lip}(\gamma)}\lesssim 1+\|\rho\|_{{s+3}}^{\textnormal{Lip}(\gamma)}.
\end{align}
Thus, in view of  \eqref{def Hilbert trans}, we get
\begin{align*}
\mathbf{H}_0[h](\lambda,\varphi,\theta)=\partial_\theta\int_{\mathbb{T}}\mathbb{K}_0(\lambda,\varphi,\theta,\eta)h(\lambda,\varphi,\eta)\,d\eta,
\end{align*}
where
\begin{align*}
\mathbb{K}_0(\lambda,\varphi,\theta,\eta)&\triangleq\log\left|\frac{e^{\ii\,(\theta+\widehat{\beta_1}(\lambda,\varphi,\theta))}-e^{\ii \,(\eta+\widehat{\beta_{1}}(\lambda,\varphi,\eta))}}{e^{\ii\,\theta}-e^{\ii\,\eta}} \right|\\
&=\varepsilon^2\hbox{Im}\bigg\{\frac{e^{\ii\,\theta}b_{2}(\varphi,\theta)-e^{\ii\,\eta}b_{2}(\varphi,\eta)}{e^{\ii\,\theta}-e^{\ii\,\eta}} \bigg\}+\varepsilon^{2+\mu}\mathbb{K}_1(\varphi,\theta,\eta),
\end{align*}
with the kernel $\mathbb{K}_1$ satisfying the estimate
\begin{align}\label{estim bbK1}
\|\mathbb{K}_1\|_{s}^{\textnormal{Lip}(\gamma)}\lesssim1+\|\rho\|_{s+{3}}^{\textnormal{Lip}(\gamma)},
\end{align}
where we have used Taylor expansion combined with \eqref{expand beta1h}.
Making appeal to the expression \eqref{def b2} and using the identity
\begin{align*}
\big(e^{\ii2\theta}-e^{\ii2\eta}\big)\hbox{Im}\Big\{\tfrac{e^{\ii\eta}}{e^{\ii\theta}-e^{\ii\eta}}\Big\}
&=\tfrac{1}{2\ii}\big(e^{\ii\theta}+e^{\ii\eta}\big)^2,
\end{align*} we infer
\begin{align*}
\mathbb{K}_0(\varphi,\theta,\eta)
&=\varepsilon^2\Big(b_{2}(\varphi,\theta)-b_{2}(\varphi,\eta)\Big)\hbox{Im}\Big\{\tfrac{e^{\ii\eta}}{e^{\ii\theta}-e^{\ii\eta}}\Big\}+\varepsilon^{2+\mu}\mathbb{K}_1(\varphi,\theta,\eta)\\
&=\varepsilon^2\textnormal{Im}\Big\{\mathtt{w}_2\big(p(\varphi)\big)\big(e^{\ii2\theta}-e^{\ii2\eta}\big)\Big\}\hbox{Im}\Big\{\tfrac{e^{\ii\eta}}{e^{\ii\theta}-e^{\ii\eta}}\Big\}+\varepsilon^{2+\mu}\mathbb{K}_1( \varphi, \theta,\eta)\\
&=-\tfrac{\varepsilon^2}{2}\textnormal{Re}\Big\{\mathtt{w}_2\big(p(\varphi)\big)\big(e^{\ii\theta}+e^{\ii\eta}\big)^2\Big\}+\varepsilon^{2+\mu}\mathbb{K}_1(\varphi,\theta,\eta).
\end{align*}
As a consequence, using in particular the fact that $h$ has zero space average, we can write
\begin{equation}\label{commutator with H}
\mathbf{H}_0[h]=-\varepsilon^2\partial_{\theta}\widehat{\mathbf{Q}}_1[h]+\varepsilon^{2+\mu}\mathbf{H}_1[h],
\end{equation}
where
\begin{align}
	\widehat{\mathbf{Q}}_1[h](\varphi,\theta)&\triangleq\int_{\mathbb{T}}\textnormal{Re}\Big\{\mathtt{w}_2\big(p(\varphi)\big)e^{\ii(\theta+\eta)}\Big\}h(\varphi,\eta)\,d\eta,\label{def hatbfQ1}\\
	\mathbf{H}_1[h](\varphi,\theta)&\triangleq\partial_\theta\int_{\mathbb{T}}\mathbb{K}_1(\varphi,\theta,\eta)h(\varphi,\eta)\,d\eta.\nonumber
\end{align}
Besides, one has
\begin{align}\label{commutator bfQ1}
\mathscr{B}^{-1}\partial_\theta\mathbf{Q}_{1}  \mathscr{B}&=\partial_\theta\mathbf{Q}_{1} +\mathscr{B}^{-1}\partial_\theta\mathbf{Q}_{1}  \mathscr{B}-\partial_\theta\mathbf{Q}_{1}. 
\end{align}
Plugging \eqref{commutant H}, \eqref{commutator bfQ1} and  \eqref{commutator with H} into \eqref{conjug A} implies
\begin{align*}
\mathscr{B}^{-1}\mathscr{L}_1 \mathscr{B}&=\varepsilon^2\omega(\lambda)\partial_{\varphi}+\mathtt{c}\partial_{\theta}+\mathtt{E}_{n}^{0}-\tfrac{1}{2}\mathbf{H} +\varepsilon^2\partial_\theta\mathfrak{Q}_1+\partial_{\theta}\mathfrak{R}_1^{\varepsilon},
\end{align*}
where 
\begin{equation}\label{def frakQ1 proof}
	\mathfrak{Q}_1\triangleq\mathbf{Q}_1+\tfrac{1}{2}\widehat{\mathbf{Q}}_1
\end{equation}
and
\begin{equation}\label{def frakR1eps}
\begin{aligned}
\partial_{\theta}\mathfrak{R}_1^\varepsilon&\triangleq\varepsilon^3\mathscr{B}^{-1}\partial_\theta\mathcal{R}^\varepsilon_0(\rho)\mathscr{B}-\tfrac{1}{2}[\mathscr{B}_2^{-1}\mathbf{H}_0\mathscr{B}_2-\mathbf{H}_0]-\tfrac{1}{2}\big[\mathscr{B}_2^{-1}\mathbf{H}\mathscr{B}_2-\mathbf{H}\big]\\ &\quad-\varepsilon^2\big[\mathscr{B}^{-1}\partial_\theta\mathbf{Q}_{1}\mathscr{B}-\partial_\theta\mathbf{Q}_{1}\big]-{\tfrac{1}{2}}\varepsilon^{2+\mu}\mathbf{H}_1.
\end{aligned}
\end{equation}
Combining \eqref{def frakQ1 proof}, \eqref{def hatbfQ1}, \eqref{ttw2}, \eqref{def bfQ1} and \eqref{ttQ} yields  the following integral representation
\begin{align*}
\mathfrak{Q}_1[h](\varphi,\theta)&\triangleq  \int_{\mathbb{T}} h( \varphi, \eta)\mathtt{Q}_1\big(p(\varphi),\theta,\eta\big)d\eta,\qquad
\mathtt{Q}_1(p,\theta,\eta)\triangleq\tfrac{\cos(\theta-\eta)}{r_{\mathbf{D}}^2(p)}-\tfrac{1}{2}\textnormal{Re}\left\lbrace\big(\partial_{z}\mathcal{R}_{\mathbf{D}}(p)\big)^2e^{\ii(\theta+\eta)}\right\rbrace.
\end{align*}
Now, by virtue of \cite[Lem. 2.36]{BM20}, \eqref{estim beta2} and \eqref{smallness CVAR2} we can write
\begin{align*}
		(\mathscr{B}_2^{-1}\mathbf{H} \mathscr{B}_2-\mathbf{H})h(\varphi, \theta) 
		&=\int_{\mathbb{T}}{\mathscr K}_1(\varphi,\theta,\eta)h(\varphi,\eta)\,d\eta,
\end{align*}
		with $\mathscr{K}_1$ satisfying the following estimates, for any $s\in[s_0,S],$  
\begin{align}\label{e-Box H}
\|\mathscr{K}_1\|_{s}^{\textnormal{Lip}(\gamma)}\lesssim\varepsilon^{{6}}\gamma^{-1}\left(1+\| \rho\|_{s+2\tau+{5}}^{\textnormal{Lip}(\gamma)}\right).
\end{align}
Moreover, according to Lemma \ref{lem CVAR kernel}, \eqref{est beta1 b1 frakr bfV2}, \eqref{estim beta2} and \eqref{smallness CVAR2}, we have
\begin{align*}
\big(\mathscr{B}^{-1}\partial_\theta\mathbf{Q}_{1}\mathscr{B}-\partial_\theta\mathbf{Q}_{1}\big)[h](\varphi,\theta)&=\int_{\mathbb{T}}h(\varphi,{\eta})\mathscr{K}_2(\varphi,\theta,\eta)d{\eta},\\
\big(\mathscr{B}_2^{-1}\mathbf{H}_0\mathscr{B}_2-\mathbf{H}_0\big)[h](\varphi,\theta)&=\int_{\mathbb{T}}h(\varphi,\eta){\mathscr  K}_3(\varphi,\theta,\eta)d\eta,\\
\mathscr{B}^{-1}\partial_\theta\mathcal{R}^\varepsilon_0(\rho)\mathscr{B}[h](\varphi,\theta)&=\int_{\mathbb{T}}h(\varphi,\eta)\mathscr{K}_4(\varphi,\theta,\eta)d\eta,
\end{align*}
with
\begin{align*}
\|\mathscr{K}_2\|_{s}^{\textnormal{Lip}(\gamma)}&\lesssim\varepsilon\left(1+\|\rho\|_{s+2\tau+{5}}^{\textnormal{Lip}(\gamma)}\right),\\
\|\mathscr{K}_3\|_{s}^{\textnormal{Lip}(\gamma)}&\lesssim\varepsilon^{6}\gamma^{-1}\left(1+\|\rho\|_{s+2\tau+{5}}^{\textnormal{Lip}(\gamma)}\right),\\ 
\|\mathscr{K}_4\|_{s}^{\textnormal{Lip}(\gamma)}&\lesssim1+\|\rho\|_{s+2\tau+{5}}^{\textnormal{Lip}(\gamma)}.
\end{align*}
Thus, the operator $\partial_{\theta}\mathfrak{R}_1^\varepsilon$, introduced in \eqref{def frakR1eps}, is an integral operator with the kernel 
\begin{align*}
	\mathcal{K}_1^{\varepsilon}\triangleq\varepsilon^3 {\mathscr K}_4-\tfrac{1}{2}{\mathscr K}_3-\tfrac{1}{2}{\mathscr  K}_1-\varepsilon^2{\mathscr  K}_2+\varepsilon^3\mathcal{K}_0^{\varepsilon}-\tfrac{1}{2}\varepsilon^{2+\mu}\partial_{\theta}\mathbb{K}_1,
\end{align*}
which satisfies the estimate 
\begin{align*}
	\|\mathcal{K}_1^{\varepsilon}\|_{s}^{\textnormal{Lip}(\gamma)}\lesssim\big(\varepsilon^{2+{\mu}}+\varepsilon^{6}\gamma^{-1}\big)\left(1+\|\rho\|_{s+2\tau+{5}}^{\textnormal{Lip}(\gamma)}\right).
\end{align*}
We conclude the estimate of $\mathfrak{R}_1^{\varepsilon}$ by applying Lemma \ref{lem CVAR kernel}. This ends the proof of Proposition \ref{prop CVAR}.
\end{proof}
\subsection{Invertibility of the linearized operator}
The main goal this section is to transform the linearized operator $\mathcal{L}_0$, as defined in \eqref{def L0}, into a Fourier multiplier up to a small error. This step is required along Nash-Moser scheme and will be  done in the spirit of \cite[Sec. $5$]{HHM23}. As the leading term of this operator  degenerates on the modes $\pm1,$ we shall first split the phase space into two parts: a subspace  localized on the modes $\pm1$ and its complement. This leads to explore the invertibility of a   matrix-valued operator that will be done in different steps as detailed below.  \\
Let  $\Pi_{1}$ stand for  the orthogonal projection on the modes $\pm1$ acting on $L^{2}(\mathbb{T}^2,\mathbb{R})$ as follows
\begin{align}\label{def Pi1}
	h(\varphi,\theta)=\sum_{j\in\mathbb{Z}}h_{j}(\varphi) e^{\ii j\theta}\quad\Rightarrow\quad\Pi_1h\triangleq\sum_{j=\pm1}h_{j}(\varphi) e^{\ii j\theta}.
	\end{align}
We also define 
$$\Pi_1^\perp\triangleq\textnormal{Id}-\Pi_1.$$ As the function $h$ is real-valued then 
$$
\forall j\in \mathbb{Z},\quad h_{-j}(\varphi)=\overline{h_j(\varphi)}.
$$
Now, we decompose the phase space $\textnormal{Lip}_{\gamma}(\mathscr{O},H_0^{s})$ as follows,
\begin{align}\label{phase space Xs}
	X^s\triangleq\textnormal{Lip}_{\gamma}(\mathscr{O},H_0^{s})=X^{s}_{1}\overset{\perp}{\oplus}X^s_{\perp},
\end{align}
with
\begin{align*}
	X^{s}_{1}\triangleq\big\{h\in\textnormal{Lip}_{\gamma}(\mathscr{O},H^{s}_0)\quad\textnormal{s.t.}\quad\Pi_1h=h\big\},\\
	X^{s}_{\perp}\triangleq\big\{h\in\textnormal{Lip}_{\gamma}(\mathscr{O},H^{s}_0)\quad\textnormal{s.t.}\quad\Pi_1^\perp h=h\big\}.
\end{align*}
Recall from Proposition \ref{prop CVAR} and Proposition \ref {prop CVAR2} that the operator $\mathcal{L}_1$ decomposes as below
\begin{equation}\label{def cal L1}
	\mathcal{L}_1=\mathbb{L}_1+\partial_{\theta}\mathfrak{R}_{1}+\mathtt{E}_{n},\qquad\mathbb{L}_1\triangleq\varepsilon^2\omega(\lambda)\partial_\varphi+\mathtt{c}(\lambda,\rho)\partial_{\theta}-\tfrac{1}{2}\mathbf{H}+\varepsilon^{2}\partial_\theta\mathfrak{Q}_1,
\end{equation}
with
$$\mathfrak{Q}_1[h](\varphi,\theta)\triangleq r_{\mathbf{D}}^{-2}\big(p(\varphi)\big)\textnormal{Re}\left\lbrace h_1(\varphi)e^{\ii \theta}\right\rbrace-\tfrac{1}{2}\textnormal{Re}\left\lbrace\big[\partial_{z}\mathcal{R}_{\mathbf{D}}\big(p(\varphi)\big)\big]^2h_{-1}(\varphi)e^{\ii\theta}\right\rbrace$$
and
\begin{equation}\label{def ttc2}
	\mathtt{c}=\tfrac{1}{2}+\varepsilon^{3}\mathtt{c}_1+(\mathtt{c}-\mathtt{c}_0)\triangleq\tfrac{1}{2}+\varepsilon^3\mathtt{c}_2,
\end{equation}
such that
\begin{align}\label{c2-estimate}
	\|\mathtt{c}_2\|^{\textnormal{Lip}(\gamma)}\leqslant C.
\end{align}
Upon straightforward analysis, it becomes apparent that the linear operator
$$\mathbb{L}_1+\partial_{\theta}\mathfrak{R}_1^{\varepsilon}:\textnormal{Lip}_{\gamma}(\mathscr{O},X^{s})\to \textnormal{Lip}_{\gamma}(\mathscr{O},X^{s-1})$$ is well-defined and  its action is  equivalent to the matrix operator $\displaystyle{\mathbb{M}: X^{s}_{1}\times X^{s}_{\perp}\to X^{s-1}_{1}\times X^{s-1}_{\perp}}$ with
\begin{align}\label{def bbM}
	\mathbb{M}&\triangleq\begin{pmatrix}
		\Pi_1\mathbb{L}_1\Pi_1 & 0 \\
		0 & \Pi_1^\perp\mathbb{L}_1\Pi_1^\perp
	\end{pmatrix}+\begin{pmatrix}
		\Pi_1\partial_{\theta}\mathfrak{R}_1^{\varepsilon}\Pi_1 & \Pi_1\partial_{\theta}\mathfrak{R}_1^{\varepsilon}\Pi_1^\perp \\
		\Pi_1^\perp\partial_{\theta}\mathfrak{R}_1^{\varepsilon}\Pi_1 & \Pi_1^\perp\partial_{\theta}\mathfrak{R}_1^{\varepsilon}\Pi_1^\perp
	\end{pmatrix}\\
	\nonumber &\triangleq\mathbb{M}_1+\partial_{\theta}\mathbb{R}_1,
\end{align}
where, according to  Proposition \ref{prop CVAR} and \eqref{hao11},  for any $H=(h_1,h_2)\in X_1^s\times X_{\perp}^s$, $s\in[s_0,S]$ and $N\geqslant0$, one has
		\begin{equation}\label{est sur R1}
		\|\partial_\theta\mathbb{R}_1H\|_{s,N}^{\textnormal{Lip}(\gamma)}
		\leqslant C\varepsilon^{2+{\mu}} \Big( \|H\|_{{s}}^{\textnormal{Lip}(\gamma)}\big(1+\|\rho\|_{s_0+N+2\tau+5}^{\textnormal{Lip}(\gamma)}\big)+\|\rho\|_{s+N+2\tau+5}^{\textnormal{Lip}(\gamma)}  \|H\|_{{s_0}}^{\textnormal{Lip}(\gamma)}\Big).
	\end{equation}
Hence, we will examine the invertibility of the scalar operator $\mathbb{L}_1+\partial_{\theta}\mathfrak{R}_1^{\varepsilon}$ by analyzing the invertibility of the matrix operator $\mathbb{M}$. This analysis involves inverting its main part $\mathbb{M}_1$ alongside employing perturbative arguments. To achieve this, we will break  the process down into several steps.
\subsubsection{Degeneracy of the modes $\pm1$ and monodromy matrix}\label{section mono}
The main objective is to demonstrate the invertibility of the operator $\Pi\mathbb{L}_1\Pi$,  defined through \eqref{def bbM} and \eqref{def cal L1}. More precisely,  by virtue  of  \eqref{def cal L1}, \eqref{def ttc2} and the identity
	$$
	(\partial_\theta-\mathbf{H})\Pi_1=0,
	$$
	one has
\begin{align*}
\nonumber \mathbb{L}_{1,1}\triangleq\Pi_1\mathbb{L}_1\Pi_1=\varepsilon^2\Big(\omega(\lambda)\partial_\varphi+\varepsilon\mathtt{c}_2(\lambda)\Pi_1\partial_{\theta}+\Pi_1\partial_\theta\mathfrak{Q}_1\Pi_1\Big).
\end{align*}
 Notably, this operator localizes in the Fourier  spatial modes $\pm1$ and displays a degeneracy in $\varepsilon$. 
To invert it, we consider an arbitrary  real-valued function 
$$(\varphi,\theta)\mapsto g(\varphi,\theta)=\displaystyle{\sum_{j=\pm1}g_j(\varphi)e^{\ii j\theta}\in X^{s-1}_{1}}
$$ 
and we shall solve in the real space  $ X^{s}_{1}$ the equation
$$
\mathbb{L}_{1,1}h= g.
$$
Since $g$ and $h$ are real then  $g_{-n}=\overline{g_n}$ and  $h_{-n}=\overline{h_n}$. Thus, 
using Fourier expansion,  the last equation is equivalent to 
\begin{equation}\label{equation on H}
\partial_\varphi H-\mathbf{A}(\varepsilon,\lambda,\varphi) H=\varepsilon^{-2} G,
\end{equation}
where
\begin{align}\label{Mat A}
\mathbf{A}(\varepsilon,\lambda,\varphi)\triangleq \begin{pmatrix}
\mathtt{u}_{\lambda}(\varphi)-\frac{\ii\varepsilon\mathtt{c}_2(\lambda)}{2\omega(\lambda)}&\mathtt{v}_{\lambda}(\varphi) 
\\
\overline{\mathtt{v}_{\lambda}(\varphi)} &\overline{\mathtt{u}_{\lambda}(\varphi)}+\frac{\ii\varepsilon\mathtt{c}_2(\lambda)}{2\omega(\lambda)} 
\end{pmatrix}, \,\qquad\begin{aligned}
	 \mathtt{u}_{\lambda}(\varphi)&=-\tfrac{\ii}{2\omega(\lambda)}\tfrac{1}{r_{\mathbf{D}}^{2}(p)},\\\mathtt{v}_{\lambda}(\varphi)&= \tfrac{\ii}{4\omega(\lambda)}\big(\partial_{z}\mathcal{R}_{\mathbf{D}}(p)\big)^2
\end{aligned}
\end{align}
and $$
G(\varphi)\triangleq \frac{1}{\omega(\lambda)}\begin{pmatrix}
g_1(\varphi) \\
\overline{g_1(\varphi)}
\end{pmatrix}\in\mathbb{C}^2 ,\,\qquad H(\varphi)\triangleq \begin{pmatrix}
h_1(\varphi) \\
\overline{h_1(\varphi)}
\end{pmatrix}\in\mathbb{C}^2.
$$

We intend to prove the following result.
\begin{proposition}\label{propo-monod}
Assume \eqref{non degeneracy monodromy thm}, there exists  $\varepsilon_1>0$ such that for all $\varepsilon\in(0,\varepsilon_1)$,  the operator $\mathbb{L}_{1,1}:  X^{s}_{1}\to  X^{s-1}_{1}$ is invertible with inverse $\mathbb{L}_{1,1}^{-1}$ satisfying
 	\begin{align*}
		\|\mathbb{L}_{1,1}^{-1}g\|_{s}^{\textnormal{Lip}(\gamma)}\leqslant C\varepsilon^{-2}\|g\|_{s-1}^{\textnormal{Lip}(\gamma)}.
	\end{align*}
	\end{proposition}
\begin{proof}
Let $(\varphi,\phi)\in\mathbb{R}^2 \mapsto \mathscr{F}(\varepsilon,\lambda,\varphi,\phi)$ be  the fundamental matrix defined through the $2\times2$ matrix ODE
	\begin{equation*}
		\begin{cases}
			\partial_\varphi \mathscr{F}(\varepsilon,\lambda,\varphi,\phi)- \mathbf{A}(\varepsilon,\lambda,\varphi) \mathscr{F}(\varepsilon,\lambda,\varphi,\phi)=0,
			\\
			\mathscr{F}(\phi,\phi)=\textnormal{Id}.
		\end{cases}
	\end{equation*}
	Then the solution $H$ can be expressed in the form
	\begin{align}\label{Def-M1}
		H(\varphi)=\mathscr{F}(\varepsilon,\lambda,\varphi,0) H(0)+\varepsilon^{-2}\int_0^\varphi \mathscr{F}(\varepsilon,\lambda,\varphi,\phi)G(\phi) d\phi.
	\end{align}
	Observe that  the matrix $A$, defined by \eqref{Mat A},  is $2\pi-$periodic in the variable $\varphi$. Then $H$ is $2\pi-$periodic if and only if
	$$
	H(2\pi)=H(0).
	$$
Equivalently, 
	\begin{align}\label{Iden-P1}
		\big(\hbox{Id}-\mathscr{F}(\varepsilon,\lambda,2\pi,0)\big)H(0)= \varepsilon^{-2}\int_0^{2\pi} \mathscr{F}(\varepsilon,\lambda,2\pi,\phi)G(\phi) d\phi.
	\end{align}
	To find a unique solution to this equation it is enough to show  that the matrix $\hbox{Id}-\mathscr{F}(\varepsilon,\lambda,2\pi,0)$ is invertible.
	To this end, we use the decomposition
$$\mathbf{A}(\varepsilon,\lambda,\varphi)=\mathbb{A}_\lambda(\varphi)+\mathbf{B}(\varepsilon,\lambda,\varphi),\qquad \mathbb{A}_\lambda(\varphi)\triangleq \mathbf{A}(0,\lambda,\varphi),\qquad  \mathbf{B}(\varepsilon,\lambda,\varphi)\triangleq \mathbf{A}(\varepsilon,\lambda,\varphi)-\mathbf{A}(0,\lambda,\varphi).$$
According to \eqref{c2-estimate}, one has
\begin{equation}\label{estimate A1}
\sup_{\varphi\in\mathbb{R}}\|\mathbf{B}(\varepsilon,\lambda,\varphi)\|\lesssim\varepsilon.
\end{equation}
	Now, consider the fundamental solution of the unperturbed problem
	\begin{equation}\label{Matrix-Fund1}
		\begin{cases}
			\partial_\varphi \mathscr{M}_\lambda(\varphi)- \mathbb{A}_\lambda(\varphi) \mathscr{M}_\lambda(\varphi)=0,
			\\
			\mathscr{M}_\lambda(0)=\textnormal{Id}.
		\end{cases}
	\end{equation}
	Then one may write
	$$\mathscr{F}(\varepsilon,\lambda,\varphi,0)=\mathscr{M}_\lambda(\varphi)+ \mathscr{G}(\varepsilon,\lambda,\varphi),
	$$
	with 
	$$
	\partial_\varphi \mathscr{G}(\varepsilon,\lambda,\varphi)- \mathbb{A}_\lambda(\varphi)\, \mathscr{G}(\varepsilon,\lambda,\varphi)=-\mathbf{B}(\varepsilon,\lambda,\varphi)\mathscr{F},\quad \mathscr{G}(\varepsilon,\lambda,0)=0.
	$$
	and
	\begin{align*}
		\sup_{\varphi\in[0,2\pi]}\|\mathscr{G}(\varepsilon,\lambda,\varphi)\|\leqslant C\varepsilon,
	\end{align*}
	where we have used for the last inequality the estimate \eqref{estimate A1}.
The resolvant matrix must have the following form
\begin{equation}\label{resolvante shape}
	\mathscr{M}_\lambda(\varphi)=\begin{pmatrix}
		a_{\lambda}(\varphi) & b_{\lambda}(\varphi)\\
		\overline{b_{\lambda}(\varphi)} & \overline{a_{\lambda}(\varphi)}
	\end{pmatrix},\qquad a_{\lambda}(0)=1,\qquad b_{\lambda}(0)=0.
\end{equation}
Remark that the property $\textnormal{Tr}(\mathbb{A}_{\lambda})\equiv0$ gives
$$|a_{\lambda}(\varphi)|^2-|b_{\lambda}(\varphi)|^2=\det\big(\mathscr{M}_\lambda(\varphi)\big)=\exp\left(\int_{0}^{\varphi}\textnormal{Tr}\big(\mathbb{A}_\lambda(u)\big)du\right)=1.$$
Hence,
$$\det(\mathscr{M}_\lambda-\textnormal{Id})=|a_{\lambda}-1|^2-|b_{\lambda}|^2=2\big(1-\textnormal{Re}(a_{\lambda})\big)=2-\textnormal{Tr}\big(\mathscr{M}_\lambda(2\pi)\big).$$
In particular,
\begin{equation}\label{cond reso}
	\det\big(\mathscr{M}_\lambda(2\pi)-\textnormal{Id}\big)\neq0\qquad\Longleftrightarrow\qquad\textnormal{Tr}\big(\mathscr{M}_\lambda(2\pi)\big)\neq2.
\end{equation}
	Thus, under the assumption \eqref{non degeneracy monodromy thm}, the matrix $\mathscr{M}_{\lambda}(2\pi)-\textnormal{Id}$ is invertible and there exist   $C>0$  such that 
	$$
	\|\big(\mathscr{M}_\lambda(2\pi)-\textnormal{Id}\big)^{-1}\|^{\textnormal{Lip}(\gamma)}\leqslant C.
	$$
		Therefore, by using perturbation arguments we conclude that  the matrix $\mathscr{F}(\varepsilon,\lambda,2\pi)-\textnormal{Id}$ is invertible on $[\lambda_*,\lambda^*]$ and 
	we have
	$$
	\|\big(\mathscr{F}(\varepsilon,2\pi)-\textnormal{Id}\big)^{-1}\|^{\textnormal{Lip}(\gamma)}\leqslant C.
	$$
	Consequentely, the equation \eqref{Iden-P1} admits a unique solution satisfying 
	\begin{align}\label{Iden-P2}
		\|H(0)\|^{\textnormal{Lip}(\gamma)}_{s}\leqslant C\varepsilon^{-2}\|G\|_{L^2(\mathbb{T})},
	\end{align}
	where we have used, for $\varepsilon$ small enough,  the estimate
	$$
	\sup_{\varphi,\phi\in[0,2\pi]}\| \mathscr{F}(\varepsilon,\lambda,\varphi,\phi)\|\leqslant C.
	$$
	From the previous analysis we conclude that the equation  \eqref{equation on H} admits a unique  solution which satisfies, in view of  \eqref{Iden-P2} and   \eqref{Def-M1} the estimate:  for all $s>1$
	\begin{align*}
		\|H\|_{s}^{\textnormal{Lip}(\gamma)}\leqslant C\varepsilon^{-2}\|G\|_{s-1}^{\textnormal{Lip}(\gamma)}.
	\end{align*}	
This implies  that  the linear operator $\mathbb{L}_{1,1}:X^{s}_{1}\to X^{s-1}_{1}$ is invertible on $[\lambda_*,\lambda^*]$  and 
	\begin{align*}
		\|\mathbb{L}_{1,1}^{-1}g\|_{s}^{\textnormal{Lip}(\gamma)}\leqslant C\varepsilon^{-2}\|g\|_{s-1}^{\textnormal{Lip}(\gamma)}.
	\end{align*}
	This achieves the proof of the desired result.
	\end{proof}	
We shall end this section with some comments on the monodromy matrix that can be linked to Riccatti equation.
As we have seen before,
\begin{align*}
	\det(\mathscr{M}_\lambda)=1,
\end{align*}
which implies that the map 
\begin{align*}
	\varphi\in\mathbb{R} \mapsto \mathscr{M}_\lambda(\varphi)\in SL(2;\mathbb{C})
\end{align*}
is well defined. We will associate to each element $\mathscr{M}_\lambda(\varphi)$   the M\"obius transform
\begin{align*}
	T(\varphi): z\in \mathbb{D}\mapsto \frac{a_{\lambda}(\varphi)\, z+b_{\lambda}(\varphi)}{\overline{b_{\lambda}(\varphi)}\,z+\overline{a_{\lambda}(\varphi)}}\in  \mathbb{D}.
\end{align*}
Notice that $T(\varphi):\mathbb{D}\to \mathbb{D}$ is an automorphism with $T(0)=\hbox{Id}.$ Straightforward computations based on \eqref{Matrix-Fund1}  lead to \begin{equation*}
	\begin{cases}
		\partial_\varphi T=-\ii \mu\, T-\overline{\mathtt{v}_{\lambda}}\,T^2+\mathtt{v}_{\lambda},
		\\
		T(0,z)=z\in \mathbb{D},
	\end{cases}
\end{equation*}
with $T^2(\varphi,z)\triangleq\big(T(\varphi,z)\big)^2$ and 
\begin{equation*}
	\mu\triangleq\tfrac{e^{2\lambda}}{\omega(\lambda)}, \qquad\mathtt{v}_{\lambda}(\varphi)\triangleq \tfrac{\ii}{4\omega(\lambda)}\Big(\partial_{z}\mathcal{R}_{\mathbf{D}}\big(p(\varphi)\big)\Big)^2.
\end{equation*}
By setting
\begin{align*}
	\mathtt{T}(\varphi)\triangleq e^{-i\mu \varphi} T(\varphi),\qquad \varrho(\varphi)\triangleq e^{i\mu \varphi}\mathtt{v}_{\lambda}(\varphi),
\end{align*}
we get Riccatti equation
\begin{equation*}
	\begin{cases}
		\partial_\varphi \mathtt{T}=-\overline{\varrho}\,\,\mathtt{T}^2+\varrho,
		\\
		\mathtt{T}(0,z)=z \in \mathbb{D}.
	\end{cases}
\end{equation*}

\subsubsection{Invertibility of $\Pi_1^\perp\mathbb{L}_1\Pi_1^\perp$}\label{section inv normal direc}
The purpose here is to find a right inverse for the operator $$\mathbb{L}_{1,\perp}\triangleq\Pi_1^\perp\mathbb{L}_1\Pi_1^\perp,$$ where $\mathbb{L}_1$ is defined through \eqref{def cal L1} and \eqref{def Pi1}. Recall from Proposition \ref{prop CVAR}-2 that the operator $\mathfrak{Q}_1$ localizes on the modes $\pm1$ implying that
$$\Pi_1^{\perp}\mathfrak{Q}_1\Pi_{1}^{\perp}=0.$$
Therefore, the linear operator 
$$ \mathbb{L}_{1,\perp}:X^s_{\perp}\to X^{s-1}_{\perp}
$$ is well-defined and assumes the structure 
\begin{align}\label{def bbL1perp}
	\mathbb{L}_{1,\perp}
	&=\varepsilon^2\omega(\lambda)\partial_\varphi+\mathcal{D}_{1,\perp},\qquad\mathcal{D}_{1,\perp}\triangleq\mathtt{c}(\lambda)\partial_{\theta}-\tfrac{1}{2}\mathbf{H}.
\end{align}
In view of  \eqref{def ttc2}, one has
\begin{align}\label{D-1-LL}
	\forall l\in\mathbb{Z},\quad\forall|j|\geqslant 2,\quad\mathcal{D}_{1,\perp}\mathbf{e}_{l,j}=\ii\mu_{j,2}\,\mathbf{e}_{l,j},
\end{align}
with
\begin{align}\label{mu-j1}
	\mu_{j,2}(\lambda)&=j\left( \tfrac{1}{2}+\varepsilon^3\mathtt{c}_2(\lambda)\right)-\tfrac{1}{2}\tfrac{j}{|j|}\cdot
\end{align}
The main result of this subsection reads as follows.
\begin{proposition}\label{prop-perp}
	Let $(\tau,\gamma,s_0,S)$ as in \eqref{cond1}. There exists $\varepsilon_0>0$ small enough such that for any $\varepsilon\in(0,\varepsilon_0),$
	there exists a family of linear operators $\big(\mathtt{T}_n\big)_{n\in\mathbb{N}}$ satisfying for any $N\geqslant 0$ and $s\in[s_0,S],$
	$$\sup_{n\in\mathbb{N}}\|\mathtt{T}_nh\|_{s,N}^{\textnormal{Lip}(\gamma)}\lesssim\gamma^{-1}\|h\|_{s,2\tau+N}^{\textnormal{Lip}(\gamma)}$$
	and for $\lambda$ in the Cantor set
	\begin{equation}\label{Cantor second}
		\mathscr{O}_{n}^2(\rho)\triangleq\bigcap_{(l,j)\in\mathbb{Z}^2\atop2\leqslant |j|\leqslant N_n}\Big\{\lambda\in\mathscr{O}\quad\textnormal{s.t.}\quad|\varepsilon^2\omega(\lambda)l+\mu_{j,2}(\lambda)|\geqslant\gamma|j|^{-\tau}\Big\},
	\end{equation}
	we get
	$$\mathbb{L}_{1,\perp}\mathtt{T}_n=\textnormal{Id}+{\mathtt{E}_{n}^{2}},$$
	where $\mathtt{E}_{n}^{2}$ satisfies for any $s\in[s_0,S],$
	\begin{align*}
		\|\mathtt{E}_{n}^{2}h\|_{s_0}^{\textnormal{Lip}(\gamma)}&\lesssim\gamma^{-1}N_n^{s_0-s}\|h\|_{{s,2\tau+1}}^{\textnormal{Lip}(\gamma)}.
	\end{align*}	
\end{proposition}

\begin{proof}
	In view of  \eqref{def bbL1perp}, one may write	\begin{equation}\label{decompostion L1perp}
		\mathbb{L}_{1,\perp}=\mathtt{L}_n+\Pi_{N_n}^\perp\mathcal{D}_{1,\perp},\qquad\mathtt{L}_n\triangleq\varepsilon^2\omega(\lambda)\partial_\varphi+\Pi_{N_n}\mathcal{D}_{1,\perp},
	\end{equation}
	where the projectors $\Pi_{N_n}$ and $\Pi^{\perp}_{N_n}$ are defined  by
	\begin{equation*}
			\Pi_{N_n} h\triangleq\sum_{\underset{\langle l,j\rangle\leqslant N_n}{(l,j)\in\mathbb{Z}^{2}}}h_{l,j}\mathbf{e}_{l,j}\qquad\textnormal{and}\qquad \Pi^{\perp}_{N_n}\triangleq\textnormal{Id}-\Pi_{N_n}.
		\end{equation*}
		Therefore,
	$$\mathtt{L}_n\mathbf{e}_{l,j}=\begin{cases}
		\ii\big(\varepsilon^2\omega(\lambda)\,l+\mu_{j,2}\big)\mathbf{e}_{l,j},& \textnormal{if }2\leqslant |j|\leqslant N_n\textnormal{ and }l\in\mathbb{Z},\\
		\ii\,\varepsilon^2\omega(\lambda)\,l\mathbf{e}_{l,j},& \textnormal{if }|j|>N_n\textnormal{ and }l\in\mathbb{Z}.
	\end{cases}$$
	Define the  operator  $\mathtt{T}_n$ by 
	$$\mathtt{T}_{n}h(\varphi,\theta)\triangleq-\ii\sum_{(l,j)\in\mathbb{Z}^2\atop2\leqslant |j|\leqslant N_n}\tfrac{\chi\left((\varepsilon^2\omega(\lambda)\,l+\mu_{j,2})\gamma^{-1}|j|^{\tau}\right)}{\varepsilon^2\omega(\lambda)\,l+\mu_{j,2}}h_{l,j}\,\mathbf{e}_{l,j}(\varphi,\theta),$$
	where $\chi\in\mathscr{C}^\infty(\mathbb{R},[0,1])$ is an even positive cut-off function  such that 
	\begin{equation*}
		\chi(\xi)=\begin{cases}
			0,& \textnormal{if }|\xi|\leqslant\frac13,\\
			1,&\textnormal{if }|\xi|\geqslant\frac12.
		\end{cases}
	\end{equation*}
	In the Cantor set $\mathscr{O}_{n}^{2}(\rho)$ one has
	\begin{align}\label{LTn}
		\mathtt{L}_n\mathtt{T}_{n}=\hbox{Id}+\Pi_{N_n}^\perp.
	\end{align}
	Arguing in a similar way to \cite[Prop. 5.3]{HHM23} one can show that, 
 for any $N\geqslant 0$,
	\begin{align}\label{Tn-01}
		\sup_{n\in\mathbb{N}}\|\mathtt{T}_nh\|_{s,N}^{\textnormal{Lip}(\gamma)}&\leqslant C\gamma^{-1}\|h\|_{s,2\tau+N}^{\textnormal{Lip}(\gamma)}.
	\end{align}
	Now, from \eqref{decompostion L1perp} and \eqref{LTn} we conclude that  on   the Cantor set $\mathscr{O}_{n,2}$ we have the identity
	\begin{align*}
		\mathbb{L}_{1,\perp}\mathtt{T}_{n}=\hbox{Id}+\mathtt{E}_{n}^{2}, \qquad \mathtt{E}_{n}^{2}\triangleq\Pi_{N_n}^\perp+\Pi_{N_n}^\perp\mathcal{D}_{1,\perp}\mathtt{T}_{n}.
	\end{align*}
	Finally, the estimate on the remainder $\mathtt{E}_{n}^{2}$ follows immediately from \eqref{Tn-01} and \eqref{D-1-LL}.
	This ends the proof of the desired result.
\end{proof}
\subsubsection{Invertibility of $\mathbb{M}$}
The next goal is to revisit   the matrix operator introduced in \eqref{def bbM} and explore its invertibility.  Here is our key result.
\begin{proposition}\label{propo-Hm-1}
	Under the assumptions of Proposition $\ref{prop-perp}$ and the smallness condition
	\begin{equation}\label{hao11}
		\varepsilon^{-1}\gamma +{\varepsilon^{2+{\mu}} \gamma^{-1}\leqslant\varepsilon_0},\qquad  \|\rho\|_{s_0+4\tau+5}^{\textnormal{Lip}(\gamma)} \leqslant 1,
	\end{equation}
	the following holds true. 
	There exists  a family of linear operators $\mathbb{P}_n$ satisfying for any $s\in[s_0,S],$
	\begin{align}\label{estiPnH}
	\|\mathbb{P}_nH\|_{s}^{\textnormal{Lip}(\gamma)}\leqslant C\gamma^{-1}\Big(\|H\|_{s+2\tau}^{\textnormal{Lip}(\gamma)}+\|\rho\|_{s+2\tau}^{\textnormal{Lip}(\gamma)}\|H\|_{s_0+2\tau}^{\textnormal{Lip}(\gamma)}\Big)
	,
	\end{align}
	such that in the Cantor set $\mathscr{O}_{n}^2$,  defined in Proposition $\ref{prop-perp}$,    we have
	$$
	\mathbb{M}\mathbb{P}_n=\textnormal{Id}_{X_{1}^s\times X_{\perp}^s}+\mathbb{E}_{n}^{2},
	$$
	with
	\begin{align}\label{estimation En2}
		\|\mathbb{E}_{n}^{2}H\|_{s_0}^{\textnormal{Lip}(\gamma)}
		&\leqslant C\gamma^{-1}N_n^{s_0-s}\Big( \|H\|_{{s+2\tau+1}}^{\textnormal{Lip}(\gamma)}+\|\rho\|_{s+4\tau+5}^{\textnormal{Lip}(\gamma)}  \|H\|_{{s_0+2\tau}}^{\textnormal{Lip}(\gamma)}\Big).
	\end{align}

\end{proposition}
\begin{proof}
Consider the operator 
	\begin{align*}
		\mathbb{K}_n&=\begin{pmatrix}
			\mathbb{L}_{1,1}^{-1} &0 \\
			0 & \mathtt{T}_n
		\end{pmatrix},
	\end{align*}
	where  the operator $\mathtt{T}_n$ was defined in Proposition \ref{prop-perp}.
	For all $\lambda\in \mathscr{O}_{n}^2$ one has the identity
	\begin{align}\label{id MK}
		\mathbb{M}_1\mathbb{K}_n&=\hbox{Id}_{X_1^s\times X_{\perp}^s}+\begin{pmatrix}
			0&0 \\
			0 & \mathtt{E}_{n}^{2}
		\end{pmatrix}.
	\end{align}
Moreover,  according to Proposition \ref{propo-monod}, Proposition \ref{prop-perp} and \eqref{hao11}, for any $H=(h_1,h_2)\in X_1^s\times X_{\perp}^s$  we have
	\begin{align}\label{est Kn}
		 \|\mathbb{K}_nH\|_{s}^{\textnormal{Lip}(\gamma)} 
		&\leqslant C\gamma^{-1}\|H\|_{s,2\tau}^{\textnormal{Lip}(\gamma)}.
	\end{align}
Inserting \eqref{est sur R1} into \eqref{est Kn} gives, according to  \eqref{hao11} and  $\mu\in(0,1)$,
	\begin{align}\label{estiKnR1}
		\|\mathbb{K}_n\partial_\theta\mathbb{R}_1H\|_{s}^{\textnormal{Lip}(\gamma)} 
		&\leqslant C\varepsilon^{2+\mu}\gamma^{-1}\Big( \|H\|_{{s}}^{\textnormal{Lip}(\gamma)}+\|\rho\|_{s+4\tau+5}^{\textnormal{Lip}(\gamma)}  \|H\|_{{s_0}}^{\textnormal{Lip}(\gamma)}\Big).
	\end{align}
	Consequently,  under the condition  \eqref{hao11},  the operator 
	$$\hbox{Id}_{X_1^{s_0}\times X_{\perp}^{s_0}}+\mathbb{K}_n\partial_\theta\mathbb{R}_1: {X_1^{s_0}\times X_{\perp}^{s_0}}\to {X_1^{s_0}\times X_{\perp}^{s_0}}
	$$ is invertible, with
	$$
	\|(\hbox{Id}_{X_1^s\times X_{\perp}^s}+\mathbb{K}_n\partial_\theta\mathbb{R}_1)^{-1}H\|_{s_0}^{\textnormal{Lip}(\gamma)}\leqslant 2 \|H\|_{s_0}^{\textnormal{Lip}(\gamma)}.
	$$
	As to the invertibility for $s\in[s_0,S]$, one can check by induction from \eqref{estiKnR1}, $\forall m\geqslant 1,$
	that under the smallness condition \eqref{hao11} one has 
	\begin{align*}
		\sum_{m\geqslant 0}\|(\mathbb{K}_n\partial_\theta\mathbb{R}_1)^mH\|_{s}^{\textnormal{Lip}(\gamma)} &\leqslant C \|H\|_{{s}}^{\textnormal{Lip}(\gamma)}+C\|\rho\|_{s+4\tau+5}^{\textnormal{Lip}(\gamma)}  \|H\|_{{s_0}}^{\textnormal{Lip}(\gamma)}.
	\end{align*}
	This implies that
	\begin{align}\label{estiIdKnR1}
		\|(\hbox{Id}_{X_1^s\times X_{\perp}^s}+\mathbb{K}_n\partial_\theta\mathbb{R}_1)^{-1}H\|_{s}^{\textnormal{Lip}(\gamma)} &\leqslant C \|H\|_{{s}}^{\textnormal{Lip}(\gamma)}+C\|\rho\|_{s+4\tau+5}^{\textnormal{Lip}(\gamma)}  \|H\|_{{s_0}}^{\textnormal{Lip}(\gamma)}.
	\end{align}
	Now define the operator
	\begin{align}\label{defPn}
		\mathbb{P}_n=\big(\hbox{Id}_{X_1^s\times X_{\perp}^s}+\mathbb{K}_n\partial_\theta\mathbb{R}_1\big)^{-1}\mathbb{K}_n.
	\end{align}
From \eqref{defPn}, \eqref{estiIdKnR1}, \eqref{est Kn} and \eqref{hao11}, we obtain \eqref{estiPnH}.
		Moreover, in view  of \eqref{id MK} and \eqref{defPn} we deduce that in the Cantor set $ \mathscr{O}_n^2$ one has
	\begin{align}\label{Inv-PP1}
		\mathbb{M}\mathbb{P}_n&=\left(\mathbb{M}_1+\left(\mathbb{M}_1\mathbb{K}_n-\begin{pmatrix}
			0&0 \\
			0 & \mathtt{E}_{n}^{2}
		\end{pmatrix}\right)\partial_\theta\mathbb{R}_1\right)\left(\hbox{Id}_{X_1^s\times X_{\perp}^s}+\mathbb{K}_n\partial_\theta\mathbb{R}_1\right)^{-1}\mathbb{K}_n\\
		\nonumber&=\mathbb{M}_1\mathbb{K}_n-\begin{pmatrix}
			0&0 \\
			0 & \mathtt{E}_{n}^{2}
		\end{pmatrix}\partial_\theta\mathbb{R}_1\mathbb{P}_n\\
		\nonumber&=\hbox{Id}_{X_1^s\times X_{\perp}^s}+\mathbb{E}_{n}^{2},
	\end{align}
	where
	$$
	\,\mathbb{E}_{n}^{2}\triangleq\begin{pmatrix}
				0&0 \\
				0 & \mathtt{E}_{n}^{2}
			\end{pmatrix}-\begin{pmatrix}
				0&0 \\
				0 & \mathtt{E}_{n}^{2}
			\end{pmatrix}\partial_\theta\mathbb{R}_1\mathbb{P}_n.
	$$
	By virtue of Proposition \ref{prop-perp}, \eqref{est sur R1}, \eqref{estiPnH} and  \eqref{hao11} one gets \eqref{estimation En2}.
			Finally, by construction, one has the algebraic properties
	\begin{align}\label{Alg-Iden-Jul}
		\mathbb{M}=\begin{pmatrix}
			\Pi_1&0 \\
			0 & \Pi_1^\perp
		\end{pmatrix} \mathbb{M}\begin{pmatrix}
			\Pi_1&0 \\
			0 & \Pi_1^\perp
		\end{pmatrix}\qquad\hbox{and}\qquad \mathbb{P}_n=\begin{pmatrix}
			\Pi_1&0 \\
			0 & \Pi_1^\perp
		\end{pmatrix} \mathbb{P}_n\begin{pmatrix}
			\Pi_1&0 \\
			0 & \Pi_1^\perp
		\end{pmatrix}.
	\end{align}
	This completes the proof of the desired result.
\end{proof}
\subsubsection{Invertibility of ${\mathcal{L}_1}$ and ${\mathcal{L}_0}$}
In this section we intend to explore the   existence of an approximate right inverse to   the operators ${\mathcal{L}_1}$ and ${\mathcal{L}_0}$ defined in \eqref{def cal L1} and Proposition \ref{prop CVAR}, respectively.
\begin{proposition}\label{prop-inverse}
	Given the conditions \eqref{cond1}, \eqref{choix mu2}. 
	There exists $\epsilon_0>0$ such that under the assumptions
	\begin{equation}\label{hao11MM}
		\varepsilon^{-1}\gamma+\varepsilon^{4+\mu}\gamma^{-1}N_0^{\mu_2}+{\varepsilon^{2+{\mu}} \gamma^{-1}\leqslant\epsilon_0},\qquad  \|\rho\|_{2s_0+2\tau+4+\frac32\mu_2}^{\textnormal{Lip}(\gamma)} \leqslant 1,
	\end{equation}
	the following assertions  holds true. 
	
	\begin{enumerate}
		\item There exists a family of linear  operators $\big({\mathbf{T}}_{n}\big)_{n\in\mathbb{N}}$ satisfying
		\begin{equation*}
			\forall \, s\in\,[ s_0, S],\quad\sup_{n\in\mathbb{N}}\|{\mathbf{T}}_{n}h\|_{s}^{\textnormal{Lip}(\gamma)}\leqslant  C\gamma^{-1}\Big(\|h\|_{s+2\tau}^{\textnormal{Lip}(\gamma)}+\|\rho\|_{s+2\tau}^{\textnormal{Lip}(\gamma)}\|h\|_{s_0+2\tau}^{\textnormal{Lip}(\gamma)}\Big)
		\end{equation*}
		and such that in the Cantor set  $\mathscr{O}_{n}^2$, where $\mathscr{O}_{n}^2$ is introduced in Proposition $\ref{prop-perp}$,
		we have
		$$
		{\mathcal{L}_1}\,{\mathbf{T}}_{n}=\textnormal{Id}+\mathbf{E}_n,
		$$
		where $\mathbf{E}_n$ satisfies the following estimate
		\begin{align*}
			\forall\, s\in [s_0,S],\quad  \|\mathbf{E}_n h\|_{s_0}^{\textnormal{Lip}(\gamma)}
			\nonumber&\leqslant C \gamma^{-1}N_n^{s_0-s}\Big( \|h\|_{s+2\tau+1}^{\textnormal{Lip}(\gamma)}+\| \rho\|_{s+4\tau+5}^{\textnormal{Lip}(\gamma)}\|h\|_{s_{0}+2\tau}^{\textnormal{Lip}(\gamma)} \Big)\\
			&\quad +\varepsilon^6\gamma^{-1}N_{0}^{\mu_{2}}N_{n+1}^{-\mu_{2}}\|h\|_{s_{0}+2\tau+2}^{\textnormal{Lip}(\gamma)}.
		\end{align*}
		
		\item Set 
	$
	\mathcal{T}_n\triangleq\mathscr{B}{\mathbf{T}_n}\mathscr{B}^{-1}.
	$
 Then,
		\begin{equation*}
			\forall \, s\in\,[ s_0, S],\quad\sup_{n\in\mathbb{N}}\|{\mathcal{T}}_{n}h\|_{s}^{\textnormal{Lip}(\gamma)}\leqslant C\gamma^{-1}\Big(\|h\|_{s+2\tau}^{\textnormal{Lip}(\gamma)}+\|\rho\|_{s+{4\tau+5}}^{\textnormal{Lip}(\gamma)}\|h\|_{s_0+2\tau}^{\textnormal{Lip}(\gamma)}\Big).
		\end{equation*}
		Moreover, in the Cantor set $\mathscr{O}_n^{1}(\rho)\cap \mathscr{O}_n^{2}(\rho)$ $($see \eqref{Cantor set0} for $\mathscr{O}_n^{1}(\rho))$
		one has the identity
		$$
		{\mathcal{L}_0}\,{\cal{T}}_{n}=\textnormal{Id}+{\cal{E}}_{n},
		$$
		where $\mathcal{E}_{n}:=\mathscr{B}{\mathbf{E}_n}\mathscr{B}^{-1}$ satisfies the estimate
 \begin{align*}
			\|\mathcal{E}_{n}h\|_{s_0}^{\textnormal{Lip}(\gamma)}
			\nonumber&\leqslant C\gamma^{-1}N_n^{s_0-s}\big( \|h\|_{s+2\tau+1}^{\textnormal{Lip}(\gamma)}+\| \rho\|_{s+4\tau{+6}}^{\textnormal{Lip}(\gamma)}\|h\|_{s_{0}+2\tau}^{\textnormal{Lip}(\gamma)}\big)\\ &\quad +C\varepsilon^6\gamma^{-1}N_{0}^{\mu_{2}}N_{n+1}^{-\mu_{2}}\|h\|_{s_{0}+2\tau+2}^{\textnormal{Lip}(\gamma)}.
		\end{align*}
	\end{enumerate}
\end{proposition}
\begin{proof}
	In view of  \eqref{def bbM}, \eqref{defPn} and \eqref{Inv-PP1} we may write
	\begin{align*}
		\mathbb{M}&=\begin{pmatrix}
			M_1&M_2 \\
			M_3 & M_4
		\end{pmatrix},\qquad\mathbb{P}_n=\begin{pmatrix}
			P_1&P_2 \\
			P_3&P_4
		\end{pmatrix}\qquad\hbox{and}\qquad \mathbb{E}_{n}^{2}=\begin{pmatrix}
			E_{n,1}^2&E_{n,2}^2 \\
			E_{n,3}^2&E_{n,4}^2
		\end{pmatrix}.
	\end{align*}
	It follows from \eqref{def cal L1} that
	$$
	\mathcal{L}_1=\sum_{j=1}^4M_j+\mathtt{E}_{n}.
	$$
	By denoting 
	$$
	{\mathbf{T}_n}= \sum_{i=1}^4P_i \qquad\hbox{and}\qquad {\mathbf{E}_n}=\sum_{i=1}^4 E_{n,i}^2+\mathtt{E}_{n}{\mathbf{T}_n}
	$$
	and using the algebraic structure \eqref{Alg-Iden-Jul} and the identity \eqref{Inv-PP1} we conclude that
		$$
	\mathcal{L}_1{\mathbf{T}_n}=\hbox{Id}+ {\mathbf{E}_n}.
	$$
	The estimates of $\mathbf{T}_n$ and $\mathbf{E}_n$ can be easily achieved  from Proposition \ref{propo-Hm-1} and  Proposition  \ref{prop CVAR2}.
	The second point follows immediately from the fact that $\mathcal{L}_1=\mathscr{B}^{-1}\mathcal{L}_0\mathscr{B},$ Proposition \ref{prop-inverse}-1 and Proposition~\ref{prop CVAR} -1 together with \eqref{hao11MM}.
	This achieves the proof of Proposition \ref{prop-inverse}.\end{proof}

\subsection{Construction of solutions and Cantor set measure}\label{Nash Moser}
To construct solutions to the nonlinear equation 
\begin{equation}\label{main-eq1}
\mathbf{F}(\rho) \triangleq \tfrac{1}{\varepsilon^{2+\mu}} \mathbf{G}( \varepsilon r_\varepsilon + \varepsilon^{\mu+1}\rho) = 0
\end{equation}
 introduced in \eqref{def func F}, we apply a modified Nash-Moser scheme, in the spirit of \cite{HHM23}. The tame estimates on the approximate right inverse of $\mathcal{L}_0=d_\rho \mathbf{F}(\rho)$, obtained in Proposition \ref{prop-inverse} are identical to those found in \cite[Prop. 5.5]{HHM23}. Consequently, the proof of Proposition \ref{Nash-Moser} follows the same lines of \cite[Prop. 6.1]{HHM23}. Therefore, we shall omit the proof and only provide a complete statement of the main results. For this end,
we define the finite-dimensional space  
$$
E_{n}\triangleq\Big\{h: [\lambda_*,\lambda^*]\times\T^2\to\mathbb{R}\quad\textnormal{s.t.}\quad\Pi_{N_n}h=h\Big\},$$
where $\Pi_{N_n}$ is the projector defined by
$$
h(\varphi,\theta)=\sum_{l\in\mathbb{Z}\atop j\in\mathbb{Z}^*}h_{l,j}e^{\ii(l \varphi+j\theta)},\qquad\Pi_{N_n}h(\varphi,\theta)\triangleq\sum_{|l|+|j|\leqslant N_{n}}h_{l,j}e^{\ii(l \varphi+j\theta)},
$$
{and  the sequence  $(N_n)_{n}$ was defined  in \eqref{suite Nn}. Here, we will utilize the parameters introduced in \eqref{cond1}, along with the following additional value. 
\begin{align}\label{def b0}
 \quad \mathtt{b}_0=3-\mu.
\end{align}
We shall also fix the values of  $N_0$ and $\gamma$  as below
				\begin{equation}\label{lambda-choice}
N_{0}\triangleq\varepsilon^{-{\delta}}, \qquad \gamma= \varepsilon^{2+{\delta}}.\end{equation}
 Moreover, we shall  impose the following constraints required along the  Nash-Moser scheme,
\begin{equation}\label{Assump-DRP1}\left\lbrace\begin{array}{rcl}
						1+\tau&<&a_2,\\3s_0+12\tau+15+\tfrac32 a_2&<& a_1,\\
												\frac{2}{3}a_{1}&< & \mu_{2},\\
	0<\delta&<& \min\big(\mu,\tfrac{1-\mu}{a_1+2},{\tfrac{2+\mu}{1+\mu_2}\big)},\\
						12\tau+3+\tfrac{6}{\delta}&<&\mu_1,\\
						\max\left(s_{0}+4\tau+3+\frac{2}{3}\mu_{1}+a_{1} +{\frac4\delta},3s_0+4\tau+6+3\mu_2\right)&< & b_{1}.
					\end{array}\right.
				\end{equation}
We have the following result.					
\begin{proposition}[Nash-Moser scheme]\label{Nash-Moser}
Given the conditions \eqref{def b0}, \eqref{lambda-choice} and \eqref{Assump-DRP1}. There exist $C_{\ast}>0$ and ${\varepsilon}_0>0$ such that for any $\varepsilon\in[0,\varepsilon_0]$   we get  for all $n\in\mathbb{N}$ the following properties,
\begin{itemize}
\item  $(\mathcal{P}1)_{n}$ There exists a Lipschitz function 
$$\rho_{n}:\begin{array}[t]{rcl}
[\lambda_*,\lambda^*] & \rightarrow &  E_{n-1}\\
\lambda & \mapsto & \rho_n
\end{array}$$
satisfying 
$$
\rho_{0}=0\quad\mbox{ and }\quad\,\| \rho_{n}\|_{{2s_0+2\tau+3}}^{\textnormal{Lip}(\gamma)}\leqslant C_{\ast}\varepsilon^{\mathtt{b}_0}\gamma^{-1}\quad \hbox{for}\quad n\geqslant1.
$$
By setting 
$$
\quad {u}_{n} =\rho_{n}-\rho_{n-1} \quad \hbox{for}\quad n\geqslant1,
$$
 we have 
$$
\forall s\in[s_0,S], \,\| {u}_{1}\|_{s}^{{\textnormal{Lip}(\gamma)}}\leqslant\tfrac12 C_{\ast}\varepsilon^{\mathtt{b}_0}\gamma^{-1}\quad\hbox{and}\quad \| {u}_{k}\|_{{{2s_0+2\tau+3}}}^{{\textnormal{Lip}(\gamma)}}\leqslant C_{\ast}\varepsilon^{\mathtt{b}_0}\gamma^{-1}N_{k-1}^{-a_{2}}\quad  \forall\,\, 2\leqslant k\leqslant n.
$$
\item $(\mathcal{P}2)_{n}$ Set 
\begin{align}\label{def set Am}
\mathcal{A}_{0}=[\lambda_*,\lambda^*]\quad \mbox{ and }\quad \mathcal{A}_{n+1}=\mathcal{A}_{n}\cap\mathcal{O}_{n}^1(\rho_n)\cap\mathcal{O}_{n}^2(\rho_n), \quad\forall n\in\mathbb{N}.
\end{align}
Then we have the following estimate 
$$\|\mathbf{F}(\rho_{n})\|_{s_{0},\mathcal{A}_{n}}^{{\textnormal{Lip}(\gamma)}}\leqslant C_{\ast}\varepsilon^{\mathtt{b}_0} N_{n-1}^{-a_{1}}.
$$
\item $(\mathcal{P}3)_{n}$ High regularity estimate: $\| \rho_{n}\|_{b_1}^{{\textnormal{Lip}(\gamma)}}\leqslant C_{\ast}\varepsilon^{\mathtt{b}_0}\gamma^{-1} N_{n-1}^{\mu_1}.$ Here, we have used  the notation \eqref{Norm-not}.
\end{itemize}

\end{proposition}

Next, we shall study the convergence of Nash-Moser  scheme, stated in Proposition \ref{Nash-Moser}, show that the limit is a solution to the problem \eqref{main-eq1} and establish a lower bound for the Lebesgue measure of the final Cantor. 
\begin{coro}\label{prop-construction}
There exists $\lambda\in[\lambda_*,\lambda^*]\mapsto \rho_\infty$ satisfying
$$\|\rho_\infty\|_{2s_{0}+2\tau+3}^{\textnormal{Lip}(\gamma)}\leqslant C_{\ast}\varepsilon^{\mathtt{b}_0}\gamma^{-1}\quad \hbox{and}\quad \|\rho_\infty-\rho_m\|_{2s_{0}+2\tau+3}^{\textnormal{Lip}(\gamma)}\leqslant  C_{\ast}\varepsilon^{\mathtt{b}_0}\gamma^{-1}N_{m}^{-a_{2}}
$$
such that  
\begin{align}\label{solution}\forall \lambda\in \mathtt{C}_{\infty}\triangleq\bigcap_{m\in\mathbb{N}}\mathcal{A}_{m}, \qquad \mathbf{F}\big(\rho_{\infty}(\lambda)\big)=0.
\end{align}
\end{coro}
\begin{proof}
In view of Proposition \ref{Nash-Moser} one has for all $m\geqslant1,$
$$
\rho_m=\sum_{n=1}^m u_n.
$$
Define the formal infinite sum
$$
\rho_\infty\triangleq\,\sum_{n=1}^\infty u_n.
$$
Then the claimed estimates follow immediately from the estimates of $(\mathcal{P}1)_{m}$, \eqref{suite Nn} and \eqref{lambda-choice}. 
Thus, the sequence $(\rho_m)_{m\geqslant1}$ converges pointwise to $\rho_\infty$. In view of  $(\mathcal{P}2)_{m}$ we obtain
$$\|\mathbf{F}(\rho_{m})\|_{s_{0},\mathtt{C}_\infty}^{{\textnormal{Lip}(\gamma)}}\leqslant C_{\ast}\varepsilon^{\mathtt{b}_0} N_{m-1}^{-a_{1}}.
$$
Passing to the limit leads to 
$$\forall \lambda\in \mathtt{C}_{\infty}, \quad \mathbf{F}\big(\rho_{\infty}(\lambda)\big)=0.
$$
This concludes the proof of Corollary \ref{prop-construction}.
\end{proof}

The following result is proved using standard argument, in similar way to \cite[Lem. 6.3]{HHM23}.
\begin{lemma}\label{Lem-Cantor-measu}
Under \eqref{Assump-DRP1},  we have the inclusion
\begin{align}\label{inclusion}
\mathcal{C}_{\infty}^{1}\cap \mathcal{C}_{\infty}^{2}\subset \mathtt{C}_{\infty},
\end{align}
where, for $k\in\{1,2\},$
$$\mathcal{C}_{\infty}^{k}\triangleq\bigcap_{l\in\mathbb{Z}\atop|j|\geqslant k}\Big\{\lambda\in[\lambda_*,\lambda^*]\quad\textnormal{s.t.}\quad|\varepsilon^2\omega(\lambda)l+\mu_{j,k}(\lambda,\rho_\infty)|\geqslant\tfrac{2\gamma}{|j|^\tau}\Big\}$$
and
\begin{align*}
					\nonumber \mu_{j,k}(\lambda,\rho)&= j\mathtt{c}(\lambda,\rho)-\tfrac{k-1}{2}\tfrac{j}{|j|}
					\cdot
								\end{align*}
In addition, there exists $\varepsilon_0>0$ and $C>0$ such that for any $\varepsilon\in(0,\varepsilon_0)$ 
$$|[\lambda_*,\lambda^*]\setminus  \mathtt{C}_{\infty}|\leqslant C\varepsilon^{\delta}.
$$
\end{lemma}
\begin{proof}
In view  of \eqref{Cantor set0}, \eqref{Cantor second}, \eqref{def set Am} and \eqref{solution} we have
\begin{align*}
\mathtt{C}_{\infty}=\mathcal{O}_{\infty,1}\cap\mathcal{O}_{\infty,2},
\end{align*}
with 
$$\mathcal{O}_{\infty,k}=\bigcap_{(l,j,n)\in\mathbb{Z}^2\times\mathbb{N}\atop k\leqslant |j|\leqslant N_n}\Big\{\lambda\in[\lambda_*,\lambda^*]\quad\textnormal{s.t.}\quad|\varepsilon^2\omega(\lambda)\,l+\mu_{j,k}(\lambda,\rho_n)|\geqslant\tfrac{\gamma}{|j|^\tau}\Big\}.$$
We shall prove that for $k\in\{1,2\},$
\begin{equation}\label{inclusion C in O}
\mathcal{C}_{\infty}^{k}\subset \mathcal{O}_{\infty,k},
\end{equation}
which implies the inclusion \eqref{inclusion}. Let $\lambda\in \mathcal{C}_{\infty}^{k}$ and $n\in\mathbb{N}$. Then, from the triangle inequality, we obtain for all $k\leqslant |j|\leqslant N_n,$  
\begin{align}\label{triangle inq}
\nonumber |\varepsilon^2 \omega(\lambda)\,l+\mu_{j,k}(\lambda,\rho_n)|&\geqslant |\varepsilon^2 \omega(\lambda)\,l+\mu_{j,k}(\lambda,\rho_\infty)|-|\mu_{j,k}(\lambda,\rho_\infty)-\mu_{j,k}(\lambda,\rho_n)|\\
&\geqslant\tfrac{2\gamma}{|j|^\tau}-|\mu_{j,k}(\lambda,\rho_\infty)-\mu_{j,k}(\lambda,\rho_n)|.
\end{align}
Using \eqref{diff ttc}, \eqref{lambda-choice}  and Corollary \ref{prop-construction} we get for $\varepsilon$ small enough 
\begin{align*}
|\mu_{j,k}(\lambda,\rho_\infty)-\mu_{j,k}(\lambda,\rho_n)|&
\leqslant C\varepsilon^{3}|j|\|\rho_{\infty}-\rho_n\|_{2s_{0}+2\tau+3}^{\textnormal{Lip}(\gamma)}\\
&  \leqslant C_{\ast} \varepsilon^{3+\mathtt{b}_0}\gamma^{-1}|j|N_n^{-a_2}\\
&\leqslant  C_{\ast}\gamma|j|^{-\tau}\varepsilon^{\mathtt{b}_0-1-2\delta}N_n^{1+\tau-a_2}.
\end{align*}
From \eqref{def b0} and the assumption $1+\tau<a_2$ and $\delta<\frac{2-\mu}{2}$ (see \eqref{Assump-DRP1}), we conclude that   for small $\varepsilon$ we have
\begin{align*}
|\mu_{j,k}(\lambda,\rho_\infty)-\mu_{j,k}(\lambda,\rho_n)|&
\leqslant  C_{\ast}\gamma|j|^{-\tau}\varepsilon^{2-\mu-2\delta}N_n^{1+\tau-a_2}\\
&\leqslant  \gamma|j|^{-\tau}.
\end{align*} 
Combining the last estimate with \eqref{triangle inq} we get, for all $n\in\mathbb{N}$,  $l\in\mathbb{Z}$ and $k\leqslant |j|\leqslant N_n$ we get
\begin{align*}
|\varepsilon^2 \omega(\lambda)\,l+\mu_{j,k}(\lambda,\rho_n)|&\geqslant \tfrac{\gamma}{|j|^\tau}\cdot
\end{align*}
Hence, we deduce that $\lambda\in\mathcal{O}_{\infty,k}$, concluding the proof of \eqref{inclusion C in O}. Consequently, one has 
$$
\big|[\lambda_*,\lambda^*]\setminus\mathtt{C}_{\infty}\big|\leqslant \big|[\lambda_*,\lambda^*]\setminus  \mathcal{C}_{\infty}^{1}\big|+\big|[\lambda_*,\lambda^*]\setminus \mathcal{C}_{\infty}^{2}\big|.
$$ 
Thus,  the measure estimate we only need to prove that
\begin{align}\label{measure estimate Ck}
  \big|[\lambda_*,\lambda^*]\setminus \mathcal{C}_{\infty}^{k}\big|&\leqslant  C \varepsilon^{\delta}.
  \end{align}
 For this aim we write
$$\mathcal{C}_{\infty}^{k}=\bigcap_{n\in\mathbb{N}}\bigcap_{(l,j)\in\mathbb{Z}^{2}\atop k\leqslant |j|\leqslant N_{n}}\mathcal{U}_{l,j}^k,
$$
with
$$\mathcal{U}_{l,j}^k\triangleq\left\lbrace\lambda\in [\lambda_*,\lambda^*]\quad\textnormal{s.t.}\quad\big|\varepsilon^2\omega(\lambda)\,l+\mu_{j,k}(\lambda,\rho_\infty)\big|\geqslant2\tfrac{\varepsilon^{2+\delta}}{|j|^{\tau}}\right\rbrace,$$
having used the fact that $\gamma=\varepsilon^{2+\delta}.$ Notice that
$$\bigcap_{(l,j)\in\mathbb{Z}^{2}\atop k\leqslant |j|}\mathcal{U}_{l,j}^k\subset \mathcal{C}_{\infty}^{k}.
$$
Moreover, by continuity and non-degeneracy assumption on the period \ref{non degeneracy period thm}, we infer
$$
\forall \lambda\in[\lambda_*,\lambda^*],\quad 0<\overline{b}\leqslant \omega(\lambda)\leqslant\overline{a},\qquad\overline{a}\triangleq\max_{\lambda\in[\lambda_{*},\lambda^{*}]}\omega_{0}(\lambda),\qquad\overline{b}\triangleq\min_{\lambda\in[\lambda_{*},\lambda^{*}]}\omega_{0}(\lambda).
$$
Moreover, from \eqref{def ttc2}, for  all $0<\varepsilon\leqslant \varepsilon_0$ one has  
\begin{align}\label{est-c_0}
\tfrac13\leqslant \mathtt{c}(\lambda,\rho_\infty)\leqslant \tfrac23\cdot
\end{align}
We distinguish different cases.\\
\\
$\bullet$ Case $\varepsilon^2 \overline{b}|l|\geqslant |j|\geqslant k$: 
 For $k=1$, from the triangle inequality, we get 
\begin{align*}
 \big|\varepsilon^2\omega(\lambda)  \ell+j \mathtt{c}(\lambda,\rho_\infty)\big|&\geqslant \varepsilon^2\overline{b}|l|-\tfrac23|j|\\
 &\geqslant\tfrac13|j|\\
 &\geqslant \tfrac{\varepsilon^{2+\delta}}{| j|^{\tau}}\cdot
 \end{align*}
Similarly, for $k=2$, we have
\begin{align*}
 \big|\varepsilon^2\omega(\lambda)\,l+j\mathtt{c}(\lambda,\rho_\infty)-\tfrac{j}{2|j|}\big|&\geqslant\varepsilon^2\overline{b}|l|-\tfrac23|j|-\tfrac12\\
 &\geqslant\tfrac13|j|-\tfrac12\geqslant\tfrac16\\
 &\geqslant\tfrac{\varepsilon^{2+\delta}}{|j|^{\tau}}\cdot
 \end{align*}
 Therefore, for $k\in\{1,2\},$
 $$\mathcal{U}_{l,j}^k=[\lambda_*,\lambda^*].$$
 $\bullet$ Case $|j|\geqslant 24\varepsilon^2 \overline{a}|l|$: For $k=1$, one gets by the triangle inequality and \eqref{est-c_0}, together with $|j|\geqslant 1$,
 \begin{align*}
 \big|\varepsilon^2\omega(\lambda)\,l+j\mathtt{c}(\lambda,\rho_\infty)\big|&\geqslant\tfrac13|j|-\varepsilon^2\overline{a}|l|\\
 &\geqslant\big(\tfrac{1}{24}|j|-\varepsilon^2\overline{a}|l|\big) +\tfrac{7}{24}|j|\\
 &\geqslant\tfrac{7}{24}\cdot
 \end{align*}
  Thus, for small $\varepsilon$ we infer
  \begin{align*}
 \big|\varepsilon^2\omega(\lambda)\,l+j\mathtt{c}(\lambda,\rho_\infty)\big|
 &\geqslant\tfrac{\varepsilon^{2+\delta}}{|j|^{\tau}}\cdot
 \end{align*}
  For  $k=2$, one obtains by the triangle inequality and \eqref{est-c_0}, together with $|j|\geqslant 2$,
 \begin{align*}
 \big|\varepsilon^2\omega(\lambda)\,l+j\mathtt{c}(\lambda,\rho_\infty)-\tfrac{j}{2|j|}\big|&\geqslant\tfrac13|j|-\varepsilon^2\overline{a}|l|-\tfrac12\\
 &\geqslant\big(\tfrac{1}{24}|j|-\varepsilon^2\overline{a}|l|\big)+\tfrac{7}{24}|j|-\tfrac12\\
 &\geqslant\tfrac{1}{12}\cdot
 \end{align*}
 Then, for small $\varepsilon$ we infer
 \begin{align*}
 \big|\varepsilon^2\omega(\lambda)\,l+j\mathtt{c}(\lambda,\rho_\infty)-\tfrac{j}{2|j|}\big|&\geqslant\tfrac{\varepsilon^{2+\delta}}{| j|^{\tau}}\cdot 
 \end{align*}
 Hence, for $k\in\{1,2\},$
 $$\mathcal{U}_{l,j}^k=[\lambda_*,\lambda^*].$$ Consequently, one has 
 \begin{align}\label{Cinfty}\bigcap_{|j|\leqslant 24\varepsilon^2 \overline{a}|l|\atop \varepsilon^2 \overline{b}|l|\leqslant |j|}\mathcal{U}_{l,j}^k\subset \mathcal{C}_{\infty}^{k}.
 \end{align}
Set
 $$f_{l,j}(\lambda)\triangleq\varepsilon^2\omega(\lambda)\,l+j\mathtt{c}(\lambda,\rho_\infty)-(k-1)\tfrac{j}{2|j|}\cdot$$
 Differentiating $f_{l,j}$ with respect to  $\lambda$ and using \eqref{def ttc2} give
 $$f_{l,j}^\prime(\lambda)=\varepsilon^2\omega^\prime(\lambda)\,l+\varepsilon^{3}j\mathtt{c}_2^\prime(\lambda).$$
By non-degeneracy assumption and from the estimate \eqref{c2-estimate}, we conclude that
 $$\forall\lambda\in[\lambda_*,\lambda^*],\quad0<\underline{c}\leqslant|\omega^\prime(\lambda)|\quad\hbox{and}\quad |\mathtt{c}_2^\prime(\lambda)|\leqslant c_1.
 $$ 
Since $|j|\leqslant24\varepsilon^2\overline{a}|l|,$ then for all $\lambda\in[\lambda_*,\lambda^*],$ we have
   \begin{align*}
   |f_{j,l}^\prime(\lambda)|&\geqslant\underline{c}\varepsilon^2|l|-\varepsilon^{3}|j|c_1.
  \\ &\geqslant |j|\big(\tfrac{\underline{c}}{24\overline{a}}-c_1\varepsilon^{3}\big).
  \end{align*}
Thus, for  $\varepsilon$ small enough, 
   \begin{align*}
  |f_{j,l}^\prime(\lambda)|\geqslant\tfrac{\underline{c}}{48\overline{a}}|j|.
   \end{align*}
  Applying Lemma \ref{Piralt}, we deduce that
  \begin{align*}
  \left|\left\lbrace\lambda\in[\lambda_*,\lambda^*]\quad\textnormal{s.t.}\quad\big|\varepsilon^2\omega(\lambda)\,l+j\mathtt{c}(\lambda,\rho_\infty)-(k-1)\tfrac{j}{2|j|}\big|<2\tfrac{\varepsilon^{2+\delta}}{|j|^{\tau}}\right\rbrace\right|\leqslant C\tfrac{\varepsilon^{2+\delta}}{|j|^{1+\tau}}\cdot
  \end{align*}
  Consequently,  from \eqref{Cinfty}, we get
  \begin{align*}
  \big|[\lambda_*,\lambda^*]\setminus \mathcal{C}_{\infty}^{k}\big|&\leqslant C \sum_{ \varepsilon^2 \overline{b}|\ell|\leqslant |j|}\tfrac{\varepsilon^{2+\delta}}{|j|^{1+\tau}} \\
  &\leqslant C \varepsilon^{\delta} \sum_{j\in\mathbb{Z}^*}\tfrac{1}{|j|^{\tau}} \\
  &\leqslant C \varepsilon^{\delta},
  \end{align*}
proving \eqref{measure estimate Ck}.
It is important to mention that we are slightly stretching the argument here, as the function $f_{j,l}$  is Lipschitz continuous rather than $C^1$. Nevertheless, we can rigorously support the preceding argument by working with the Lipschitz norm. This leads to the proof of the desired result.
\end{proof}

We finally  recall the following classical result concerning bi-Lipschitz functions and measure theory.
\begin{lemma}\label{Piralt}
	Let  $(\alpha,\beta)\in(\mathbb{R}_{+}^{*})^{2}.$ Consider $f:[\lambda_*,\lambda^*]\to \mathbb{R}$ a bi-Lipschitz function such that 
	$$\forall(\lambda_1,\lambda_2)\in[\lambda_*,\lambda^*]^2,\quad |f(\lambda_1)-f(\lambda_2)|\geqslant\beta|\lambda_1-\lambda_2|.$$
	Then, there exists $C>0$ independent of $\|f\|_{\textnormal{Lip}}$ such that 
	$$\Big|\big\lbrace \lambda\in[\lambda_*,\lambda^*]\quad\textnormal{s.t.}\quad|f(\lambda)|\leqslant\alpha\big\rbrace\Big|\leqslant C\tfrac{\alpha}{\beta}\cdot$$
\end{lemma}

\appendix
\section{Symplectic changes of coordinates and integral operators}
This section's major objective is to highlight useful  findings related to some change of coordinates system.  We refer the reader to the papers \cite{BM21,BFM24,FGMP19,HR21} for the proofs.
Let $\beta: \mathscr{O}\times \mathbb{T}^{2}\to \mathbb{R}$ be a smooth function such that $\displaystyle\sup_{\lambda\in \mathscr{O}}\|\beta(\lambda,\cdot,\centerdot)\|_{\textnormal{Lip}}<1$ 
then there exists $\widehat\beta: \mathscr{O}\times \mathbb{T}^{2}\to \mathbb{R}$ smooth 
such that
\begin{equation}\label{def betahat}
	y=\theta+\beta(\lambda,\varphi,\theta)\Longleftrightarrow \theta=y+\widehat\beta(\lambda,\varphi,y).
\end{equation}
Define the operators
\begin{equation}\label{def Symp Chan Var}
	\mathscr{B}=(1+\partial_{\theta}\beta)\mathrm{B}, \qquad \mathrm{B}h(\lambda,\varphi,\theta)=h\big(\lambda,\varphi,\theta+\beta(\lambda,\varphi,\theta)\big).
\end{equation}
Then
\begin{equation}\label{mathscrB}
	\mathscr{B}^{-1}=(1+\partial_{\theta}\widehat\beta)\mathrm{B}^{-1},\qquad\mathrm{B}^{-1}h(\lambda,\varphi,y)=h\big(\lambda,\varphi,y+\widehat{\beta}(\lambda,\varphi,y)\big).
\end{equation}
Now, we will provide some basic algebraic properties for the aforementioned operators.
\begin{lemma}\label{lem tranfo CVAR}
	The following assertions hold true.
	\begin{enumerate}
		\item Let $\mathscr{B}_1,\mathscr{B}_2$  be two periodic  change of variables as  in \eqref{def Symp Chan Var}, then
		$$\mathscr{B}_{{1}}\mathscr{B}_2=(1+\partial_{\theta}\beta)\mathrm{B},$$
		with
		$$
		\beta(\varphi,\theta)\triangleq \beta_1(\varphi,\theta)+\beta_2\big(\varphi,\theta+\beta_1(\varphi,\theta)\big).
		$$					\item
		The conjugation of the transport operator by $\mathscr{B}$  keeps the same structure
		$$
		\mathscr{B}^{-1}\Big(\omega\cdot\partial_\varphi+\partial_\theta\big(V(\varphi,\theta)\cdot\big)\Big)\mathscr{B}=\omega\cdot\partial_\varphi+\partial_y\big(\mathscr{V}(\varphi,y)\cdot\big),
		$$
		with
		$$\mathscr{V}(\varphi,y)\triangleq\,\mathrm{B}^{-1}\Big(\omega\cdot\partial_{\varphi} \beta(\varphi,\theta)+V(\varphi,\theta)\big(1+\partial_\theta \beta(\varphi,\theta)\big)\Big).$$
	\end{enumerate}
\end{lemma}
In what follows, and in the rest of this appendix, we assume that $(\lambda,s,s_0)$ satisfy \eqref{cond1} and we consider  $\beta\in \textnormal{Lip}_{\gamma}(\mathscr{O},H^{s}(\mathbb{T}^{2})) $ satisfying the smallness condition 
\begin{equation*}
	\|\beta \|_{2s_0}^{\textnormal{Lip}(\gamma)}\leqslant \varepsilon_0,
\end{equation*}
with $\varepsilon_{0}$ small enough.
The following result can be found in \cite{FGMP19,BFM24}.
\begin{lemma}\label{lemma estimates CVAR}
	The following assertions hold true.
	\begin{enumerate}
		\item The linear operators $\mathrm{B},\mathscr{B}:\textnormal{Lip}_{\gamma}(\mathscr{O},H^{s}(\mathbb{T}^{2}))\to \textnormal{Lip}_{\gamma}(\mathscr{O},H^{s}(\mathbb{T}^{2}))$ are continuous and invertible, with 
		\begin{align}
			\|\mathrm{B}^{\pm1}h\|_{s}^{\textnormal{Lip}(\gamma)}\leqslant \|h\|_{s}^{\textnormal{Lip}(\gamma)}\left(1+C\|\beta\|_{s_{0}}^{\textnormal{Lip}(\gamma)}\right)+C\|\beta\|_{s}^{\textnormal{Lip}(\gamma)}\|h\|_{s_{0}}^{\textnormal{Lip}(\gamma)}, \label{tame comp}
			\\
			\|\mathscr{B}^{\pm1}h\|_{s}^{\textnormal{Lip}(\gamma)}\leqslant \|h\|_{s}^{\textnormal{Lip}(\gamma)}\left(1+C\|\beta\|_{s_{0}}^{\textnormal{Lip}(\gamma)}\right)+C\|\beta\|_{s+1}^{\textnormal{Lip}(\gamma)}\|h\|_{s_{0}}^{\textnormal{Lip}(\gamma)}. \label{tame comp symp}
		\end{align}
		\item The functions $\beta$ and $\widehat{\beta}$ defined through \eqref{def betahat} satisfy the estimates
		\begin{equation*}
			\|\widehat{\beta}\|_{s}^{\textnormal{Lip}(\gamma)}\leqslant C\|\beta\|_{s}^{\textnormal{Lip}(\gamma)}.
		\end{equation*}
	\end{enumerate}
\end{lemma}
%
 The next result deals with some integral  operator estimates. For the proof, we refer to  \cite[Lem. 4.4]{HR21} and 
%
 \cite[Lem. 2.3]{BM20}.			
\begin{lemma}\label{lem CVAR kernel}
	Consider a smooth real-valued kernel
	$$K:(\lambda,\varphi,\theta,\eta)\mapsto K(\lambda,\varphi,\theta,\eta)$$
	and let   
	\begin{equation}\label{Top-op1}(\mathcal{T}_Kh)(\lambda,\varphi,\theta)\triangleq\int_{\mathbb{T}}K(\lambda,\varphi,\theta,\eta)h(\lambda,\varphi,\eta)d\eta,
\end{equation} be an integral operator. 
Then, for any $s_1,s_2\geqslant0,$
	\begin{align*}
		\| \mathcal{T}_Kh\|_{s_1,s_2}^{\textnormal{Lip}(\gamma)}&\lesssim \|h\|_{s_0}^{\textnormal{Lip}(\gamma)}\|K\|_{s_1+s_2}^{\textnormal{Lip}(\gamma)} +\|h\|_{s_1}^{{\textnormal{Lip}(\gamma)}} \|K\|_{s_0+s_2}^{\textnormal{Lip}(\gamma)}.
	\end{align*}
Given the periodic change of variables $\mathscr{B}$ defined by \eqref{def Symp Chan Var}.
	Then, the operators $\mathscr{B}^{-1}\mathcal{T}_K\mathscr{B}$ and $\mathscr{B}^{-1}\mathcal{T}_{K}\mathscr{B}-\mathcal{T}_{K}$ are  integral operators, 		
	\begin{align*}\big(\mathscr{B}^{-1}\mathcal{T}_{K}\mathscr{B}\big)h(\lambda,\varphi,\theta)&=\int_{\mathbb{T}}h(\lambda,\varphi,{\eta})\widehat{K}(\lambda,\varphi,\theta,{\eta})d{\eta}\\
		\big(\mathscr{B}^{-1}\mathcal{T}_{K}\mathscr{B}-\mathcal{T}_{K}\big)h(\lambda,\varphi,\theta)&=\int_{\mathbb{T}}h(\lambda,\varphi,{\eta})\widetilde{K}(\lambda,\varphi,\theta,{\eta})d{\eta}
	\end{align*}
	with
	\begin{align*}
		\|\widehat{K}\|_{s}^{\textnormal{Lip}(\gamma)}&\lesssim \|K\|_{s}^{\textnormal{Lip}(\gamma)}+\|K\|_{s_0}^{\textnormal{Lip}(\gamma)}\|\beta\|_{s+1}^{\textnormal{Lip}(\gamma)},\\
		\|\widetilde{K}\|_{s}^{\textnormal{Lip}(\gamma)}&\lesssim\|K\|_{s+1}^{\textnormal{Lip}(\gamma)}\|\beta\|_{s_0}^{\textnormal{Lip}(\gamma)}+\|K\|_{s_0}^{\textnormal{Lip}(\gamma)}\|\beta\|_{s+1}^{\textnormal{Lip}(\gamma)}.
	\end{align*}
\end{lemma}

	\textbf{Acknoledgments :} The work of Z. Hassainia is supported by Tamkeen under the NYU Abu Dhabi Research Institute grant of the center SITE. T. Hmidi has been supported by Tamkeen under the NYU Abu Dhabi Research Institute grant. The work of E. Roulley is supported by PRIN 2020XB3EFL, "Hamiltonain and Dispersive PDEs".\\

			\vspace{1cm}
			\noindent\textbf{Zineb Hassainia}\\
			Department of Mathematics\\
			New York University in Abu Dhabi\\
			Saadiyat Island, P.O. Box 129188, Abu Dhabi, UAE\\
			Email: zh14@nyu.edu\\
				
			\noindent\textbf{Taoufik Hmidi}\\
			Department of Mathematics\\
			New York University in Abu Dhabi\\
			Saadiyat Island, P.O. Box 129188, Abu Dhabi, UAE\\
			Email: th2644@nyu.edu\\
			
			\noindent\textbf{Emeric Roulley}\\
			SISSA\\
			Via Bonomea 265, 34136, Trieste, Italy\\
			Email : eroulley@sissa.it
\end{document}